\documentclass[twoside,openright]{amsbook}
\usepackage[english]{babel}
\usepackage{indentfirst}
\usepackage[T1]{fontenc}
\usepackage[utf8]{inputenc}
\usepackage{paralist}
\usepackage{graphicx}
\usepackage{amssymb}
\usepackage{amsthm}
\usepackage{amsmath}
\usepackage{amscd}
\usepackage{mathabx} 

\usepackage[colorlinks=true]{hyperref}
\newtheorem{thm}{Theorem}[section] 
\newtheorem{cor}[thm]{Corolary} 
\newtheorem{lemma}[thm]{Lemma}
\newtheorem{prop}[thm]{Proposition}
\newtheorem{defn}[thm]{Definition}
\theoremstyle{remark}

\numberwithin{section}{chapter}
\numberwithin{equation}{section} 
\numberwithin{figure}{section}
\usepackage[figurewithin=none]{caption}
  
\newcommand{\Z}{\ensuremath{\mathbb{Z}}}
\newcommand{\R}{\ensuremath{\mathbb{R}}}
\newcommand{\N}{\ensuremath{\mathbb{N}}}
\newcommand{\Rn}{\ensuremath{\mathbb{R}^n}}

\newcommand{\frlap}{\ensuremath{(-\Delta)^s}}
\newcommand{\eee}{\ensuremath{\varepsilon}}
\newcommand{\E}{\ensuremath{\mathcal{E}}}
\newcommand{\C}{\ensuremath{\mathcal{C}}}
\newcommand{\Co}{\ensuremath{\mathbb{C}}}
\newcommand{\K}{\ensuremath{\mathcal{K}}}
\newcommand{\bgs}[1]{\begin{equation*} \begin{aligned} #1 \end{aligned} \end{equation*}}
\newcommand{\eqlab}[1]{\begin{equation}  \begin{aligned}#1 \end{aligned}\end{equation}}
\newcommand{\al}{\ensuremath{&\;}}
\newcommand{\Lm}{\ensuremath{\mathcal{L}}}

\begin{document}

\title[Nonlocal diffusion and applications]{Nonlocal diffusion and applications}
\author{Claudia Bucur, Enrico Valdinoci}
\email{c.bucur@unimelb.edu.au}
\email{enrico.valdinoci@wias-berlin.de}
\thanks{It is a pleasure to thank Serena Dipierro, Rupert Frank, 
Richard Mathar, Alexander Nazarov, Joaquim Serra 
and Fernando Soria
for very interesting and pleasant discussions.
We are also indebted with all the 
participants of the seminars and minicourses
from which this set of notes generated
for the nice feedback received, and we hope that this work, 
though somehow sketchy and informal, can be useful to stimulate new 
discussions and further develop this rich and interesting subject.}

\date{}

\begin{abstract}
We consider the fractional Laplace framework and provide models and theorems related to nonlocal diffusion phenomena. Some applications are presented, including: a simple probabilistic interpretation, water waves, crystal dislocations, nonlocal phase transitions, nonlocal minimal surfaces and Schr\"{o}dinger equations. Furthermore, an example of an $s$-harmonic function, the harmonic extension and some insight on a fractional version of a classical
conjecture formulated by De Giorgi
are presented. Although this book aims at gathering some 
introductory material on the applications of the fractional Laplacian, some proofs and results are original. Also, the work is self contained, and the reader is invited to consult the rich bibliography for further details, whenever a subject is of interest.  
\end{abstract}

\maketitle

\tableofcontents

\chapter{Introduction}

In the recent years the fractional Laplace operator has received
much attention both in pure and in applied
mathematics.

The purpose of these pages is to collect a set of notes that are a result of 
several talks and minicourses delivered here and there in the world 
(Milan, Cortona, Pisa, Roma, Santiago del Chile, Madrid, Bologna, 
Porquerolles, Catania to name a few). 
We will present here some mathematical models related to nonlocal
equations, providing some introductory material and examples. 

Starting from the basics of the nonlocal equations, we will discuss in
detail some recent developments in four topics of research on which we focused our
attention, namely:
\begin{itemize}
\item a problem arising in crystal dislocation (which is related to a
classical model introduced by Peierls and Nabarro),
\item a problem arising in phase transitions
(which is related to
a nonlocal version of the classical
Allen–Cahn equation),
\item the limit interfaces arising in the above
nonlocal phase transitions
(which turn out
to be nonlocal minimal surfaces,
as introduced by
Caffarelli, Roquejoffre and Savin), and
\item a nonlocal version of the Schr\"odinger equation
for standing waves (as introduced by Laskin).
\end{itemize}

This set of notes is organized as follows.
To start with, in Chapter~\ref{S:1}, we will give a motivation
for the fractional Laplacian (which is the typical nonlocal operator
for our framework), that
originates from probabilistic considerations.
As a matter of fact, no advanced knowledge of probability theory
is assumed from the reader, and the topic is dealt with
at an elementary level.

In Chapter~\ref{S:2}, we will recall some basic properties of
the fractional Laplacian, discuss some explicit examples in detail
and point out some structural inequalities, that are due to a fractional
comparison principle. This part continues with a quite surprising
result, which states that every function can be locally approximated
by functions with vanishing fractional Laplacian
(in sharp contrast with the rigidity of the classical harmonic functions). We also give an example of a function with constant fractional Laplacian on the ball.

In Chapter~\ref{S:3} we deal with extended problems.
It is indeed a quite remarkable fact that
in many occasions nonlocal operators can be equivalently
represented as local (though possibly degenerate or singular)
operators in one dimension more. Moreover, as a counterpart,
several models arising in a local framework
give rise to nonlocal equations, due to boundary effects.
So, to introduce the extension problem and give
a concrete intuition of it, we will present some models in physics
that are naturally set on an extended space
to start with, and will show their relation with the fractional Laplacian
on a trace space. We will also give a detailed justification
of this extension procedure by means of the Fourier transform.

As a special example of problems arising in physics
that produce a nonlocal equation, we consider a problem related to crystal dislocation, present some mathematical results
that have been recently obtained on this topic, and
discuss the relation between these results and
the observable phenomena.

Chapter~\ref{S:NP}, \ref{nlms} and~\ref{S:NP:2}
present topics of contemporary research.
We will discuss in particular:
some phase transition
equations of nonlocal type, 
their limit interfaces,
which (below a critical threshold of the fractional parameter)
are surfaces that minimize a nonlocal perimeter functional, and
some nonlocal equations arising in quantum mechanics.
\bigskip

We remark that the introductory part of these notes
is not intended to be separated from the one which is more
research oriented: namely,
even the chapters whose main goal is to develop the basics
of the theory contain some parts related to contemporary research trends.
\bigskip

Of course, these notes and the results presented
do not aim to be comprehensive and cannot take into
account all the material that would deserve to be included.
Even a thorough
introduction to nonlocal (or even just fractional) equations
goes way beyond the purpose of this book.

Many fundamental
topics slipped completely out of these notes:
just to name a few, the topological methods
and the fine regularity theory in the fractional cases are
not presented here, 
the fully nonlinear or singular/degenerate equations
are not taken into account,
only very few applications are discussed briefly,
important models such as
the quasi-geostrophic equation and the fractional
porous media equation are not covered in these notes,
we will not consider models arising in game theory
such as the nonlocal tug-of-war,
the parabolic equations are not taken into account in detail,
unique continuation and overdetermined problems will not be studied here
and the link to probability theory that we consider here
is not rigorous and only superficial
(the reader interested in these important topics
may look, for instance, at~\cite{MOLICA, MINGIONE, kuusi_nonlocal, kuusi_frac, FULLY-NONLINEAR,QUASI1, QUASI2, POROUSMEDIA1, POROUSMEDIA2,BONFVAZ,BON_SIR_VAZ, 
TUG-OF-WAR, PARABOLIC, Peral, UNIQUE-CONTINUATION, Fall-over, Soave,Athanasopoulos, BONVAZBDD, BONVAZBDDI,CAFFVAZ, CAFFVAZ_ASYM, chengrennest, Duvaut, grubb_fractional,kul,fractional_porous, dipierrocontdens}).
Also, a complete discussion of the nonlocal equations
in bounded domains is not available here (for this,
we refer to the recent
survey~\cite{RO15}).
In terms of surveys, collections
of results and open problems,
we also mention the very nice website~\cite{MWIKI},
which gets\footnote{It seems to be known that
Plato did not like books because they cannot respond to questions.
He might have liked websites.}
constantly updated.\bigskip

Using a metaphor with fine arts,
we could say that the picture that we painted here
is not even impressionistic, it is just na\"{\i}f.
Nevertheless, we hope that these pages may be of some help
to the young researchers of all ages who are willing to
have a look at the exciting nonlocal scenario
(and who are willing to tolerate
the partial and incomplete point of view offered by this modest observation
point).

\chapter{A probabilistic motivation}\label{S:1}

The fractional Laplacian will be the main operator
studied in this book.
We consider a function~$u\colon \Rn\to \R$ (which is supposed\footnote{To
write~\eqref{frlap2def} it is sufficient, for simplicity, to take
here~$u$ in the Schwartz space $\mathcal{S}(\Rn)$
of smooth and rapidly decaying functions, or in~$ C^2(\Rn)\cap
L^{\infty}(\Rn)$.
We refer to \cite{S05} for a refinement of the space
of definition.}
to be regular enough) and a fractional parameter~$s\in (0,1)$. Then, the 
fractional Laplacian of~$u$ is given by
\begin{equation}\label{frlap2def}
\frlap u(x)= \frac{C(n,s)}{2} \int_{\Rn} \frac{2 u(x) -u(x+y)-u(x-y)} {|y|^{n+2s} } \, dy,\end{equation} where $C(n,s)$ is a dimensional\footnote{The explicit
value of~$C(n,s)$ is usually unimportant. Nevertheless,
we will compute its value explicitly in formulas~\eqref{cnsgalattica}
and~\eqref{EXPL}.
The reason for which it is convenient to divide~$C(n,s)$ by a factor~$2$
in~\eqref{frlap2def} will be clear later on, in formula~\eqref{frlapdef}.}
constant.

One sees from~\eqref{frlap2def}
that $(-\Delta)^s$
is an operator of order~$2s$, namely, it arises from
a differential quotient of order~$2s$ weighted in the whole
space. Different fractional operators have been
considered in literature
(see e.g. \cite{CAPUTO, SeV14, MUSINA-NAZAROV}), and all of them
come from interesting problems in pure or/and applied
mathematics. We will focus here
on the operator in~\eqref{frlap2def}
and we will motivate it by probabilistic considerations
(as a matter of fact, many other motivations are possible).

The probabilistic model under consideration is a random process
that allows long jumps
(in further generality, it is known that the fractional Laplacian 
is an infinitesimal generator of L\`evy processes,
see e.g.~\cite{B96,Applebaum} for further details). 
A more detailed
mathematical introduction
to the fractional Laplacian
is then presented in the subsequent
Section \ref{sspn}.

\section{The random walk with arbitrarily long jumps}\label{srw}

We will show here that the fractional heat equation (i.e. the ``typical'' equation that drives the fractional diffusion and that can be written, up to dimensional constants, as~$\partial_t u+ (-\Delta)^s u=0$) naturally arises from a probabilistic process in which a particle moves randomly in the space subject to a probability that allows long jumps with a polynomial tail.

For this scope, we introduce a probability distribution on the natural numbers~$\N^*:=\{1,2,3,\cdots\}$ as follows. If $I\subseteq \N^*$, then the probability of $I$ is defined to be 
	\[ P(I):= c_s \, \sum_{k\in I} \frac{1}{|k|^{1+2s}}.\]
The constant $c_s$ is taken in order to normalize $P$ to be a probability measure. Namely, we take 
	\[ c_s:=\left( \sum_{k\in \N^*} \frac{1}{|k|^{1+2s}}\right)^{-1},\]
so that we have~$P(\N^*)=1$.

Now we consider a particle that moves in $\R^n$ according to a probabilistic process. The process will be discrete both in time and space
(in the end, we will formally take the limit when these time and space steps are small).
We denote by $\tau$ the discrete time step, and by $h$ the discrete space step. We will take the scaling $\tau=h^{2s}$ and we denote by $u(x,t)$ the probability of finding the particle at the point $x$ at time $t$.

The particle in $\R^n$ is supposed to move according to the following probabilistic law: at each time step~$\tau$, the particle selects randomly both a direction $v\in \partial B_1$, according to the uniform distribution on $\partial B_1$, and a natural number~$k\in\N^*$,
according to the probability law $P$, and it moves by a discrete space step $khv$. Notice that long jumps are allowed with small probability. Then, if the particle is at time~$t$ at the point~$x_0$ and, following the probability law, it picks up a direction~$v\in \partial B_1$
and a natural number~$k\in\N^*$, then the particle at time~$t+\tau$ will lie at~$x_0+khv$.

Now, the probability~$u(x,t+\tau)$ of finding the particle at~$x$ at time~$t+\tau$ is the sum of the probabilities of finding the
particle somewhere else, say at~$x+khv$, for some direction~$v\in \partial B_1$ and some natural number~$k\in\N^*$, times the probability of having selected such a direction and such a natural number.
\begin{center}
\begin{figure}[htpb]
	\hspace{0.5cm}
	\begin{minipage}[b]{0.90\linewidth}
	\centering
	\includegraphics[width=0.90\textwidth]{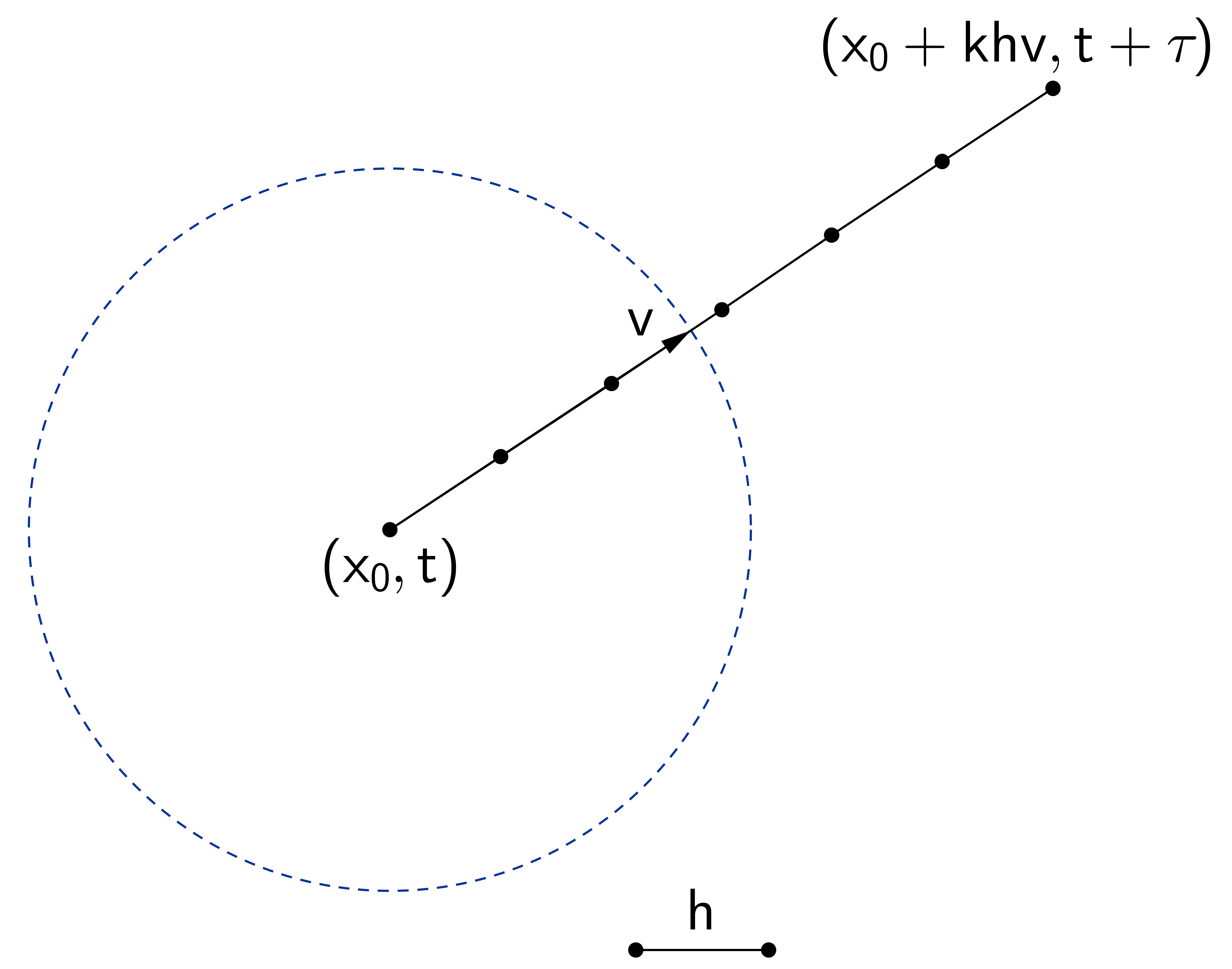}
	\caption{The random walk with jumps}   
	\label{fign:rndw}
	\end{minipage}
\end{figure} 
	\end{center}
This translates into 	
	\[ u(x,t+\tau) =\frac{c_s}{|\partial B_1|} \sum_{k\in\N^*} \int_{\partial B_1}
\frac{u(x+khv,t)}{|k|^{1+2s}} \,d\mathcal{H}^{n-1}(v).\]
Notice that the factor $c_s/|\partial B_1|$ is a normalizing probability
constant, hence we subtract~$u(x,t)$ and we obtain
	\bgs{
	u(x,t+\tau) -u(x,t)=\al
	\frac{c_s}{|\partial B_1|} \sum_{k\in\N^*} \int_{\partial B_1}
	 \frac{u(x+khv,t)}{|k|^{1+2s}} \,d\mathcal{H}^{n-1}(v) - u(x,t)\\
	=\al
	\frac{c_s}{|\partial B_1|} \sum_{k\in\N^*} \int_{\partial B_1}
	\frac{u(x+khv,t)-u(x,t)}{|k|^{1+2s}} \,d\mathcal{H}^{n-1}(v).
	}
As a matter of fact, by symmetry, we can change~$v$ to~$-v$ in the integral above, so we find that
	\[ u(x,t+\tau) - u(x,t) = \frac{c_s}{|\partial B_1|}  \sum_{k\in\N^*} \int_{\partial B_1} \frac{u(x-khv,t)-u(x,t)}{|k|^{1+2s}} \,d\mathcal{H}^{n-1}(v).\]
Then we can sum up these two expressions (and divide by $2$) and obtain that
	\bgs { u(&x,t+ \tau)  -u(x,t) \\
	\al = \frac{c_s}{2\,|\partial B_1|}  \sum_{k\in\N^*} \int_{\partial B_1} \frac{u(x+khv,t)+u(x-khv,t)-2u(x,t)}{|k|^{1+2s}}\,d\mathcal{H}^{n-1}(v) .} 
Now we divide by~$\tau=h^{2s}$, we recognize a Riemann sum, we take a formal limit and we use polar coordinates, thus obtaining:
	\bgs{
		\partial_t u(x,t)\simeq & \frac{ u(x,t+\tau)-u(x,t)}{\tau} 
		\\ =&
		\frac{c_s\,h}{2\,|\partial B_1|} \sum_{k\in\N^*} \int_{\partial B_1}
		\frac{u(x+khv,t)+u(x-khv,t)-2u(x,t)}{|hk|^{1+2s}}d\mathcal{H}^{n-1}(v) 
		\\ \simeq&  \frac{c_s}{2\,|\partial B_1|} \int_0^{+\infty}
		\int_{\partial B_1}
		\frac{u(x+rv,t)+u(x-rv,t)-2u(x,t)}{|r|^{1+2s}}d\mathcal{H}^{n-1}(v) \,dr \\
		=& \frac{c_s}{2\,|\partial B_1|} \int_{\R^n}
		\frac{u(x+y,t)+u(x-y,t)-2u(x,t)}{|y|^{n+2s}} \,dy\\
		=& - c_{n,s} \,(-\Delta)^s u(x,t)	}
for a suitable $c_{n,s}>0$.

This shows that, at least formally, for small time and space steps, the above probabilistic process approaches a fractional heat equation.\bigskip

We observe that processes of this type occur in nature quite
often, see in particular the biological observations
in~\cite{ALBA, MARINE}, other interesting observations in \cite{reynolds_levy,schouten,woycz} and the mathematical discussions
in~\cite{KAO, FRIEDMAN, MONTEFUSCO-PELLACCI-VERZINI, MASSACCESI,metzler}.

Roughly speaking, let us say that it is not unreasonable that
a predator may decide to use a nonlocal dispersive strategy
to hunt its preys more efficiently
(or, equivalently, that the natural selection may favor some kind of
nonlocal diffusion):
small fishes will not wait to be eaten by a big fish once they
have seen it, so it may be more convenient for the big fish
just to pick up a random direction, move rapidly
in that direction, stop quickly and eat the small fishes
there (if any) and then go on with the hunt. And this ``hit-and-run''
hunting procedure
seems quite related to that described in Figure~\ref{fign:rndw}.

\section{A payoff model}\label{pym}

Another probabilistic motivation for the fractional Laplacian arises from a payoff approach. Suppose to move in a domain~$\Omega$ according to a random walk with jumps as discussed in Section~\ref{srw}. Suppose also that exiting the domain~$\Omega$ for the first time by jumping to an outside point~$y\in\Rn \setminus\Omega$, means earning~$u_0(y)$ sestertii. A relevant question is, of course, how rich we expect to become
in this way. That is, if we start at a given point~$x\in\Omega$ and we denote by~$u(x)$ the amount of sestertii that we expect to gain, is there a way to obtain information on~$u$?

The answer is that 
(in the right scale limit of the random walk with jumps presented in Section~\ref{srw}) the expected payoff~$u$ is determined by the equation
	\begin{equation}\label{HAS}
			\begin{cases}
				(-\Delta)^s u =0 & {\mbox{ in }}\Omega,\\
				u=u_0 &{\mbox{ in }}\R^n\setminus\Omega.
			\end{cases}
	\end{equation}
To better explain this,
let us fix a point~$x\in\Omega$. The expected value of the payoff at~$x$ is the average of all the payoffs at the points~$\tilde x$ from which one can reach~$x$, weighted by the probability of the jumps. That is, by writing~$\tilde x=x+khv$, with~$v\in\partial B_1$, $k\in\N^*$ and~$h>0$, as in the previous Chapter~\ref {srw}, we have that the probability of jump is~$\displaystyle \frac{c_s}{|\partial B_1|\,|k|^{1+2s}}$. This leads to the formula 
	\[ u(x)= \frac{c_s}{|\partial B_1|} \sum_{k\in\N^*} \int_{\partial B_1} \frac{u(x+khv)}{|k|^{1+2s}}\,d{\mathcal{H}}^{n-1}(v).\]
By changing~$v$ into~$-v$ we obtain 
  \[ u(x)= \frac{c_s}{|\partial B_1|} \sum_{k\in\N^*}\int_{\partial B_1} \frac{u(x-khv)}{|k|^{1+2s}}\,d{\mathcal{H}}^{n-1}(v)\]
and so, by summing up,
	\[  2u(x)= \frac{c_s}{|\partial B_1|} \sum_{k\in\N^*}\int_{\partial B_1} \frac{u(x+khv)+u(x-khv)}{|k|^{1+2s}}\,d{\mathcal{H}}^{n-1}(v).\]
Since the total probability is~$1$, we can subtract~$2u(x)$ to both sides and obtain that
 \[ 0 = \frac{c_s}{|\partial B_1|} \sum_{k\in\N^*}\int_{\partial B_1} \frac{u(x+khv)+u(x-khv)-2u(x) }{|k|^{1+2s}}\,d{\mathcal{H}}^{n-1}(v). \]
We can now divide by~$h^{1+2s}$ and recognize a Riemann sum, which, after passing to the limit as~$h\searrow0$, gives~$0=-(-\Delta)^s
u(x)$,
that is~\eqref{HAS}.

\chapter{An introduction to the fractional Laplacian}\label{S:2}

We introduce here some preliminary notions on the fractional Laplacian and on fractional Sobolev spaces. Moreover, we present an 
explicit example of an $s$-harmonic function on the positive half-line $\R_{+}$, an example of a function with constant Laplacian on the ball, 
discuss some maximum principles and
a Harnack inequality, and 
present
a quite surprising
local density property
of $s$-harmonic functions into the space of smooth functions.

\section{Preliminary notions}\label{sspn}
We introduce here the fractional Laplace operator, the fractional 
Sobolev spaces and give some useful pieces of notation.
We also refer to~\cite{DNPV12} for further details related to the topic.

We consider the  Schwartz space of rapidly decaying functions defined as 
	\[ \mathcal{S} ({\Rn}): = \left \{  f \in C^\infty(\Rn)\;  \Big|   \; \forall \alpha, \, \beta \in \mathbf{N}^n_0, \,  \sup_{x\in {\Rn}} |x^{\alpha} \partial_{\beta} f(x)|<  \infty \right \}.  \]
For any $f \in \mathcal{S}(\Rn)$, denoting the space variable $x \in \Rn$ and the frequency variable $\xi \in \Rn$, the Fourier transform
and the inverse Fourier transform are defined,
respectively, as
		\begin{equation} 
		\widehat f(\xi) := \mathcal{F} f(\xi):= 
\int_{{\Rn}} f(x) e^{-  2\pi i x \cdot \xi} \, dx
		\label{transF} 
	\end{equation}
and
	\begin{equation}	\label{invF}
		f(x)=  \mathcal{F}^{-1} \widehat f(x)  = \int_{\Rn} \widehat f (\xi) e^{2\pi i x \cdot \xi} \, d\xi.
	\end{equation} 
Another useful notion is the one of principal value, namely we consider
the definition
\begin{equation}\label{PV-1} \mbox{P.V.}
\int_{\Rn}\frac{u(x) - u(y)}{|x- y|^{n+2s}}\, dy :=
\lim_{\eee \to 0} \int_{{\Rn}\setminus
B_{\eee}(x)} \frac{u(x) - u(y)}{|x- y|^{n+2s}}\, dy.\end{equation}
Notice indeed that the integrand above
is singular when~$y$ is in a neighborhood
of~$x$, and this singularity is, in general,
not integrable (in the sense of Lebesgue):
indeed notice that, near~$x$,
we have that~$u(x) - u(y)$ behaves at the first order
like~$\nabla u(x)\cdot(x-y)$,
hence the integral above behaves at the first order like
\begin{equation}\label{PV-2}
\frac{ \nabla u(x)\cdot(x-y) }{|x- y|^{n+2s}}\end{equation}
whose absolute value gives an infinite integral near~$x$
(unless either~$\nabla u(x)=0$
or~$s<1/2$).

The idea of the definition in~\eqref{PV-1}
is that the term in~\eqref{PV-2} averages out
in a neighborhood of~$x$
by symmetry, since the term is odd with respect to~$x$,
and so it does not contribute to the integral if we perform it
in a symmetric way. In a sense,
the principal value in~\eqref{PV-1}
kills the first order of the function at the numerator,
which produces a linear growth, and focuses
on the second order remainders.\bigskip

The notation in~\eqref{PV-1} allows us to write~\eqref{frlap2def}
in the following more compact form:
\begin{equation*}\begin{split}
	\frlap u(x)\,& =\,\frac{C(n,s)}{2}  \int_{\Rn} \frac{2 u(x) -u(x+y)-u(x-y)} {|y|^{n+2s}}\, dy\\
	&=\, \,\frac{C(n,s)}{2}\lim_{\eee \to 0} \int_{{\Rn}\setminus B_{\eee}}\frac{2 u(x) -u(x+y)-u(x-y)} {|y|^{n+2s}}\, dy \\
	&=\,\frac{C(n,s)}{2}	\lim_{\eee \to 0} \left[\int_{{\Rn}\setminus B_{\eee}}\frac{u(x) -u(x+y)} {|y|^{n+2s}}\, dy+\int_{{\Rn}\setminus B_{\eee}}\frac{u(x) -u(x-y)} {|y|^{n+2s}}\, dy\right]\\
	&=\,\frac{C(n,s)}{2}\lim_{\eee \to 0}\left[\int_{{\Rn}\setminus B_{\eee}(x)}\frac{u(x) -u(\eta)} {|x-\eta|^{n+2s}}\, d\eta+\int_{{\Rn}\setminus B_{\eee}(x)}\frac{u(x) -u(\zeta)} {|x-\zeta|^{n+2s}}\, d\zeta\right]\\
	&=\,C(n,s)\,\lim_{\eee \to 0}\int_{{\Rn}\setminus B_{\eee}(x)}\frac{u(x) -u(\eta)} {|x-\eta|^{n+2s}}\, d\eta,
	\end{split}\end{equation*}
where the changes of variable~$\eta:=x+y$ and~$\zeta:=x-y$
were used, i.e.
\begin{equation}\label{frlapdef}
	\frlap u(x) =C(n,s)\, \mbox{P.V.} \int_{\Rn}\frac{u(x) - u(y)}{|x- y|^{n+2s}}\, dy .\end{equation}
The simplification
above also explains why it was convenient to
write~\eqref{frlap2def} with the factor~$2$
dividing~$C(n,s)$.
Notice that the
expression in~\eqref{frlap2def}
does not require the P.V. formulation since, for instance, taking $u\in L^{\infty}(\Rn)$ and locally $C^2$, using a Taylor expansion of $u$ in $B_1$, one observes that 
	\[ \begin{split}
		\int_{\Rn}\al  \frac{|2 u(x) -u(x+y)-u(x-y)|} {|y|^{n+2s} } \, dy \\
		\leq \al   \|u\|_{L^{\infty}(\Rn)} \int_{\Rn \setminus B_1} |y|^{-n-2s} \, dy + \int_{B_1} \frac{|D^2u(x)| |y|^2}{|y|^{n+2s}}\, dy \\
	\leq \al \|u\|_{L^{\infty}(\Rn)}\int_{\Rn \setminus B_1}   |y|^{-n-2s}\, dy + \|D^2 u\|_{L^{\infty}(\Rn)}\int_{B_1} |y|^{-n-2s +2 }\, dy ,\\
\end{split} \] 
and the integrals above provide a finite quantity.\bigskip

Formula~\eqref{frlapdef}
has also a stimulating analogy with the classical Laplacian.
Namely, the classical Laplacian (up to normalizing constants)
is the measure of the infinitesimal displacement of
a function in average
(this is the ``elastic'' property
of harmonic functions, whose value at a given point
tends to revert to the average in a ball). Indeed,
by canceling the odd contributions, and using that
\begin{eqnarray*}&&\int_{B_r(x)} |x-y|^2 \,dy =
\sum_{k=1}^n \int_{B_r(x)} (x_k-y_k)^2 \,dy
= n \int_{B_r(x)} (x_i-y_i)^2 \,dy,
\\&&\qquad{\mbox{ for any }}i\in\{1,\dots,n\},\end{eqnarray*}
we see that
\begin{equation}\label{KJ896}
\begin{split}
& \lim_{r\to0} \frac{1}{r^2}
\left( u(x)-\frac{1}{|B_r(x)|}\, \int_{B_r(x)} u(y)\,dy\right)\\
=\;& 
\lim_{r\to0} -\frac{1}{r^2|B_r(x)| }
\int_{B_r(x)}\left( u(y)-u(x)\right) \,dy
\\
=\;& 
\lim_{r\to0} -\frac{1}{r^{n+2}\,|B_1|}\, \int_{B_r(x)} \nabla u(x)\cdot(x-y)
+\frac{1}{2} D^2u(x) (x-y)\cdot(x-y)
\\ &\qquad+{\mathcal{O}} (|x-y|^3)\,dy
\\
=\;&
\lim_{r\to0} -\frac{1}{2r^{n+2}\,|B_1|}\, \sum_{i,j=1}^n \int_{B_r(x)} 
\partial^2_{i,j} u(x)\,(x_i-y_i)(x_j-y_j) \,dy \\
=\;&
\lim_{r\to0} -\frac{1}{2r^{n+2}\,|B_1|}\, \sum_{i=1}^n \int_{B_r(x)}
\partial^2_{i,i} u(x)\,(x_i-y_i)^2 \,dy \\
=\;&
\lim_{r\to0} -\frac{1}{2n \,r^{n+2}\,|B_1|}\, \sum_{i=1}^n 
\partial^2_{i,i} u(x)\,
\int_{B_r(x)} |x-y|^2 \,dy \\
=\;& -C_n \Delta u(x),
\end{split}\end{equation}
for some~$C_n>0$. In this spirit,
when we compare the latter formula with~\eqref{frlapdef},
we can think that the fractional Laplacian corresponds to
a weighted average of the function's oscillation.
While the average in~\eqref{KJ896}
is localized in the vicinity of a point~$x$, the one in~\eqref{frlapdef}
occurs in the whole space (though it decays at infinity).
Also, the spacial homogeneity of the average in~\eqref{KJ896}
has an extra factor that is proportional to the space variables to
the power~$-2$, while the corresponding power in the average 
in~\eqref{frlapdef} is~$-2s$ (and this is consistent for~$s\to1$).
\bigskip

Furthermore, for $u \in \mathcal{S}(\Rn)$ the fractional Laplace
operator can be 
expressed in Fourier frequency variables multiplied
by~$(2\pi|\xi|)^{2s}$,
as stated in the following lemma. 
\begin{lemma}\label{frlaphdeflem}
We have that
	\begin{equation} 
		\frlap u(x)=\mathcal{F}^{-1} \big((2\pi|\xi|)^{2s} \widehat u (\xi)\big) .
		\label{frlaphdef} 
	\end{equation}
\end{lemma}

Roughly speaking, formula~\eqref{frlaphdef}
characterizes the fractional Laplace operator in the Fourier
space, by taking the $s$-power of the multiplier associated
to the classical Laplacian operator.
Indeed, by using the inverse Fourier transform, one has that
\begin{eqnarray*}
&& -\Delta u(x) =-\Delta ({\mathcal{F}}^{-1}(\widehat u))(x)
=-\Delta
\int_{{\Rn}} \widehat u(\xi) e^{2\pi i x \cdot \xi} \, d\xi
\\ &&\qquad = 
\int_{{\Rn}} (2\pi |\xi|)^2
\widehat u(\xi) e^{2\pi i x \cdot \xi} \, d\xi
= {\mathcal{F}}^{-1} \big( (2\pi |\xi|)^2
\widehat u(\xi)\big),\end{eqnarray*}
which gives that the classical Laplacian acts in a Fourier space
as a multiplier of~$(2\pi |\xi|)^2$. {F}rom this and
Lemma~\ref{frlaphdeflem} it also follows
that the classical Laplacian is the limit case of the 
fractional one, namely for any $u \in \mathcal{S}(\Rn)$
	\[\lim_{s\to 1} \frlap u =-\Delta u  \quad \mbox{ and also }\quad \lim_{s\to 0}\frlap  u = u.\]

Let us now
prove that indeed the two formulations \eqref{frlap2def} and \eqref{frlaphdef} are equivalent.
\begin{proof}[Proof of Lemma \ref{frlaphdeflem}]
Consider identity \eqref{frlap2def} and apply the Fourier transform to obtain 
	\begin{equation}\label{3.7bis} \begin{split}
		  \mathcal{F} \Big(\frlap u(x)\Big)  = &\; \frac{C(n,s)}{2} \int_{\Rn} \frac{\mathcal{F}\Big(2u(x) -u(x+y)-u(x-y)\Big)}{|y|^{n+2s}} \, dy \\ 
		= &\; \frac{C(n,s)}{2} \int_{\Rn} \widehat{u}(\xi)   \frac{2- e^{2\pi
i\xi \cdot  y} - e^{-2\pi
i\xi \cdot  y} }  {|y|^{n+2s}}\, dy   \\
		= &\; C(n,s) \, \widehat{u}(\xi)  \int_{\Rn}   \frac{1-\cos(2\pi\xi \cdot y) }  {|y|^{n+2s}}\, dy.   
	\end{split}\end{equation} 
Now, we use the change of variable $z= |\xi| y$ and obtain that
		\[ \begin{split}
		J (\xi)  : = &\; \int_{\Rn} \frac{1-\cos(2\pi\xi \cdot  y)}{|y|^{n+2s}} \, dy\\
		= &\;|\xi|^{2s} \int_{\Rn} \frac{1-\cos\frac{2\pi\xi}{|\xi|}\cdot z} {|z|^{n+2s}} \, dz.	 
		\end{split}\] 
Now we use that~$J$ is rotationally invariant. More precisely,
we consider a rotation $R$ that sends~$e_1=(1,0,\dots,0)$ into~$\xi/|\xi|$,
that is~$Re_1= \xi/|\xi|$,
and we
call $R^T$ its transpose. Then, by using the change of variables $\omega=R^T z$ we have that
	\begin{equation*}
\begin{split}
		 J(\xi) = &\;	|\xi|^{2s} \int_{\Rn} \frac{1-\cos (2\pi
Re_1 \cdot z)}{|z|^{n+2s}}\, dz\\
				= &\;	|\xi|^{2s} \int_{\Rn} \frac{1-\cos(2\pi R^T z \cdot e_1)}{|R^Tz|^{n+2s}}\, dz \\
				= &\;|\xi|^{2s}  \int_{\Rn} \frac{1-\cos (2\pi
\omega_1)}{|\omega|^{n+2s}}\, d\omega.
	\end{split}\end{equation*}
Changing variables $\tilde \omega = 2\pi\omega$ (we still write $\omega $ as a variable of integration), we obtain that
	\eqlab{ \label{37BB}
	 	J(\xi)  = \left(2\pi|\xi|\right)^{2s}  \int_{\Rn} \frac{1-\cos \omega_1}{|\omega|^{n+2s}}\, d\omega.}
Notice that this latter integral is finite.
Indeed, integrating outside the ball $B_1$ we have that 
	\[ 	 \int_{{\Rn}\setminus B_1} \frac{|1-\cos\omega_1|}{|\omega|^{n+2s}}  \, d\omega  \leq \int_{{\Rn}\setminus B_1} \frac{2}{|\omega|^{n+2s}} < \infty, 	\]
while inside the ball we can use the Taylor expansion of the cosine function and observe that
	\[ \int_{B_1} \frac{|1-\cos\omega_1|}{|\omega|^{n+2s}}  \, d\omega \leq \int_{B_1} \frac{|\omega|^2}{|\omega|^{n+2s}}  \, d\omega \leq \int_{B_1} \frac{ d\omega}{|\omega|^{n+2s-2}} < \infty. \]
Therefore, by taking
	\begin{equation}
		 C(n, s) := \bigg(\int_{\Rn} \frac{1 - \cos\omega_1}{|\omega|^{n+2s}} \, d\omega\bigg)^{-1} 
		\label{cnsgalattica}
	\end{equation}
it follows from~\eqref{37BB} that
	\[J(\xi) = \frac{\left(2\pi |\xi|\right)^{2s}}{C(n,s)}.\]
By inserting this into~\eqref{3.7bis}, we obtain that
$$ \mathcal{F} \Big(\frlap u(x)\Big)= C(n,s)\,\widehat u(\xi)\,J(\xi)=
(2\pi|\xi|)^{2s}\widehat u(\xi),$$
which concludes the proof.
\end{proof}

Notice that the renormalization constant $C(n,s)$ introduced in~\eqref{frlap2def} is now computed in~\eqref{cnsgalattica}.

\bigskip

Another approach to the fractional Laplacian comes
from the theory of semigroups (or, equivalently from the fractional calculus
arising in subordination identities). This technique
is classical (see~\cite{YOSIDA}),
but it has also been efficiently used in recent research papers
(see for
instance~\cite{LLAVE-VALDINOCI-AIHP, STINGA-TORREA, STINGA-CAFFARELLI}).
Roughly speaking, the main idea underneath
the semigroup approach comes from the following explicit
formulas for the Euler's function: for any~$\lambda>0$, one
uses an integration by parts and the substitution~$\tau=\lambda t$
to see that
\bgs{  -s\Gamma(-s) =\al \Gamma(1-s) \\= \al\int_0^{+\infty} \tau^{-s} e^{-\tau}\,d\tau \\ 
	=\al - \int_0^{+\infty} \tau^{-s} \frac{d}{d\tau} (e^{-\tau}-1)\,d\tau
\\=\al
-s\int_0^{+\infty} \tau^{-s-1} (e^{-\tau}-1)\,d\tau 
\\ =\al  -s\lambda^{-s} \int_0^{+\infty} t^{-s-1} (e^{-\lambda t}-1)\,dt,}
that is
\begin{equation}\label{L-EQ} \lambda^s =
\frac{1}{\Gamma(-s)} \int_0^{+\infty} t^{-s-1} (e^{-\lambda t}-1)\,dt.\end{equation}
Then one applies formally this identity to~$\lambda:=-\Delta$. Of course, this formal step needs to be justified, but if
things go well one obtains
$$ (-\Delta)^s =
\frac{1}{\Gamma(-s)} \int_0^{+\infty} t^{-s-1} (e^{t\Delta}-1)\,dt,$$
that is (interpreting~$1$ as the identity operator)
\begin{equation}\label{FP}
(-\Delta)^s u(x) =
\frac{1}{\Gamma(-s)} \int_0^{+\infty} t^{-s-1} (e^{t\Delta} u(x)-u(x))\,dt.
\end{equation}
Formally, if~$U(x,t):=e^{t\Delta}u(x)$, we have that~$U(x,0)=u(x)$ and
$$ \partial_t U=\frac{\partial}{\partial t} ( e^{t\Delta} u(x) ) = \Delta
e^{t\Delta} u(x) =\Delta U,$$
that is~$U(x,t)=e^{t\Delta}u(x)$ can be interpreted as the solution
of the heat equation with initial datum~$u$.
We indeed point out that these formal computations can be justified:

\begin{lemma}\label{L32}
Formula~\eqref{FP} holds true. That is, if~$u\in{\mathcal{S}}(\R^n)$
and~$U=U(x,t)$ is
the solution of the heat equation
$$ \left\{
\begin{matrix}
\partial_t U=\Delta U & {\mbox{ in }}t>0,\\
U\big|_{t=0} = u,
\end{matrix}
\right. $$
then
\begin{equation}\label{EX-CO} (-\Delta)^s u(x) = 
\frac{1}{\Gamma(-s)} \int_0^{+\infty} t^{-s-1} (U(x,t)-u(x))\,dt.
\end{equation}
\end{lemma}

\begin{proof} {F}rom Theorem~1 on page~47
in~\cite{EVANS}
we know that~$U$ is obtained by Gaussian convolution
with unit mass, i.e.
\eqlab{\label{G-EQ} 
& U(x,t)=\int_{\R^n} G(x-y,t)\,u(y)\,dy=\int_{\R^n} G(y,t)\,u(x-y)\,dy,\\
&\qquad{\mbox{ where }}\;
G(x,t):= (4\pi t)^{-n/2} e^{-|x|^2/(4t)}.}
As a consequence, using the substitution~$\tau:=|y|^2/(4t)$,
\bgs{ \al  \int_0^{+\infty} t^{-s-1} (U(x,t)-u(x))\,dt \\ =\al  \int_0^{+\infty}\left[ \int_{\R^n} t^{-s-1}  G(y,t)\,\big( u(x-y)-u(x)\big) \,dy \right]\,dt \\
	=\al \int_0^{+\infty}\left[ \int_{\R^n} (4\pi t)^{-n/2} t^{-s-1}  e^{-|y|^2/(4t)} \,\big( u(x-y)-u(x)\big) \,dy \right]\,dt\\
	=\al \int_0^{+\infty}\left[\int_{\R^n} \tau^{n/2} (\pi |y|^2)^{-n/2} |y|^{-2s} (4\tau)^{s+1}  e^{-\tau} \,\big( u(x-y) - u(x)\big) \,dy \right]\,\frac{ d\tau }{4\tau^2 } \\
	=\al 2^{2s-1}\pi^{-n/2} \int_0^{+\infty}\left[ \int_{\R^n}  \tau^{\frac{n}{2}+s-1} e^{-\tau}\frac{ u(x+y)+u(x-y)-2u(x) }{ |y|^{n+2s} }\,dy \right]\,d\tau.
}
Now we notice that
$$ \int_0^{+\infty}\tau^{\frac{n}{2}+s-1} e^{-\tau}\,d\tau=\Gamma\left(
\frac{n}{2}+s\right),$$
so we obtain that
\begin{eqnarray*} &&\int_0^{+\infty} t^{-s-1} (U(x,t)-u(x))\,dt
\\&=& 2^{2s-1}\pi^{-n/2}\Gamma\left(\frac{n}{2}+s\right)
\int_{\R^n} \frac{ u(x+y)+u(x-y)-2u(x) }{ |y|^{n+2s} } \,dy
\\&=& -\frac{ 2^{2s}\,\pi^{-n/2}\,\Gamma\left(\frac{n}{2}+s\right) }{ C(n,s) }(-\Delta)^s u(x).\end{eqnarray*}
This proves~\eqref{EX-CO}, by choosing $C(n,s)$ appropriately.
And, as a matter of fact, gives the explicit value of the constant~$C(n,s)$
as
\begin{equation}\label{EXPL}
C(n,s)=-\frac{2^{2s}\,\Gamma\left( \frac{n}{2}+s\right)}{\pi^{n/2} 
\Gamma(-s)}
=
\frac{2^{2s}\,s\,\Gamma\left( \frac{n}{2}+s\right)}{\pi^{n/2}
\Gamma(1-s)},
\end{equation}
where we have used again that~$\Gamma(1-s)=-s\Gamma(-s)$, for any~$s\in(0,1)$.
\end{proof}

It is worth pointing out that the renormalization constant~$C(n,s)$
introduced in~\eqref{frlap2def}
has now been explicitly computed in~\eqref{EXPL}.
Notice that the choices of~$C(n,s)$
in~\eqref{cnsgalattica}
and~\eqref{EXPL} must agree (since we have computed
the fractional Laplacian in two different ways):
for a computation that shows that the quantity
in~\eqref{cnsgalattica} coincides with the one
in~\eqref{EXPL}, see Theorem~3.9
in \cite{B15}. For completeness, we give below a direct proof
that the settings in~\eqref{cnsgalattica}
and~\eqref{EXPL} are the same, by using
Fourier methods and~\eqref{L-EQ}.

\begin{lemma}\label{LM-COST}
For any~$n\in\N$, $n\ge1$, and~$s\in(0,1)$, we have that
\begin{equation}\label{EQW-IO-1}
\int_{\R^n} \frac{1-\cos(2\pi\omega_1)}{|\omega|^{n+2s}}\,d\omega
=\frac{\pi^{\frac{n}{2}+2s}\,\Gamma(1-s)}{
s\,\Gamma\left( \frac{n}{2}+s\right)}.\end{equation}
Equivalently, we have that
\begin{equation}\label{EQW-IO-2}
\int_{\R^n} \frac{1-\cos\omega_1}{|\omega|^{n+2s}}\,d\omega
=\frac{\pi^{\frac{n}{2}}\,\Gamma(1-s)}{
	2^{2s}s\,\Gamma\left( \frac{n}{2}+s\right)}. \end{equation}
\end{lemma}

\begin{proof} Of course, formula~\eqref{EQW-IO-1}
is equivalent to~\eqref{EQW-IO-2} (after the substitution~$\tilde\omega:=
2\pi\omega$).
Strictly speaking, in Lemma~\ref{frlaphdeflem}
(compare \eqref{frlap2def}, \eqref{frlaphdef}, and~\eqref{cnsgalattica})
we have proved that 
\eqlab{ \label{PoP-0} \frac{1} {2\displaystyle \int_{\Rn}  \displaystyle \frac{1 -\cos\omega_1}{|\omega|^{n+2s}} \, d\omega } \, \int_{\Rn} \frac{2 u(x) -u(x+y)-u(x-y)} {|y|^{n+2s} } \, dy 
=  {\mathcal{F}}^{-1} \big((2\pi|\xi|)^{2s} \widehat u(\xi)\big).}
Similarly, by means of Lemma~\ref{L32}
(compare \eqref{frlap2def}, 
\eqref{EX-CO} and~\eqref{EXPL}) we know that
\begin{equation}\label{PoP-1}\begin{split}&
\frac{2^{2s-1}\,s\,\Gamma\left( \frac{n}{2}+s\right)}{\pi^{n/2}
\Gamma(1-s)} \int_{\Rn} \frac{2 u(x) -u(x+y)-u(x-y)} {|y|^{n+2s} } \, dy\\&\qquad=
\frac{1}{\Gamma(-s)} \int_0^{+\infty} t^{-s-1} (U(x,t)-u(x))\,dt.\end{split}\end{equation}
Moreover (see~\eqref{G-EQ}), we have that~$U(x,t):=\Gamma_t * u(x)$, where
$$ \Gamma_t(x):=G(x,t)=(4\pi t)^{-n/2} e^{-|x|^2/(4t)}.$$
We recall that the Fourier transform of
a Gaussian is a Gaussian itself, namely
$$ {\mathcal{F}}(e^{-\pi|x|^2}) =e^{-\pi|\xi|^2},$$
therefore, for any fixed~$t>0$, using the substitution~$y:=x/\sqrt{4\pi t}$,
\begin{eqnarray*}
{\mathcal{F}}\,\Gamma_t (\xi) &=& 
\frac{1}{(4\pi t)^{n/2}} \int_{\R^n} 
e^{-|x|^2/(4t)} e^{-2\pi ix\cdot \xi}\,dx
\\ &=& 
\int_{\R^n}
e^{-\pi|y|^2} e^{-2\pi iy\cdot (\sqrt{4\pi t}\xi)}\,dy
\\ &=& e^{-4\pi^2 t|\xi|^2}.
\end{eqnarray*}
As a consequence
\begin{eqnarray*}
&& {\mathcal{F}} \big(U(x,t)-u(x)\big)
={\mathcal{F}}\big( \Gamma_t*u(x)-u(x)\big)
\\ &&\qquad ={\mathcal{F}}(\Gamma_t*u)(\xi)-\widehat u(\xi)
=\big({\mathcal{F}}\,\Gamma_t(\xi)-1\big)\widehat u(\xi)
\\ &&\qquad = (e^{-4\pi^2 t|\xi|^2}-1)\widehat u(\xi).
\end{eqnarray*}
We multiply by~$ t^{-s-1}$ and integrate over~$t>0$,
and we obtain
\begin{eqnarray*}
{\mathcal{F}} \int_0^{+\infty} t^{-s-1}\big(U(x,t)-u(x)\big)\,dt
&=& \int_0^{+\infty} t^{-s-1}(e^{-4\pi^2 t|\xi|^2}-1)\,dt\,
\widehat u(\xi)
\\&=&\Gamma(-s)\, (4\pi^2 |\xi|^2)^s\,\widehat u(\xi),\end{eqnarray*}
thanks to~\eqref{L-EQ} (used here with~$\lambda:=4\pi^2 |\xi|^2$).
By taking the inverse Fourier transform, we have
$$ \int_0^{+\infty} t^{-s-1}\big(U(x,t)-u(x)\big)\,dt =
\Gamma(-s)\, (2 \pi)^{2s} {\mathcal{F}}^{-1}
\big(|\xi|^{2s}\,\widehat u(\xi)\big).$$
We insert this information into~\eqref{PoP-1}
and we get
$$ \frac{2^{2s-1}\,s\,\Gamma\left( \frac{n}{2}+s\right)}{\pi^{n/2}
\Gamma(1-s)} \int_{\Rn} \frac{2 u(x) -u(x+y)-u(x-y)} {|y|^{n+2s} } \, dy=
(2 \pi)^{2s} {\mathcal{F}}^{-1}
\big(|\xi|^{2s}\,\widehat u(\xi)\big).$$
Hence, recalling~\eqref{PoP-0},
\bgs{ &\frac{2^{2s-1}\,s\,\Gamma\left( \frac{n}{2}+s\right)}{\pi^{n/2}
\Gamma(1-s)} \int_{\Rn} \frac{2 u(x) -u(x+y)-u(x-y)} {|y|^{n+2s} } \, dy
\\ &\qquad=
\frac{1}{2\displaystyle\int_{\Rn} \displaystyle \frac{1 -
\cos\omega_1}{|\omega|^{n+2s}} \, d\omega}\,
\int_{\Rn} \frac{2 u(x) -u(x+y)-u(x-y)} {|y|^{n+2s} } \, dy
,}
which gives the desired result.
\end{proof}

For the sake of completeness,
a different proof of Lemma~\ref{LM-COST} will be given in
Appendix~\ref{APP-APP}. There, to prove Lemma~\ref{LM-COST},
we will use the theory of special functions rather than
the fractional Laplacian. For other approaches to the proof of
Lemma~\ref{LM-COST} see also the recent PhD dissertations~\cite{Felsinger} (and related \cite{Felsinger2})
and~\cite{Jarohs}.\bigskip

\section{Fractional Sobolev Inequality and Generalized Coarea Formula} \label{sobineq}
Fractional Sobolev spaces enjoy quite
a number of important functional inequalities.
It is almost impossible to list here all the results
and the possible applications, therefore we will
only present two important inequalities which have a simple and nice proof,
namely the fractional Sobolev Inequality and the Generalized Coarea Formula.

The fractional Sobolev Inequality can be written as follows:

\begin{thm}
For any~$s\in(0,1)$, $p\in\left(1,\frac{n}{s}\right)$ and~$u\in C^\infty_0(\R^n)$,
\begin{equation}\label{OK:sob:p}
\| u \|_{L^{\frac{np}{n-sp}}(\R^n)}\le C\left(
\int_{\R^n}\int_{\R^n}\frac{|u(x)-u(y)|^p}{|x-y|^{n+sp}}
\,dx\,dy\right)^{\frac{1}{p}},\end{equation}
for some~$C>0$, depending only on~$n$ and~$p$.
\end{thm}

\begin{proof} We follow here the very nice proof given in~\cite{PONCE} (where, in turn, the proof is attributed to Ha\"{\i}m Brezis).
We fix~$r>0$, $a>0$, $P>0$ and~$x\in\R^n$. Then, for any~$y\in\R^n$,
$$ |u(x)|\le |u(x)-u(y)|+|u(y)|,$$
and so, integrating over~$B_r(x)$, we obtain
\begin{eqnarray*}
|B_r|\,|u(x)|&\le& \int_{B_r(x)} |u(x)-u(y)|\,dy+\int_{B_r(x)}|u(y)|\,dy
\\ &=&
\int_{B_r(x)} \frac{|u(x)-u(y)|}{|x-y|^a}\cdot|x-y|^a\,dy+
\int_{B_r(x)}|u(y)|\,dy\\
&\le&
r^a \int_{B_r(x)} \frac{|u(x)-u(y)|}{|x-y|^a}\,dy+
\int_{B_r(x)}|u(y)|\,dy.
\end{eqnarray*}
Now we choose~$a:=\frac{n+sp}{p}$
and we make use of the H\"older Inequality (with
exponents~$p$ and~$\frac{p}{p-1}$
and with exponents~$\frac{np}{n-sp}$ and~$\frac{np}{n(p-1)+sp}$),
to obtain
\begin{eqnarray*}
|B_r|\,|u(x)|
&\le&
r^{\frac{n+sp}{p}} \int_{B_r(x)} \frac{|u(x)-u(y)|}{
|x-y|^{\frac{n+sp}{p}}}\,dy+
\int_{B_r(x)}|u(y)|\,dy \\
&\le&
r^{\frac{n+sp}{p}} \left( \int_{B_r(x)} \frac{|u(x)-u(y)|^p}{
|x-y|^{{n+sp}}}\,dy\right)^\frac{1}{p} 
\left( \int_{B_r(x)} \,dy\right)^\frac{p-1}{p}\\&&\qquad
+
\left( \int_{B_r(x)}|u(y)|^{\frac{np}{n-sp}}\,dy\right)^{\frac{n-sp}{np}}
\left( \int_{B_r(x)}\,dy\right)^{\frac{n(p-1)+sp}{np}}
\\ &\le& C r^{n+s}\,\left( \int_{B_r(x)} \frac{|u(x)-u(y)|^p}{
|x-y|^{{n+sp}}}\,dy\right)^\frac{1}{p} 
\\&&\qquad+ Cr^{\frac{n(p-1)+sp}{p}}
\left( \int_{B_r(x)}|u(y)|^{\frac{np}{n-sp}}\,dy\right)^{\frac{n-sp}{np}}
,\end{eqnarray*}
for some~$C>0$. So, we divide by~$r^n$ and we possibly rename~$C$.
In this way, we obtain
$$ |u(x)|\le 
C r^{s} \left[
\,\left( \int_{B_r(x)} \frac{|u(x)-u(y)|^p}{
|x-y|^{{n+sp}}}\,dy\right)^\frac{1}{p}
+ r^{-\frac{n}{p}}
\left( \int_{B_r(x)}|u(y)|^{\frac{np}{n-sp}}\,dy\right)^{\frac{n-sp}{np}}
\right].$$
That is, using the short notation
\begin{eqnarray*}
&&\alpha:= \int_{\R^n} \frac{|u(x)-u(y)|^p}{
|x-y|^{{n+sp}}}\,dy\\
{\mbox{and }}&&\beta:=\int_{\R^n}|u(y)|^{\frac{np}{n-sp}}\,dy
\end{eqnarray*}
 we have that
\bgs{ \label{OPKgg89} |u(x)|\leq Cr^s \left( \alpha^\frac{1}p + r^{-\frac{n}p }\beta ^{\frac{n-sp}{np}} \right),}
hence, raising both terms at the appropriate
power $\frac{np}{n-sp}$ and renaming $C$
\eqlab{ \label{OPKgg89}  
 |u(x)|^{\frac{np}{n-sp}} 
&\leq& C r^{\frac{nsp}{n-sp}} \left( \alpha^{\frac{1}p}+ r^{-\frac{n}p }\beta ^{\frac{n-sp}{np}} \right)^{\frac{np}{n-sp}}.}
We take now
$$ r:=\frac{ \beta^{\frac{n-sp}{n^2}} }{\alpha^{\frac1n}}.$$
With this setting, we have that $r^{-\frac{n}p}\beta^{\frac{n-sp}{np}} $ is equal to~$
\alpha^{\frac{1}{p}}$. Accordingly,  possibly renaming~$C$, we infer from~\eqref{OPKgg89}
that
\begin{eqnarray*}
 |u(x)|^{\frac{np}{n-sp}} &\le& C \alpha\,\beta^{\frac{sp}{n}} \\
&=& C \,
\int_{\R^n} \frac{|u(x)-u(y)|^p}{
|x-y|^{{n+sp}}}\,dy\;
\left(
\int_{\R^n}|u(y)|^{\frac{np}{n-sp}}\,dy\right)^{\frac{sp}{n}},\end{eqnarray*}
for some~$C>0$, and so, integrating over~$x\in\R^n$,
$$ \int_{\R^n}|u(x)|^{\frac{np}{n-sp}}\,dx\le
C \left(\int_{\R^n}
\int_{\R^n} \frac{|u(x)-u(y)|^p}{
|x-y|^{{n+sp}}}\,dx\,dy\right)\,
\left(
\int_{\R^n}|u(y)|^{\frac{np}{n-sp}}\,dy\right)^{\frac{sp}{n}}.$$
This, after a simplification, gives~\eqref{OK:sob:p}.
\end{proof}

What follows is the Generalized Co-area Formula of~\cite{VISENTIN}
(the link with the classical Co-area Formula will be indeed
more evident in terms of the fractional perimeter functional
discussed in Chapter~\ref{nlms}).

\begin{thm}
For any~$s\in(0,1)$ and any measurable function~$u:\Omega\to[0,1]$,
$$ \frac12 \int_\Omega\int_\Omega \frac{|u(x)-u(y)|}{|x-y|^{n+s}}\,dx\,dy
=\int_0^1 \left(\int_{\{x\in\Omega,\;
u(x)>t\}}\int_{\{y\in\Omega,\;u(y)\le t\}}
\frac{dx\,dy}{|x-y|^{n+s}}\right)\,dt.$$
\end{thm}

\begin{proof} We claim that for any~$x$, $y\in\Omega$
\begin{equation}\label{CV-OA:1}
|u(x)-u(y)|=\int_0^1 \Big(
\chi_{\{u>t\}}(x) \,\chi_{\{u\le t\}}(y)+
\chi_{\{u\le t\}}(x)\,\chi_{\{u>t\}}(y) \Big)\,dt.
\end{equation}
To prove this, we fix $x $ and $y$ in $\Omega$, and by possibly exchanging them, we can suppose that~$
u(x)\ge u(y)$. Then, we define
$$\varphi(t):=
\chi_{\{u>t\}}(x) \,\chi_{\{u\le t\}}(y)+
\chi_{\{u\le t\}}(x)\,\chi_{\{u>t\}}(y) .$$
By construction
\bgs{ \varphi(t)=\left\{\begin{matrix}
		0 & {\mbox{ if }} & &t<u(y) {\mbox{ and }} t\ge u(x),\\
		1 & {\mbox{ if }} && u(y) \leq t<u(x) ,\\
		\end{matrix}
\right.}
therefore
$$ \int_0^1 \varphi(t)\,dt =
\int_{u(y)}^{u(x)} \,dt=u(x)-u(y),$$
which proves \eqref{CV-OA:1}.

So, multiplying by the singular kernel
and integrating~\eqref{CV-OA:1} over~$\Omega\times\Omega$, we obtain
that
\begin{eqnarray*}&&
\int_\Omega\int_\Omega \frac{|u(x)-u(y)|}{|x-y|^{n+s}}\,dx\,dy\\&=&
\int_0^1 \left(
\int_\Omega\int_\Omega
\frac{
\chi_{\{u>t\}}(x) \,\chi_{\{u\le t\}}(y)+
\chi_{\{u\le t\}}(x)\,\chi_{\{u>t\}}(y) }{|x-y|^{n+s}}\,dx\,dy\right)\,dt
\\ &=&
\int_0^1 \left(
\int_{\{u>t\}}\int_{\{u\le t\}}
\frac{dx\,dy}{|x-y|^{n+s}}
+\int_{\{u\le t\}}
\int_{\{u>t\}}
\frac{dx\,dy}{|x-y|^{n+s}}
\right)\,dt
\\ &=& 2\int_0^1\left( \int_{\{u>t\}}\int_{\{u\le t\}}
\frac{dx\,dy}{|x-y|^{n+s}}\right)\,dt,
\end{eqnarray*}
as desired.
\end{proof}

\section{Maximum Principle and Harnack Inequality}\label{S:PGRAG}

The Harnack Inequality and \label{PGRAG} the Maximum Principle for harmonic functions are classical topics in elliptic regularity theory. Namely, in the classical case, if a non-negative function is harmonic in~$B_1$ and~$r\in(0,1)$, then its minimum and maximum in~$B_{r}$ must always be comparable (in particular, the function cannot touch the level zero
in~$B_r$).

It is worth pointing out that the fractional counterpart of these facts is, in general, false, as this next simple
result shows (see \cite{K15}):

\begin{thm} \label{NH}
There exists a bounded function~$u$ which does not vanish identically on $\Rn$, is~$s$-harmonic in~$B_1$, non-negative in~$B_1$, but such that~$\displaystyle \inf_{B_\frac12} u =0$.
\end{thm}
\begin{proof}[Sketch of the proof]
The main idea is that we are able to take the datum of~$u$ outside~$B_1$ in a suitable way as to ``bend down'' the function inside~$B_{\frac12}$ until it reaches the level zero. Namely, let~$M\ge 0$ and we take~$u_M$ to be the function satisfying
	\begin{equation}\label{e56}
		\begin{cases}
			(-\Delta)^s u_M =0 & {\mbox{ in }} B_1,\\
			u_M = 1-M & {\mbox{ in }} B_3\setminus B_2,\\
			u_M =1 & {\mbox{ in }} \R^n\setminus \big((B_3\setminus B_2)\cup B_1\big).
	\end{cases}
	\end{equation}
When~$M=0$, the function~$u_M$ is identically~$1$.
When~$M>0$, we expect~$u_M$ to bend down, since the fact that the fractional Laplacian vanishes in $B_1$ forces the second order quotient to vanish 
in average (recall~\eqref{frlap2def},
or the equivalent formulation in~\eqref{frlapdef}).
Indeed, we claim that there exists~$M_\star>0$ such that~$u_{M_\star}\ge0$ in~$B_1$ with~$\displaystyle \inf_{B_\frac12} u_{M_\star}=0$. Then, the result of
Theorem~\ref{NH} would be reached by taking~$u:=u_{M_\star}$.

To check the existence of such~$M_\star$,
we show that
\[ \inf_{B_\frac12} u_M\to -\infty\; \mbox{ as } M\to+\infty.\] 
Indeed, we argue by contradiction and suppose this cannot happen. Then, for any~$M\ge0$, we would have that
\begin{equation}\label{e1234}
\inf_{B_\frac12} u_M\ge -a,\end{equation} for some fixed~$a\in\R$.
We set
	\[  v_M:= \frac{u_M+M-1}{M}.\]
Then, by~\eqref{e56}, 
	\[  \begin{cases}
			(-\Delta)^s v_M =0 & {\mbox{ in }} B_1,\\
			v_M = 0 & {\mbox{ in }} B_3\setminus B_2,\\
			v_M =1 & {\mbox{ in }} \R^n\setminus \big((B_3\setminus B_2)\cup B_1\big).
	\end{cases}\]
Also, by~\eqref{e1234}, for any~$x\in {B_\frac12}$,
	\[  v_M(x)\ge \frac{-a+M-1}{M}.\]
By taking limits, one obtains that~$v_M$ approaches a function~$v_\infty$ that satisfies
 	\[ \begin{cases}
(-\Delta)^s v_\infty =0 & {\mbox{ in }} B_1,\\
v_\infty= 0& {\mbox{ in }} B_3\setminus B_2,\\
v_\infty = 1 & {\mbox{ in }} \R^n\setminus \big((B_3\setminus B_2)\cup 
B_1\big)
\end{cases} \]
and, for any~$x\in {B_\frac12}$,	
		\[ v_\infty (x)\ge 1.\] 
		On the other hand, by the Maximum Principle (that we prove in the next Theorem \ref{THM-MA-1}), we have that in $B_1$
		\[ v_\infty(x)\leq 1.\]  Thus the maximum of $v_\infty$ is attained at some point~$x_\star\in {B_\frac12}$, with~$v_\infty(x_\star)\ge 1$. Accordingly, 
	\begin{eqnarray*}
&& 0 = P.V. \int_{\R^n} \frac{v_\infty(x_\star)-v_\infty(y)}{|x_\star-y|^{n+2s}}\,dy
\ge P.V. \int_{B_3\setminus B_2} 
\frac{v_\infty(x_\star)-v_\infty(y)}{|x_\star-y|^{n+2s}}\,dy
\\ &&\qquad\ge
P.V. \int_{B_3\setminus B_2} \frac{1-0}{|x_\star-y|^{n+2s}}\,dy>0,
\end{eqnarray*}
which is a contradiction.
\end{proof}

The example provided by Theorem~\ref{NH} is not the end
of the story concerning the Harnack Inequality
in the fractional setting.
On the one hand, Theorem~\ref{NH} is just a particular case of the very dramatic effect that the datum at infinity may have on the fractional Laplacian (a striking example of this phenomenon will be given in Section \ref{afsh}).
On the other hand, the Harnack Inequality and the Maximum Principle hold true if, for instance, the sign of the function~$u$ is controlled in the whole of~$\Rn$.

We refer to \cite{BASS-KASSMANN, SILV-HOELDER, K15,FerrariFranchi} and to the references therein for a detailed introduction to the fractional Harnack Inequality, and to \cite{DCKP14} for general estimates of this type.

Just to point out
the validity of a global Maximum Principle, 
we state in detail the following simple result:

\begin{thm}\label{THM-MA-1}
If~$(-\Delta)^s u\ge0$ in~$B_1$ and~$u\ge0$ in~$\R^n\setminus B_1$,
then~$u\ge 0$ in~$B_1$.\end{thm}

\begin{proof} Suppose by contradiction that the minimal point~$x_\star\in B_1$ satisfies $u(x_\star)<0$. Then $u(x_*)$ is a minimum in $\Rn$ (since $u$ is positive outside $B_1$), if $y \in B_2$ we have that $2u(x_\star) -u(x_\star+y)-u(x_\star-y)  \leq 0$. On the other hand, in $\Rn \setminus B_2$ we have that $x_\star\pm y \in \Rn \setminus B_1$, hence $u(x_\star\pm y)\geq 0$. We thus have	\[\begin{split}  0\leq \al \int_{\R^n} \frac{2u(x_\star)-u(x_\star+y)-u(x_\star-y)}{|y|^{n+2s}}\,dy\\
		\leq \al  \int_{\R^n\setminus B_2} \frac{2u(x_\star)-u(x_\star+y)-u(x_\star-y)}{|y|^{n+2s}}\,dy \\ 
		\leq \al \int_{\R^n\setminus B_2}  \frac{2u(x_\star)}{|y|^{n+2s}}\,dy<0. \end{split}\]
This leads to a contradiction.\end{proof}

Similarly to Theorem~\ref{THM-MA-1}, one can prove a Strong Maximum Principle,
such as:

\begin{thm}\label{THM-MA-1-STRONG}
If~$(-\Delta)^s u\ge0$
in~$B_1$ and~$u\ge0$ in~$\R^n\setminus B_1$,
then~$u>0$ in~$B_1$, unless~$u$ vanishes identically.\end{thm}

\begin{proof}
We observe that we already know that~$u\ge0$
in the whole of~$\R^n$, thanks to Theorem~\ref{THM-MA-1}.
Hence, if $u$ is not strictly positive, there
exists~$x_0\in B_1$
such that~$u(x_0)=0$. This gives that
\[  0\le \int_{\R^n} \frac{2u(x_0)-u(x_0+y)-u(x_0-y)}{|y|^{n+2s}}\,dy
=- \int_{\Rn}\frac{u(x_0+y)+u(x_0-y)}{|y|^{n+2s}}\,dy.\]
Now both $u(x_0+y)$ and  $u(x_0-y)$ 
are non-negative, hence the latter
integral is less than or equal to zero,
and so it must vanish identically, proving that~$u$ also
vanishes identically.
\end{proof}

A simple version of a Harnack-type
inequality in the fractional setting can be also obtained as follows:

\begin{prop}
Assume that~$(-\Delta)^s u\ge0$ in~$B_2$, with~$u\ge0$ in
the whole of~$\R^n$. Then
$$ u(0)\ge c\int_{B_1} u(x)\,dx,$$
for a suitable~$c>0$.
\end{prop}

\begin{proof} Let~$\Gamma\in C^\infty_0(B_{1/2})$,
with~$\Gamma(x)\in[0,1]$ for any~$x\in\R^n$, and~$\Gamma(0)=1$.
We fix~$\epsilon>0$, to be taken arbitrarily small at the end of this
proof and set
\begin{equation}\label{78HJP}
\eta:=u(0)+\epsilon>0.\end{equation}
We define~$\Gamma_a(x):= 2\eta\,\Gamma(x)-a$.
Notice that if~$a>2\eta$, then~$\Gamma_a (x) \le2\eta-a<0\leq u(x)$
in the whole of~$\R^n$, hence the set $\{ \Gamma_a<u {\mbox{ in $\R^n$}}\}$ is not empty, and we can define
$$ a_* := \inf_{a\in \R} \{ \Gamma_a<u {\mbox{ in $\R^n$}}\}.$$
By construction
\begin{equation}\label{89U}
a_* \le 2\eta.
\end{equation}
If~$a<\eta$ then~$\Gamma_a(0)=2\eta-a > \eta >u(0)$, therefore
\begin{equation}\label{U7-a-be}
a_* \ge \eta.
\end{equation}
Notice that
\begin{equation}\label{x0EX0}
{\mbox{$\Gamma_{a_*}\le u$ in the whole of $\R^n$.}}\end{equation}
We claim that 
\begin{equation}\label{x0EX}
{\mbox{there exists~$x_0\in \overline{B_{1/2}}$ such that $\Gamma_{a_*}(x_0)=u(x_0)$.}}
\end{equation}
To prove this, we suppose by contradiction that~$u>\Gamma_{a_*} $
in~$\overline{B_{1/2}}$, i.e.
$$ \mu:=\min_{ \overline{B_{1/2}} } (u-\Gamma_{a_*} )>0.$$
Also, if~$x\in \R^n\setminus \overline{B_{1/2}}$, we have that
$$ u(x)-\Gamma_{a_*}(x) = u(x)
-2\eta\,\Gamma(x)+a_* = u(x)+a_*\ge a_*\ge\eta,$$
thanks to~\eqref{U7-a-be}.
As a consequence, for any~$x\in\R^n$,
$$ u(x)-\Gamma_{a_*}(x) \ge\min\{\mu,\eta\}=:\mu_*>0.$$
So, if we define~$a_\sharp:= a_*-(\mu_*/2)$,
we have that~$a_\sharp<a_*$ and
$$ u(x)-\Gamma_{a_\sharp}(x) = u(x)-\Gamma_{a_*}(x) -\frac{\mu_*}{2}
\ge \frac{\mu_*}{2} >0.$$
This is in contradiction with the definition of~$a_*$
and so it proves~\eqref{x0EX}.

{F}rom~\eqref{x0EX} we have that $x_0\in \overline{ B_{1/2}}$, hence ~$(-\Delta)^s u(x_0)\ge0$.
Also~$|(-\Delta)^s \Gamma_{a_*}(x)|=
2\eta\,|(-\Delta)^s\Gamma(x)|\le C\eta$, for any~$x\in\R^n$, therefore,
recalling~\eqref{x0EX0} and~\eqref{x0EX},
\begin{eqnarray*}
C\eta &\ge& (-\Delta)^s \Gamma_{a_*}(x_0) -(-\Delta)^s u(x_0)\\
&=& C(n,s)\,\mbox{P.V.} \int_{\R^n} 
\frac{ \big[ \Gamma_{a_*}(x_0)-\Gamma_{a_*}(x_0+y) \big]- 
\big[ u(x_0)-u(x_0+y)\big] }{|y|^{n+2s}}\,dy \\
&=& C(n,s)\,\mbox{P.V.} \int_{\R^n} 
\frac{ u(x_0+y)-\Gamma_{a_*}(x_0+y)}{|y|^{n+2s}}\,dy
\\ &\ge& C(n,s)\,\mbox{P.V.} \int_{B_1(-x_0)}
\frac{ u(x_0+y)-\Gamma_{a_*}(x_0+y)}{|y|^{n+2s}}\,dy
.\end{eqnarray*}
Notice now that if~$y\in B_1(-x_0)$, then~$|y|\le |x_0|+1<2$,
thus we obtain
$$ C\eta \ge \frac{C(n,s)}{2^{n+2s}}\,\int_{B_1(-x_0)}
\big[ u(x_0+y)-\Gamma_{a_*}(x_0+y) \big]\,dy.$$
Notice now that $\Gamma_{a_*}(x)= 2\eta\Gamma(x) -a_* \le \eta$, due to~\eqref{U7-a-be},
therefore we conclude that
$$ C\eta \ge \frac{C(n,s)}{2^{n+2s}}\,\left( \int_{B_1(-x_0)}
u(x_0+y)\,dy-\eta|B_1|\right).$$
That is,
using the change of variable~$x:=x_0+y$, recalling~\eqref{78HJP}
and renaming the constants, we have
$$ C\big(u(0)+\epsilon\big)=
C\eta\ge \int_{B_1} u(x)\,dx,$$
hence the desired result follows by sending~$\epsilon\to0$.
\end{proof}

\section{An $s$-harmonic function} \label{shf1}

We provide here an explicit
example of a function that is $s$-harmonic on the positive line $\R_+:=
(0,+\infty)$. Namely, we prove the following result: 
\begin{thm}\label{G1}
For any $x\in\R$, let $w_s(x):=x_+^s=\max\{x,0\}^s$. Then
	\[ (-\Delta)^s w_s (x) = \left\{
	\begin{matrix}
		-c_s |x|^{-s} & \mbox{ if  } x<0,\\
		0 &\mbox{ if } x>0,
	\end{matrix}
	\right.\]
for a suitable constant~$c_s>0$.
\end{thm}
\begin{center}
\begin{figure}[htpb]
	\hspace{0.8cm}
	\begin{minipage}[b]{0.95\linewidth}
	\centering
	\includegraphics[width=0.95\textwidth]{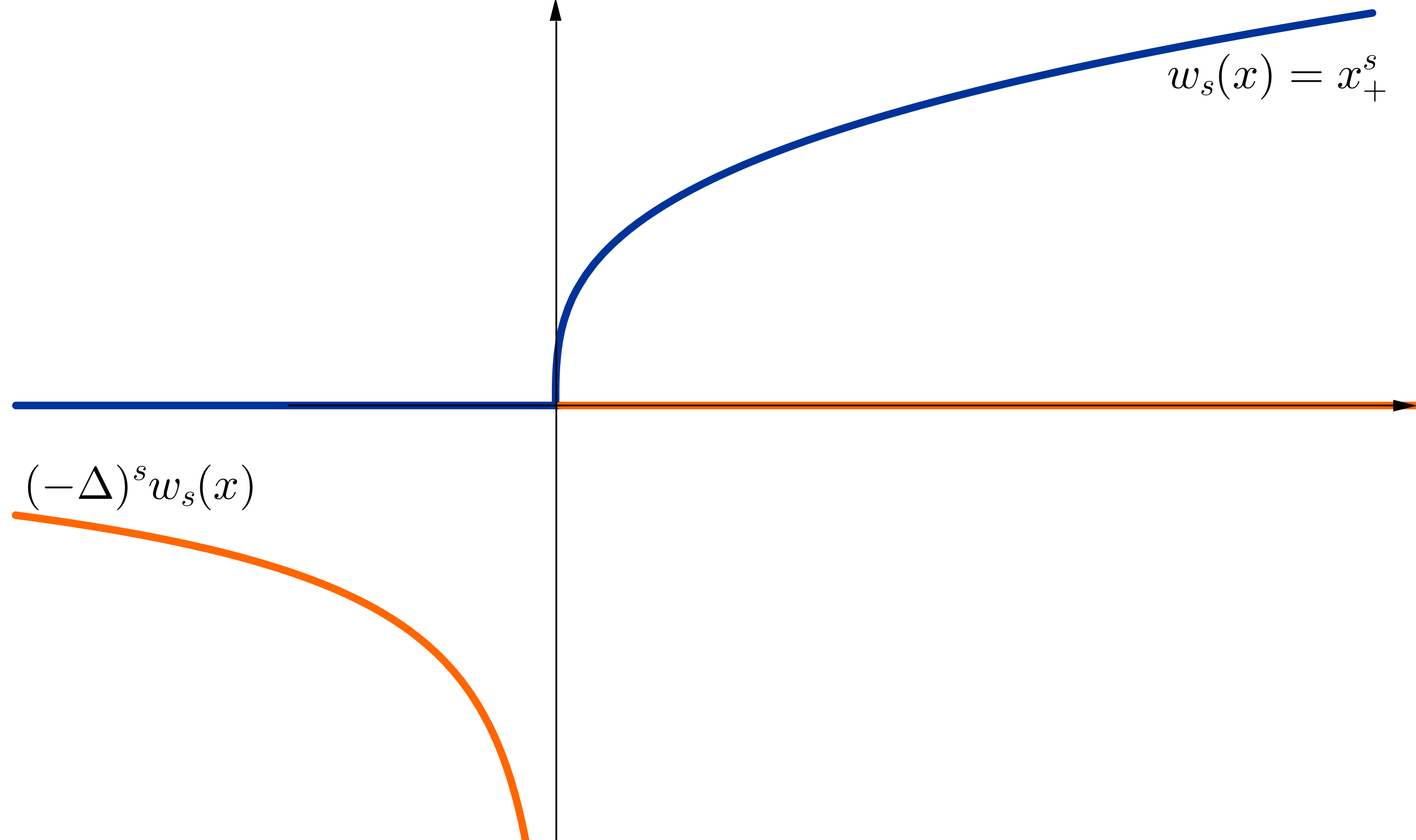}
	\caption{An $s$-harmonic function}   
	\label{fign:sh}
	\end{minipage}
\end{figure}
\end{center}
At a first glance, it may be quite surprising that the function~$x_+^s$ is $s$-harmonic in~$(0,+\infty)$, since such function is not smooth (but only continuous) uniformly up to the boundary, so let us try to give some heuristic explanations for it. 

We try to understand why the function~$x_+^s$ is $s$-harmonic in, say, the interval~$(0,1)$ when~$s\in(0,1]$. From the discussion in Section~\ref{pym}, we know that the $s$-harmonic function in~$(0,1)$ that agrees with~$x_+^s$ outside~$(0,1)$ coincides with the expected value of a payoff, subject to a random walk (the random walk is classical when~$s=1$ and it presents jumps when~$s\in(0,1)$). If~$s=1$ and we start from the middle of the interval, we have the same probability of being moved randomly to the left and to the right. This means that we have the same probability of exiting the interval~$(0,1)$ to the right (and so ending the process at~$x=1$, which gives~$1$ as payoff) or to the left~(and so ending the process at~$x=0$, which gives~$0$ as payoff). Therefore the expected value starting at~$x=1/2$ is exactly the average between~$0$ and~$1$, which is~$1/2$. Similarly, if we start the process at the point~$x=1/4$, we have the same probability of reaching the point~$0$ on the left and the point~$1/2$ to the right. Since we know that the payoff at~$x=0$ is~$0$ and the expected value of the payoff at~$x=1/2$ is~$1/2$, we deduce in this case that the expected value for the process starting at~$1/4$ is the average between~$0$ and~$1/2$, that is exactly~$1/4$. We can repeat this argument over and over, and obtain the (rather obvious)
fact that the linear function is indeed harmonic in the classical sense.

The argument above, which seems either
trivial or unnecessarily complicated in the classical case,
can be adapted when~$s\in(0,1)$ and it can give a qualitative picture of the corresponding~$s$-harmonic function. Let us start again the random walk, this time with jumps, at~$x=1/2$: in presence of jumps, we have the same probability of reaching the left interval~$(-\infty,0]$ and the right interval~$[1,+\infty)$. Now, the payoff at~$(-\infty,0]$ is~$0$, while the payoff at~$[1,+\infty)$ is {\em bigger} than~$1$. This implies that the expected value at~$x=1/2$ is the average between~$0$ and something bigger than~$1$, which produces a value larger than~$1/2$. When repeating this argument over and over, we obtain a concavity property enjoyed by the~$s$-harmonic functions in this case (the exact values prescribed in~$[1,+\infty)$ are not essential here, it would be enough that these values were monotone increasing and larger than~$1$).
\begin{center}
\begin{figure}[htpb]
	\hspace{0.8cm}
	\begin{minipage}[b]{1.00\linewidth}
	\centering
	\includegraphics[width=1.00\textwidth]{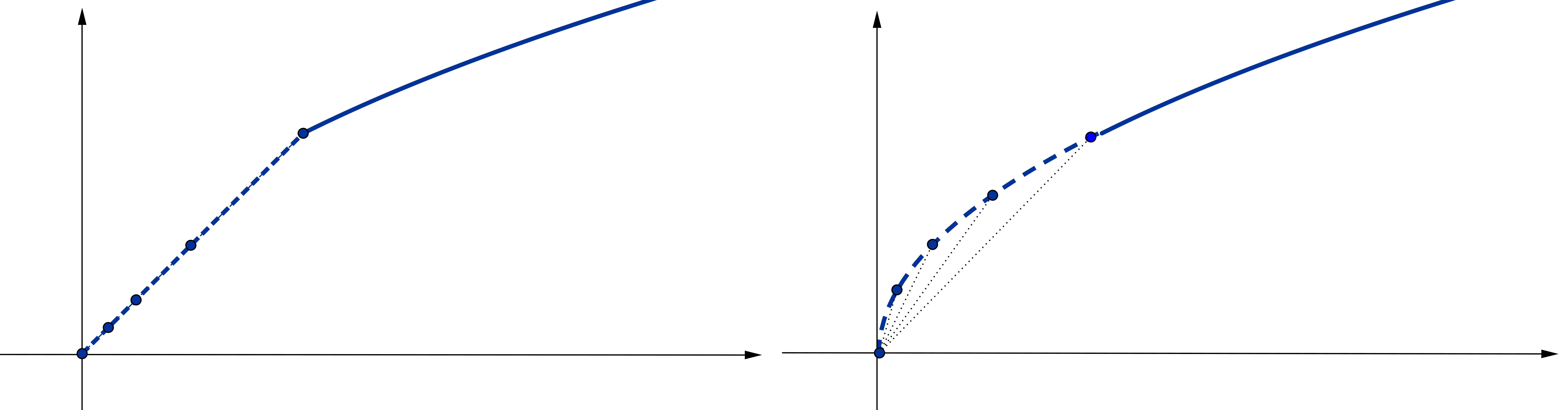}
	\caption{A payoff model: case $ s=1$ and $s \in (0,1)$}   
	\label{fign:py}
	\end{minipage}
\end{figure}
\end{center}

In a sense, therefore, this concavity properties and loss of Lipschitz regularity near the minimal boundary values is typical of nonlocal diffusion and it is due to the possibility of ``reaching far away points'', which may increase the expected payoff.

\bigskip

Now we present a complete, elementary proof
of Theorem \ref{G1}. The proof originated from some pleasant discussions
with Fernando Soria and it is based on some rather surprising
integral cancellations.
The reader who wishes to skip this proof
can go directly to Section~\ref{S:PGRAG}
on page~\pageref{PGRAG}. Moreover, a shorter, but
technically more advanced proof,
is presented in Appendix~\ref{1.APP-APP-TH}.

Here, we start with
some preliminary computations.  

\begin{lemma}\label{L1}
For any $s\in(0,1)$
\[ \int_0^{1} \frac{(1+t)^s+(1-t)^s-2}{t^{1+2s}}\,dt+ \int_1^{+\infty}\frac{(1+t)^s}{t^{1+2s}}\, dt = \frac{1}{s}.\]
\end{lemma}

\begin{proof} Fixed $\eee>0$, we integrate by parts:
	\eqlab{ \label{6f7gg}		
	\al \int_\eee^{1}  \frac{(1+t)^s+(1-t)^s-2}{t^{1+2s}}\,dt \;\\\ =\al  -\frac{1}{2s} \int_\eee^{1} \Big[ (1+t)^s+(1-t)^s-2\Big] \frac{d}{dt} t^{-2s}\,dt \\
	 =&\;\frac{1}{2s} \bigg[\frac{ (1+\eee)^s+(1-\eee)^s-2 }{ \eee^{2s} } -2^s+2\bigg]
	+\frac{1}{2} \int_\eee^{1} \frac{ (1+t)^{s-1}-(1-t)^{s-1}}{ t^{2s}} \,dt \\ 
	 =&\;\frac{1}{2s} \left[ o(1) -2^s+2\right]+\frac{1}{2}\bigg( \int_\eee^{1} (1+t)^{s-1} t^{-2s}\,dt
 	-\int_\eee^{1}(1-t)^{s-1} t^{-2s}\,dt\bigg), }
		with~$o(1)$ infinitesimal as~$\eee\searrow0$.
Moreover, by changing variable~$\tilde t:=t/(1-t)$, that is~$t:=\tilde t/(1+\tilde t)$,
we have that
\begin{equation*}
	\int_\eee^{1} (1-t)^{s-1} t^{-2s}\,dt
= \int_{\eee/(1-\eee)}^{+\infty} (1+\tilde t)^{s-1}\tilde t^{-2s}\,d\tilde t.
\end{equation*}
Inserting this into \eqref{6f7gg} (and writing $t$ instead of~$\tilde t$ as variable of integration), we obtain
	\begin{equation}\label{98-1}
		\begin{split}
		&\int_\eee^{1} \frac{(1+t)^s+(1-t)^s-2}{t^{1+2s}}\,dt \\
		&\quad=\frac{1}{2s} \big[o(1)	-2^s+2\big]+\frac{1}{2}\bigg[\int_\eee^{1} (1+t)^{s-1} t^{-2s}\,dt	- \int_{\eee/(1-\eee)}^{+\infty}(1+t)^{s-1}t^{-2s}\,dt\bigg]\\
		&\quad=	\frac{1}{2s} \big[ o(1)	-2^s+2\big]+\frac{1}{2}\bigg[	\int_\eee^{\eee/(1-\eee)} (1+t)^{s-1} t^{-2s}\,dt- \int_1^{+\infty}(1+t)^{s-1}t^{-2s}\,dt\bigg] .\end{split}\end{equation}
Now we remark that
	\[ \int_\eee^{\eee/(1-\eee)} (1+t)^{s-1} t^{-2s}\,dt\le  \int_\eee^{\eee/(1-\eee)} (1+\eee)^{s-1} \eee^{-2s}\,dt
	=\eee^{2-2s}(1-\eee)^{-1}(1+\eee)^{s-1} ,\]
therefore
	\[\lim_{\eee\searrow0} \int_\eee^{\eee/(1-\eee)} (1+t)^{s-1} t^{-2s}\,dt =0.\]
So, by passing to the limit in~\eqref{98-1}, we get
	\begin{equation}\label{98c}
		\int_0^{1} \frac{(1+t)^s+(1-t)^s-2}{t^{1+2s}}\,dt=
		\frac{-2^s+2}{2s} -\frac{1}{2}
		\int_1^{+\infty} (1+t)^{s-1}t^{-2s}\,dt.
	\end{equation}
Now, integrating by parts we see that
	\begin{eqnarray*}
		 \frac{1}{2}\int_1^{+\infty} (1+t)^{s-1}t^{-2s}\,dt
	&=& \frac{1}{2s}\int_1^{+\infty} t^{-2s} \frac{d}{dt}(1+t)^s\,dt
		\\ &=&-\frac{2^s}{2s}+
		\int_1^{+\infty} t^{-1-2s} (1+t)^s\,dt.
	\end{eqnarray*}
By plugging this into~\eqref{98c} we obtain that
\[\int_0^{1} \frac{(1+t)^s+(1-t)^s-2}{t^{1+2s}}\,dt=
\frac{-2^s+2}{2s}+
\frac{2^s}{2s}-
\int_1^{+\infty} t^{-1-2s} (1+t)^s\,dt , \]
which gives the desired result.
\end{proof}

 {F}rom Lemma~\ref{L1} we deduce the following (somehow unexpected)
cancellation property:

\begin{cor}\label{L2}
Let~$w_s$ be as in the statement of Theorem~\ref{G1}. Then
$$ (-\Delta)^s w_s (1) = 0.$$
\end{cor}

\begin{proof} The function~$t\mapsto (1+t)^s+(1-t)^s-2$
is even, therefore
$$ \int_{-1}^1 \frac{(1+t)^s+(1-t)^s-2}{|t|^{1+2s}}\,dt=2
\int_{0}^1 \frac{(1+t)^s+(1-t)^s-2}{t^{1+2s}}\,dt.$$
Moreover, by changing variable~$\tilde t:=-t$, we have that
$$ \int_{-\infty}^{-1}\frac{(1-t)^s-2}{|t|^{1+2s}}\,dt
=\int_1^{+\infty}\frac{(1+\tilde t)^s-2}{\tilde t^{1+2s}}\,d\tilde t.$$
Therefore
\begin{eqnarray*}
&& \int_{-\infty}^{+\infty}\frac{w_s(1+t)+w_s(1-t)-2w_s(1)}{|t|^{1+2s}}\,dt\\
&=& 
\int_{-\infty}^{-1}\frac{(1-t)^s-2}{|t|^{1+2s}}\,dt
+
\int_{-1}^{1}\frac{(1+t)^s+(1-t)^s-2}{|t|^{1+2s}}\,dt
+
\int_{1}^{+\infty}\frac{(1+t)^s-2}{|t|^{1+2s}}\,dt
\\ &=&
2\int_{0}^{1}\frac{(1+t)^s+(1-t)^s-2}{t^{1+2s}}\,dt
+2\int_{1}^{+\infty}\frac{(1+t)^s-2}{t^{1+2s}}\,dt
\\ &=&
2\left[
\int_{0}^{1}\frac{(1+t)^s+(1-t)^s-2}{t^{1+2s}}\,dt
+\int_{1}^{+\infty}\frac{(1+t)^s}{t^{1+2s}}\,dt
-2\int_{1}^{+\infty}\frac{dt}{t^{1+2s}}
\right]
\\ &=&
2\left[\frac1s-2\int_{1}^{+\infty}\frac{dt}{t^{1+2s}}\right],
\end{eqnarray*}
where Lemma~\ref{L1} was used in the last line.
Since
$$ \int_{1}^{+\infty}\frac{dt}{t^{1+2s}} = \frac1{2s},$$
we obtain that
$$ \int_{-\infty}^{+\infty}\frac{w_s(1+t)+w_s(1-t)-2w_s(1)}{|t|^{1+2s}}\,dt=0,$$
that proves the desired claim.
\end{proof}

The counterpart of Corollary~\ref{L2} is given by the following
simple observation:

\begin{lemma}\label{L3}
Let~$w_s$ be as in the statement of Theorem~\ref{G1}. Then
$$ -(-\Delta)^s w_s (-1) > 0.$$
\end{lemma}

\begin{proof} We have that
$$ w_s(-1+t)+w_s(-1-t)-2w_s(-1)=(-1+t)_+^s+(-1-t)_+^s\ge0$$
and not identically zero, which implies the desired result.
\end{proof}

We have now all the elements to proceed to the proof of Theorem~\ref{G1}.
\begin{proof} [Proof of Theorem~\ref{G1}]

We let~$\sigma\in\{+1,-1\}$ denote the sign of a fixed $x\in\R\setminus\{0\}$.
We claim that
\begin{equation}\label{Si}\begin{split}
&\int_{-\infty}^{+\infty}
\frac{w_s(\sigma(1+t))+w_s(\sigma(1-t))-2w_s(\sigma)}{ |t|^{1+2s}}\,dt
\\=\;&\int_{-\infty}^{+\infty}
\frac{w_s(\sigma+t)+w_s(\sigma-t)-2w_s(\sigma)}{ |t|^{1+2s}}\,dt.
\end{split}\end{equation}
Indeed, the formula above is obvious when~$x>0$ (i.e. $\sigma=1$),
so we suppose~$x<0$ (i.e. $\sigma=-1$) and we change variable~$\tau:=-t$,
to see that, in this case,
\begin{equation*}
	\begin{split}
&\int_{-\infty}^{+\infty}
\frac{w_s(\sigma(1+t))+w_s(\sigma(1-t))-2w_s(\sigma)}{ |t|^{1+2s}}\,dt\\ =\; &
\int_{-\infty}^{+\infty}
\frac{w_s(-1-t)+w_s(-1+t)-2w_s(\sigma)}{ |t|^{1+2s}}\,dt \\
 =\; &\int_{-\infty}^{+\infty}
\frac{w_s(-1+\tau)+w_s(-1-\tau)-2w_s(\sigma)}{ |\tau|^{1+2s}}\,d\tau
\\=\; &\int_{-\infty}^{+\infty}
\frac{w_s(\sigma+\tau)+w_s(\sigma-\tau)-2w_s(\sigma)}{ |\tau|^{1+2s}}\,d\tau
,\end{split}
	\end{equation*}
thus checking~\eqref{Si}.

Now we observe that, for any~$r\in\R$,
$$ w_s(|x|r)=(|x|r)_+^s =|x|^s r_+^s =|x|^s w_s(r).$$
That is
$$ w_s(xr)=w_s(\sigma |x| r)=|x|^sw_s(\sigma r).$$ 
So we change variable~$y=tx$ and we obtain that
\begin{equation*}
	\begin{split}
& \int_{-\infty}^{+\infty}
\frac{w_s(x+y)+w_s(x-y)-2w_s(x)}{|y|^{1+2s}}\,dy\\=\; &
\int_{-\infty}^{+\infty}
\frac{w_s(x(1+t))+w_s(x(1-t))-2w_s(x)}{|x|^{2s} |t|^{1+2s}}\,dt
\\=\; &|x|^{-s} \int_{-\infty}^{+\infty}
\frac{w_s(\sigma(1+t))+w_s(\sigma(1-t))-2w_s(\sigma)}{ |t|^{1+2s}}\,dt
\\=\; & |x|^{-s} \int_{-\infty}^{+\infty}
\frac{w_s(\sigma+t)+w_s(\sigma-t)-2w_s(\sigma)}{ |t|^{1+2s}}\,dt,
\end{split}
	\end{equation*}
where~\eqref{Si} was used in the last line.
This says that
$$ (-\Delta)^s w_s(x)
= \left\{\begin{matrix}
|x|^{-s} \,(-\Delta)^s w_s(-1)& {\mbox{ if $x<0$,}}\\
|x|^{-s}\,(-\Delta)^s w_s(1)&{\mbox{ if $x>0$,}}\end{matrix}
\right.$$
hence the result in Theorem~\ref{G1} follows from Corollary~\ref{L2}
and Lemma~\ref{L3}.
\end{proof}

\section[All functions are locally $s$-harmonic up to a small error]{All functions are locally $s$-harmonic up to a small error } \label{afsh}

Here we will show
that $s$-harmonic functions can locally approximate any given function, without any geometric constraints.
This fact is rather surprising and it
is a purely nonlocal feature, in the sense that it has no classical
counterpart.
Indeed, in the classical setting, harmonic functions are quite rigid, for instance they cannot have a strict local maximum, and therefore cannot approximate a function with a strict local maximum. The nonlocal picture is, conversely, completely different, as the oscillation of a function ``from far'' can make the function locally harmonic, almost independently from its local behavior.

We want to give here some
hints on the proof of this approximation result:

\begin{thm}\label{ALL FUNCTIONS}
Let $k\in \N$ be fixed. Then for any $f\in C^k(\overline{B_1})$ and any $\eee>0$ there exists $R>0$ and $u\in H^s(\Rn)\cap C^s(\Rn)$ such that
	\begin{equation}\label{cc1} \begin{cases} \frlap u(x)= 0 \quad &\mbox{ in }  B_1\\
					u=0 \quad &\mbox{ in }  \Rn \setminus  B_R	
	\end{cases}\end{equation}
and \[ \|f-u\|_{C^k(\overline{B_1})} \leq \eee.\]					
\end{thm}

\begin{proof}[Sketch of the proof]
For the sake of convenience, we divide
the proof into three steps. Also, for simplicity, we give the sketch of the proof in the one-dimensional case. See \cite{DSV14} for the entire and more general proof.\\
\textbf{Step 1. Reducing to monomials}\\
Let $k\in \N$ be fixed. We use first of all the Stone-Weierstrass Theorem and we have that for any $\eee>0$ and any $f \in C^k \big ([0,1]\big)$ there exists a polynomial $P$ such that
\[ \|f-P\|_{C^k{(\overline{B_1})}} \leq \eee.\]
Hence it is enough to prove Theorem~\ref{ALL FUNCTIONS}
for polynomials. Then, by linearity, it is enough to prove it
for monomials. Indeed, if~$P(x)= \displaystyle \sum_{m=0}^N c_m x^{m} $
and one finds
an $s$-harmonic function  $u_m$ such that  
	\[  \|u_m - x^{m} \|_{C^k{(\overline{B_1})}} \leq \frac{\eee}{|c_m|}, \]
then by taking $u:=\displaystyle\sum_{m=1}^N c_m u_m $ we have that 
	\[ \| u-P\|_{C^k{(\overline{B_1})}} 
\leq \sum_{m=1}^N |c_m| \|u_m -x^{m} \|_{C^k{(\overline{B_1})}} \leq {\eee}.\] 
Notice that the function $u$ is still $s$-harmonic, since the fractional Laplacian is a linear operator. 
  \bigskip

\noindent \textbf{Step 2. Spanning the derivatives}\\ We prove
the existence of an $s$-harmonic function in $B_1$, vanishing outside a compact set and with arbitrarily large number of derivatives prescribed.
That is, we show that for any~$m\in\N$
there exist $R > r > 0$, a point~$x\in\R$ and a function~$u$
such that
\eqlab{ \label{SPA}
& (-\Delta)^s u=0 {\mbox{ in }} (x-r,x+r),\\
& u=0 {\mbox{ outside }} (x-R,x+R),}
and 
\eqlab{ \label{spnd} & D^j u(x)=0 {\mbox{ for any }} j\in\{0,\dots,m-1\},\\
\;&
D^m u(x)=1.} 
To prove this,
we argue by contradiction.

We consider $\mathcal Z$ to be the set of all pairs $(u,x)$ of $s$-harmonic functions in a neighborhood of~$x$,
and points $x\in \R$ satisfying \eqref{SPA}.
 To any pair, we associate the vector \[\big(u(x), Du(x), \dots, D^m u(x) \big) \in \R^{m+1}\] and take $V$ to be the vector space 
spanned by this construction, i.e.
$$ V:= \Big\{
\big(u(x), Du(x), \dots, D^m u(x) \big),\;{\mbox{ for }}(u,x)\in {\mathcal Z}
\Big\}.$$
Notice indeed that
\begin{equation}\label{LIy}
{\mbox{$V$ is a linear space.}}\end{equation}
Indeed, let~$V_1$, $V_2\in V$ and~$a_1$, $a_2\in\R$.
Then, for any~$i\in\{1,2\}$,
we have that~$V_i = \big(u_i(x_i), Du_i(x_i), \dots, D^m u_i(x_i) \big)$,
for some~$(u_i,x_i)\in {\mathcal Z}$, i.e.~$u_i$ is $s$-harmonic
in~$(x_i-r_i,x_i+r_i)$ and vanishes outside~$(x_i-R_i,x_i+R_i)$, for some~$R_i\ge r_i>0$.
We set
$$ u_3(x):= a_1 u_1(x+x_1)+a_2u_2(x+x_2).$$
By construction, $u_3$ is $s$-harmonic in~$(-r_3,r_3)$,
and it vanishes outside~$(-R_3,R_3)$,
with~$r_3:=\min\{r_1,r_2\}$
and~$R_3:=\max\{R_1,R_2\}$, therefore~$(u_3,0)\in {\mathcal Z}$.
Moreover
$$ D^j u_3(x)= a_1 D^j u_1(x+x_1)+a_2D^ju_2(x+x_2)$$
and thus
\begin{eqnarray*}&&
a_1 V_1+a_2 V_2 \\&=& a_1\big(u_1(x_1), Du_1(x_1), \dots, D^m u_1(x_1) \big)
+a_2 \big(u_2(x_2), Du_2(x_2), \dots, D^m u_2(x_2) \big)
\\ &=& \big(u_3(0), Du_3(0), \dots, D^m u_3(0) \big).
\end{eqnarray*}
This establishes~\eqref{LIy}.

Now, to complete the proof of Step 2, it is enough to show that
\begin{equation}\label{PaV}
V =\R^{m+1}.\end{equation}
Indeed, if~\eqref{PaV}
holds true, then in particular~$(0,\dots, 0,1)\in V$,
which is the desired claim in Step~2.

To prove~\eqref{PaV}, we argue by contradiction: if
not, by~\eqref{LIy}, we have that~$V$ is a proper subspace
of~$\R^{m+1}$ and so it lies in a hyperplane.

Hence there exists a vector $c=(c_0, \dots, c_{m}) 
\in\R^{m+1}\setminus \{0\}
$ such that
$$ V\subseteq \left\{ \zeta\in\R^{m+1} {\mbox{ s.t. }} c\cdot\zeta=0\right\}.$$
That is, taking $(u,x)\in \mathcal Z$, the vector~$c=(c_0, \dots, c_{m})$ is
orthogonal to any vector $\big(u(x), Du(x), \dots, D^m u(x) \big)$, namely
\[ \sum_{j \leq m } c_j D^{j} u(x) =0.\] 
To find a contradiction, we now choose an appropriate
$s$-harmonic function~$u$
and we evaluate it at an appropriate point~$x$.
As a matter of fact,
a good candidate for the $s$-harmonic function is $x^s_+$,
as we know from Theorem~\ref{G1}:
strictly speaking, this function is not allowed
here, since it is not compactly supported,
but let us say that one can construct a compactly
supported $s$-harmonic function with the same 
behavior near the origin. With this slight caveat set aside,
we compute for a (possibly small)~$x$ in~$(0,1)$:
\[ D^{j} x^s=
s(s-1)\dots (s-j+1) x^{s-j} \] and multiplying the sum with $x^{m-s}$ (for $x\neq 0$) we have that 
	\[ \displaystyle \sum_{j\leq m} c_j  s(s-1)\dots (s-j+1) x^{m-j}  =0.\] But since $s\in (0,1)$ the product $ s(s-1)\dots (s-j+1)$ never vanishes. Hence the polynomial is identically null if and only if $c_j=0$ for any $j$, and we reach a contradiction.
This completes the proof of the existence of a function~$u$
that satisfies~\eqref{SPA} and~\eqref{spnd}.

\noindent \textbf{Step 3. Rescaling argument and completion of the proof}\\
By Step 2,
for any $m\in\N$ 
we are able to construct
a locally $s$-harmonic
function $u$ such that $u(x)=x^m +\mathcal{O}(x^{m+1})$ near the origin
(up to a translation).
By considering the blow-up 
\[ u_\lambda (x) = \frac{u(\lambda x) }{\lambda^m} = x^m +\lambda \mathcal{O}(x^{m+1})\]
 we have that for $\lambda$ small, $u_\lambda$ is arbitrarily close to the 
monomial $x^m$. As stated in Step 1, this concludes the proof
of Theorem~\ref{ALL FUNCTIONS}.
\end{proof}

It is worth pointing out that the flexibility
of $s$-harmonic functions given by Theorem~\ref{ALL FUNCTIONS}
may have concrete consequences.
For instance, as a byproduct of Theorem~\ref{ALL FUNCTIONS},
one has that a biological population with nonlocal dispersive attitudes
can better locally adapt to a given distribution of
resources (see e.g. Theorem~1.2 in~\cite{MASSACCESI}).
Namely, nonlocal biological species may efficiently use
distant resources and they can fit to the resources available nearby
by consuming them (almost) completely,
thus making more difficult for a different
competing species to come into place.

\section {A function with constant fractional Laplacian on the ball}\label{ctfrlap}
We complete this chapter with the explicit computation of the fractional Laplacian of the function $ \mathcal U(x) = (1-|x|^2)_+^s$. In $B_1$ we have that
	\[ \frlap \mathcal U(x) = C(n,s)\frac{\omega_n}{2} B(s, 1-s),\]
	where $C(n,s)$ is introduced in \eqref{cnsgalattica} and $B$ is the special Beta function (see section 6.2 in \cite{ABRAMOWITZ}). Just to give an idea of how such computation can be obtained, with small modifications respect to \cite{DYDA,DYDA1} we go through the proof of the result. The reader can find the more general result, i.e. for $\mathcal U(x)= (1-|x|^2)_+^p$ for $p>-1$, in the above mentioned \cite{DYDA,DYDA1}. 	
	  
 Let us take $u \colon \R \to \R$ as $u(x) = (1-|x|^2)^s_+$. Consider the regional  fractional Laplacian restricted to $(-1,1)$ 
	\[\Lm u(x) := P.V. \int_{-1}^1 \frac{u(x)-u(y)}{|x-y|^{1+2s}}\, dy  \] and compute its value at zero. By symmetry we have that
	\[	 		\Lm u(0) = 2 \lim_{\eee\to 0}\int_{\eee}^1 \frac{1-(1-y^2)^s}{y^{1+2s}}\, dy . 	\]
 Changing the variable $\omega= y^2$ and integrating by parts we get that
	\[
		\begin{split}
		\Lm u(0) =& \; 2 \lim_{\eee\to 0} \bigg( \int_{\eee}^1 y^{-1-2s}\, dy - \int_{\eee}^1 (1-y^2)^s y^{-1-2s}\, dy \bigg) \\
		 = & \; -\frac{1}{s} + \lim_{\eee\to 0} \bigg( \frac{\eee^{-2s}}{s}  - \int_{\eee^2}^1 (1-\omega)^s \omega^{-s-1}  \, d\omega\bigg) \\
	= & \; -\frac{1}{s} + \lim_{\eee\to 0} \bigg( \frac{\eee^{-2s}-\eee^{-2s} (1-\eee^2)^s}{s} + \int_{\eee^2}^1 \omega^{-s}(1-\omega)^{s-1}\, d\omega \bigg).
		\end{split} \] 
Using the integral representation of the Gamma function (see \cite{ABRAMOWITZ}, formula 6.2.1), i.e. 
\[ B(\alpha, \beta)= \int_0^1 t^{\alpha-1} (1-t)^{\beta-1} \, dt,\] it yields that
	\[
		\Lm u(0) =B(1-s,s) -\frac{1}{s}.
	\]	
	For $x\in B_1$ we use the change of variables $ \omega=\frac{x-y}{1-xy}$. We obtain that
\begin{equation}
		\begin{split}\label{llpux}
		\Lm u(x) =  &\; P.V. \int_{-1}^1 \frac{ (1-x^2)^s -(1-y^2)^s}{|x-y|^{1+2s}} \, dy \\
			= &\; (1-x^2)^{-s} P.V. \int_{-1}^1 \frac{(1-\omega x)^{2s-1} -(1-\omega^2)^s(1-\omega x)^{-1}} {|\omega|^{2s+1}}\, d\omega\\
			= &\; (1-x^2)^{-s}P.V. \Bigg( \int_{-1}^1 \frac{1-(1-\omega^2)^s}{|\omega|^{2s+1}} \, d\omega + \int_{-1}^1 \frac{(1-\omega x)^{2s-1}-1}{|\omega|^{2s+1}} \, d\omega    \\
			&\;  +\int_{-1}^1 \frac{(1-\omega^2)^s \Big ( 1- (1-\omega x)^{-1} \Big)}{|\omega|^{2s+1}} \, d\omega  \Bigg) \\
		=&\; (1-x^2)^{-s}\Bigg(  \Lm u(0) + J(x) +I(x) \Bigg),
		\end{split}
	\end{equation}
where we have recognized the regional fractional Laplacian and denoted
	\[ \begin{split} J(x) := &\;P.V. \int_{-1}^1 \frac{(1-\omega x)^{2s-1}-1}{|\omega|^{2s+1}} \, d\omega \qquad \qquad \mbox{ and } \\
	 I(x) :=&\; \text{P.V.} \int_{-1}^1 \frac{1-(1-\omega x)^{-1} } {|\omega|^{2s+1}} (1-\omega^2)^s \, d\omega. \end{split} \]
In $J(x)$ we have that 
	\[ \begin{split}
		J(x)  =&\; P.V. \bigg( \int_{-1}^1 \frac{(1-\omega x)^{2s-1}}{|\omega|^{2s+1}}  \, d\omega - \int_{-1}^1 |\omega|^{-1-2s} \, d\omega \bigg)\\
	=&\;  \lim_{\eee \to 0}\bigg( \int_{\eee}^1\frac{(1+\omega x)^{2s-1} +(1-\omega x)^{2s-1} }{|\omega|^{2s+1}} \, d\omega -  2 \int_{\eee}^1 \omega^{-1-2s} \, d\omega \bigg) .\end{split}  \] With the change of variable $t=\displaystyle \frac{1}{\omega}$
	\begin{equation}\label{jxah} \begin{split}
		J(x)  =&\; \frac{1}{s} + \lim_{\eee \to 0}   \bigg(\int_1^{1/\eee} \Big[ (t+x)^{2s-1} + (t-x)^{2s-1} \Big] \, dt - \frac{\eee^{-2s} }{s}\bigg)\\
		=&\; \frac{1}{s}  - \frac{(1+x)^{2s} +(1-x)^{2s}}{2s} + \frac{1}{2s} \lim_{\eee \to 0}\frac{ (1+\eee x)^{2s} +(1-\eee x)^{2s}   -2}{\eee^{2s}}\\
		=&\; \frac{1}{s}  - \frac{(1+x)^{2s} +(1-x)^{2s}}{2s}.
		\end{split} \end{equation}
To compute $I(x)$, with a Taylor expansion of $ (1-\omega x)^{-1}$ at $0$ we have that  
	\[	I(x) =  \text{P.V.} \int_{-1}^1 \displaystyle \frac{-\sum_{k=1}^\infty (x\omega)^k } {|\omega|^{2s+1}} (1-\omega^2)^s \, d\omega.\]
The odd part of the sum vanishes by symmetry, and so
	\begin{equation*}
		\begin{split}
		I(x)=\;&  -2 \lim_{\eee\to 0} 	\int_{\eee}^1 \displaystyle \frac{\sum_{k=1}^{\infty} (x\omega)^{2k} } {\omega^{2s+1}} (1-\omega^2)^s \, d\omega\\
			=\;&- 2 \lim_{\eee\to 0} \displaystyle  \sum_{k=1}^{\infty} x^{2k}  \int_{\eee}^1 \omega^{2k-2s-1} (1-\omega^2)^s\, d\omega. \\
				\end{split}	
	\end{equation*}
We change the variable $t =\omega^2$ and integrate by parts to obtain
	\begin{equation*}
		\begin{split}
	  I(x)=	\;&  - \lim_{\eee\to 0}   \sum_{k=1}^{\infty} x^{2k}   \int_{\eee^2 }^1 t^{k-s-1} (1-t)^s\, dt,\\
	     = \;&  \sum_{k=1}^{\infty} x^{2k} \lim_{\eee\to 0}   \bigg[ \frac{ \eee^{2k-2s}(1-\eee^2)^s}{k-s} - \frac{s}{k-s}  \int_{\eee^2}^1 t^{k-s}(1-t)^{s-1} \, dt \bigg].
		\end{split}	
	\end{equation*}
For $k\geq 1$, the limit for $\eee$ that goes to zero is null, and using the integral representation of the Beta function, we have that
	\begin{equation*}	
			I(x) = \sum_{k=1}^{\infty} x^{2k} \frac{-s}{k-s} B(k+1-s,s) .
		\end{equation*}
We use the Pochhammer symbol defined as
		\begin{equation}\label{Pochh}
		(q)_k = \begin{cases}   
			1   &\mbox{ for } k = 0 ,\\
  		q(q+1) \cdots (q+k-1) &\mbox{ for } k > 0
		 \end{cases}
		\end{equation} and with some manipulations, we get
	\[	\begin{split}
			\frac{-s}{k-s}  B(k+1-s,s) =  &\; \frac{(-s) \Gamma(k+1-s) \Gamma(s)}{(k-s) \Gamma(k+1) }\\
			 = &\; \frac{(-s) \Gamma(k-s)\Gamma(s)}{k!}\\
			= &\; B(1-s,s) \frac{(-s)_k}{k!}. 
		\end{split}\]	
And so
\[ I(x) =  B(1-s,s)  \sum_{k=1}^{\infty} x^{2k} \frac{(-s)_k}{k!}.\] By the definition of the hypergeometric function (see e.g. page~211 in~\cite{Oberhettinger}) we obtain 
 \[ \begin{split}
 I(x)  =&\; -B(1-s,s) + B(1-s,s) \sum_{k=0}^{\infty}  { (-s)_k} \frac{x^{2k}}{k!} \\
  = &\; B(1-s,s) \bigg(  F\Big(-s,\frac{1}{2},\frac{1}{2}, x^2\Big)- 1\bigg).\end{split}\]
Now, by 15.1.8 in \cite{ABRAMOWITZ} we have that
\[ F\Big(-s,\frac{1}{2},\frac{1}{2}, x^2\Big) = (1-x^2)^s\]
and therefore
	\begin{equation*}  I(x)  =  B(1-s,s) \Big( (1-x^2)^s-1\Big).\end{equation*}
Inserting this and \eqref{jxah} into \eqref{llpux} we obtain
	\begin{equation} \label{onedimfin} \Lm u(x) = B(1-s,s) -(1-x^2)^{-s}   \frac{(1+x)^{2s} +(1-x)^{2s}}{2s}.\end{equation}
Now we write the fractional Laplacian of $u$ as
	\[\begin{split}
	\frac{\frlap u(x) }{C(1,s)} =&\; \Lm u(x) + \int_{-\infty}^{-1}\frac {u(x)}{|x-y|^{1+2s}} \, dy + \int_1^{\infty}\frac {u(x)}{|x-y|^{1+2s}} \, dy \\
							 =&\;\Lm u(x)+ (1-x^2)^s \bigg(  \int_{-\infty}^{-1} |x-y|^{-1-2s} \, dy +\int_1^{\infty}|x-y|^{-1-2s}  \, dy \bigg) \\
							 =&\; \Lm u(x) +  (1-x^2)^s \frac {(1+x)^{-2s} +(1-x)^{-2s}}{2s}.
	\end{split}\]		 	
Inserting \eqref{onedimfin} into the computation, we obtain		 	
\begin{equation} \label{infrl122}	\frlap u(x) =C(1,s) B(1-s,s).\end{equation}
To pass to the $n$-dimensional case, without loss of generality and up to rotations, we consider $x=(0,0,\dots, x_n)$ with $x_n\geq 0$. We change into polar coordinates $x-y=th$, with $h\in \partial B_1$ and $t\geq0$. We have that
	\begin{equation}\label{lapcaln1} \begin{split}
			\frac{\frlap \mathcal U (x)}{C(n,s)} =&\; P.V. \int_{\Rn} \frac{(1-|x|^2)^s - (1-|y|^2)^s}{|x-y|^{n+2s}} \, dy\\
					=&\;\frac{1}{2}\int_{\partial B_1} \bigg( P.V. \int_{\R}  \frac{(1-|x|^2)^s - (1-|x+ht|^2)^s}{|t|^{1+2s}} \,dt\bigg) \, d\mathcal{H}^{n-1} (h).	
	\end{split}
		\end{equation}
Changing the variable $t= -|x|h_n +\tau \sqrt{|h_nx|^2-|x|^2+1}$, we notice that \[1-|x+ht|^2 = (1-\tau^2)(1-|x|^2 +|h_nx|^2)\] and so 
	 \[\begin{split}
	 P.V. \int_{\R}  &\frac{(1-|x|^2)^s - (1-|x+ht|^2)^s}{|t|^{1+2s}}  \, dt\\
	  = &P.V. \int_{\R} \frac{(1-|x|^2)^s - (1-\tau^2)^s(|h_nx|^2-|x|^2+1)^s } {\Big| -|x|h_n +\tau \sqrt{|h_nx|^2-|x|^2+1}\Big|^{1+2s}}\sqrt{|h_nx|^2-|x|^2+1} \, d\tau \\
	 	=& P.V. \int_{\R}  \frac{ \bigg(1-\displaystyle \frac{|x|^2h_n^2}{|h_nx|^2-|x|^2+1} \bigg)^s- (1-\tau^2)^s}{\bigg|\tau -\displaystyle\frac{|x|h_n}{ \sqrt{|h_nx|^2-|x|^2+1}}\bigg| ^{1+2s}}\, d\tau \\
	 		=& \frac{\frlap u\bigg( \displaystyle\frac{|x|h_n}{ \sqrt{|h_nx|^2-|x|^2+1}}\bigg)}{C(1,s)}\\
	 		= &  B(1-s,s),
	 \end{split}		\]
where the last equality follows from identity \eqref{infrl122}. Hence from \eqref{lapcaln1} we conclude that \[  \frlap \mathcal U(x) = C(n,s) B(1-s,s)\frac{\omega_n}{2}.\]
\chapter{Extension problems} \label{S:3}
We dedicate this part of the book
to the harmonic extension of the fractional Laplacian.
We present at first two applications, the water wave model and the
Peierls-Nabarro model related to crystal dislocations, making clear how the extension problem appears in these models. We conclude this part
by discussing\footnote{Though we do not develop this approach here,
it is worth mentioning that extended problems arise naturally
also from the probabilistic interpretation described in Section~\ref{S:1}.
Roughly speaking, a stochastic process with jumps in~$\R^n$
can often be seen as the ``trace''
of a classical stochastic process in~$\R^n\times [0,+\infty)$
(i.e., each time that the classical stochastic process 
in~$\R^n\times [0,+\infty)$
hits~$\R^n\times\{0\}$ it induces a jump process over~$\R^n$).
Similarly, stochastic process with jumps may also be seen
as classical processes at discrete, random, time steps.}
in detail the extension problem via the Fourier transform.

The harmonic extension of the fractional Laplacian in the framework
considered here
is due to Luis Caffarelli and Luis Silvestre (we
refer to \cite{CS07} for details). We also recall that this extension procedure
was obtained by S.~A. Mol{\v{c}}anov and E.~Ostrovski{\u\i}
in~\cite{MO-OS} by probabilistic methods (roughly speaking
``embedding'' a long jump random walk in~$\R^n$
into a classical random walk in one dimension more, see Figure \ref{fign:rwmm}).
\begin{center}
\begin{figure}[htpb]
	\hspace{0.6cm}
	\begin{minipage}[b]{0.95\linewidth}
	\centering
	\includegraphics[width=0.95\textwidth]{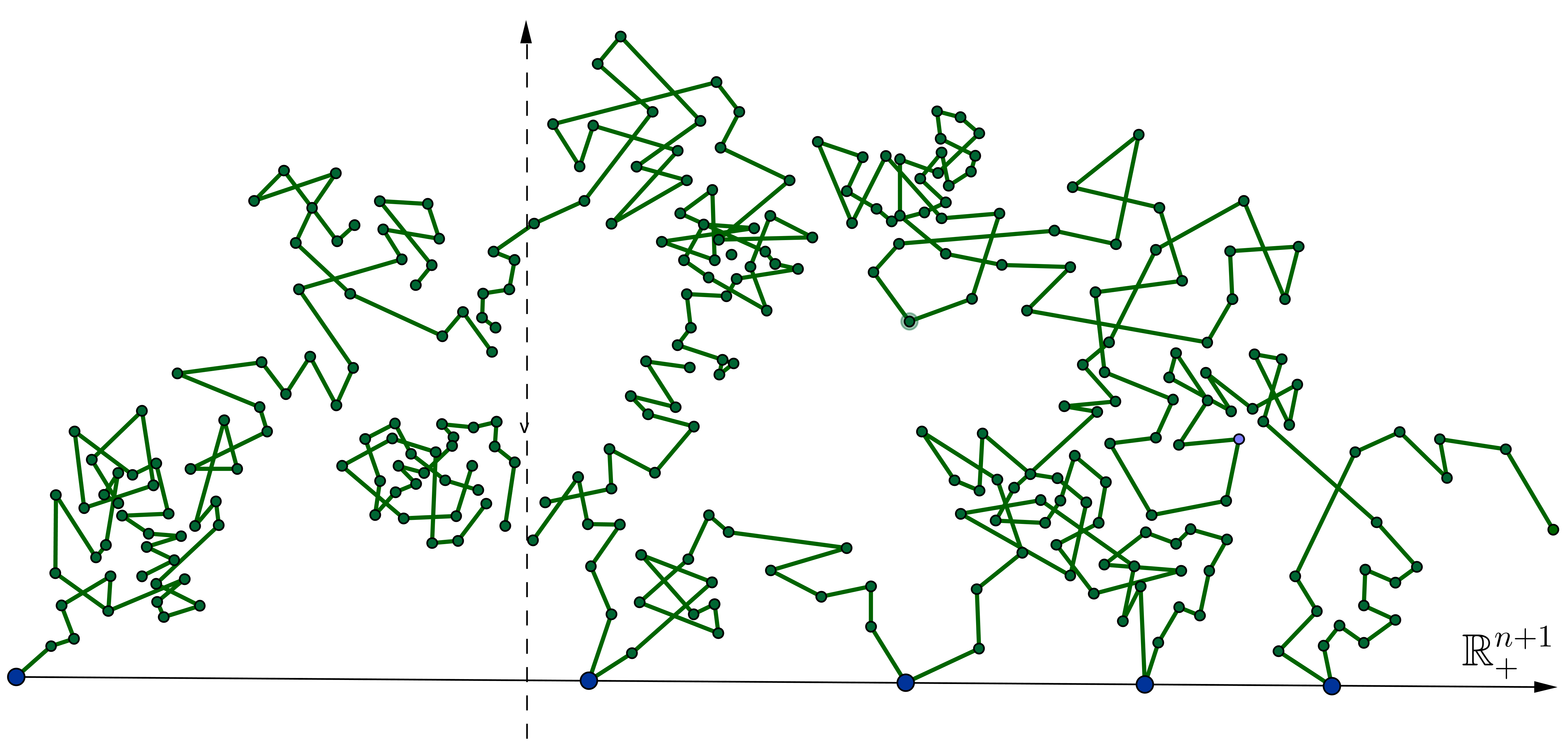}
	\caption{The random walk with jumps in $\Rn$ can be seen as a classical random walk in $\R^{n+1}$}   
	\label{fign:rwmm}
	\end{minipage}
\end{figure} 
\end{center}
The idea of this extension procedure
is that the nonlocal operator $\frlap $ acting on functions defined on $\Rn$ may be reduced to a local operator, acting on functions defined in the higher-dimensional half-space $\R^{n+1}_+ :=\Rn \times (0,+\infty).$ Indeed, take $U \colon \R^{n+1}_+  \to \R$ such that $U(x,0)=u(x)$ in $\Rn$, solution to the equation 	
\[ \text{div} \Big(y^{1-2s} \nabla U(x,y)\Big) =0 \quad \mbox{in}\quad \R^{n+1}_+.\] 
Then up to constants one has that\[ -\lim_{y\to 0^+} \Big( y^{1-2s} \partial_y U (x,y) \Big)= \frlap u(x).\]

\section{Water wave model}

Let us consider the half space $\R_{+}^{n+1} = {\Rn} \times (0,+\infty)$ endowed with the coordinates $x \in {\Rn}$ and $y\in (0,+\infty)$. We show that the half-Laplacian (namely when $s=  1/2 )$ arises when looking for a harmonic function in $\R_{+}^{n+1}$ with given data on $\Rn \times \{y=0\}$. Thus, let us consider the following local Dirichlet-to-Neumann problem: 
	\begin{equation*}
		\begin{cases}
		\Delta U = 0 & \text{ in }  \R_{+}^{n+1} ,\\
		U(x,0) =u(x) & \text{ for }  x\in  { \Rn}.
		\end{cases}
	\end{equation*}
The function $U$ is the harmonic extension of $u$, we write $U=E u$, and define the operator $\mathcal{L}$ as 
	\eqlab { \label{D1} \mathcal{L} u (x):= -\partial_y U (x,0) .}
We claim that 
	\eqlab{\label{D2} \mathcal{L} = \sqrt{-\Delta_x}, }
	in other words
	\[ \mathcal{L} ^2 = -\Delta_x. \]
Indeed, by using the fact that $E(\mathcal{L} u) = -\partial_y U$ (that can be proved, for instance, by using the Poisson kernel representation for the solution), we obtain that	
	\[ \begin{split}
		\mathcal{L}^2 u (x) =\; & \mathcal{L}\big( \mathcal{L}u \big) (x) \\
		=\; &-\partial_y E\big( \mathcal{L}u\big)  (x,0) \\
		=\; & -\partial_y \big( -\partial_y  U\big) (x,0) \\
		=\; &\big( \partial_{yy} U + \Delta_x U- \Delta_x U \big)(x,0)\\
		=\; &	\Delta U (x,0) - \Delta u (x)	\\
		=\; &  -\Delta u(x),
	\end{split}\]
which concludes the proof of~\eqref{D2}.

One  remark in the above calculation lies in the choice of the sign of the square root of the operator. Namely, if we set~$\tilde{\mathcal{L}}u(x):=\partial_y U(x,0)$, the same computation as above would give that~$\tilde{\mathcal{L}}^2 = -\Delta$. In a sense, there is no surprise that a quadratic equation offers indeed two possible solutions. But a natural question is how to choose the ``right'' one.

There are several reasons to justify the sign convention in \eqref{D1}. One 
reason
is given by spectral theory, that makes the (fractional) Laplacian a negative definite operator. Let us discuss a purely geometric justification, in the simpler $n=1$-dimensional case. We wonder how the solution of problem
\begin{equation}\label{EQ-D1-1}
 \begin{cases}
(-\Delta)^s u=\,1    &\mbox{ in }\;  (-1,1),\\
	 u=\,0   &\mbox{ in } \; \R \setminus (-1,1) \end{cases}.
\end{equation}
should look like in the extended variable~$y$. First of all, 
by Maximum Principle (recall Theorems~\ref{THM-MA-1}
and~\ref{THM-MA-1-STRONG}),
we have that~$u$ is positive\footnote{As a matter of fact,
the solution of~\eqref{EQ-D1-1} is explicit and it is given by~$(1-x^2)^s$,
up to dimensional constants. See~\cite{DYDA} for a list of functions
whose fractional Laplacian can be explicitly computed
(unfortunately, differently from the classical cases, explicit computations
in the fractional setting
are available only for very few functions).}
when~$x\in(-1,1)$
(since this is an $s$-superharmonic function,
with zero data outside). 

Then the harmonic extension~$U$ in~$y>0$ of a 
function~$u$ which is positive in~$(-1,1)$ and vanishes outside~$(-1,1)$ 
should have the shape of an elastic membrane over the halfplane~$\R^2_+$ 
that is constrained to the graph of~$u$ on the trace~$\{y=0\}$.
\begin{center}
\begin{figure}[htpb]
	\hspace{0.9cm}
	\begin{minipage}[b]{1.00\linewidth}
	\centering
	\includegraphics[width=1.00\textwidth]{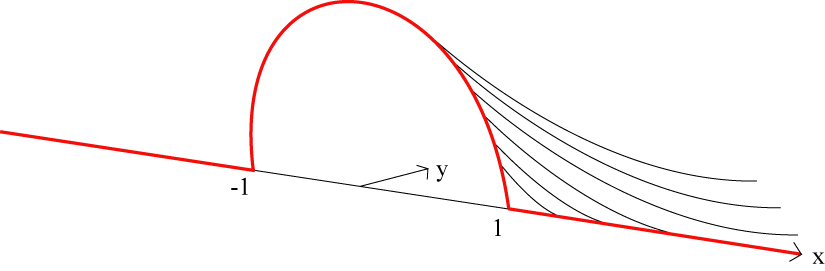}
	\caption{The harmonic extension}   
	\label{fign:FE}
	\end{minipage}
\end{figure} 
\end{center}
We give a picture
of this function~$U$ in Figure~\ref{fign:FE}. Notice from the picture that $\partial_y U(x,0)$ is negative, for any~$x\in(-1,1)$. Since $(-\Delta)^s u (x)$ is positive, we deduce that, to make our picture consistent with the maximum principle, we need to take the sign of~${\mathcal{L}}$ opposite to that of~$\partial_y U(x,0)$. This gives a geometric justification of~\eqref{D1}, which is only based on maximum principles (and on ``how classical harmonic functions look like''). 

\bigskip

\textbf{Application to the water waves.}

We show now that the operator $\mathcal{L}$ arises in the theory of water waves of irrotational, incompressible, inviscid fluids in the small amplitude, long wave regime. 

Consider a particle moving in the sea, which is, for us, the space $\Rn \times (0,1)$, where the bottom of the sea is set at level $1$ and the surface at level $0$ (see Figure \ref{fign:ww}). The velocity of the particle is $v\colon \Rn \times (0,1) \to \R^{n+1}$ and we write $v(x,y)=\big(v_x(x,y),v_y(x,y)\big)$, where $v_x \colon \Rn \times(0,1) \to \Rn$ is the horizontal component and $v_y\colon \Rn \times(0,1)\to \R$ is the vertical component. We are interested in the vertical velocity of the water at the surface of the sea. 
\begin{center}
\begin{figure}[htpb]
	\hspace{0.6cm}
	\begin{minipage}[b]{0.85\linewidth}
	\centering
	\includegraphics[width=0.85\textwidth]{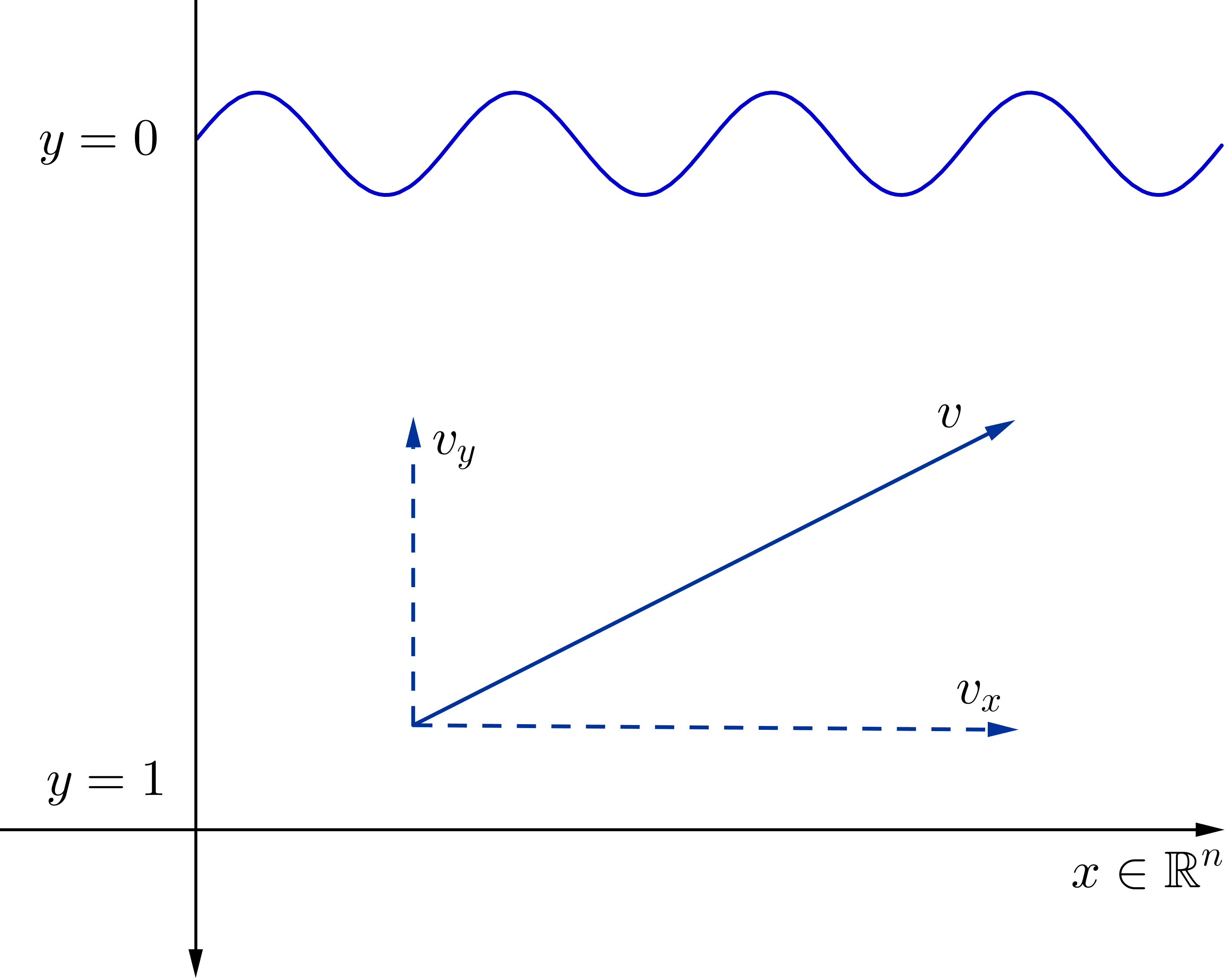}
	\caption{The water waves model}   
	\label{fign:ww}
	\end{minipage}
\end{figure} 
\end{center}
In our model, the water is incompressible, thus $\text{div } v=0 $ in $\Rn \times (0,1)$.  Furthermore, on the bottom of sea (since water cannot penetrate into the sand), the velocity has only a non-null horizontal component, hence $v_y(x,1)=0$. Also, in our model we
assume that there are no vortices: at a mathematical level,
this gives that~$v$ is irrotational, thus we may write it as the gradient of a function $U \colon \R^{n+1} \to \R$. This says that the vertical component of the velocity at the surface of the sea is $ v_y(x,0)=\partial_y U(x,0)$.
We are led to the problem 
	\begin{equation} \label{ww}
		\begin{cases}
		\Delta U = 0 & \text{ in }  \R_{+}^{n+1} ,\\
		 \partial_yU (x,1) =0 & \text{ in }  {\Rn}, \\
		U(x,0) =u(x) & \text{ in }  {\Rn}.
		\end{cases}
	\end{equation}

Let $\mathcal {L}$ be, as before, the operator $\mathcal{L}u(x):=-\partial_yU(x,0)$.
We solve the problem \eqref{ww} by using the Fourier transform and, up to a normalization factor, we obtain that
	\[ \mathcal{L}u=  \mathcal{F}^{-1} \bigg( |\xi| \frac{ e^{|\xi|} - e^{-|\xi|}}{e^{|\xi|}+e^{-|\xi|}} \widehat{u}(\xi) \bigg).\]
Notice that for large frequencies $\xi$, this operator is asymptotic to the square root of the Laplacian: 
	\[ \mathcal{L} u \simeq \mathcal{F}^{-1} \bigg( |\xi| \widehat{u}(\xi) \bigg) = \sqrt{-\Delta} u.\]

The operator $\mathcal{L}$ in the two-dimensional case has an interesting property, that is in analogy to a conjecture of
De Giorgi (the forthcoming Section \ref{sbsdg} 
will give further details about it):
more precisely,
one considers
entire, bounded, smooth, monotone solutions of the 
equation $\mathcal{L} u= f(u)$ for given $f$, and proves
that the solution only depends on one variable. More precisely:
\begin{thm}\label{DLV}
Let $f \in C^1(\R)$ and $u$ be a bounded smooth solution of 
 	\begin{equation*}
		\begin{cases}
		\mathcal{L}u = f(u)  & \text{ in } \R^2,\\
		\partial_{x_2} u >0   & \text{ in } \R^2.
		\end{cases}
	\end{equation*}
Then there exist a direction $\omega \in S^1$ and a function $u_0\colon \R \to \R$ such that, for any $x \in \R^2$, 
	\[ u(x)=u_0(x \cdot \omega).\]
\end{thm}
See Corollary 2 in \cite{LV09} for a proof of Theorem \ref{DLV} and to Theorem 1 in \cite{LV09} for a more general result (in higher dimension).

\section{Crystal dislocation}

A crystal is a material whose atoms are displayed in a regular way.
Due to some impurities in the material or to an external stress,
some atoms may move from their rest positions. The system
reacts to small modifications by pushing back towards
the equilibrium. Nevertheless, slightly larger modifications may lead
to plastic deformations.
Indeed, if an atom dislocation is of the order
of the periodicity size of the crystal, it can be perfectly
compatible with the behavior of the material
at a large scale, and it can lead to a permanent modification.

Suitably superposed atom dislocations
may also produce macroscopic deformations of the material,
and the atom dislocations may be moved
by a suitable external force,
which may be more effective if it happens to be compatible with the periodic
structure of the crystal.

These simple considerations may be framed into a mathematical
setting, and they also have concrete
applications in many industrial branches
(for instance,
in the production of a soda can, in order to change the shape of an aluminium
sheet, it is reasonable to believe that applying the right force to it can be simpler and less expensive than melting the metal).

It is also quite popular (see e.g.~\cite{CATERPILLAR})
to describe
the atom dislocation motion in crystals
in analogy
with the movement of caterpillar 
(roughly speaking, it is less expensive for
the caterpillar to produce a defect in the alignment
of its body and to dislocate this displacement, rather
then rigidly translating his body on the ground).
\bigskip

The mathematical framework of crystal dislocation
presented here is related to the 
Peierls-Nabarro model, that is
a hybrid model in which a discrete dislocation 
occurring along a slide line is incorporated in a continuum medium.
The 
total energy in the Peierls-Nabarro model combines the elastic energy of 
the material in reaction to the single dislocations, and the potential energy of 
the misfit along the glide plane. The main result is that, at a 
macroscopic level, dislocations tend to concentrate at single points, 
following the natural periodicity of the crystal.
\begin{center}
	\begin{figure}[htpb]
	\hspace{0.8cm}
	\begin{minipage}[b]{0.85\linewidth}
	\centering
	\includegraphics[width=0.85\textwidth]{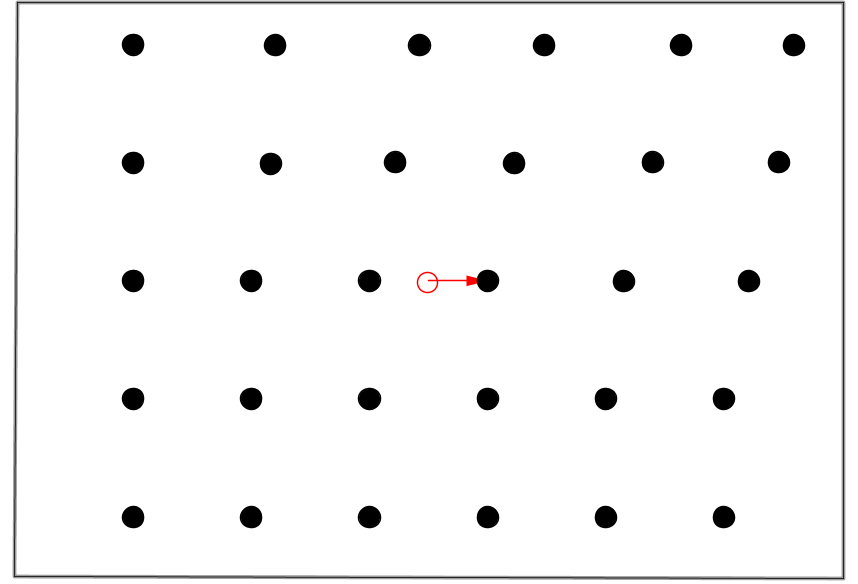}
	\caption{Crystal dislocation}   
	\label{fign:figure1}
	\end{minipage}
	\end{figure}
\end{center}

To introduce a mathematical framework
for crystal dislocation,
first, we ``slice'' the crystal with a plane. The mathematical setting will be then, by symmetry arguments, the half-plane $\R^2_{+} =\{ (x,y)\in   \R^2 \text{ s.t. } y \geq 0\}$ and the glide line will be the $x$-axis. In a crystalline structure, the atoms display periodically. Namely, the atoms on the $x$-axis have the preference of occupying integer sites. If atoms move out of their rest position due to a misfit, the material will have an elastic reaction, trying to restore the crystalline configuration. The tendency is to move back the atoms to their original positions, or to recreate, by translations, the natural periodic configuration. This effect may be modeled by defining $v^0(x):=v(x,0)$ to be the discrepancy between the position of the atom $x$ and its rest position. Then, the misfit energy is 
	\begin{equation}\label{YU-1}
\mathcal{M}(v^0):=\int_{\R} W\Big(v^0(x)\Big) \, dx,\end{equation}
where~$W$ is a smooth periodic potential,
normalized in such a way that~$W(u+1)=W(u)$ for any $u\in \R$ and~$0=W(0)<
W(u)$ for any~$u\in(0,1)$. We also assume that the minimum of~$W$
is nondegenerate, i.e.~$W''(0)>0$.

We consider
the dislocation function $v(x,y)$ on the half-plane $\R^2_{+}$. The
elastic energy of this model is given by
	\begin{equation}\label{YU-2}
 \mathcal{E}(v):=\frac{1}{2} \int_{\R^2_{+}} \Big|\nabla v(x,y)\Big|^2 \, dx \, dy.\end{equation}
The total energy of the system is therefore
	\begin{equation}\label{YU-3}
\mathcal{F} (v):= \mathcal{E}(v)+ \mathcal{M}(v^0) =\frac{1}{2} \int_{\R^2_{+}} \Big|\nabla v(x,y)\Big|^2 \, dx \, dy+\int_{\R} W\Big(v(x,0)\Big) \, dx.\end{equation}
Namely, the total energy of the system is the superposition
of the energy in~\eqref{YU-1},
which tends to settle all the atoms in their rest position
(or in another position equivalent to it from the
point of view of the periodic crystal),
and the energy in~\eqref{YU-2},
which is the elastic energy of the material itself.
\bigskip

Notice that some approximations have been performed
in this construction. For instance,
the atom dislocation changes the structure of the crystal itself:
to write~\eqref{YU-1}, one is making the assumption
that the dislocations of the single atoms do not destroy
the periodicity of the crystal at a large scale,
and it is indeed this ``permanent'' periodic structure
that produces the potential~$W$.

Moreover, in~\eqref{YU-2}, we are supposing that
a ``horizontal'' atom displacement
along the line~$\{y=0\}$ causes a horizontal
displacement at~$\{y=\epsilon\}$ as well.
Of course, in real life, if an atom at~$\{y=0\}$ moves, say, to the right,
an atom at level~$\{y=\epsilon\}$ is dragged to the right as well,
but also slightly downwards towards the slip line~$\{y=0\}$.
Thus, in~\eqref{YU-2} we are neglecting this ``vertical''
displacement. This approximation
is nevertheless reasonable since, on the one hand,
one expects the vertical displacement to be negligible with
respect to the horizontal one and, on the other hand,
the vertical periodic structure of the crystal tends to avoid
vertical displacements of the atoms outside the periodicity range
(from the mathematical point of view,
we notice that taking into account vertical displacements
would make the dislocation function vectorial,
which would produce a system of equations, rather than one single
equation for the system).

Also, the initial assumption of slicing the crystal
is based on some degree of simplification,
since this comes to studying dislocation curves in spaces
which are ``transversal'' to the slice plane.

In any case, we will take these (reasonable, after all)
simplifying assumptions
for granted, we will study their
mathematical consequences and see how the results
obtained agree with the physical experience.\bigskip

To find the Euler-Lagrange 
equation associated to~\eqref{YU-3}, let us consider a perturbation~$\phi \in C_0^{\infty}(\R^2)$, with $\varphi (x):=\phi(x,0)$ and let $v$ be a minimizer. Then
\[ \frac{d}{d\eee} \mathcal{F} (v+\eee \phi) \Big|_{\eee=0} =0,\]
which gives
\[ \int_{\R^2_+} \nabla v\cdot \nabla \phi \, dx \,dy + \int_{\R} W'(v^0)  \varphi \, dx=0.\]
Consider at first the case in which $\text{supp} \phi\cap \partial \R^2_+ = \emptyset$, thus $\varphi=0$. By the Divergence Theorem we obtain that 
	\[\int_{\R^2_+} \phi \, \Delta v \, dx \, dy=0 \quad \mbox{ for any } \phi\in C_0^{\infty}(\R^2),\] 
thus $\Delta v =0$ in $\R^2_+$. 
If $\text{supp} \phi\cap \partial \R^2_+ \neq \emptyset$ then we have that
	\[ \begin{split}
		0=&\;\int_{\R^2_+} \text{div} (\phi \nabla v)\, dx \, dy + \int_{\R} W'(v^0)  \varphi \, dx \\
		=&\; \int_{\partial \R^2_+} \phi \frac{\partial v}{\partial \nu} \, dx  + \int_{\R} W'(v^0)  \varphi \, dx \\
			=&\; - \int_{\R} \varphi \frac{\partial v}{\partial y} \, dx  + \int_{\R} W'(v^0)  \varphi \, dx 
			\end{split}
			\]
			for an arbitrary $\varphi \in C_0^{\infty}(\R)$ therefore $\displaystyle\frac{\partial v}{\partial y}(x,0)=W'(v^0(x))$ for $x \in \R$. 
Hence the critical points of $\mathcal{F}$ are solutions of the problem
	\begin{equation*}
		\begin{cases}
		\Delta v(x,y)=0 &\text{ for } x\in {\R} \text{ and } y>0, \\
		v(x,0)=v^0(x) &\text{ for  } x \in \R, \\
		\partial_y v(x,0)=W'\Big(v(x,0)\Big)  & \text{ for } x \in \R
		\end{cases}
	\end{equation*}
and up to a normalization constant, recalling \eqref{D1} and \eqref{D2}, we have that
	\[-\sqrt{-\Delta} v(x,0)=W'\big(v(x,0)\big), \text{ for any } x \in \R.\]
The corresponding parabolic evolution equation is $ \partial_t v (x,0)= -\sqrt{-\Delta} v(x,0)- W'\big(v(x,0)\big)$.\bigskip

After this discussion, one is lead to
consider the more general case of the fractional Laplacian $\frlap$ for any~$s\in(0,1)$
(not only the half Laplacian), and the corresponding parabolic equation
	\[\partial_t v=  - \frlap v - W'(v) +\sigma,\]
where $\sigma$ is a (small) external stress.

If we take the lattice of size $\epsilon$ and rescale $v$ and $\sigma$ as
	\[v_{\epsilon} (t,x)= v\bigg( \frac{t}{\epsilon^{1+2s}}, \frac{x}{\epsilon}\bigg)\quad \mbox{ and } \quad \sigma = \eee^{2s} \sigma \bigg( \frac{t}{\epsilon^{1+2s}}, \frac{x}{\epsilon}\bigg),\]
then the rescaled function satisfies 
	\begin{equation}
		\partial_t v_{\epsilon} = \displaystyle \frac{1}{\epsilon} \big( - \frlap v_{\epsilon} -\frac{1}{\epsilon^{2s}} W'(v_{\epsilon}) + \sigma \big)\text{ in } (0, +\infty) \times \R
		\label{pareq}
	\end{equation}
with the initial condition
	 \[v_{\epsilon}(0,x)= v_{\epsilon}^0(x) \text{ for } x \in \R.\]
To suitably choose the initial condition $v_{\epsilon}^0$, we introduce the basic layer\footnote{As a matter of fact,
the solution of~\eqref{allencahn} coincides
with the one
of a one-dimensional fractional Allen-Cahn equation, 
that will be discussed in further detail in the forthcoming
Section \ref{sbsac}.}
solution $u$, that is, the unique solution of the problem
	\begin{equation}
		\begin{cases}
		- \frlap u(x)=W'(u)  \quad \text { in } \R, \\
		u'>0 \text{ and } u(-\infty)=0, u(0) = 1/2, u(+\infty)=1.
		\label{allencahn}
		\end{cases}
	\end{equation}
For the existence of such solution and its main properties see \cite{PSV13} and \cite{CS15}. Furthermore, the solution decays polynomially at $\pm \infty$ (see \cite{DPV15} and \cite{DFV14}), namely
	\begin{equation}
		\bigg| u(x) - H(x) + \frac {1}{2sW''(0)}\frac{x}{|x|^{1+2s}} \bigg| \leq \frac{C}{|x|^{\vartheta}} \quad \text { for any } x \in \Rn,
		\label{layersol}
	\end{equation}
where $\vartheta>2s$ and~$H$ is the Heaviside step function
	\[H(x) =\begin{cases} 1, & x\geq 0\\
						0,& x<0.
			\end{cases} \]
We take the initial condition of the solution of~\eqref{pareq}
to be the superposition of transitions all occurring with the same orientation, i.e. we set
	\begin{equation} \label{cdin1} v_{\epsilon}(x,0):= \frac{\epsilon^{2s}}{W''(0)}\sigma(0,x)+ \sum_{i=1}^N u\bigg(\frac{x-x_i^0}{\epsilon}\bigg),\end{equation}
where $x_1^0,\dots, x_N^0$ are $N$ fixed points. 

\begin{center}\begin{figure}[htpb]
	\hspace{0.6cm}
	\begin{minipage}[b]{0.95\linewidth}
	\centering
	\includegraphics[width=0.95\textwidth]{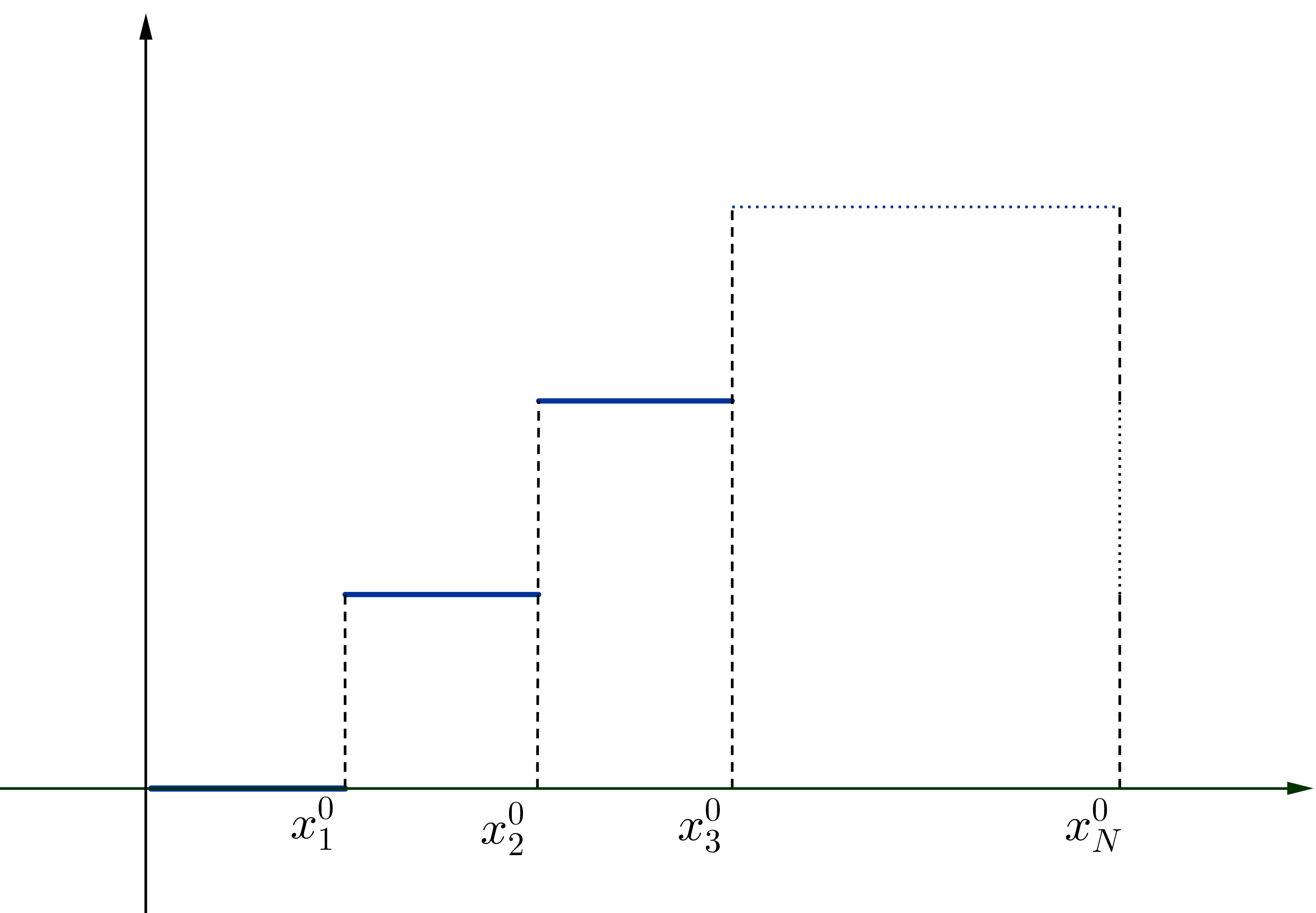}
	\caption{The initial datum when $\eee \to 0$}   
	\label{fign:cryind}
	\end{minipage}
	\end{figure}
\end{center}
The main result in this setting is that the solution~$v_\epsilon$ approaches,
as~$\epsilon\to0$, the superposition of step functions.
The discontinuities of the limit function occur at some points~$\big(
x_i(t)\big)_{i=1, \dots, N}$ which
move accordingly to the following\footnote{
The system of ordinary differential equations
in~\eqref{solx}
has been extensively studied in~\cite{MONNEAU-FORCADEL-IMBERT}.}
dynamical system
	\begin{equation}
		\begin{cases}
			\dot{x_i} = \gamma \bigg( -\sigma (t,x_i) +\displaystyle \sum_{j\neq i} \displaystyle \frac{x_i-x_j} {2s |x_i-x_j|^{2s+1}}\bigg)  \quad \text{ in } (0,+\infty) ,\\
			x_i(0)=x_i^0,
		\end{cases}
	\label{solx}
	\end{equation}
where 
	\begin{equation} \label{cdg}\gamma = \bigg(\int_{\R} (u')^2\bigg)^{-1}.\end{equation}

More precisely,
the main result obtained here is the following.

\begin{thm}\label{VISCOSITY}
There exists a
unique viscosity solution of 	
	\begin{equation*}
		\begin{cases}
		\partial_t v_{\epsilon} = \displaystyle \frac{1}{\epsilon} \Big( - \frlap v_{\epsilon} -\frac{1}{\epsilon^{2s}} W'(v_{\epsilon}) + \sigma \Big)&\text{ in } (0, +\infty) \times \R,\\
		v_{\epsilon}(0,x)=  \displaystyle \frac{\epsilon^{2s}}{W''(0)}\sigma(0,x)+ \sum_{i=1}^N u\bigg(\frac{x-x_i^0}{\epsilon}\bigg)&\text{ for } x \in \R
		\end{cases}
	\end{equation*}
such that 
	\begin{equation}\label{SENSE}
\lim_{\epsilon \to 0} v_{\epsilon} (t,x) = \sum_{i=1}^N H\big(x-x_i(t)\big),\end{equation}
where $\big(x_i(t)\big)_{i=1, \dots, N}$ is solution to \eqref{solx}.
\end{thm}

We refer to \cite{GM12} for the case $\displaystyle s=\frac{1}{2}$, 
to~\cite{DPV15} for the case 
$s>\displaystyle \frac{1}{2}$, and \cite{DFV14} for the case~$s<\displaystyle 
\frac{1}{2}$
(in these papers, it is also carefully stated
in which sense the limit in~\eqref{SENSE} holds true).

We would like to give now a formal (not rigorous) justification of the ODE system in \eqref{solx} that drives the motion of the transition layers.

\begin{proof}[Justification of ODE system \eqref{solx}]
We assume for simplicity that the external stress $\sigma $ is null. We use the notation $\simeq$ to denote the equality up to negligible terms in $\epsilon$. Also, we denote
	\[ u_i(t,x):=u\bigg(\frac{x-x_i(t)}{\epsilon}\bigg)\] 
and, with a slight abuse of notation 
	\[u'_i(t,x):= u'\bigg(\frac{x-x_i(t)}{\epsilon}\bigg).\]
By $\eqref{layersol}$ we have that the layer solution is approximated by
	\begin{equation}
	 	u_i(t,x) \simeq H\bigg(\frac{x-x_i(t)}{\epsilon}\bigg) - \frac{\epsilon^{2s}\big(x-x_i(t)\big)}{2s W''(0)\big|x-x_i(t)\big|^{1+2s}}.
		\label{ui}
	\end{equation}
We use the assumption that the solution $v_\epsilon$ is well approximated by the sum of $N$ transitions and write
	\[v_\epsilon(t,x)\simeq \sum_{i=1}^Nu_i(t,x) = \sum_{i=1}^N u\Big(\frac{x-x_i(t)}{\epsilon}\Big).\]
For that
	\[\partial_t v_\epsilon (t,x)= -\frac{1}{\epsilon} \sum_{i=1}^N u'_i(t,x) \dot{x_i}(t)\]
and, since the basic layer solution $u$ is the solution of $\eqref{allencahn}$, we have that 		
	\[ \begin{split}
		 -\frlap v_\epsilon \simeq \; &-\sum_{i=1}^N \frlap u_i(t,x) \\
					= \;&- \frac{1}{\epsilon^{2s}}\sum_{i=1}^N  \frlap u \Big(\frac{x-x_i(t)}{\epsilon}\Big) \\
					= \;& \frac{1}{\epsilon^{2s}} \sum_{i=1}^N W' \bigg(u\Big(\frac{x-x_i(t)}{\epsilon}\Big)\bigg)\\
					= \;& \frac{1}{\epsilon^{2s}} \sum_{i=1}^N  W' \big(u_i(t,x)\big).		
		\end{split} \]
Now, returning to the parabolic equation $\eqref{pareq}$ we have that
	\begin{equation}
		 -\frac{1}{\epsilon}\sum_{i=1}^N u'_i(t,x) \dot{x_i}(t) = \frac{1}{\epsilon^{2s+1} } \bigg( \sum_{i=1}^N  W' \big(u_i(t,x)\big) - W'\Big(\sum_{i=1}^Nu_i(t,x) \Big) \bigg).
		\label{upareq}
	\end{equation}
Fix an integer $k$ between $1$ and $N$, multiply $\eqref{upareq}$ by $u_k'(t,x)$ and integrate over $\R$. We obtain
	\begin{equation}
		\begin{split}
		  -\frac{1}{\epsilon}& \sum_{i=1}^N   \dot{x_i}(t) \int_{\R} u'_i(t,x) u_k'(t,x) \, dx  \\ 
		  \;&=  \frac{1}{\epsilon^{2s+1} }  \bigg(  \sum_{i=1}^N \int_{\R}   W' \big(u_i(t,x)\big) u'_k(t,x) \, dx - \int_{\R}  W'\Big(\sum_{i=1}^Nu_i(t,x)\Big)u_k'(t,x) \, dx  \bigg).
		\label{ukpar}
		\end{split}
	\end{equation}
We compute the left hand side
of~\eqref{ukpar}. First, we take the $k^{\text{th}}$ term of the sum (i.e.
we consider the case~$i=k$). By using the change of variables
	\begin{equation}
		y:= \frac{x-x_k(t)}{\epsilon}
		\label{chvar}
	\end{equation}
we have that
	\begin{equation}
		\begin{split}
		 -\frac{1}{\epsilon}\dot{x_k}(t) \int_{\R} (u_k')^2(t,x) \,dx =&\;-\frac{1}{\epsilon}\dot{x_k}(t) \int_{\R} (u')^2\bigg(\frac{x-x_k(t)}{\epsilon}\bigg) \,dx \\	
												=&\;-\dot{x_k}(t) \int_{\R} (u')^2 (y) \, dy \\
												= &\; -\frac{\dot{x_k}(t)}{\gamma},
		\end{split}
		\label{first}
	\end{equation}
where $\gamma$ is defined by \eqref{cdg}.

Then, we consider
the $i^{\text{th}}$ term of the sum on the left hand side
of~\eqref{ukpar}.
By performing again the substitution $\eqref{chvar}$, we see that this term
is
	\[\begin{split}
			-\frac{1}{\epsilon}\dot{x_i}(t) \int_{\R} u_i'(t,x)u_k'(t,x) \,dx = \;&- \frac{1}{\epsilon} \dot{x_i}(t) \int_{\R} u'\bigg(\frac{x-x_i(t)}{\epsilon}\bigg)u'\bigg(\frac{x-x_k(t)}{\epsilon}\bigg) \, dx\\
			= \;& -\dot{x_i}(t) \int_{\R} u'\bigg(y+\frac{x_k(t)-x_i(t)}{\epsilon}\bigg)u'(y) \, dy \\
			\simeq\; & \, 0,
			\end{split}\]
where, for the last equivalence we have used that for  $\epsilon$ small, $ \displaystyle u' \bigg(y + \frac{x_k(t)-x_i(t)}{\epsilon}\bigg)$ is  asymptotic to $u'(\pm\infty) =0$.

We consider the first member on the right hand side of the identity \eqref{ukpar}, and, as before, take the $k^{\text{th}}$ term of the sum. We do the substitution \eqref{chvar} and have that
	\[ \begin{split}
	\frac{1}{\epsilon} \int_{\R} W' \big(u_k(t,x) \big) u_k'(t,x) \, dx =&\;\int_{\R} W' \big(u(y) \big) u'(y) \, dy	\\	=&\; W\big(u(y)\big)\Big|_{-\infty}^{+\infty}  \\
			=&\;W(1)-W(0)=0\end{split}\]
by the periodicity of $W$.
Now we use \eqref{ui}, the periodicity of $W'$ and we perform a Taylor expansion, noticing that $W'(0)=0$. We see that
		\[ \begin{split}
			 		W' \big(u_i(t,x)\big) \simeq & \;W' \Bigg( H\bigg(\frac{x-x_i(t)}{\epsilon}\bigg) - \frac{\epsilon^{2s}\big(x-x_i(t)\big)}{2s W''(0)\big|x-x_i(t)\big|^{1+2s}}\Bigg)\\
			 	\simeq & \;  W' \bigg( - \frac{\epsilon^{2s}\big(x-x_i(t)\big)}{2s W''(0)\big|x-x_i(t)\big|^{1+2s}}\bigg)\\
				\simeq & \;\frac{-\epsilon^{2s} \big(x-x_i(t)\big) }{ 2s \big|x-x_i(t)\big|^{1+2s} } .
		\end{split}\]		
Therefore, the $i^{\text{th}}$ term of the sum on the right hand side of the identity \eqref{ukpar} for $i\neq k$, by using the above approximation and doing one more time the substitution \eqref{chvar},  for $\epsilon$ small becomes
	\begin{equation}	
		\begin{split}\label{second}
		\frac{1}{\epsilon}\int_{\R} W' \big(u_i(t,x) \big) u_k'(t,x) \, dx = & \; -\frac{1}{\epsilon} \int_{\R} \frac{\epsilon^{2s} \big(x-x_i(t)\big) }{ 2s \big|x-x_i(t)\big|^{1+2s}} u' \bigg(\frac{x-x_k(t)}{\epsilon} \bigg) \, dx \\	
		=&  \;- \int_\R  \frac{\epsilon^{2s} \big( \epsilon y +x_k(t)-x_i(t)\big) }{ 2s \big|\epsilon y +x_k(t)-x_i(t)\big|^{1+2s}} u' (y) \, dy \\
		\simeq & \;- \frac{\epsilon^{2s} \big( x_k(t)-x_i(t)\big) }{ 2s \big|x_k(t)-x_i(t)\big|^{1+2s}} \int_\R u' (y) \, dy \\	
		= & \; - \frac{\epsilon^{2s} \big( x_k(t)-x_i(t)\big) }{ 2s \big|x_k(t)-x_i(t)\big|^{1+2s}}.
		\end{split}
	\end{equation}
We also observe that, for $\epsilon$ small, the second member on the right hand side of the identity \eqref{ukpar}, by using the change of variables \eqref{chvar}, reads
		\[ \begin{split}
		  		\frac{1}{\epsilon}\int_{\R} W'\Big(\sum_{i=1}^Nu_i(t,x)\Big)& u_k'(t,x) \, dx  \\ = &\;\frac{1}{\epsilon}\int_\R W'\Big(u_k(t,x) + \sum_{i\neq k} u_i(t,x)\Big)u_k'(t,x) \, dx \\ 
		=&\; \int_\R W'\bigg(u(y) + \sum_{i\neq k} u\Big(y + \frac{x_k(t)-x_i(t)}{\epsilon}\Big)\bigg) u'(y) \, dy.
		\end{split}\]
For $\epsilon$ small, $\displaystyle u\bigg(y + \frac{x_k(t)-x_i(t)}{\epsilon}\bigg) $ is asymptotic either to $u(+\infty)=1$ for $x_k>x_i$, or to $u(-\infty)=0$ for $x_k<x_i$. By using the periodicity of $W$, it follows that
	\begin{equation*}
		 \frac{1}{\epsilon}\int_{\R} W'\Big(\sum_{i=1}^Nu_i(t,x)\Big)u_k'(t,x) \, dx  = \int_\R W'\Big(u(y) \Big)u'(y) \, dy= W(1) -W(0)=0,
	\end{equation*}
again by the asymptotic behavior of $u$.
Concluding, by inserting the results \eqref{first} and \eqref{second} into \eqref{ukpar} we get that 
	\begin{equation*}
		\frac{\dot{x_k}(t)}{\gamma} =  \sum_{i\neq k} \frac{  x_k(t)-x_i(t) }{ 2s \big|x_k(t)-x_i(t)\big|^{1+2s}},
	\end{equation*}
which ends the justification of the system \eqref{solx}.
\end{proof}

We recall that, till now,
in Theorem~\ref{VISCOSITY}
we considered the initial data as a superposition of transitions all 
occurring with the same orientation (see \eqref{cdin1}),
i.e. the initial dislocation is a monotone function
(all the atoms are initially moved to the right).

Of course, for concrete applications,
it is interesting to consider also
the case in which
the atoms may dislocate in both directions,
i.e.
the transitions can occur with different orientations
(the atoms may be initially 
displaced to the left or to the right of their equilibrium position).

To model the different orientations of the dislocations,
we introduce a parameter $\xi_i \in \{-1,1\}$
(roughly speaking~$\xi_i=1$ corresponds to a dislocation to the right
and~$\xi_i=-1$ to a dislocation to the left). 

The main result in this case is the following (see~\cite{PV15}):

\begin{thm}\label{VISCOSITY2}
There exists a viscosity solution of 	
	\begin{equation*}
		\begin{cases}
		\partial_t v_{\epsilon} = \displaystyle \frac{1}{\epsilon} \Big( - \frlap v_{\epsilon} -\frac{1}{\epsilon^{2s}} W'(v_{\epsilon}) + \sigma_{\epsilon} \Big)&\text{ in } (0, +\infty) \times \R,\\
		v_{\epsilon}(0,x)=  \displaystyle \frac{\epsilon^{2s}}{W''(0)}\sigma(0,x)+ \sum_{i=1}^N u\bigg(\xi_i \frac{x-x_i^0}{\epsilon}\bigg)&\text{ for } x \in \R
		\end{cases}
	\end{equation*}
such that 
	\[\lim_{\epsilon \to 0} v_{\epsilon} (t,x) = \sum_{i=1}^N H\Big(\xi_i \big(x-x_i(t)\big)\Big),\]
where $\big(x_i(t)\big)_{i=1, \dots, N}$ is solution to 
\begin{equation}
		\begin{cases}
			\dot{x_i} = \gamma \bigg( -\xi_i \sigma (t,x_i) + \displaystyle\sum_{j\neq i} \xi_i \xi_j \displaystyle \frac{x_i-x_j} {2s |x_i-x_j|^{2s+1}}\bigg)  \quad \text{ in } (0,+\infty) ,\\
			x_i(0)=x_i^0.
		\end{cases}
	\label{solx1}
	\end{equation} 
\end{thm}

We observe that
Theorem~\ref{VISCOSITY2} reduces to Theorem~\ref{VISCOSITY}
when~$\xi_1=\dots=\xi_n=1$. In fact, the case discussed
in Theorem~\ref{VISCOSITY2} is richer than the one in Theorem~\ref{VISCOSITY},
since, in the case of different initial orientations,
collisions can occur, i.e. it may happen that~$x_i(T_c)=x_{i+1}(T_c)$
for some~$i\in\{1,\dots,N-1\}$ at a collision time~$T_c$.

For instance, in the case $N=2$, for $\xi_1=1$ and $\xi_2=-1$
(two initial dislocations with different orientations)
we have that
	\[ \begin{split}&\text{ if } \sigma \leq 0  \mbox{ then } T_c\leq\frac{s\theta_0^{1+2s}}{(2s+1)\gamma},\\
		&\text{ if } \theta_0 < (2s \|\sigma\|_{\infty})^{-\frac{1}{2s}} \mbox{ then } T_c \leq \frac{s\theta_0^{1+2s}}{\gamma(1-2s \theta_0\|\sigma\|_{\infty} )},
		\end{split} \]
where~$\theta_0:=x_2^0-x_1^0$ is the initial
distance between the dislocated atoms.
That is, if either the external force has the right sign,
or the initial distance is suitably small with respect to the
external force, then the dislocation time is finite,
and collisions occur in a finite time
(on the other hand, when these conditions are violated,
there are examples in which collisions do not occur).\bigskip

This and more general cases of collisions, with precise
estimates on the collision times, are discussed in detail
in~\cite{PV15}.\bigskip

An interesting feature of the system
is that the dislocation function~$v_\epsilon$
does not annihilate
at the collision time. More precisely, in the appropriate
scale, we have that~$v_\epsilon$ at the collision time
vanishes outside the collision
points, but it still preserves a non-negligible
asymptotic contribution exactly
at the collision points.
A formal statement is the following (see~\cite{PV15}):

\begin{thm}\label{J78K}
Let~$N=2$ and assume that a
collision occurs. Let~$x_c$ be the collision point, namely $x_c=x_1(T_c)=x_2(T_c)$. Then 
\begin{equation} \label{C1} \lim_{t\to T_c} \lim_{\eee\to 0} v_\eee(t,x)=0 \quad \mbox { for any } \quad x\neq x_c,\end{equation} but
\begin{equation}\label{C2}\limsup_{\substack{ {t\to T_c}\\{\eee\to 0}}} v_\eee(t,x_c) \geq 1.\end{equation} 
\end{thm}

Formulas~\eqref{C1} and~\eqref{C2} describe what happens in the crystal at the collision time. On the one hand, formula~\eqref{C1} states that at any point that is not the collision point and at a large scale, the system relaxes at the collision time. On the other hand, formula~\eqref{C2} states that the behavior at the collision points at the collision time is quite ``singular''. Namely, the system does not relax immediately
(in the appropriate scale). As a matter of fact, in related numerical simulations
(see e.g. \cite{RAABE}) one may notice that the dislocation
function may persists after collision and, in higher dimensions,
further collisions may change the topology of the dislocation curves.

What happens is that a slightly larger time is needed before the system relaxes exponentially fast: a detailed description of this relaxation phenomenon
is presented in \cite{RELAX}. For instance,
in the case~$N=2$,
the dislocation function decays to zero
exponentially fast, slightly after collision, as given by the following
result:

\begin{thm}\label{thmexponentialdecay}
Let~$N=2$, $\xi_1=1$, $\xi_2=-1$, and let
$v_\epsilon$ be the solution given by Theorem~\ref{VISCOSITY2},
with~$\sigma\equiv0$.
Then there exist
$\epsilon_0>0$, $c>0$, $T_\epsilon>T_c$ and~$\rho_\epsilon>0$
satisfying
\begin{eqnarray*}
&& \lim_{\epsilon\to0} T_\epsilon=T_c\\
{\mbox{and }}&& \lim_{\epsilon\to0}
\varrho_\epsilon=0\end{eqnarray*}
such that for any $\epsilon<\epsilon_0$ we have
\begin{equation}\label{4.23bis}
|v_\epsilon (t,x)|\leq \varrho_\epsilon
e^{c\frac{T_\epsilon-t}{\epsilon^{2s+1}}},\quad\text{for any }x\in\R\text{ and
}t\geq T_\epsilon.\end{equation} \end{thm}

The estimate in~\eqref{4.23bis} states, roughly speaking, that
at a suitable time~$T_\epsilon$ (only slightly bigger than the collision
time~$T_c$) the dislocation function gets below a small threshold~$\rho_\epsilon$,
and later it decays exponentially fast (the constant of this
exponential becomes large when $\epsilon$ is small).

The reader may compare Theorem~\ref{J78K}
and~\ref{thmexponentialdecay} and notice that different asymptotics
are considered by the two results.
A result
similar to Theorem~\ref{thmexponentialdecay}
holds for a larger number of dislocated atoms.
For instance,
in the case of three atoms with alternate
dislocations, one has that, slightly after collision,
the dislocation function decays
exponentially fast to
the basic layer solution. More precisely (see again~\cite{RELAX}),
we have that:

\begin{thm}\label{mainthmbeforecoll3}
Let~$N=3$, $\xi_1=\xi_3=1$, $\xi_2=-1$, and let
$v_\epsilon$ be the solution given by Theorem~\ref{VISCOSITY2},
with~$\sigma\equiv0$.
Then there exist
$\epsilon_0>0$, $c>0$, 
$T_\epsilon^1,T_\epsilon^2>T_c$ and~$\rho_\epsilon>0$
satisfying
\begin{eqnarray*}
&& \lim_{\epsilon\to0} T_\epsilon^1=
\lim_{\epsilon\to0} T_\epsilon^2
=T_c,\\
{\mbox{and }}&& \lim_{\epsilon\to0}
\varrho_\epsilon=0\end{eqnarray*}
and points~$\bar y_\epsilon$ and~$\bar z_\epsilon$ satisfying
$$ \lim_{\epsilon\to0}
|\bar z_\epsilon-\bar y_\epsilon|=0$$
such that for any $\epsilon<\epsilon_0$ we have
\begin{equation}
\label{SDF-1}
v_\epsilon(t,x)\leq
   u\left(\frac{x-\bar y_\epsilon}{\epsilon}\right)+
\varrho_\epsilon e^{-\frac{c
(t-T^1_\epsilon)}{\epsilon^{2s+1}}},\qquad
   \text{for any }x\in\R \text{ and }t\geq T^1_\epsilon,
\end{equation}
and
\begin{equation}
\label{SDF-2}
v_\epsilon(t,x)\geq
u\left(\frac{x-\bar z_\epsilon}{\epsilon}\right)
-\varrho_\epsilon e^{-\frac{c
(t-T^2_\epsilon)}{\epsilon^{2s+1}}},\qquad
\text{for any }x\in\R
\text{ and }t\geq
T^2_\epsilon,\end{equation}
where $u$ is the basic layer solution introduced in~\eqref{allencahn}.
\end{thm}

Roughly speaking, formulas~\eqref{SDF-1}
and~\eqref{SDF-2}
say that for times~$T^1_\epsilon$, $T^2_\epsilon$
just slightly bigger than the collision time~$T_c$,
the dislocation function~$v_\epsilon$ gets trapped
between two basic layer solutions (centered at points~$\bar y_\epsilon$
and~$\bar z_\epsilon$), up to a small error.
The error gets eventually to zero, exponentially fast in time,
and the two basic layer solutions which trap~$v_\epsilon$
get closer and closer to each other as~$\epsilon$
goes to zero (that is, the distance between~$\bar y_\epsilon$
and~$\bar z_\epsilon$ goes to zero with~$\epsilon$).

We refer once more
to~\cite{RELAX} for a series of figures
describing in details
the results of Theorems~\ref{thmexponentialdecay}
and~\ref{mainthmbeforecoll3}. 
We observe that
the results presented in Theorems \ref{VISCOSITY}, \ref{VISCOSITY2}, \ref{J78K}, \ref{thmexponentialdecay} and \ref{mainthmbeforecoll3} describe the crystal at different space and time scale.
As a matter of fact, the mathematical study of a crystal typically goes from an atomic description (such as a classical discrete model presented by Frenkel-Kontorova
and Prandtl-Tomlinson) to a macroscopic scale in which a plastic deformation occurs.

In the theory discussed here, we join this atomic and macroscopic scales by a series
of intermediate scales, such as a microscopic scale, in which the Peierls-Nabarro model is introduced,
a mesoscopic scale, in which we studied the dynamics of the dislocations (in particular,
Theorems \ref{VISCOSITY} and \ref{VISCOSITY2}),
in order to obtain at the end a macroscopic theory
leading to the relaxation of the model to a permanent
deformation (as given in Theorems \ref{thmexponentialdecay} and \ref{mainthmbeforecoll3} ,
while Theorem\ref{J78K} somehow describes the further intermediate
features between these schematic scalings).

\section{An approach to the extension problem via the Fourier transform}

We will discuss here the extension operator of the fractional Laplacian
via the Fourier transform approach (see~\cite{CS07} and~\cite{STINGA-TORREA} 
for other approaches and further results and also~\cite{FGMT},
in which a different extension formula is obtained
in the framework of the Heisenberg groups). 

Some readers may find the details of this part rather technical:
if so, she or he can jump directly to Section~\ref{S:NP} on page~\pageref{S:NP:P},
without affecting the subsequent reading.

We fix at first a few pieces of
notation. We denote points in~$\R^{n+1}_+:=\R^n\times(0,+\infty)$ as~$X=(x,y)$, with~$x\in\R^n$ and~$y>0$. When taking gradients in~$\R^{n+1}_+$, we write~$\nabla_X=(\nabla_x,\partial_y)$, where~$\nabla_x$ is the gradient in~$\R^n$. Also, in~$\R^{n+1}_+$, we will often take the
Fourier transform in the variable~$x$ only, for fixed~$y>0$. We also set \[ a:=1-2s\in(-1,1).\]

We will consider the fractional Sobolev space $\widehat H^s(\Rn)$
defined as the 
set of functions $u$ that satisfy
\[\|u\|_{L^2(\Rn)}+[\widehat u]_G <+\infty,\]
where
	\[ [v]_G:=\sqrt{ \int_{\Rn} |\xi|^{2s}\,|v(\xi)|^2\,d\xi}. \]
For any~$u\in W^{1,1}_{\rm loc}((0,+\infty))$, we consider the functional 	
	\begin{equation} \label{Gu}
		 G(u):=\int_0^{+\infty}t^a \Big(\big|u(t)\big|^2+\big|u'(t)\big|^2\Big) \,dt.
	\end{equation}
By Theorem~4 of~\cite{SeV14}, we know that the functional~$G$ attains its minimum among all the functions~$u\in W^{1,1}_{\rm loc}((0,+\infty))\cap C^0([0,+\infty))$ with~$u(0)=1$. We call~$g$ such minimizer and
	\begin{equation}\label{F-0}
		C_\sharp:=G(g)=\min_{{u\in W^{1,1}_{\rm loc}((0,+\infty)) 
		\cap C^0([0,+\infty))}\atop{u(0)=1}} G(u).
	\end{equation}

The main result of this section is the following.

\begin{thm} \label{thmext}
Let~$u\in \mathcal{S}(\R^n)$ and let
	\begin{equation}\label{F-7-bis--}
		U(x,y):={\mathcal{F}}^{-1} \Big(\widehat u(\xi)\,g(|\xi|y)\Big).
	\end{equation}
Then
	\begin{equation}\label{78}
		{\rm div}\, (y^a \nabla U)=0
	\end{equation}
for any~$X=(x,y)\in\R^{n+1}_+$. 
In addition,
	\begin{equation}\label{0S}
		-y^a\partial_y U \Big|_{\{y=0\}} = C_\sharp (-\Delta)^s u
	\end{equation}
in~$\Rn$, both in the sense of distributions and as a pointwise limit.
\end{thm}

In order to prove Theorem \ref{thmext}, we need to make some preliminary computations. At first, let us recall a few useful properties of the minimizer function $g$ of the operator $G$ introduced in \eqref{Gu}.

We know from formula~(4.5) in~\cite{SeV14} that 
	\begin{equation}\label{gg}
			0\le g\le 1,
	\end{equation}
and from formula~(2.6) in~\cite{SeV14} that
	\begin{equation}\label{89}
		g'\le0.
	\end{equation}
We also cite formula~(4.3) in~\cite{SeV14}, according to which~$g$ is a solution of
	\begin{equation}\label{89-bis}
		g''(t)+a t^{-1} g'(t)=g(t)
	\end{equation}
for any~$t>0$, and formula~(4.4) in~\cite{SeV14}, according to which
	\begin{equation}\label{89-bis2}
		\lim_{t\rightarrow0^+} t^a g'(t)\,=\,-C_\sharp.
	\end{equation}

Now, for any~$V\in W^{1,1}_{\rm loc}(\R^{n+1}_+)$ we set
	\[ [V]_a := \sqrt{\int_{\R^{n+1}_+} y^a |\nabla_X V(X)|^2\,dX}.\]
Notice that~$[V]_a$ is well defined (possibly infinite) on such space. Also, one can compute~$[V]_a$ explicitly in the following interesting case:

\begin{lemma} \label{lem1ext}
Let~$\psi\in \mathcal{S}(\Rn)$
and
	\begin{equation}\label{F-7}
		U(x,y):={\mathcal{F}}^{-1} \Big(\psi(\xi)\,g(|\xi|y)\Big).
	\end{equation}
Then
	\begin{equation}\label{F-8}
		[U]_a^2 = C_\sharp\,[\psi]_G^2.
	\end{equation}
\end{lemma}
 
\begin{proof} By~\eqref{gg}, for any fixed~$y>0$, the function~$\xi\mapsto \psi(\xi)\,g(|\xi|y)$ belongs to~$L^2(\Rn)$, and so we may consider its (inverse) Fourier transform. This says that the definition of~$U$ is well posed. 

By the inverse Fourier transform definition \eqref{invF}, we have that 
	\[ \begin{split}
		 \nabla_x U(x,y)= &\; \nabla_x
			\int_{\Rn} \psi(\xi)\,g(|\xi|y) \,e^{ix\cdot\xi}\,d\xi \\
			=&\;
			\int_{\Rn} i\xi\psi(\xi)\,g(|\xi|y) \,e^{ix\cdot\xi}\,d\xi\\
			=&\;{\mathcal{F}}^{-1} \Big( i\xi\psi(\xi) g(|\xi|y)\Big) (x).
	\end{split} \]
Thus, by Plancherel Theorem,
	\[ \int_{\Rn} |\nabla_x U(x,y)|^2\,dx =\int_{\Rn} \big|\xi\psi(\xi) g(|\xi|y)\big|^2\,d\xi.\]
Integrating over~$y>0$, we obtain that
	\begin{equation}\label{F-2}
		\begin{split}
			\int_{\R^{n+1}_+} y^a |\nabla_x U(X)|^2\,dX\,&
			=\int_{\Rn} |\xi|^2 \big|\psi(\xi)\big|^2 \left[\int_0^{+\infty} 
			y^a \big|g(|\xi|y)\big|^2 \,dy\right]\,d\xi\\
			&=\int_{\Rn} |\xi|^{1-a} \big|\psi(\xi)\big|^2 \left[\int_0^{+\infty}
			t^a \big|g(t)\big|^2 \,dt\right]\,d\xi
			\\ &=\int_0^{+\infty}
			t^a \big|g(t)\big|^2 \,dt\cdot
			\int_{\Rn} |\xi|^{2s} \big|\psi(\xi)\big|^2 \, d\xi
		\\ &= [\psi]_{G}^2 \;\int_0^{+\infty}
			t^a \big|g(t)\big|^2 \,dt.
		\end{split}
	\end{equation}
Let us now prove that the  following identity is well posed
	\begin{equation}\label{4.3bis}
\partial_y U(x,y)= {\mathcal{F}}^{-1} \Big(|\xi|\,\psi(\xi)\,g'(|\xi|y)\Big) .\end{equation}
For this, we observe that
	\begin{equation}\label{98}
		|g'(t)|\le C_\sharp t^{-a}.
	\end{equation}
To check this, we define~$\gamma(t):= t^a \,|g'(t)|$. {F}rom~\eqref{89} and~\eqref{89-bis}, we obtain that
	\[ \gamma'(t) =-\frac{d}{dt}\big(t^a g'(t)\big)=-t^a \,\big(g''(t)+a t^{-1} g'(t)\big)=-t^a g(t)\le0.\]
Hence
$$ \gamma(t)\le \lim_{\tau\rightarrow 0^+} \gamma(\tau)=
\lim_{\tau\rightarrow 0^+} \tau^a |g'(\tau)|= C_\sharp,$$
where formula~\eqref{89-bis2} was used in the last identity,
and this establishes~\eqref{98}.

{F}rom~\eqref{98} we have that~$|\xi|\,|\psi(\xi)|\,|g'(|\xi|y)|\le
 C_\sharp y^{-a}\,|\xi|^{1-a}\,|\psi(\xi)|\in L^2(\Rn)$, and so~\eqref{4.3bis}
follows.

Therefore, by~\eqref{4.3bis} and the Plancherel Theorem,
$$ \int_{\Rn} |\partial_y U(x,y)|^2\,dx= \int_{\Rn} |\xi|^2\,\big|\psi(\xi)\big|^2\,\big|g'(|\xi|y)\big|^2\,d\xi.$$
Integrating over~$y>0$ we obtain
\[ \begin{split}
	\int_{\R^{n+1}_+} y^a |\partial_y U(x,y)|^2\,dx =&\;
	\int_{\Rn} |\xi|^2\,\big|\psi(\xi)\big|^2\,\left[\int_0^{+\infty}
	y^a \big|g'(|\xi|y)\big|^2\,dy \right]\,d\xi
	\\ =&\;
	\int_{\Rn} |\xi|^{1-a} \,\big|\psi(\xi)\big|^2\,\left[\int_0^{+\infty}
	t^a \big|g'(t)\big|^2 dt \right]\,d\xi
	\\ =&\;
	\int_0^{+\infty}
	t^a \big|g'(t)\big|^2 dt\cdot
	\int_{\Rn} |\xi|^{2s} \,\big|\psi(\xi)\big|^2\,d\xi
	\\ =&\;
	 [\psi]_{G}^2 \; \int_0^{+\infty}
	t^a \big|g'(t)\big|^2 dt.
\end{split} \]
By summing this with~\eqref{F-2}, and recalling~\eqref{F-0}, we obtain the desired result $[U]_a^2 = C_\sharp\,[\psi]_G^2$. This concludes the proof of the Lemma.
\end{proof}

Now, given~$u \in L^1_{\rm loc}(\Rn)$, we consider the space~$X_u$ of all the functions~$V\in W^{1,1}_{\rm loc}(\R^{n+1}_+)$ such that, for any~$x\in\Rn$, the map~$y\mapsto V(x,y)$ is in~$C^0\big([0,+\infty)\big)$, with~$V(x,0)=u(x)$ for any~$x\in\Rn$.
Then the problem of minimizing~$[\,\cdot\,]_a$ over~$X_u$ has a somehow explicit solution.

\begin{lemma}\label{L7s}
Assume that~$u\in \mathcal{S}(\Rn)$. Then
\begin{equation}\label{A6}
\min_{V\in X_u} [V]_a^2 = [U]_a^2
=C_\sharp\,[\widehat u]_G^2,\end{equation}

\begin{equation}\label{F-7-bis}
U(x,y):={\mathcal{F}}^{-1} \Big(\widehat u(\xi)\,g(|\xi|y)\Big).\end{equation}
\end{lemma}

\begin{proof} We remark that~\eqref{F-7-bis} is simply~\eqref{F-7}
with~$\psi:=\widehat u$, and by Lemma \ref{lem1ext} we have that \[ [U]_a^2= C_{\sharp}[\widehat u]^2_G.\]
Furthermore, we claim that
\begin{equation}\label{98-bis}
U\in X_u. \end{equation}
In order to prove this, we first observe that
\begin{equation}\label{98-tris}
|g(T)-g(t)|\le \frac{C_\sharp\,|T^{2s}-t^{2s}|}{2s} .
\end{equation}
To check this, without loss of generality, we may suppose that~$T\ge t\ge0$.
Hence, by~\eqref{89} and~\eqref{98},
	\[ \begin{split}
		|g(T)-g(t)|\le &\; \int_t^T |g'(r)|\,dr \\
		\le &\; C_\sharp\int_t^T r^{-a}\,dr\\
		=&\; \frac{C_\sharp\,(T^{1-a}-t^{1-a})}{1-a},
	\end{split}\]
that is~\eqref{98-tris}.

Then, by~\eqref{98-tris}, for any~$y$, $\tilde y\in(0,+\infty)$,
we see that
\[ \Big| g(|\xi|\,y)-g(|\xi|\,\tilde y)\Big|\le
\frac{C_\sharp\,|\xi|^{2s} |y^{2s}-{\tilde y}^{2s}|}{2s} .\]
Accordingly,
	\[ \begin{split}
		\big|U(x,y)-U(x,\tilde y)\big|
		=& \; \bigg|{\mathcal{F}}^{-1} \bigg(\widehat u(\xi)\,\Big( g(|\xi|\,y)-g(|\xi|\,\tilde y)\Big)\bigg)\bigg|\\
		 \le & \; 	\int_{\Rn}  	\Big|\widehat u(\xi)\,\Big( g(|\xi|\,y)-g(|\xi|\,\tilde y)\Big)\Big|\,d\xi \\ 
		\le & \; 	\frac{C_\sharp\,|y^{2s}-{\tilde y}^{2s}|}{2s} \int_{\Rn} |\xi|^{2s} |\widehat u(\xi)|\,d\xi,
	\end{split} \]
and this implies~\eqref{98-bis}.

Thanks to~\eqref{98-bis} and~\eqref{F-8}, in order to complete the proof
of~\eqref{A6},
it suffices to show that, for any~$V\in X_u$, we have that
\begin{equation}\label{A6-c}
[V]_a^2 \ge [U]_a^2.\end{equation}
To prove this,
let us take~$V\in X_u$. Without loss of generality,
since~$[U]_a<+\infty$ thanks to~\eqref{F-8},
we may suppose that~$[V]_a<+\infty$. Hence, 
fixed a.e.~$y>0$, we have that
$$ y^a \int_{\Rn}|\nabla_x V(x,y)|^2\,dx
\,\le\,
y^a \int_{\Rn}|\nabla_X V(x,y)|^2\,dx\,<\,+\infty,$$
hence the map~$x\in |\nabla_x V(x,y)|$ belongs to~$L^2(\Rn)$.
Therefore, by Plancherel Theorem,
\begin{equation}\label{F-9}
\int_{\Rn}|\nabla_x V(x,y)|^2\,dx =
\int_{\Rn}\Big|{\mathcal{F}}\big(\nabla_x V(x,y)\big) (\xi)\Big|^2\,d\xi
.\end{equation}
Now by the Fourier transform definition \eqref{transF}
\begin{eqnarray*}  {\mathcal{F}}\big(\nabla_x V(x,y)\big) (\xi)&=&
\int_{\Rn} \nabla_x V(x,y)\,e^{-ix\cdot\xi}\,dx\\
&=&
\int_{\Rn} i\xi\,V(x,y)\,e^{-ix\cdot\xi}\,dx \\ &=&
i\xi \,{\mathcal{F}}\big(V(x,y)\big) (\xi),\end{eqnarray*}
hence~\eqref{F-9} becomes
\begin{equation}\label{F-11}
\int_{\Rn}|\nabla_x V(x,y)|^2\,dx =
\int_{\Rn} |\xi|^2 \,|{\mathcal{F}}\big(V(x,y)\big) (\xi)|^2\,d\xi.\end{equation}
On the other hand
$$ {\mathcal{F}}\big(\partial_y V(x,y)\big) (\xi)=
\partial_y {\mathcal{F}}\big(V(x,y)\big) (\xi)$$
and thus, by Plancherel Theorem,
$$ \int_{\Rn} |\partial_y V(x,y)|^2 \,dx=
\int_{\Rn}\big| {\mathcal{F}}\big(\partial_y V(x,y)\big) (\xi)\big|^2\,d\xi
=\int_{\Rn} |\partial_y {\mathcal{F}}\big(V(x,y)\big) (\xi)|^2\,d\xi.$$
We sum up this latter result with identity \eqref{F-11} and we use the notation~$\phi(\xi,y):=
{\mathcal{F}}\big(V(x,y)\big) (\xi)$
to conclude that
\begin{equation}\label{A5} 
\int_{\Rn}|\nabla_X V(x,y)|^2\,dx =
\int_{\Rn} |\xi|^2 \,|\phi(\xi,y)|^2 + |\partial_y \phi(\xi,y)|^2\,d\xi.\end{equation}
Accordingly, integrating over~$y>0$, we deduce that
\begin{equation}\label{F-13}
[V]_a^2 =\int_{\R^{n+1}_+} y^a\Big(
 |\xi|^2 \,|\phi(\xi,y)|^2 + |\partial_y \phi(\xi,y)|^2 \Big)\,d\xi\,dy.
\end{equation}
Let us first consider the integration over~$y$, for any fixed $\xi\in
\Rn\setminus\{0\}$, that we now omit from the notation when
this does not generate any confusion. We set~$h(y):= 
\phi(\xi, |\xi|^{-1} y)$.
We have that~$h'(y)=|\xi|^{-1} \partial_y\phi(\xi, |\xi|^{-1} y)$ and
therefore, using the substitution~$t=|\xi|\,y$, we obtain
	\begin{equation}\label{F-17}
		\begin{split}
		&\int_0^{+\infty} y^a \Big(|\xi|^2\,|\phi(\xi, y)|^2 +  \big|\partial_y\phi(\xi, y)\big|^2 \Big)\,dy
		\\ = &\; |\xi|^{1-a} 
		\int_0^{+\infty} t^a \Big( 
		|\phi(\xi, |\xi|^{-1} t)|^2 +
		|\xi|^{-2} \big|\partial_y\phi(\xi, |\xi|^{-1} t)\big|^2
		\Big)\,dt\\
		= &\; |\xi|^{1-a}
		\int_0^{+\infty} t^a \Big( 
		|h(t)|^2 +|h'(t)|^2
		\Big)\,dt
		\\ = &\; |\xi|^{2s} \,G(h).
\end{split}\end{equation}
Now, for any~$\lambda\in\R$, we show that
\begin{equation}\label{F-15}
\min_{\substack{{w\in W^{1,1}_{\rm loc}((0,+\infty))
\cap C^0([0,+\infty))}\\ {w(0)=\lambda}}}  G(w) = \lambda^2 \,C_\sharp.
\end{equation}
Indeed, when~$\lambda=0$, the trivial function is an allowed
competitor and~$G(0)=0$, which gives~\eqref{F-15}
in this case. If, on the other hand,~$\lambda\ne 0$, given~$w$
as above with~$w(0)=\lambda$ we set~$w_\lambda(x):=\lambda^{-1} w(x)$.
Hence we see that~$w_\lambda(0)=1$ and thus~$G(w)=
\lambda^2 \, G(w_\lambda)\le \lambda^2\,G(g)=\lambda^2\,C_\sharp$, due to the minimality
of~$g$. This proves~\eqref{F-15}.
{F}rom~\eqref{F-15} and the fact that
$$h(0)=
\phi(\xi, 0)={\mathcal{F}}\big(V(x,0)\big) (\xi) = \widehat u(\xi),$$
we obtain that
$$ G(h)\ge C_\sharp\, \big|\widehat u(\xi)\big|^2.$$
As a consequence, we get from~\eqref{F-17} that
$$ \int_0^{+\infty} y^a \Big(
|\xi|^2\,|\phi(\xi, y)|^2 +
\big|\partial_y\phi(\xi, y)\big|^2
\Big)\,dy \ge C_\sharp\,|\xi|^{2s} \,\big|\widehat u(\xi)\big|^2.$$
Integrating over~$\xi\in\Rn\setminus\{0\}$ we obtain that
\[ \int_{\R^{n+1}_+} y^a \Big(
|\xi|^2\,|\phi(\xi, y)|^2 +
\big|\partial_y\phi(\xi, y)\big|^2
\Big)\,d\xi\,dy \ge  C_\sharp\, [\widehat u]_G^2.\]
Hence, by~\eqref{F-13},
\[ [V]_a^2 \ge C_\sharp\, [\widehat u]_G^2,\]
which proves~\eqref{A6-c},
and so~\eqref{A6}.
\end{proof}

We can now prove the main result of this section.

\begin{proof}[Proof of Theorem \ref{thmext}] Formula~\eqref{78} follows from the minimality
property in~\eqref{A6}, by writing that
$[U]_a^2 \le [U+\epsilon \varphi]_a^2$
for any~$\varphi$ smooth and compactly supported inside~$\R^{n+1}_+$
and any~$\epsilon\in\R$.

Now we take~$\varphi\in C^\infty_0(\Rn)$ (notice that its support
may now hit~$\{y=0\}$). We define $u_\epsilon:= u+\epsilon\varphi$,
and~$U_\epsilon$
as in~\eqref{F-7-bis--}, with~$\widehat u$ replaced by~$\widehat u_\epsilon$
(notice that~\eqref{F-7-bis--}
is nothing but~\eqref{F-7-bis}), hence we will be able to exploit
Lemma~\ref{L7s}.

We also set
$$\varphi_* (x,y):= {\mathcal{F}}^{-1} \Big(\widehat\varphi(\xi)\,g(|\xi|y)\Big).$$
We observe that
\begin{equation}\label{9956}
\varphi_* (x,0)=
{\mathcal{F}}^{-1} \Big(\widehat\varphi(\xi)\,g(0)\Big)=
{\mathcal{F}}^{-1} \Big(\widehat\varphi(\xi)\Big)=\varphi(x)\end{equation}
and that
$$ U_\epsilon= U+\epsilon
{\mathcal{F}}^{-1} \Big(\widehat\varphi(\xi)\,g(|\xi|y)\Big)=
U+\epsilon\varphi_*.$$
As a consequence
$$ [U_\epsilon]_a^2 \, =[U_\epsilon]_a^2 +2\epsilon
\int_{\R^{n+1}_+} y^a \nabla_X U\cdot \nabla_X \varphi_*\,dX+o(\epsilon).$$
Hence, using~\eqref{78}, \eqref{9956} and the Divergence Theorem,
	\begin{equation}\label{67}
		\begin{split}
			[U_\epsilon]_a^2= &\; [U]_a^2 +2\epsilon 	\int_{\R^{n+1}_+} {\rm div}\,
			\Big( \varphi_* \; y^a\nabla_X U \Big) \,dX+o(\epsilon)\\
			 =&\;   [U]_a^2 -2\epsilon \int_{\R^{n}\times\{0\}}
			\varphi \; y^a\partial_y U \,dx+o(\epsilon).
		\end{split}
	\end{equation}
Moreover, from Plancherel Theorem, and the fact that the image of~$\varphi$ is in the reals,
\[ \begin{split}
[\widehat u_\epsilon]_G^2 =&\; [\widehat u]_G +2\epsilon \int_{\Rn}
|\xi|^{2s}\widehat u(\xi)\,\overline{\widehat \varphi(\xi)}\,d\xi+o(\epsilon)\\
=&\;[\widehat u]_G +2\epsilon \int_{\Rn}
{\mathcal{F}}^{-1}\Big( |\xi|^{2s}\widehat u(\xi)\Big)(x)
\,\overline{\varphi(x)}\,dx+o(\epsilon)\\
=&\;[\widehat u]_G +2\epsilon \int_{\Rn} (-\Delta)^s u(x)\,
{\varphi(x)}\,dx+o(\epsilon).
\end{split} \] 
By comparing this with~\eqref{67} and recalling \eqref{A6}
we obtain that
\[ \begin{split}
& [U]_a^2 -2\epsilon
\int_{\R^{n}\times\{0\}}
\varphi \; y^a\partial_y U \,dx+o(\epsilon)\\=&\; [U_\epsilon]_a^2
\\ =&\; C_\sharp [u_\epsilon]_G^2
\\ =&\;
C_\sharp [\widehat u]_G +2C_\sharp
\epsilon \int_{\Rn} (-\Delta)^s u(x)\,
{\varphi(x)}\,dx+o(\epsilon)
\\=&\;[U]_a^2+2C_\sharp
\epsilon \int_{\Rn} (-\Delta)^s u\,
{\varphi}\,dx+o(\epsilon) 
\end{split}\]
and so
$$ -\int_{\R^{n}\times\{0\}}
\varphi \; y^a\partial_y U \,dx = C_\sharp
\int_{\Rn} (-\Delta)^s u\,
{\varphi}\,dx,$$
for any~$\varphi\in C^\infty_0(\Rn)$, that is the
distributional formulation of~\eqref{0S}.

Furthermore, by \eqref{F-7-bis--}, we have that
$$ y^a \partial_y
U(x,y)={\mathcal{F}}^{-1} \Big(|\xi|\,\widehat u(\xi)\,y^a\,g(|\xi|y)\Big)
= {\mathcal{F}}^{-1} \Big(|\xi|^{1-a}\,\widehat u(\xi)\,(|\xi|y)^a\,g(|\xi|y)\Big)
.$$
Hence, by~\eqref{89-bis2}, we obtain
\[ \begin{split}
 	\lim_{y\rightarrow0^+} y^a \partial_y U(x,y) =&\; -C_\sharp {\mathcal{F}}^{-1} \Big(|\xi|^{1-a}\,\widehat u(\xi) \Big)\\
 	=&\;-C_\sharp {\mathcal{F}}^{-1} \Big(|\xi|^{2s}\,\widehat u(\xi)  \Big)\\
	=&\;-(-\Delta)^s u(x),
	\end{split} \]
that is the pointwise limit formulation of~\eqref{0S}. This concludes the proof of Theorem \ref{thmext}.
\end{proof}

\chapter{Nonlocal phase transitions }\label{S:NP}

Now, we consider
a nonlocal phase transition model, \label{S:NP:P} in particular described by the Allen-Cahn equation. A fractional analogue of
a conjecture of De Giorgi, that deals with possible one-dimensional symmetry of entire solutions, naturally arises from treating this model, and will be consequently presented. There is a very interesting connection with nonlocal minimal surfaces, that will be studied in Chapter \ref{nlms}.

  We introduce briefly the classical case\footnote{We would like to thank Alberto Farina who, during a summer-school in Cortona (2014), gave a beautiful introduction on phase transitions in the classical case.}. The Allen-Cahn equation has various applications, for instance, in the study of interfaces (both in gases and solids), in the theory of superconductors and superfluids or in cosmology.  We deal here with a two-phase transition model, in which a fluid can reach two pure phases (say $1$ and $-1$) forming an interface of separation. The aim is to describe the pattern and the separation of the two phases. 
  
The formation of the interface is driven by a variational principle. Let $u(x)$ be the function describing the state of the fluid at position $x$ in a bounded region $\Omega$. As a first guess, the phase separation
can be modeled
via the minimization of the energy 
  \[ \E_0(u)= \int_{\Omega} W \big(u(x)\big) \, dx,\]
where $W$ is a double-well potential. More precisely, $W\colon [-1,1]\to [0,\infty)$ such that 
	\eqlab{ \label{dwp} W\in C^2\left([-1,1]\right), W(\pm 1)=0, W>0 \mbox{ in } (-1,1),\\
	 W'(\pm 1)= 0 \mbox{ and } W''(\pm 1) >0. }
The classical example is 
\begin{equation}\label{DF-WELL}
W(u):=\displaystyle \frac{(u^2-1)^2}{4}.\end{equation} 
On the other hand, the functional in~$\E_0$ 
produces an ambiguous
outcome, since any function $u$ that attains
only the values $\pm 1$ is a minimizer for the energy. That is,
the energy functional in~$\E_0$ alone cannot detect any
geometric feature of the interface.

To avoid this, one is led to consider an additional energy term that penalizes 
the formation of unnecessary interfaces. The typical energy
functional provided by this procedure has the form
\begin{equation}\label{5.2PRE}
\mathcal{E} (u):= \int_{\Omega} W\big(u(x)\big)\, dx + \frac{\eee^2}{2} \int_{\Omega} |\nabla u(x)|^2 \, dx.\end{equation}
In this way, the potential energy that forces the pure phases is compensated by a small term, that is due to the elastic effect of the reaction of the particles. 
As a curiosity, we point out that
in the classical mechanics framework, the analogue
of~\eqref{5.2PRE} is a Lagrangian action of a particle,
with~$n=1$, $x$ representing a time coordinate and $u(x)$
the position of the particle at time~$x$. In this framework
the term involving the square of the derivative of~$u$
has the physical meaning of a
kinetic energy. With a slight abuse of notation, we will keep
referring to the gradient term in~\eqref{5.2PRE}
as a kinetic energy. Perhaps a more appropriate term would be elastic
energy, but in concrete applications also 
the potential may arise from elastic reactions,
therefore the only purpose of these names in our framework
is to underline the fact that~\eqref{5.2PRE}
occurs as a superposition of two terms,
a potential one, which only depends on $u$,
and one, which will be called kinetic, which only depends
on the variation of $u$ (and which, in principle, possesses no
real ``kinetic'' feature). 
	
	The energy minimizers will be smooth functions, taking values between $-1$ and $1$, forming layers of interfaces of $\eee$-width. If we send $\eee \to 0$, the transition layer will tend to a minimal surface. To better explain this, consider the energy
	\eqlab{\label{gle1}  J(u)= \int \frac{1}2 |\nabla u|^2 +W(u) \,dx,} whose minimizers solve the Allen-Cahn equation
	\begin{equation}\label{ALC} -\Delta u+ W'(u)=0. \end{equation}  
In particular, for the explicit potential in~\eqref{DF-WELL},
equation~\eqref{ALC} reduces (up to normalizations constants) to
\begin{equation}\label{ALC-DG} -\Delta u= u-u^3.\end{equation}
In this setting,
the behavior of $u$ in large domains reflects into the behavior of the rescaled function $u_{\eee} (x)=u\big( \frac{x}{\eee}\big)$ in $B_1$. Namely, the minimizers of $J$ in $B_{1/ \eee}$ are the minimizers of $J_\eee$ in $B_1$, where $J_\eee$ is the rescaled energy functional
	\eqlab{ J_{\eee} (u)= \int_{B_1} \frac{\eee}2 |\nabla u|^2 + \frac{1}\eee W(u) \, dx.    \label{rescef1} }
We notice then that
	\[ J_\eee(u) \geq \int_{B_1} \sqrt{ 2W(u)}\, |\nabla u| \, dx\] 
which, using the Co-area Formula, gives 	\[  J_\eee(u) \geq
\int_{-1}^1 \sqrt{2W(t)}\,\mathcal {H}^{n-1} \left(\{u=t\}\right) \, dt.\] 
The above formula may suggest that the minimizers of~$J_\eee$ 
have the tendency to minimize the $(n-1)$-dimensional measure of
their level sets. It turns out that indeed the level sets of the
minimizers of $J_\eee$ converge 
to a minimal surface as $\eee \to 0$: for more details see, for instance,
\cite{savinminimal} and the references therein.
															
In this setting, a famous De Giorgi conjecture comes into place. In the
late 70's, De Giorgi conjectured that entire, smooth, monotone (in one direction), bounded solutions of~\eqref{ALC-DG}
in the whole of~$\R^n$ 
are necessarily one-dimensional, i.e., there exist $\omega \in S^{n-1}$ and $u_0: \R\to \R$ such that 	
	\[ u(x)=u_0(\omega \cdot x) \quad \mbox{for any} \quad x\in \Rn.\] In other words, the conjecture above asks
if the level sets of the entire, smooth, monotone (in one direction),
bounded solutions are necessarily hyperplanes, at least in dimension $n\leq 8$. 

One may wonder why the number eight has a relevance in the problem above.
A possible explanation for this is given by the Bernstein Theorem,
as we now try to describe.
 
The Bernstein problem asks on whether or not all 
minimal graphs (i.e.
surfaces that locally minimize the perimeter and that
are graphs in a given direction)
in $\R^n$ must be necessarily affine.
This is indeed true in dimensions $n$ at most eight.
On the other hand, in dimension $n\geq 9$ there are global minimal graphs that are not hyperplanes (see e.g.~\cite{GIUSTI}).

The link between the problem of Bernstein and the conjecture of De Giorgi could be suggested
by the fact that
minimizers approach minimal surfaces in the limit.
In a sense, if one is able to prove that the limit interface is a hyperplane
and that this rigidity property gets inherited by the level sets of
the minimizers~$u_\eee$ (which lie nearby such limit hyperplane),
then, by scaling back, one obtains that the level sets
of~$u$ are
also hyperplanes. Of course, this link between the two
problems, as stated here, is only heuristic,
and much work is needed to deeply understand the connections
between the problem of Bernstein and the conjecture of De Giorgi.
We refer to~\cite{FARINA-VALDINOCI-STATE} for a more detailed introduction
to this topic.
\bigskip
	
We recall that this conjecture by De Giorgi
was proved for $n\leq 3$, see~\cite{GGUI, BCN97, AC00}.
Also,  the case $4\leq n \leq 8$ with the additional assumption that 	
\begin{equation} \label{limdgs}
 \lim_{x_n\to \pm \infty} u(x',x_n)=\pm 1, \quad \mbox{for any} \quad x'\in \R^{n-1}
\end{equation}
was proved in \cite{S09}. 

For $n\geq 9$ a counterexample can be found in \cite{PKW08}.  Notice 
that, if the above limit is uniform (and the De Giorgi conjecture with 
this additional assumption is known as the Gibbons conjecture), the 
result extends to all possible $n$ (see for 
instance~\cite{Farina-Gibbons, FARINA-VALDINOCI-STATE}
for further details).

The goal of the next part of this book
is then to discuss an analogue of these questions
for the nonlocal case and present related results.

\section{The fractional Allen-Cahn equation} \label{sbsac} 

The extension of the Allen-Cahn equation in~\eqref{ALC}
from a local to a nonlocal 
setting has theoretical interest and concrete applications. 
Indeed, the study of
long range interactions naturally leads to the analysis of
phase transitions and interfaces of nonlocal type.

Given an open domain 
$\Omega\subset \Rn$ and the double well potential $W$
(as in~\eqref{DF-WELL}), our goal here is 
to study the fractional Allen-Cahn equation	
\[ \frlap u +W'(u)=0 \quad \mbox{in} \quad \Omega, \]
for~$s\in(0,1)$ (when~$s=1$, this equation reduces to~\eqref{ALC}).
The solutions are the critical points of the nonlocal energy 	
	\begin{equation} \label{enfac}
		 \mathcal{E}(u,\Omega) := \int_{\Omega} W\big(u(x)\big) \, dx + \frac12
		 \iint_{\R^{2n}\setminus (\Omega^{\C})^2} \frac{|u(x)-u(y)|^2}{|x-y|^{n+2s} }\, dx\, dy,	
	\end{equation}
up to normalization constants that we omitted for simplicity.
The reader can compare~\eqref{enfac} with~\eqref{5.2PRE}. Namely, in~\eqref{enfac} the kinetic energy
is modified, in order to take into account long range interactions.
That is, the new kinetic energy still depends on the variation of the
phase parameter.
But, in this case, far away changes in phase may influence each other
(though the influence is weaker and weaker towards infinity).
	
Notice that in the nonlocal framework, we prescribe the function on 
$\Omega^{\C} \times \Omega^{\C}$ and consider the kinetic energy on the remaining regions (see Figure \ref{fign:eac}). 
The prescription of values in~$\Omega^{\C} \times \Omega^{\C}$
reflects into the fact that the domain of integration of the kinetic
integral in~\eqref{enfac} is~$\R^{2n}\setminus (\Omega^{\C})^2$.
Indeed, this is perfectly compatible with
the local case in~\eqref{5.2PRE}, where the domain of
integration of the kinetic term was simply~$\Omega$.
To see this compatibility, one may think that 
the domain of integration of the kinetic energy
is simply the complement of
the set in which the values of the functions are prescribed.
In the local case of~\eqref{5.2PRE}, the values are prescribed on~$\partial\Omega$, or, one may say, in~$\Omega^{\C}$: then the domain of integration
of the kinetic energy is the complement of~$\Omega^{\C}$,
which is simply~$\Omega$. 
In analogy with that, in the nonlocal case of~\eqref{enfac},
the values are prescribed on~$\Omega^{\C}\times\Omega^{\C}=(\Omega^{\C})^2$,
i.e. outside~$\Omega$ for both the variables $x$ and~$y$.
Then, the kinetic integral is set on the complement of~$(\Omega^{\C})^2$,
which is indeed~$\R^{2n}\setminus(\Omega^{\C})^2$.

Of course, the potential energy has local features, both
in the local and in the nonlocal case,
since in our model the nonlocality only occurs in the kinetic interaction,
therefore the potential integrals are set over~$\Omega$
both in~\eqref{5.2PRE} and in~\eqref{enfac}.

For the sake of shortness, given disjoint sets~$A$, $B\subseteq\R^n$
we introduce the notation
\begin{equation*} u(A,B):=\int_{A}\int_{B}  \frac{|u(x)-u(y)|^2}{|x-y|^{n+2s} }\, dx\, dy,\end{equation*}
and we write the new kinetic energy in~\eqref{enfac} as
\begin{equation}\label{kenac}
 \mathcal {K} (u,\Omega) = \frac{1}{2} u(\Omega, \Omega) + u(\Omega, \Omega^{\C}).
\end{equation} 
\begin{center}
	\begin{figure}[htb]
	\hspace{0.8cm}
	\begin{minipage}[b]{0.75\linewidth}
	\centering
	\includegraphics[width=0.75\textwidth]{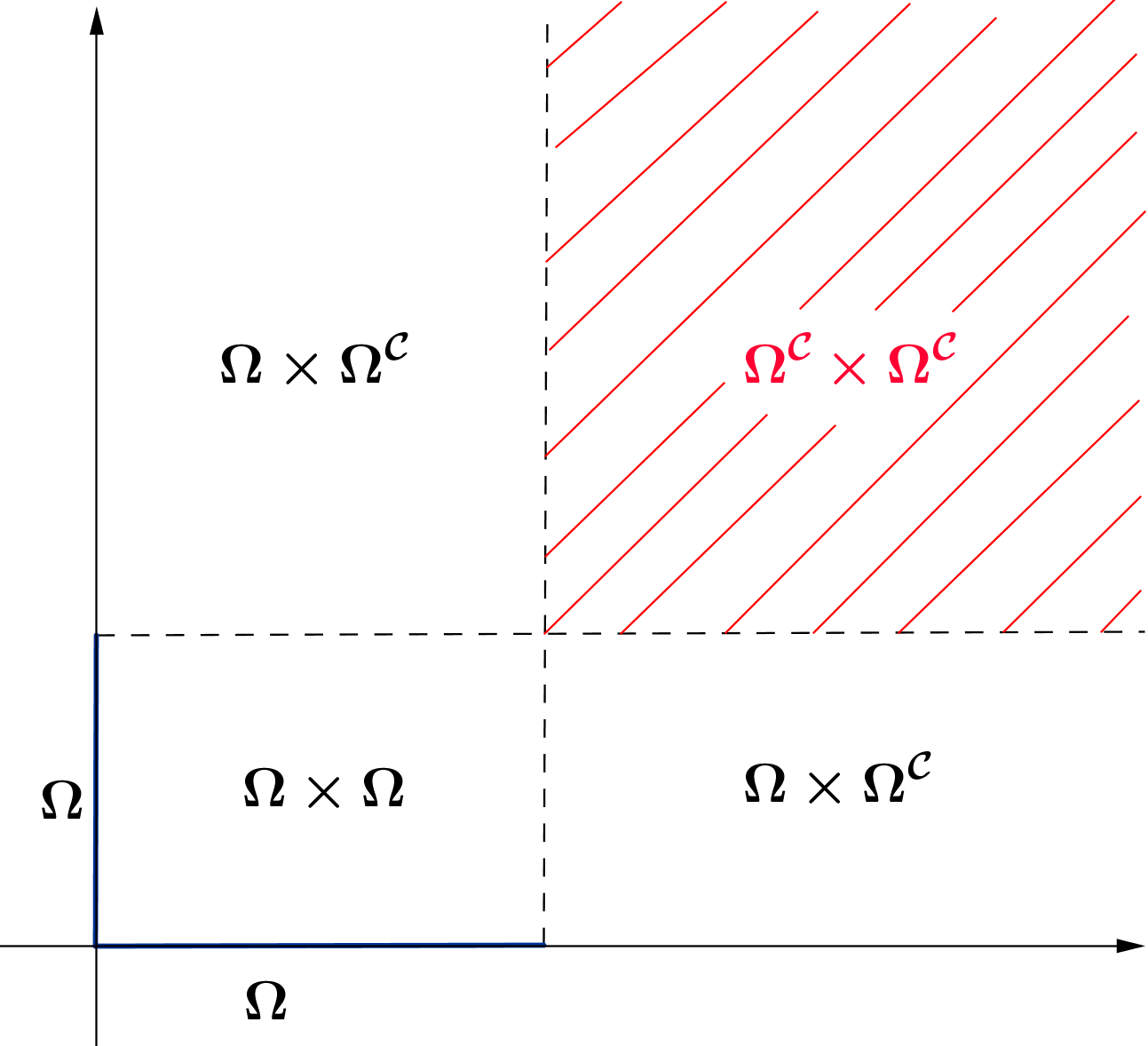}
	\caption{The kinetic energy}   
	\label{fign:eac}
	\end{minipage}
\end{figure}
\end{center}
Let us define the energy minimizers and provide a density estimate for the minimizers.
\begin{defn}
The function $u$ is a minimizer for the energy $\E$ in $B_R$ if 
$ \E (u, B_R)\leq \E (v,B_R) $ for any $v$ such that $u=v $ outside $B_R$. 
\end{defn}

The energy of the minimizers satisfy the following
uniform bound property on large balls.
\begin{thm} \label{acenest1}
Let $u$ be a minimizer in $B_{R+2}$ for a large $R$, say $R\geq1$. Then 
	\begin{equation}\label{THANKS}
\lim_{R \to +\infty} \frac{1}{R^n} \E (u,B_R) =0.\end{equation} More precisely,
	\[ \E(u,B_R) \leq 
					\begin{cases}
						CR^{n-1} \quad &\mbox{if} \quad s\in \Big(\frac{1}{2},1\Big),\\
						CR^{n-1}\log R \quad &\mbox{if} \quad s=\frac{1}{2},\\
						CR^{n-2s} \quad &\mbox{if} \quad s\in \Big(0,\frac{1}{2}\Big).
					\end{cases}
	\]
	Here, $C$ is a positive constant depending only on $n, s$ and $W$.
\end{thm}
Notice that for $\displaystyle s\in \Big(0,\frac{1}{2}\Big)$, $R^{n-2s}>R^{n-1}$. These estimates are optimal (we
refer to \cite{SV14} for further details).
\begin{proof}
We introduce at first some auxiliary functions. Let 
	\[\psi(x):=-1+2 \min\Big\{ (|x|-R-1)_+,1\Big\}, \quad v(x):=\min \Big\{u(x), \psi(x)\Big\},\]\[ d(x):= \max\Big\{ (R+1-|x|),1 \Big\}.\]
	\begin{center}
\begin{figure}[htb]
	\hspace{1.15cm}
	\begin{minipage}[b]{1.05\linewidth}
	\centering
	\includegraphics[width=0.95\textwidth]{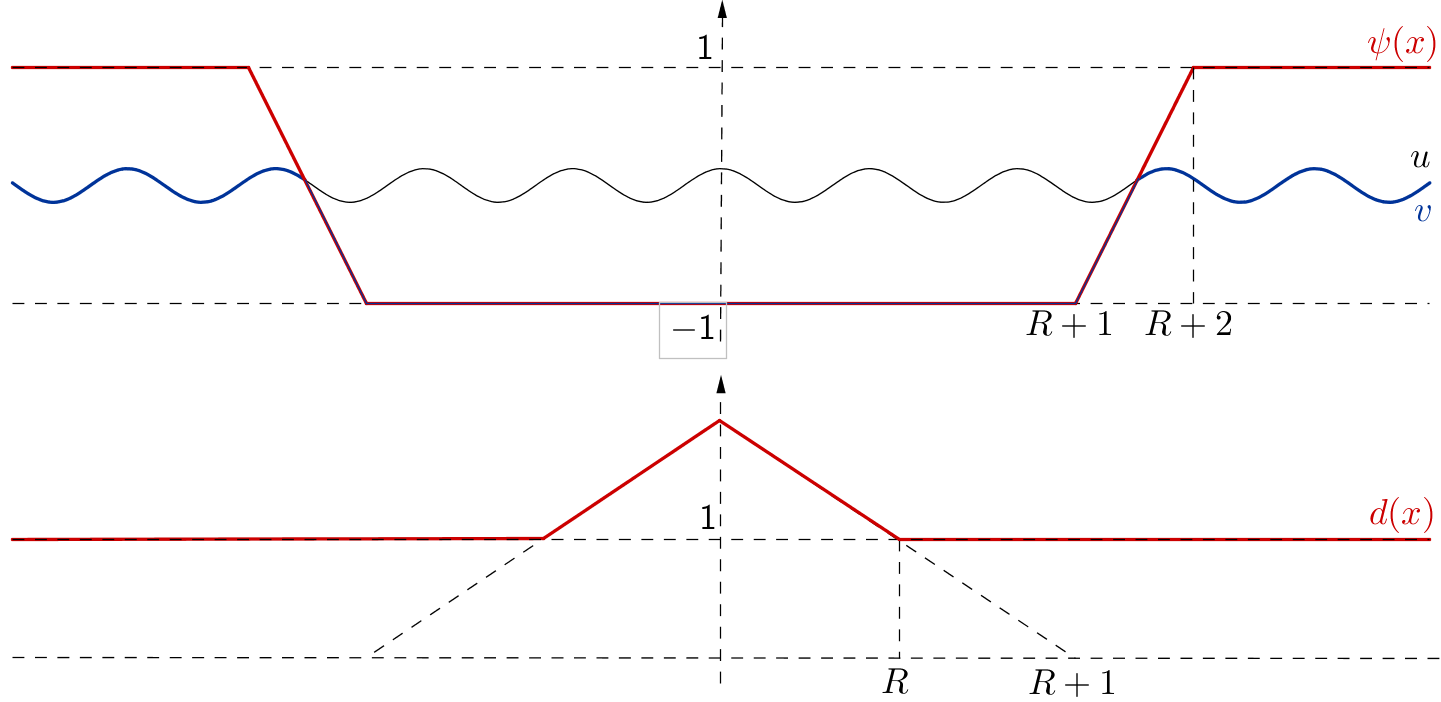}
	\caption{The functions $\psi$, $v$ and $d$}   
	\label{fign:psid}
	\end{minipage}
\end{figure}
\end{center}
Then, for $|x-y|\leq d(x)$ we have that 
	\begin{equation}\label{acpsi1} |\psi(x)-\psi(y)|\leq \frac{2 |x-y|}{d(x)}.\end{equation}
	Indeed, if $|x|\leq R$, then $d(x)=R+1-|x|$ and \[ |y|\leq |x-y|+|x|\leq d(x)+|x|\leq R+1,\] thus $\psi(x)=\psi(y)=0$ and the inequality is trivial. Else, if $|x|\geq R$, then $d(x)=1$,
and so the inequality is assured by the Lipschitz continuity of $\psi$ (with $2$ as the Lipschitz constant).
	
Also, we prove that we have the following estimates for the function $d$:
	\begin{equation} \label{acdx1} \int_{B_{R+2}} d(x)^{-2s}\, dx \leq 
		\begin{cases} 
			 CR^{n-1}  \quad & \mbox{if}  \quad s\in \Big(\frac{1}{2},1\Big),\\
			 CR^{n-1}\log R   \quad &\mbox{if} \quad s=\frac{1}{2},\\
			 CR^{n-2s}  \quad & \mbox{if}  \quad s\in \Big(0,\frac{1}{2}\Big) .
		\end{cases}
	\end{equation}
	To prove this, we observe that
in the ring $B_{R+2}\setminus B_R$, we have~$d(x)=1$. Therefore,
the contribution to the integral in~\eqref{acdx1}
that comes from the ring~$B_{R+2}\setminus B_R$
is bounded by the measure of the ring, and so it
is of order $R^{n-1}$, namely
\begin{equation}\label{5.10bis}
\int_{B_{R+2}\setminus B_R} d(x)^{-2s}\,dx=|B_{R+2}\setminus B_R|\le CR^{n-1},
\end{equation}
for some~$C>0$.
We point out that this order is always negligible with respect to
the right hand side of~\eqref{acdx1}.

Therefore, to complete the proof of~\eqref{acdx1}, it only remains to estimate
the contribution to the integral coming from~$B_R$.

For this, we use
polar coordinates and perform
the change of variables $t=\rho/(R+1)$. In this way, we obtain that
	\[ \begin{split}
		\int_{B_R} d(x)^{-2s} \, dx =\;& C\,\int_0^R \frac{\rho^{n-1}}	{(R+1-\rho)^{2s}} \, d\rho  \\
			=\;&  C\,(R+1)^{n-2s} \int_0^{1-\frac{1}{R+1}} t^{n-1} (1-t)^{-2s} \, dt\\
			\leq \;& C\,(R+1)^{n-2s} \int_0^{1-\frac{1}{R+1}}  (1-t)^{-2s} \, dt,
	 \end{split}\]
for some~$C>0$. 
Now we observe that
	\[  \int_0^{1-\frac{1}{R+1}}  (1-t)^{-2s} \, dt \leq
	 	\begin{cases} \displaystyle \int_0^1 (1-t)^{-2s} \, dt =C \quad &\mbox{if} \quad s\in \Big(0, \frac{1}{2}\Big),\\
	 					-\log (1-t) \Big|_0^{1-\frac{1}{R+1}} \leq \log R \quad &\mbox{if} \quad s= \frac{1}{2},\\
	 					-\frac{(1-t)^{1-2s}}{1-2s} \Big|_0^{1-\frac{1}{R+1}} \leq C R^{2s-1} \quad &\mbox{if} \quad s\in \Big( \frac{1}{2},1 \Big) .
		\end{cases}
	\]
The latter two formulas and~\eqref{5.10bis}
imply~\eqref{acdx1}.

Now, we
define the set \[ A:= \{ v=\psi\} \] and notice that $B_{R+1}\subseteq A \subseteq B_{R+2}$. We prove
that for any $x\in A$ and any $y\in A^{\C}$ 
	\begin{equation}\label{8uUtY}
		|v(x)-v(y)| \leq \max\Big\{ |u(x)-u(y)|,|\psi(x)-\psi(y)|\Big\}.	
	\end{equation}
Indeed, for $x\in A$ and $y\in A^{\C}$ we have that
\[ 	v(x) =\psi(x)\leq u(x) \quad \mbox{and} \quad  v(y)=u(y)\leq \psi(y),\]
therefore
	\[ v(x)-v(y)\leq u(x)-u(y) \quad \mbox{and} \quad v(y)-v(x)\leq \psi(y)-\psi(x),\]
	which establishes~\eqref{8uUtY}. This leads to 
	\begin{equation} \label{acv1}
		v(A,A^{\C}) \leq u(A,A^{\C})+\psi(A,A^{\C}).
	\end{equation}
Notice now that \[\E(u,B_{R+2})\leq \E(v,B_{R+2})\] since $u$ is a minimizer in $B_{R+2}$ and $v=u$ outside $B_{R+2}$. We have that
	\[ \begin{split}
		\E (u,B_{R+2}) =\;& \frac{1}{2} \;u(B_{R+2},B_{R+2}) + u(B_{R+2}, B_{R+2}^{\C}) +\int_{B_{R+2}} W(u) \, dx\\
						=\;& \frac{1}{2}\; u(A,A) + u(A,A^{\C}) \\
						\;& + \frac{1}{2} \;u(B_{R+2}\setminus A,B_{R+2}\setminus A) + u(B_{R+2}\setminus A, B_{R+2}^{\C}) \\
						\;& + \int_{A} W(u) \, dx+ \int_{B_{R+2}\setminus A} W(u) \, dx.
	\end{split}\]
	Since $u$ and $v$ coincide on $A^{\C}$, by using the inequality \eqref{acv1} we obtain that
	\[ \begin{split}
		0\;\leq&\;  \E(v,B_{R+2})- \E(u,B_{R+2}) \\\;= &\; \frac{1}{2}\; v(A,A)-\frac{1}{2} u(A,A) + v(A,A^{\C}) -u (A,A^{\C}) + \int_{A} \Big(W(v)-W(u)\Big) \, dx \\
		\; \leq & \;  \frac{1}{2} \;v(A,A)-\frac{1}{2} u(A,A)  + \psi (A,A^{\C}) +  \int_{A}\Big( W(v)-W(u) \Big)\, dx.
	\end{split}\]
Moreover, $v=\psi$ on $A$ and we have that
		\[ \frac{1}{2}\; u(A,A) + \int_{A} W(u) \, dx \leq  \frac{1}{2} \; \psi (A, A) + \psi (A, A^{\C})  +\int_{A} W(\psi) \, dx=\E(\psi,A),\]
		and therefore, since $B_{R+1}\subseteq A \subseteq B_{R+2}$,
		\begin{equation}\label{acbr1} \frac{1}{2} \;u(B_{R+1},B_{R+1}) + \int_{B_{R+1}} W(u)\, dx \leq \E(\psi,B_{R+2}).\end{equation}
We estimate now $\E(\psi, B_{R+2})$. For a fixed $x \in B_{R+2}$ we observe that
		\[ \begin{split} \int_{\Rn}\al \frac{|\psi(x)-\psi(y)|^2}{|x-y|^{n+2s}} \, dy \\
		=\al  \int_{|x-y|\leq d(x)} \frac{|\psi(x)-\psi(y)|^2}{|x-y|^{n+2s}} \, dy  + \int_{|x-y|\geq d(x)} \frac{|\psi(x)-\psi(y)|^2}{|x-y|^{n+2s}} \, dy \\
				\leq \al  C \bigg( \frac{1}{d(x)^{2}} \int_{|x-y|\leq d(x)} |x-y|^{-n-2s+2} \, dy + \int_{|x-y|\geq d(x)} |x-y|^{-n-2s} \, dy \bigg),\end{split}\]
		where we have used \eqref{acpsi1} and the boundedness of $\psi$. Passing to polar coordinates, we have that
		\[ \begin{split}
			\int_{\Rn} \frac{|\psi(x)-\psi(y)|^2}{|x-y|^{n+2s}} \, dy \leq&\; C\bigg( \frac{1}{d(x)^{2}} \int_0^{d(x) } \rho^{-2s+1} \, d\rho + \int_{d(x)}^{\infty} \rho^{-2s-1}\, d\rho \bigg) \\ =&\;C d(x)^{-2s}. 
	\end{split}\]
Recalling
that~$\psi(x)=-1$ on $B_{R+1}$ and $W(-1)=0$, we obtain that
 \[ \begin{split}
 	\E(\psi, B_{R+2}) =&\;\int_{B_{R+2}}\int_{\Rn} \frac{|\psi(x)-\psi(y)|^2}{|x-y|^{n+2s}} \, dy \, dx + \int_{B_{R+2}} W(\psi) \, dx \\
 	\leq&\; \int_{B_{R+2}} d(x)^{-2s} dx + \int_{B_{R+2}\setminus B_{R+1}} W(\psi)\,dx .
 \end{split}\]
Therefore, making use of~\eqref{acdx1},
\begin{equation}\label{acpsi2} \E(\psi, B_{R+2}) \leq \begin{cases} 
			CR^{n-1} \quad &\mbox{if} \quad s\in \Big(\frac{1}{2},1\Big),\\
			CR^{n-1}\log R \quad &\mbox{if} \quad s= \frac{1}{2},\\
			CR^{n-2s} \quad &\mbox{if} \quad s\in \Big(0,\frac{1}{2}\Big).
		\end{cases}\end{equation} 
For what regards the right hand-side of inequality \eqref{acbr1}, we have that
		\eqlab{\label{bla111}
			\frac{1}{2} \;u(B_{R+1},B_{R+1}) + \int_{B_{R+1}} W(u)\, dx  \geq&\;   \frac{1}{2} \; u(B_{R},B_{R}) +u(B_R, B_{R+1}\setminus B_R) \\
			&\; +\int_{B_{R}} W(u)\, dx . }
We prove now that
\begin{equation}\label{5.16bis}
u(B_R,B_{R+1}^{\C})  \leq  \int_{B_{R+2}} d(x)^{-2s}\, dx .\end{equation}
For this, we observe that if $x\in B_R$, then~$d(x)=R+1-|x|$.
So, if~$x\in B_R$ and~$y \in B_{R+1}^{\C}$, then
\[|x-y|\geq |y|-|x| \geq R+1-|x| =d(x). \] Therefore, by changing variables $z=x-y$ and then passing to polar coordinates, we have that
	\bgs{ u(B_R,B_{R+1}^{\C}) \leq&\; 4\,\int_{B_R} dx \int_{B_{d(x)}^{\C}} |z|^{-n-2s} \, dz\\
				\leq&\; C\,\int_{B_R} dx \int_{d(x) }^{\infty} \rho^{-2s-1} \, d\rho \\
				=&\;C\, \int_{B_R} d(x)^{-2s} \, dx .}
This establishes~\eqref{5.16bis}.

	 Hence, by \eqref{acdx1} and~\eqref{5.16bis}, we have that
	 	\begin{equation} \label{acpsi3} 
u(B_R,B_{R+1}^{\C})  \leq  \int_{B_{R+2}} d(x)^{-2s}\, dx \leq 
		\begin{cases} 
			 CR^{n-1}  \quad & \mbox{if}  \quad s\in \Big(\frac{1}{2},1\Big),\\
			 CR^{n-1}\log R   \quad &\mbox{if} \quad s=\frac{1}{2},\\
			 CR^{n-2s}  \quad & \mbox{if}  \quad s\in \Big(0,\frac{1}{2}\Big) .
		\end{cases}\end{equation}
We also observe that, by adding $u(B_R,B_{R+1}^{\C}) $ to inequality \eqref{bla111}, we obtain that
	\[ \begin{aligned} 	\frac{1}{2} \al  u(B_{R+1},B_{R+1}) + \int_{B_{R+1}} W(u)\, dx  + u(B_R,B_{R+1}^{\C}) 
	\\
		&\qquad\ge   \frac{1}{2} \; u(B_{R},B_{R}) +u(B_R, B_{R+1}\setminus B_R) +\int_{B_{R}} W(u) \, dx +u(B_R,B_{R+1}^{\C})
\\ &\qquad=\E(u,B_R)
.\end{aligned}\] 
This and \eqref{acbr1} give that 
		\[ \E(u,B_R)\leq \E(\psi,B_{R+2} )+  u(B_R,B_{R+1}^{\C}).\] 
Combining this with the estimates in \eqref{acpsi2} and \eqref{acpsi3},
we obtain the desired result.
\end{proof}

Another type of estimate can be given in terms of the level sets of the minimizers (see Theorem 1.4 in \cite{SV14}).

\begin{thm}\label{TY78UU}
Let u be a minimizer of $\E$ in $B_R$. Then for any $\theta_1, \theta_2 \in (-1,1)$ such that \[ u(0)>\theta_1\] we have that
there exist~$\overline R$ and~$C>0$ such that
	\[ \Big| \{ u>\theta_2\} \cap B_R\Big| \geq C R^n\]
if $R \geq \overline R(\theta_1, \theta_2)$. The constant $C$ > 0 depends only on $n$, $s$ and $W$ and $\overline R(\theta_1, \theta_2)$ is
a large constant that depends also on $\theta_1$ and $\theta_2$.
\end{thm}

The statement of Theorem \ref{TY78UU}
says that the level sets of minimizers always occupy
a portion of a large ball comparable to the ball itself.
In particular, both phases occur in a large ball,
and the portion of the ball occupied by each phase is comparable
to the one occupied by the other.

Of course, the simplest situation in which two phases split
a ball in domains with comparable, and in fact equal, size
is when all the level sets are hyperplanes.
This question is related to a fractional version
of a classical conjecture of
De Giorgi and
to nonlocal minimal surfaces, that we discuss in the following Section~\ref{sbsdg} and Chapter~\ref{nlms}.

Let us try now to give some details on the proof of the Theorem \ref{TY78UU} in the particular case in which $s$ is in the range $(0,1/2)$. The more general proof for all $s \in (0,1)$ can be found in \cite{SV14}, where one uses some estimates on the Gagliardo norm. In our particular case we will make use of the Sobolev inequality that we introduced in \eqref{OK:sob:p}. The interested reader can see \cite{SavVal11} for a more exhaustive explanation of the upcoming proof. 

\begin{proof}[Proof of Theorem \ref{TY78UU}]
Let us consider a smooth function $w$ such that $w=1$ on $B_R^{\C}$ (we will take in sequel $w$ to be a particular barrier for $u$), and define 
	\[ v(x):= \min \{ u(x),w(x)\}.\] 
	Since $|u|\leq 1$, we have that $v=u$ in $B_R^{\C}$. Calling $D=\left(\Rn\times \Rn\right)\setminus \left(B_R^{\C}\times B_R^{\C} \right)$ we have from definition \eqref{kenac} that
		\bgs{ \mathcal {K} (u-v,\al  B_R)  + \mathcal {K} (v,B_R) -\mathcal {K}(u,B_R)  \\ =\al
		\frac{1}{2} \iint_D \frac{ |(u-v)(x) -(u-v)(y)|^2 +|v(x)-v(y)|^2 -|u(x)-u(y)|^2}{|x-y|^{n+2s}} \, dx\, dy.}
		We use the algebraic identity $|a-b|^2 +b^2-a^2 = 2b(b-a)$ with $a=u(x)-u(y)$ and $b=v(x)-v(y)$ to obtain that
		\bgs{ \mathcal {K} (u-v,B_R) \al  + \mathcal {K} (v,B_R) -\mathcal {K}(u,B_R) \\
		= \al  \iint_D \frac{ \left( (u-v)(x) -(u-v)(y)\right)\left(v(y)-v(x)\right)} {|x-y|^{n+2s}} \, dx\, dy.} 
	Since $u-v=0$ on $B_R^{\C}$ we can extend the integral to the whole space $\Rn \times \Rn$, hence 
		\bgs{ \mathcal {K}   (u-v,B_R) \al + \mathcal {K} (v,B_R) -\mathcal {K}(u,B_R) \\
		= \al  \iint_{\Rn\times\Rn} \frac{ \left( (u-v)(x) -(u-v)(y)\right)\left(v(y)-v(x)\right)} {|x-y|^{n+2s}} \, dx\, dy.}
	Then  by changing variables and using the anti-symmetry of the integrals, we notice that
		\bgs{ \iint_{B_R\times B_R} \al   \frac{ \left( (u-v)(x) -(u-v)(y)\right)\left(v(y)-v(x)\right)} {|x-y|^{n+2s}} \, dx\, dy \\
		=\al   \iint_{B_R\times B_R}  \frac{ (u-v)(x) \left(v(y)-v(x)\right)} {|x-y|^{n+2s}} \, dx\, dy\\
		\al  - \iint_{B_R\times B_R} \frac{(u-v)(y)\left(v(y)-v(x)\right)} {|x-y|^{n+2s}} \, dx\, dy \\
		=\al 2  \iint_{B_R\times B_R}   \frac{ (u-v)(x) \left(v(y)-v(x)\right)} {|x-y|^{n+2s}} \, dx\, dy}
		and
		\bgs{ \iint_{B_R\times B_R^{\C}}\al   \frac{ \left( (u-v)(x) -(u-v)(y)\right)\left(v(y)-v(x)\right)} {|x-y|^{n+2s}} \, dx\, dy \\
		 \al +   \iint_{B_R^{\C} \times B_R}  \frac{ \left( (u-v)(x) -(u-v)(y)\right)\left(v(y)-v(x)\right)} {|x-y|^{n+2s}} \, dx\, dy  \\
			=\al \iint_{B_R\times B_R^{\C}}   \frac{ (u-v)(x) \left(v(y)-v(x)\right)} {|x-y|^{n+2s}} \, dx\, dy\\
			\al  -\iint_{B_R^{\C} \times B_R} \frac{ (u-v)(y) \left(v(y)-v(x)\right)} {|x-y|^{n+2s}} \, dx\, dy  \\
			= \al 2 \iint_{B_R\times B_R^{\C}}  \frac{  (u-v)(x) \left(v(y)-v(x)\right)} {|x-y|^{n+2s}} \, dx\, dy.}
	Therefore
		\bgs{ \mathcal {K}   (u-v,B_R) \al + \mathcal {K} (v,B_R) -\mathcal {K}(u,B_R) \\
		= \al 2 \iint_{\Rn\times\Rn}   \frac{\left( u(x)-v(x)\right)\left(v(y)-v(x)\right) }{|x-y|^{n+2s}}\,dx \, dy\\
		=\al 2 \int_{\Rn} (u(x)-v(x))\left(\int_{\Rn} \frac{ v(y)-v(x)}{|x-y|^{n+2s}}\, dy\right)\, dx\\
		=\al 2 \int_{B_R\cap\{u>v=w\}} (u(x)-w(x))\left(\int_{\Rn} \frac{ v(y)-w(x)}{|x-y|^{n+2s}}\, dy\right)\, dx\\
		\leq \al  2 \int_{B_R\cap\{u>v=w\}} (u-w)(x)\left(\int_{\Rn} \frac{ w(y)-w(x)}{|x-y|^{n+2s}}\, dy\right)\, dx\\
		=\al  2 \int_{B_R\cap\{u>w\}} (u-w)(x)\left(- \frlap w \right)(x)\, dx.}
Hence
	\bgs{  \mathcal {K}  (u\al-v,B_R)  \\
	 \leq \al \mathcal {K} (u,B_R) -\mathcal {K}(v,B_R) + 2 \int_{B_R\cap\{u>w\}} (u-w)\left(- \frlap w \right)\, dx.} By adding and subtracting the potential energy, we have that 
	\bgs{  \mathcal {K}  (u\al-v,B_R) \\ \leq \al \E(u,B_R)-\E(v,B_R) +\int_{B_R} W(v)-W(u) \, dx \\ \al + 2 \int_{B_R\cap\{u>w\}} (u-w)\left(- \frlap w \right)\, dx} and since $u$ is minimal in $B_R$,
	\eqlab{\label{kuvr1} \mathcal {K}  (u-v,B_R) \leq \al   \int_{B_R\cap \{ u>w=v\}} W(w)-W(u) \, dx \\ \al+ 2 \int_{B_R\cap\{u>w\}} (u-w)\left(- \frlap w \right)\, dx.}
We deduce from the properties in \eqref{dwp} of the double-well potential $W$ that there exists a small constant $c>0$ such that
	\bgs{ &W(t)-W(r) \geq c(1+r)(t-r) +c(t-r)^2 & \mbox{ when } & -1\leq r\leq t\leq -1+c \\
		&W(r)-W(t) \leq \frac{1+r}{c} &\mbox{ when } &-1\leq r\leq t\leq 1.}
	We fix the arbitrary constants $\theta_1$ and $\theta_2$, take $c$ small as here above. Let then \[ \theta_{\star}:=\min\{\theta_1,\theta_2,-1+c\}.\] It follows that
	\eqlab{ \label{wwu} \int_{B_R\cap \{ u>w\} }\al W(w)-W(u) dx \\
	=\al  \int_{B_R\cap \{ \theta_{\star}>u>w\} }W(w)-W(u) dx 	+ \int_{B_R\cap \{ u>\max\{\theta_{\star},w\} \} }W(w)-W(u) dx \\
	\leq \al -c  \int_{B_R\cap \{ \theta_{\star}>u>w\} } (1-w)(u-w) \,dx - c\int_{B_R\cap \{ \theta_{\star}>u>w\} }(u-w)^2\, dx \\ 
	\al+ \frac{1}{c} \int_{B_R\cap \{ u>\max\{\theta_{\star},w\} \} } (1+w)\, dx\\
	\leq  \al -c  \int_{B_R\cap \{ \theta_{\star}>u>w\} } (1-w)(u-w) \,dx + \frac{1}{c} \int_{B_R\cap \{ u>\max\{\theta_{\star},w\} \} } (1+w)\, dx.}
Therefore, in \eqref{kuvr1} we obtain that 
\eqlab{\label{kuvr2} \mathcal {K}  (u-v,B_R) \leq \al   -c  \int_{B_R\cap \{ \theta_{\star}>u>w\} } (1-w)(u-w) \,dx \\ \al + \frac{1}{c} \int_{B_R\cap \{ u>\max\{\theta_{\star},w\} \} } (1+w)\, dx  \\ \al+ 2 \int_{B_R\cap\{u>w\}} (u-w)\left(- \frlap w \right)\, dx.} 
We introduce now a useful barrier in the next Lemma (we just recall here Lemma 3.1 in \cite{SV14} - there the reader can find how this barrier is build):
\begin{lemma}\label{bW1} Given any $\tau\geq 0$ there exists a constant $C> 1$ (possibly depending on $n, s$ and $\tau$) such that: for any $R\geq C$ there exists a rotationally symmetric function $w \in C\left(\Rn, [-1+CR^{-2s},1]\right)$ with $w=1$ in $B_R^{\C}$ and such that for any $x\in B_R$ one has that
	\eqlab{ \label{w1} \frac{1}{C} (R+1-|x|)^{-2s} \leq 1+w(x)\leq C (R+1-|x|)^{-2s} \quad \mbox{and} }
	\eqlab{\label{w2} -\frlap w(x)\leq \tau (1+w(x)).}
\end{lemma} 
Taking $w$ as the barrier introduced in the above Lemma, thanks to \eqref{kuvr2} and to the estimate in \eqref{w2}, we have that
	\bgs{ \mathcal {K}  (u-v,B_R) \leq \al  -c  \int_{B_R\cap \{ \theta_{\star}>u>w\} } (1+w)(u-w) \,dx \\ \al + \frac{1}{c} \int_{B_R\cap \{ u>\max\{\theta_{\star},w\} \} } (1+w)\, dx  \\ \al + 		2\tau \int_{B_R\cap\{u>w\}} (u-w)(1+w) \, dx .}
Let then $\tau=\frac{c}2$, and we are left with
	\bgs{ \mathcal {K}  (u-v,B_R) \leq  \al c \int_{B_R\cap \{ u>\max\{\theta_{\star},w\} \}} (u-w)(1+w)\, dx \\ \al + \frac{1}{c} \int_{B_R\cap \{ u>\max\{\theta_{\star},w\} \} } (1+w)\, dx\\
	\leq \al C_1 \int_{B_R\cap \{ u>\max\{\theta_{\star},w\} \} } (1+w)\, dx ,}
	with $C_1$ depending on $c$ (hence on $W$).
Using again Lemma \ref{bW1}, in particular the right hand side inequality in \eqref{w1}, we have that
	\[ 	  \mathcal {K}  (u-v,B_R) \leq C_1 \cdot C \int_{B_R\cap \{ u>\max\{\theta_{\star},w\} \} } (R+1-|x|)^{-2s}.\]
We set
	\eqlab{ \label {vrrr} V(R):= |B_R\cap \{ u>\theta_{\star} \} | } and the Co-Area formula then gives
	\eqlab{ \label{strange1}  \mathcal {K}  (u-v,B_R) \leq C_2 \int_0^R (R+1-t)^{-2s} V'(t)\, dt,}
	where $C_2$ possibly depends on $n,s,W$.
	
	 We use now the Sobolev inequality \eqref{OK:sob:p} for $p=2$, applied to $u-v$ (recalling that the support of $u-v$ is a subset of $B_R$) to obtain that
	 	\eqlab{\label{sobk1} \mathcal K(u-v,B_R) \al =  \mathcal K(u-v,\Rn) = \iint_{\Rn \times \Rn} \frac{ |(u-v)(x)-(u-v)(y)|^2}{|x-y|^{n+2s}}\, dx\, dy \\
	 	\geq \al  \tilde C \|u-v\|^2_{L^{\frac{2n}{n-2s}}(\Rn)} = \tilde C \|u-v\|^2_{L^{\frac{2n}{n-2s}}(B_R)}.}
From \eqref{w1} one has that
	\[ w(x)\leq C(R+1-|x|)^{-2s} -1.\] We fix $K$ large enough so as to have $R\geq 2K$ and in $B_{R-K}$ 
		\[ w(x) \leq C(1+K)^{-2s} -1 \leq -1 + \frac{1+\theta_{\star}}2.\]
		Therefore in $B_{R-K}\cap \{u>\theta_{\star} \}$  we have that
		\[ |u-v|\geq u-w \geq u+1- \frac{1+\theta_{\star}}2\geq \frac{1+\theta_{\star}}2 .\] Using definition \eqref{vrrr}, this leads to
			\bgs{ \|u-v\|^2_{L^{\frac{2n}{n-2s}}(B_R)} =\al  \left( \int_{B_R} |u-v|^{\frac{2n}{n-2s}} dx\right) ^{\frac{n-2s}{n}}\\
			\geq \al  \left( \frac{1+\theta_{\star}}2\right)^{\frac{2n}{n-2s}} \left( \int_{ B_{R-K}\cap \{u>\theta_{\star} \} } dx \right)^{\frac{n-2s}{n}} \\
			\geq\al C_3 V(R-K)^{\frac{n-2s}{n}}.}
In \eqref{sobk1}  we thus have  	\[ \mathcal K (u-v,B_R) \geq \tilde C_3 V(R-K)^{\frac{n-2s}{n}}\] and from \eqref{strange1} it follows that
	\bgs{ C_4 V(R-K)^{\frac{n-2s}{n}} \leq \int_0^R (R+1-t)^{-2s} V'(t) \, dt.}
Let $R\geq \rho \geq 2K$. Integrating the latter integral from $\rho$ to $\displaystyle \frac{3\rho}2$ we have that
	\bgs{ C_4 \frac{\rho}2 V(\rho-K)^{\frac{n-2s}{n}} \leq \al C_4   \int_\rho^{\frac{3\rho}2} V(R-K)^{\frac{n-2s}{n}} \, dR \\
		\leq \al \int_0^{\frac{3\rho}2}\left(\int_0^R (R+1-t)^{-2s} V'(t) \, dt\right) \, dR \\	
		=\al \int_0^{\frac{3\rho}2} V'(t)  \left(\int_0^{\frac{3\rho}2} (R+1-t)^{-2s} \, dR\right)\, dt \\
		=\al\int_0^{\frac{3\rho}2} V'(t)  \frac{ \left(\frac{3\rho}2+1-t\right)^{1-2s}-1}{1-2s} \, dt.}
Since $1-2s>0$, one has for large $\rho$ that $\displaystyle \left(\frac{3\rho}2+1-t\right)^{1-2s}-1 \leq (2\rho)^{1-2s}$, hence, noticing that the function $V$ is nondecreasing,
	\bgs{  \frac{\rho}2 V(\rho-K)^{\frac{n-2s}{n}} \leq \al C_5 \rho^{1-2s} \int_0^{2\rho} V'(t)\, dt\\
	\leq \al C_5 \rho^{1-2s} V(2\rho).}
	Therefore 	\eqlab{ \label{rhovk1} \rho^{2s} V(\rho-K)^{\frac{n-2s}{n}} \leq2 C_5 V(2\rho).}
Now we use an inductive argument as in Lemma 3.2 in \cite{SV14}, that we recall here:
\begin{lemma} \label{indargl} Let $\sigma, \mu\in (0,\infty), \nu \in (\sigma,\infty) $ and $\gamma, R_0,C\in (1,\infty)$. \\ Let $V \colon (0,\infty)\to (0,\infty)$ be a nondecreasing function. For any $r\in [R_0,\infty)$, let $\alpha(r) := \min\left\{ 1, \displaystyle \frac{ \log V(r)} {\log r} \right\} $. Suppose that $V(R_0)>\mu$  and 	\[ r^\sigma \alpha(r) V(r)^{\frac{\nu-\sigma}{\nu}} \leq C V(\gamma r),\]
for any $r\in [R_0,\infty)$. Then there exist $c\in (0,1)$ and $R_{\star} \in [R_0,\infty)$, possibly depending on $\mu,\nu, \gamma, R_0, C$ such that 
\[ V(r) >c r^{\nu} ,\]
for any $r\in [R_{\star},\infty)$. 
\end{lemma}

For $R$ large, one obtains from \eqref{rhovk1} and Lemma \ref{indargl} that
	\[ V(R) \geq c_0 R^n,\] for a suitable $c_0\in(0,1)$.
Let now \[ \theta^{\star} :=\max \{ \theta_1, \theta_2, -1+c\} .\] We have that
	\eqlab { \label{uteta}|\{ u>\al  \theta^{\star}\} \cap B_R | + |\{ \theta_{\star}<u<\theta^{\star}\}\cap B_R | \\ 
	=\al  |\{ u>\theta_{\star}\} \cap B_R |\\
	 =\al  V(R)\geq c_0 R^n.}
Moreover, from \eqref{acenest1} we  have that for some $\overline c>0$
		\[ \E(u,B_R) \leq \overline c R^{n-2s} ,\] therefore
			\bgs{ \overline c R^{n-2s} \geq \al \E(u,B_R) \geq  \int_{ \{ \theta_{\star}<u<\theta^{\star}\}\cap B_R} W(u)\, dx \\
			\geq \al  \inf_{t\in (\theta_{\star},\theta^{\star})}W(t)\, |\{ \theta_{\star}<u<\theta^{\star}\}\cap B_R |  .}
			From this and \eqref{uteta} we have that
				\bgs{ c_0 R^n \leq \overline C R^{n-2s} + |\{ u>\theta^{\star}\} \cap B_R |  ,}
				and finally \bgs{ |\{ u>\theta^{\star}\} \cap B_R |  \geq C R^n,}
				with $C$ possibly depending on $n,s,W$. This concludes the proof of Theorem \ref{TY78UU} in the case $s\in (0,1/2)$.
\end{proof}

\section{A nonlocal version of a conjecture by De Giorgi} \label{sbsdg}

In this section we consider the fractional 
counterpart of
the conjecture by De Giorgi that was discussed before in the classical case.
Namely,
we consider the nonlocal Allen-Cahn equation 
	\[ -\frlap u + W(u)=0 \quad \mbox{in} \quad \Rn, \]
where $W$ is a double-well potential, and $u$ is smooth, bounded and monotone in one direction, namely $|u|\leq 1$ and $\partial_{x_n} u >0$. We wonder if it is also true, at least in low
dimension, that $u$ is one-dimensional. In this case, the conjecture was initially proved for $n=2$ and $s=\frac{1}{2}$ in \cite{CM05}. In the case $n=2$, for any $s \in (0,1)$, the result is proved using the harmonic extension of the fractional Laplacian in \cite{CS15} and \cite{SV09}. For $n=3$, the proof can be found in \cite{CC10} for $s\in \Big[\frac{1}{2},1\Big]$. The conjecture is still open for $n=3$ and  $s\in \Big[0,\frac{1}{2}\Big]$ and for $n\geq 4$. Also, the Gibbons conjecture (that is the 
De Giorgi conjecture with the additional condition that limit in \eqref{limdgs} is uniform) is also true for any $s \in (0,1)$
and in any dimension~$n$, see~\cite{INDIANA}.
 
To keep the discussion as simple as possible,
we focus here on the case $n=2$ and any $s\in (0,1)$, providing an alternative proof that does not make use of the harmonic extension. This part is completely new and not available in the literature. The proof is indeed quite
general and it will be further exploited in \cite{CV15}.
 
 We define (as in \eqref{kenac}) the total energy of the system to be 
 \begin{equation}\label{5.17bis}
\E(u, B_R) =  \mathcal{K}_R(u) + \int_{B_R} W(u) dx,\end{equation}
 where the kinetic energy is 
\begin{equation}  \label{dgkenac} {\mathcal{K}}_R(u):=\frac{1}{2}
\iint_{Q_R} \frac{|u(x)-u(\bar x)|^2}{|x-\bar x|^{n+2s}}\,dx\,d\bar x,\end{equation}
and~$Q_R:= \R^{2n}\setminus (B_R^{\C})^2= (B_R\times B_R)\cup (B_R\times(\R^n\setminus B_R))\cup
((\R^n\setminus B_R)\times B_R)$. We recall that the kinetic energy can also be written as
	\begin{equation} \label{dgkr}  {\mathcal{K}}_R(u) = \frac{1}{2} u(B_R,B_R) + u(B_R,B_R^{\C}),\end{equation}
	where for two sets $A,B$ \begin{equation} \label{dguab} u(A,B)= \int_A\int_B \frac{|u(x)-u(\bar x)|^2}{|x-\bar x|^{n+2s}}\,dx\,d\bar x.\end{equation}
	
The main result of this section is the following. 

\begin{thm}\label{dgdim2}
Let $u$ be a minimizer of the energy defined in \eqref{5.17bis}
in any ball of $\R^2$. Then $u$ is $1$-D, i.e.  there exist~$\omega \in S^{1}$ and $u_0: \R\to \R$ such that 	
	\[ u(x)=u_0(\omega \cdot x) \quad \mbox{for any} \quad x\in \R^2.\] 	
\end{thm}

The proof relies on the following estimate for the kinetic energy, that we prove
by employing a domain deformation technique.   

\begin{lemma} \label{endg}
Let~$R>1$, $\varphi\in C^\infty_0(B_1)$. Also, for any~$y\in\R^n$, let
\begin{equation}\label{5252}
\Psi_{R,+}(y):=y+\varphi \Big(\frac{ y}{R}\Big)\,e_1
\ {\mbox{ and }} \ \Psi_{R,-}(y):=y-\varphi \Big(\frac{ y}{R}\Big)\,e_1.\end{equation}
Then, for large~$R$, the maps~$\Psi_{R,+}$ and~$\Psi_{R,-}$
are diffeomorphisms on~$\R^n$.
Furthermore, if we define~$u_{R,\pm}(x):= u(\Psi_{R,\pm}^{-1}(x))$,
we have that
\begin{equation}\label{DG01}
{\mathcal{K}}_R (u_{R,+})+{\mathcal{K}}_R (u_{R,-})-2
{\mathcal{K}}_R (u)\le \frac{C}{R^2}{\mathcal{K}}_R (u),
\end{equation}
for some~$C>0$. 
\end{lemma}

\begin{proof} First of all, we compute the Jacobian of~$\Psi_{R,\pm}$.
For this, we write~$\Psi_{R,+,i}$ to denote the~$i^{\text{th}}$
component
of the vector~$\Psi_{R,+}=(\Psi_{R,+,1},\cdots,\Psi_{R,+,n})$
and we observe that
\begin{equation}\label{JA}
\frac{\partial \Psi_{R,+,i}(y)}{\partial y_j}=
\frac{\partial }{\partial y_j} \Big(y_i\pm 
\varphi \Big(\frac{ y}{R}\Big) \delta_{i1} \Big) = \delta_{ij} \pm \frac{ 1}{R}
\partial_j\varphi \Big(\frac{  y}{R}\Big) \delta_{i1}.\end{equation}
The latter term is bounded by~${\mathcal{O}}(R^{-1})$, and
this proves that~$\Psi_{R,\pm}$ is a diffeomorphism if~$R$ is large enough.

For further reference, we point out that if~$J_{R,\pm}$
is the Jacobian determinant of~$\Psi_{R,\pm}$, then
the change of variable
\begin{equation}\label{JA0}
x:=\Psi_{R,\pm}(y),\qquad
\bar x:=\Psi_{R,\pm}(\bar y)\end{equation}
gives that
\begin{equation*}
	\begin{split}
		dx\,d\bar x \;=\;& J_{R,\pm}(y)\,J_{R,\pm}(\bar y)\,dy\,d\bar y \\
		 =\;& 		\bigg(1\pm  \Big(\frac{ 1}{R} \Big)\partial_1 \varphi \Big(\frac{ y}{R}\Big) +{\mathcal{O}}\Big(\frac{ 1}{R^2}\Big)\bigg)
		\bigg(1\pm  \frac{ 1}{R}\partial_1 \varphi \Big(\frac{ \bar y}{R} \Big)+{\mathcal{O}}\Big(\frac{ 1}{R^2}\Big) \bigg) dy d\bar y
		\\  \;=\;& 1\pm  \frac{ 1}{R} \partial_1 \varphi \Big(\frac{ y}{R}\Big) \pm   \frac{ 1}{R}
		\partial_1 \varphi \Big(\frac{ \bar y}{R} \Big) +{\mathcal{O}}\Big(\frac{ 1}{R^2}\Big)	
		\,dy\,d\bar y
		,\end{split}
\end{equation*}
thanks to~\eqref{JA}. Therefore
\begin{equation}\label{JA8}
\begin{split}
& \frac{|u_{R,\pm}(x)-u_{R,\pm}(\bar x)|^2}{|x-\bar x|^{n+2s}}\,dx\,d\bar x\\
&\quad=\frac{\big|u(\Psi_{R,\pm}^{-1}(x))-
u(\Psi_{R,\pm}^{-1}(\bar x))\big|^2}{
|\Psi_{R,\pm}^{-1}(x)-\Psi_{R,\pm}^{-1}(\bar x)|^{n+2s}}\cdot
\left(
\frac{ |x-\bar x|^2
}{|\Psi_{R,\pm}^{-1}(x)-\Psi_{R,\pm}^{-1}(\bar x)|^2}\right)^{-\frac{n+2s}{2}}\,dx\,d\bar x
\\ &\quad= \frac{|u(y)-u(\bar y)|^2}{|y-\bar y|^{n+2s}} \cdot \left( 
\frac{ \Big|\Psi_{R,\pm}(y)- \Psi_{R,\pm}(\bar y)\Big|^2 }{|y-\bar 
y|^2}\right)^{-\frac{n+2s}{2}} \\&\qquad\cdot \Bigg( 1\pm  \frac{ 1}{R} \partial_1 \varphi \Big(\frac{ y}{R}\Big) \pm   \frac{ 1}{R}
		\partial_1 \varphi \Big(\frac{ \bar y}{R} \Big) +{\mathcal{O}}\Big(\frac{ 1}{R^2}\Big)	\Bigg)
		\,dy\,d\bar y. \end{split} \end{equation} Now, for 
any~$y$, $\bar y\in\R^n$ we calculate 
	\begin{equation}
		\label{J1}
			\begin{split} 
			&\Big|\Psi_{R,\pm}(y)-\Psi_{R,\pm}(\bar y)\Big|^2\\
			&\quad= \Big| (y-\bar y)\pm \bigg(\varphi \Big(\frac{y}{R}\Big)-\varphi \Big(\frac{\bar y}{R}\Big)\bigg) \,e_1\Big|^2 \\
			 &\quad= |y-\bar y|^2 +\bigg|\varphi \Big(\frac{y}{R}\Big)-\varphi \Big(\frac{\bar y}{R}\Big)\bigg|^2 \pm 2 \bigg(\varphi \Big(\frac{y}{R}\Big) -\varphi \Big(\frac{\bar y}{R}\Big) \bigg) \,(y_1-\bar y_1). 					\end{split}
	\end{equation}
 Notice also  that 
 	\begin{equation}\label{j9} 
 		\bigg|\varphi \Big(\frac{y}{R}\Big)-\varphi \Big(\frac{\bar y}{R}\Big)\bigg| \le \frac{1}{R} \|\varphi\|_{C^1(\R^n)} |y-\bar y|,
	\end{equation} 
hence~\eqref{J1} becomes
\[ \frac{ \Big|\Psi_{R,\pm}(y)-\Psi_{R,\pm}(\bar y)\Big|^2 }{|y-\bar y|^2} 
=1+\eta_{\pm}\]
 where
  \begin{equation}\label{JA7} \eta_{\pm}:= \frac{ \bigg|\varphi \Big(\frac{y}{R}\Big)-\varphi \Big(\frac{\bar y}{R}\Big)\bigg|^2}{|y-\bar y|^2}\pm 2\frac{ \bigg(\varphi \Big(\frac{y}{R}\Big)-\varphi \Big(\frac{\bar y}{R}\Big)\bigg)  \,(y_1-\bar y_1)}{|y-\bar y|^2}={\mathcal{O}}\Big(\frac{1}{R}\Big).
\end{equation} 

As a consequence 
\[  \left( \frac{ \Big|\Psi_{R,\pm}(y)- \Psi_{R,\pm}(\bar y)\Big|^2 }{|y-\bar y|^2}\right)^{-\frac{n+2s}{2}}\\ = 
(1+\eta_\pm)^{-\frac{n+2s}{2}}= 1-\frac{n+2s}{2} \eta_\pm +{\mathcal{O}}(\eta_\pm^2). \]
 We plug this information into~\eqref{JA8} and use~\eqref{JA7} to 
obtain 
\begin{eqnarray*} && \frac{|u_{R,\pm}(x)-u_{R,\pm}(\bar x)|^2}{|x-\bar x|^{n+2s}}\,dx\,d\bar x\\ 
&=& \frac{|u(y)-u(\bar y)|^2}{|y-\bar y|^{n+2s}} \cdot \left(1-\frac{n+2s}{2} \eta_\pm +{\mathcal{O}}\Big(\frac{1}{R^2}\Big) \right) \\
&&\cdot \bigg(1\pm \frac{1}{R}\partial_1 \varphi \Big(\frac{y}{R}\Big)\pm \frac{1}{R} \partial_1 \varphi \Big(\frac{\bar y}{R} \Big)+{\mathcal{O}}\Big(\frac{1}{R^2}\Big)\bigg) 
\,dy\,d\bar y\\ 
&=& \frac{|u(y)-u(\bar y)|^2}{|y-\bar y|^{n+2s}} \cdot \Bigg[ 1-\frac{n+2s}{2} \eta_\pm + \, \bigg(  \pm \frac{1}{R} \partial_1 \varphi \Big(\frac{ y}{R} \Big)   \pm \frac{1}{R} \partial_1 \varphi \Big(\frac{\bar y}{R}\Big) \bigg) \\
	&&+{\mathcal{O}}\Big(\frac{1}{R^2}\Big)\Bigg]\,dy\,d\bar y.\end{eqnarray*}
Using this and the fact that $$ \eta_+ \,+\, \eta_-= 2\,\frac{ \bigg|\varphi \Big( \frac{y}{R}\Big)-\varphi \Big(\frac{\bar y}{R}\Big)\bigg|^2}{|y-\bar y|^2}={\mathcal{O}}\Big(\frac{1}{R^2}\Big),$$ thanks 
to~\eqref{j9},
we obtain \begin{eqnarray*} && 
\frac{|u_{R,+}(x)-u_{R,+}(\bar x)|^2}{|x-\bar x|^{n+2s}}+ 
\frac{|u_{R,-}(x)-u_{R,-}(\bar x)|^2}{|x-\bar x|^{n+2s}} \,dx\,d\bar 
x\\ &=& \frac{|u(y)-u(\bar y)|^2}{|y-\bar y|^{n+2s}} \cdot \left( 2 
+{\mathcal{O}}\Big(\frac{1}{R^2}\Big)\right)
\,dy\,d\bar y 
.\end{eqnarray*}
Thus, if we integrate over~$Q_R$
we find that
$$ {\mathcal{K}}_R(u_{R,+})+{\mathcal{K}}_R(u_{R,+})
= 2{\mathcal{K}}_R(u) +
\iint_{Q_R} {\mathcal{O}}\Big(\frac{1}{R^2}\Big)\,
\frac{|u(x)-u(\bar x)|^2}{|x-\bar x|^{n+2s}}\,dx\,d\bar x.$$
This establishes~\eqref{DG01}.
\end{proof}

 \begin{proof}[Proof of Theorem \ref{dgdim2}]
We organize this proof into four steps. 
 
\noindent \textbf{Step 1.} \textbf{A geometrical consideration}\\
In order to prove that the level sets are flat, it suffices to prove that $u$ is monotone in any direction.  Indeed, if $u$ is monotone in any direction, the level set  $\{ u=0\}$ is both convex and concave, thus it is flat. 
 \bigskip
 
\noindent  \textbf{Step 2.} \textbf{Energy estimates}\\
 	 Let $\varphi \in C_0^{\infty}(B_1)$ such that $\varphi = 1 $ in $B_{1/2}$, and let $e=(1,0)$. We define as in Lemma \ref{endg} 
		\[ \Psi_{R,+}(y):=y+\varphi \Big(\frac{ y}{R}\Big)\,e \ {\mbox{ and }} \ \Psi_{R,-}(y):=y-\varphi \Big(\frac{ y}{R}\Big)\,e,\]
which are diffeomorphisms for large $R$, and the functions ~$u_{R,\pm}(x):= u(\Psi_{R,+}^{-1}(x))$. Notice that 
	 \begin{align}
		&u_{R,+} (y)= u(y) \; &\text{for} \; &y \in B_R^{\C} \label{urpiur1}\\
		&u_{R,+} (y)= u(y-e) \; &\text{for} \; &y \in B_{R/2}\label{urpiur2} .
		\end{align} 
     By computing the potential energy, it is easy to see that
    	\[ \begin{split}
    	\int_{B_R} W(u_{R,+}(x)) \, dx &+\int_{B_R} W(u_{R,-}(x)) \, dx - 2 \int_{B_R} W(u(x)) \, dx \\
    		&\leq \frac{C}{R^2}  \int_{B_R} W(u(x)) \, dx.\end{split}\]
Using this and~\eqref{DG01}, we obtain
the following estimate for the total energy
    	\begin{equation}\label{5.29bis}
\E(u_{R,+},B_R)+\E(u_{R,-},B_R) - 2\E(u,B_R) \leq \frac{C}{R^2} \E(u,B_R).\end{equation}  
Also, since $u_{R,\pm}=u$ in $B_R^{\C}$, we have that
    	\[ \E(u,B_R) \leq \E(u_{R,-},B_R).\]
This and~\eqref{5.29bis} imply that
    		\begin{equation} \label {bla13} \E(u_{R,+},B_R) -\E(u,B_R) \leq \frac{C}{R^2} \E(u,B_R).\end{equation}
As a consequence of this estimate and~\eqref{THANKS},
it follows that
  		\begin{equation} \label{dge11} \lim_{R \to+ \infty} \Big(\E(u_{R,+},B_R) -\E(u,B_R) \Big)=0.
		\end{equation}
		
\bigskip 

  \noindent  \textbf{Step 3.} \textbf{Monotonicity}\\
     We claim that $u$ is monotone. Suppose by contradiction that $u$ is not monotone. 
     That is, up to translation and dilation, we suppose that the value of~$u$ at the 
     origin stays above the values of~$e$ and~$-e$, with $e:=(1,0)$, i.e. 
     \[ u(0) > u(e) \;{\mbox{ and }}\;u(0)>u(-e).\]
     Take $R$ to be large enough, say $R>8$. Let now 
  	\begin{equation}\label{5.31bis}
v_R(x):= \min \big\{u(x),u_{R,+}(x)\big\} \quad \mbox{and} \quad w_R(x):= \max \big\{u(x),u_{R,+}(x)\big\} .\end{equation}
  	By \eqref{urpiur1} we have that $v_R =w_R =u $ outside $B_R$. Then, since $u$ is a minimizer in $B_R$ and $w_R=u$ outside $B_R$, we have that
  		\begin{equation}\label{oPGHll} \E(w_R,B_R) \geq \E(u,B_R) .\end{equation}
Moreover, the sum of the energies of
the minimum and the maximum is less than or equal to the sum
of the original energies: this is obvious in the local case,
since equality holds, and in the nonlocal case the proof
is based on the inspection of the different integral contributions, see e.g.
formula~(38) in~\cite{PSV13}. So we have that
\[  \E(v_R,B_R) + \E(w_R,B_R) \leq\E(u,B_R) +\E(u_{R,+},B_R)\]
hence, recalling~\eqref{oPGHll},
	 \begin{equation} \label{dge12}   \E(v_R,B_R)\leq   \E(u_{R,+},B_R) .\end{equation} 
  	
We claim that~$v_R$ is not identically neither $u$, nor $u_{R,+}$. Indeed, since $u(0)= u_{R,+}(e)$ and $u(-e)= u_{R,+}(0)$ we have that
 	\[\begin{split}	 v_R(0) \;=&\;\min \big\{u(0),u_{R,+}(0)\big\}=
\min\big\{u(0),u(-e)\big\}\\ =&\;u(-e) = u_{R,+}(0) <u(0)  \quad \mbox{and} \\
 					v_R(e) \;= &\;\min \big\{u(e),u_{R,+}(e)\big\}=\min\big\{u(e),u(0)\big\}\\
=&\;u(e) <u(0) = u_{R,+}(e).  \end{split}\]
 					By continuity of $u$ and $u_{R,+}$, we have that
 					\begin{equation} \label{dgcuur1} \begin{split} v_R \;=&\;u_{R,+}<u  \mbox{ in a neighborhood of } 0 \quad \mbox{and} \\
 									  v_R \;=&\;u <u_{R,+} \mbox{ in a neighborhood of } e .\end{split}\end{equation}
 We focus our attention on the energy in the smaller ball $B_2$. We claim that $v_R$ is not minimal for $\E(\cdot, B_2)$. Indeed, if $v_R$ were minimal in $B_2$, then on $B_2$ both $v_R$ and $u$ would satisfy the same equation. However, $v_R \leq u$ in $\R^2$ by definition and $v_R=u$ in a neighborhood of $e$ by the second statement in \eqref{dgcuur1}. The Strong Maximum Principle implies that they coincide everywhere, which contradicts the first line in \eqref{dgcuur1}. 

Hence $v_R$ is not a minimizer in $B_2$. Let then $v^*_R$ be a minimizer of $\E(\cdot, B_2)$, that agrees with $v_R$ outside the ball $B_2$, and we define the positive quantity
 	\begin{equation}\delta_R: = \E(v_R,B_2) -  \E(v^*_R,B_2).  \label{dgclmc1} \end{equation} 
  
 We claim that 
 \begin{equation}\label{remains}
{\mbox{as $R$ goes to infinity, $\delta_R$ remains bounded away from
zero.}}\end{equation} 
To prove this, we assume by contradiction that \begin{equation} \label{dgclaim2} \displaystyle \lim_{R \to+ \infty} \delta_R=0.\end{equation} 
Consider $\tilde u$ to be the translation of $u$, that is~$\tilde u(x):= u(x-e)$.
Let also
\[ m (x):= \min \big\{ u(x), \tilde u(x)\big\}.\] 
We notice that in $B_{R/2}$ we have that $\tilde u (x)= u_{R,+}(x)$. 
This and~\eqref{5.31bis}
give that
\begin{equation}\label{5.31ter}
{\mbox{$m=v_R$ in $B_{R/2}$.}}\end{equation}
Also, from \eqref{dgcuur1} and~\eqref{5.31ter},
it follows that $m$ cannot be identically neither $u$ nor $\tilde u$,
and
\begin{equation} \label{dgcuur2} \begin{split} m \;<&\; u  \mbox{ in a neighborhood of } 0 \quad \mbox{and} \\
 									  m \;=&\;u  \mbox{ in a neighborhood of } e .\end{split}\end{equation}
Let $z$ be a competitor for $m$ in the ball $B_2$, that agrees with $m$ outside $B_2$. We take a
cut-off function $\psi\in C_0^{\infty}(\Rn)$ such that $\psi = 1 $ in $B_{R/4}$, $\psi =0$ in $B_{R/2}^{\C}$. Let \[z_R(x):= \psi(x) z(x) + \big(1- \psi(x)\big) v_R(x).\]  Then we have that $z_R=z $ on $B_{R/4}$ and 
\begin{equation}\label{STAR}
{\mbox{$z_R=v_R$ on $B_{R/2}^{\C}$. }}\end{equation}
In addition, by~\eqref{5.31ter}, we have that~$z=m=v_R$
in~$B_{R/2}\setminus B_2$. So, it follows that
	\[ z_R(x)=\psi(x)v_R(x)+(1-\psi(x))v_R(x)= v_R(x)  =z(x)\quad \mbox{on} \quad B_{R/2}\setminus B_2 .\]
This and~\eqref{STAR} imply that
$z_R=v_R$ on $B_2^{\C}$. 

We summarize
in the next lines these useful identities (see also Figure \ref{fign:Enest}). 
	\begin{align*}
		&\text{in }\; B_2  & & u_{R,+}= \tilde u, \quad m=v_R, \quad z=z_R\\
		&\text{in }\; B_{R/2}\setminus B_2  & & u_{R,+}= \tilde u, \quad  v^*_R=v_R=m=z=z_R \\
		&\text{in }\; B_R \setminus B_{R/2}  & & v^*_R=v_R=z_R,\quad  m=z\\	
		&\text{in }\; B_R^{\C}   & &  u_{R,+}=u=v_R=v^*_R=z_R, \quad m=z.	
	\end{align*}		
\begin{center}
  \begin{figure}[htpb]
	\hspace{0.65cm}
	\begin{minipage}[b]{1.15\linewidth}
	\centering
	\includegraphics[width=1.15\textwidth]{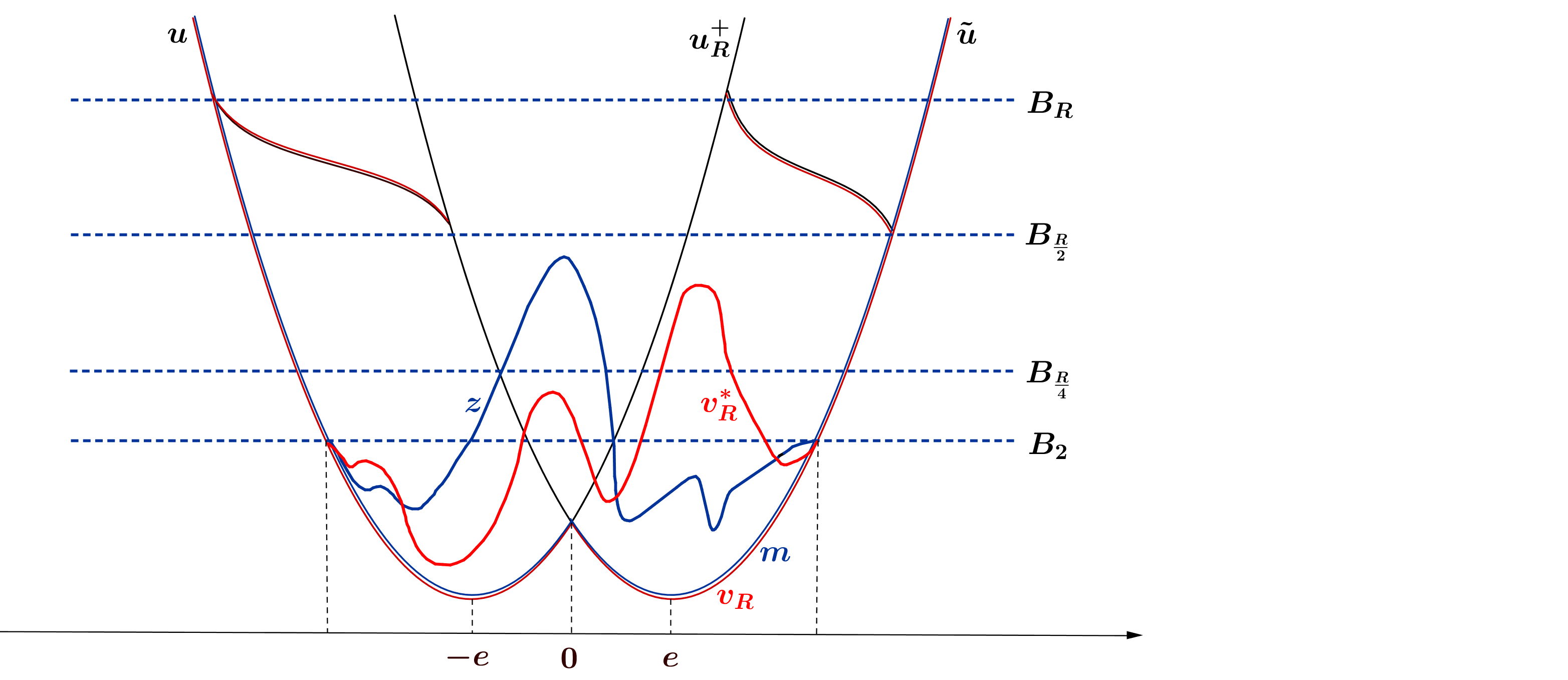}
	\caption{Energy estimates}  
	\label{fign:Enest}
	\end{minipage}
	\end{figure} 
	\end{center}										
We compute now
		\[ \begin{split} \E(m&\;,B_2) -\E(z,B_2) \\
		= &\; \E(m,B_2)-\E(v_R,B_2)+ \E(v_R,B_2)-\E(z_R,B_2) +\E(z_R,B_2)-\E(z,B_2) .\end{split}\]  \\
By the definition of $\delta_R$ in~\eqref{dgclmc1}, we have that
	\begin{equation}\label{dgeqc}\begin{split} \E(m&\;,B_2) -\E(z,B_2) \\							
				= &\;	\E(m,B_2)-\E(v_R,B_2)+\delta_R + \E(v^*_R,B_2) -\E(z_R,B_2)+\E(z_R,B_2)-\E(z,B_2).
				\end{split} \end{equation}							
Using the formula for the kinetic energy given in \eqref{dgkr} together with \eqref{dguab} we have that 
\[\begin{split} \E(m&\; ,B_2)-\E(v_R,B_2) \\
= &\; \frac{1}{2} m(B_2,B_2)+ m(B_2,B_2^{\C}) + \int_{B_2} W\big(m(x)\big)\, dx \\
&\;- \frac{1}{2} v_R(B_2,B_2) -v_R(B_2,B_2^{\C}) - \int_{B_2}W\big(v_R(x)\big) \, dx.\end{split} \]	
Since $m=v_R$ on  $B_{R/2}$ (recall~\eqref{5.31ter}), we obtain 
\[\begin{split} \E(m&\; ,B_2)-\E(v_R,B_2) \\	=&\; \int_{B_2} \, dx \int_{B_{R/2}^{\C}} \, dy \frac{ |m(x)-m(y)|^2-|m(x)-v_R(y)|^2}{|x-y|^{n+2s}} . 
				\end{split} \]	
Notice now that~$m$ and $v_R$ are bounded on $\Rn$ (since so is
$u$).
Also, if~$x\in B_2$ and~$y\in B_{R/2}^{\C}$ we have that~$|x-y|\ge|y|-|x|\ge |y|/2$ if~$R$ is large.
Accordingly,
	\begin{equation}\label{uno} \E(m ,B_2)-\E(v_R,B_2) \leq  C \int_{B_2} \, dx \int_{B_{R/2}^{\C}}\frac{1}{|y|^{n+2s}} \, dy  \le
	CR^{-2s}, 
				\end{equation}
up to renaming constants.
Similarly, 
$z_R=z$ on $B_{R/2}$ and we have the same bound  
\begin{equation}\label{due}
\E(z_R,B_2)-\E(z,B_2) \leq  C R^{-2s}.  \end{equation}
Furthermore, since $v_R^*$ is a minimizer for $\E(\cdot, B_2)$ and $v_R^*=z_R$ outside of $B_2$, we have that
 \[ \E(v^*_R,B_2) -\E(z_R,B_2)\leq 0.\]   				  			
Using this, \eqref{uno} and~\eqref{due}        
in \eqref{dgeqc}, it follows that
 	\[ \E(m,B_2)-\E(z,B_2) \leq CR^{-2s} +\delta_R.\] 
Therefore, by sending~$R\to+\infty$ and using again~\eqref{dgclaim2}, we
obtain that
 	\begin{equation}\label{MKI} \E(m,B_2) \leq \E(z,B_2) .\end{equation}
We recall that $z$ can be any competitor for $m$, that coincides with $m$ outside of $B_2$. Hence, formula~\eqref{MKI} means that $m$ is a minimizer for $\E(\cdot, B_2)$.
On the other hand,
$u$ is a minimizer of the energy in any ball.
Then, both $u$ and $m$ satisfy the same equation in $B_2$. Moreover, they coincide in a neighborhood of $e$, as stated in the second line of \eqref{dgcuur2}. By the Strong Maximum Principle, they have to coincide on $B_2$, but this contradicts the
first statement of \eqref{dgcuur2}. The proof of~\eqref{remains} is thus complete.
\bigskip

Now, since $v_R^*=v_R$ on $B_2^{\C}$, from definition \eqref{dgclmc1} we
have that 
 	\[ \delta_R  = \E(v_R,B_R)-\E(v_R^*,B_R) .\]
Also, $\E (v_R^*,B_R)\geq \E(u,B_R)$, 
thanks to the minimizing property of~$u$. 
Using these pieces of information and inequality \eqref{dge12}, it follows that 
  	\[ 		\delta_R  \leq  \E(u_{R,+},B_R) - \E(u,B_R) .\]
Now, by  sending $R \to +\infty$ and using~\eqref{remains}, we have that
	\[  \lim_{R \to +\infty} \E(u_{R,+},B_R) - \E(u,B_R) > 0 ,\] which contradicts \eqref{dge11}. This implies that indeed $u$ is monotone, and this concludes the proof of this Step.

\bigskip  					  						  						  					
 \noindent \textbf{Step 4.} \textbf{Conclusions}\\ 
         In Step 3, we have proved that $u$ is monotone, in any given direction~$e$.
Then, Step 1 gives the desired result.
This concludes the proof of Theorem \ref{dgdim2}.
 \end{proof}

We remark that the exponent two in the energy estimate \eqref{DG01} is related to the expansions of order two and not to the dimension of the space. Indeed, the energy estimates hold for any $n$. However, the two power in the estimate \eqref{DG01} allows us to prove the fractional
version of De Giorgi conjecture only in dimension two. In other words, the proof of Theorem \ref{dgdim2} is not applicable for $n> 2$. One can verify this by checking the limit in \eqref{dge11}
\bgs{ \lim_{R \to+ \infty} \Big(\E(u_{R,+},B_R) -\E(u,B_R) \Big)=0,}
which was necessary for the Proof of Theorem \ref{dgdim2} in the case $n=2$.
 We know from Theorem \ref{acenest1} that
	\[ \lim_{R \to+ \infty} \frac{C}{R^n} \E(u,B_R) =0.\]  Confronting this result with inequality \eqref{bla13}
	    		\bgs{ \E(u_{R,+},B_R) -\E(u,B_R) \leq \frac{C}{R^2} \E(u,B_R),}
 we see that we need to have $n=2$ in order for the the limit in \eqref{dge11} to be zero.  
\bigskip

Of course, the one-dimensional symmetry property in
Theorem \ref{dgdim2} is inherited by the spatial homogeneity
of the equation, which is translation and rotation invariant.
In the case, for instance, in which the potential also depends 
on the space variable, the level sets of the (minimal) solutions
may curve, in order to adapt themselves to the spatial inhomogeneity.

Nevertheless, in the case of periodic dependence, it
is possible to construct minimal solutions whose
level sets are possibly not planar, but still remain at a bounded distance
from any fixed hyperplane. As a typical result in this direction,
we recall the following one (we refer to~\cite{MATTEO-ENRICO}
for further details on the argument):
\begin{center}
  \begin{figure}[htpb]
	\hspace{0.65cm}
	\begin{minipage}[b]{0.90\linewidth}
	\centering
	\includegraphics[width=0.90\textwidth]{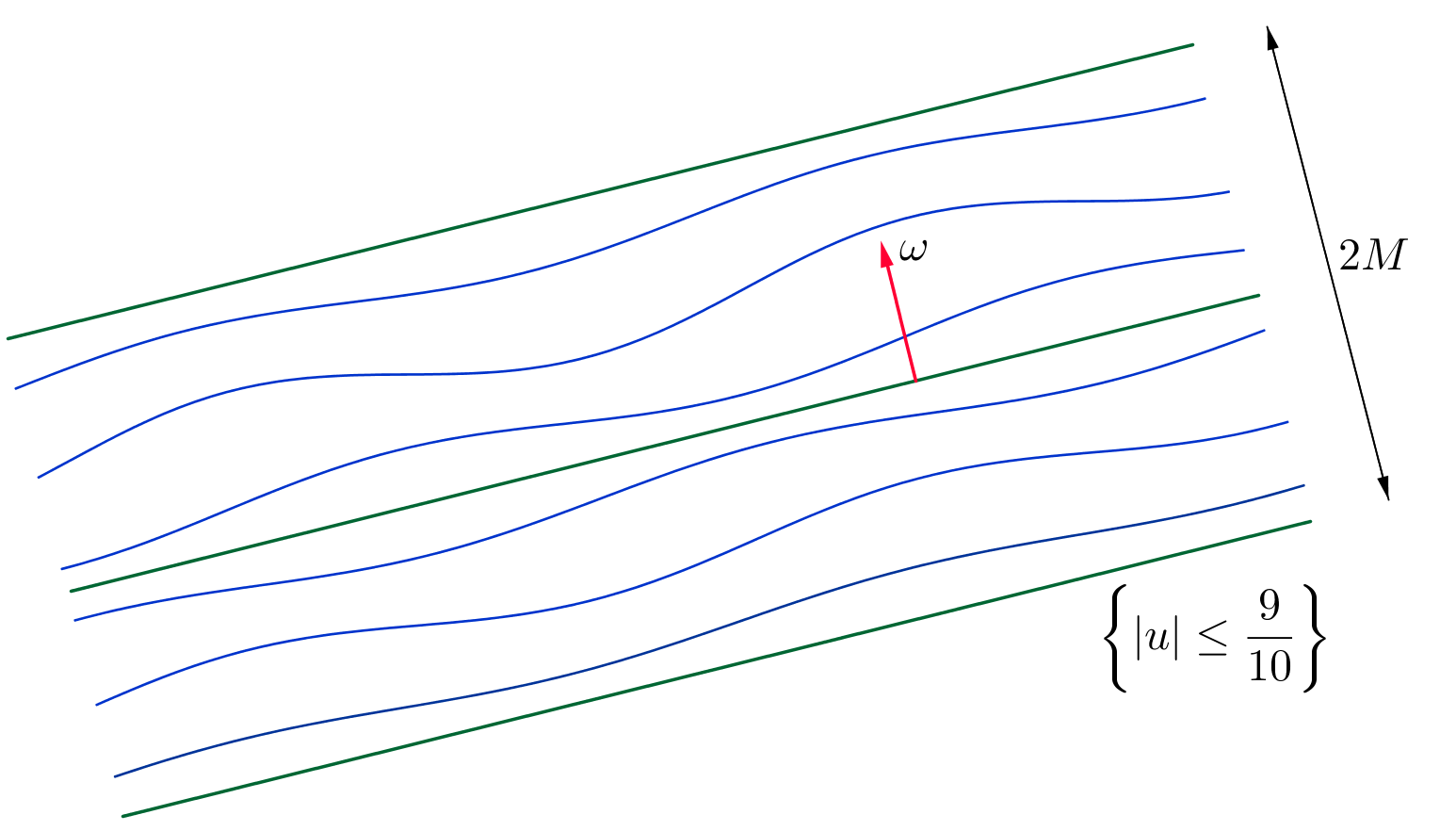}
	\caption{Minimal solutions in periodic medium}  
	\label{fign:minsperm}
	\end{minipage}
	\end{figure} 
	\end{center}
\begin{thm}
Let~$Q_+>Q_->0$ and~$Q:\R^n\to [Q_-,Q_+]$. Suppose that~$Q(x+k)=Q(x)$
for any~$k\in \Z ^n$. Let us consider, in any ball~$B_R$, the energy
defined by
$$ \E(u, B_R) =  \mathcal{K}_R(u) + \frac14 \,\int_{B_R} Q(x)\,(1-u^2)^2 dx,$$
where the kinetic energy ${\mathcal{K}}_R(u)$ is defined as in~\eqref{dgkenac}.

Then, there exists a constant $M> 0$, 
such that, given any~$\omega \in \partial B_1$,
there exists a minimal solution~$u_\omega$ of 
$$ (-\Delta)^s u_\omega (x)= Q(x)\,(u_\omega(x)-u^3_\omega(x))
\qquad{\mbox{ for any }}x\in\R^n$$
for which the level sets $\{|u_\omega | \le\frac{9}{10}\}$ are contained in the 
strip~$\{x\in\R^n {\mbox{ s.t. }} |\omega\cdot x|\le M\}$.

Moreover, if $\omega$ is rotationally dependent,
i.e. if there exists~$k_o\in\Z^n$ such that~$\omega\cdot k_o=0$,
then~$u_\omega$ is periodic with respect to~$\omega$, i.e.
$$ u_\omega(x)=u_\omega(y) {\mbox{ for any $x$, $y\in\R^n$ such that 
$x-y=k$ and $\omega\cdot k=0$.}} $$
\end{thm}
 
\chapter{Nonlocal minimal surfaces}\label{nlms}

In this chapter, we introduce nonlocal minimal surfaces and focus on two 
main results, a Bernstein type result in any dimension and the 
non-existence
of nontrivial $s$-minimal cones in dimension $2$. 
Moreover, some boundary properties will be discussed at the end 
of this chapter. For a preliminary introduction to some properties of
the nonlocal minimal surfaces, see \cite{MILAN}.

  Let $\Omega \subset \Rn$ be an open bounded domain, and $E \subset \Rn$ be a measurable set, fixed outside $\Omega$. We will consider for $s\in (0,1/2)$ minimizers of the $H^{s}$ norm 
	\begin{equation*}
		\begin{split}
		||\chi_E||^2_{H^{s}} =& \int_{\Rn} \int_{\Rn} \frac{|\chi_E(x)-\chi_E(y)|^2}{|x-y|^{n+2s}} \, dx \, dy\\
							=& 2 \int_{\Rn} \int_{\Rn} \frac{\chi_E(x)\chi_{E^{\C}}(y)}{|x-y|^{n+2s}} \, dx \, dy.
		\end{split}
	\end{equation*}
Notice that only the interactions between $E$ and $E^{\C}$ contribute to the norm. 

In order to define the fractional perimeter of $E$ in $\Omega$, we need to clarify the contribution of $\Omega$ to the $H^{s}$ norm here introduced. Namely, as $E$ is fixed outside $\Omega$, we aim at minimizing the ``$\Omega$-contribution'' to the norm among all measurable sets that ``vary'' inside $\Omega$. We consider thus interactions between $E\cap \Omega$ and $E^{\C}$ and between $E\setminus \Omega$ and $\Omega \setminus E$, 
neglecting the data that is fixed outside~$\Omega$
and that does not contribute to the minimization of the norm (see Figure \ref{fign:FracPer}).
\begin{center}
 \begin{figure}[htpb]
	\hspace{0.65cm}
	\begin{minipage}[b]{0.95\linewidth}
	\centering
	\includegraphics[width=0.95\textwidth]{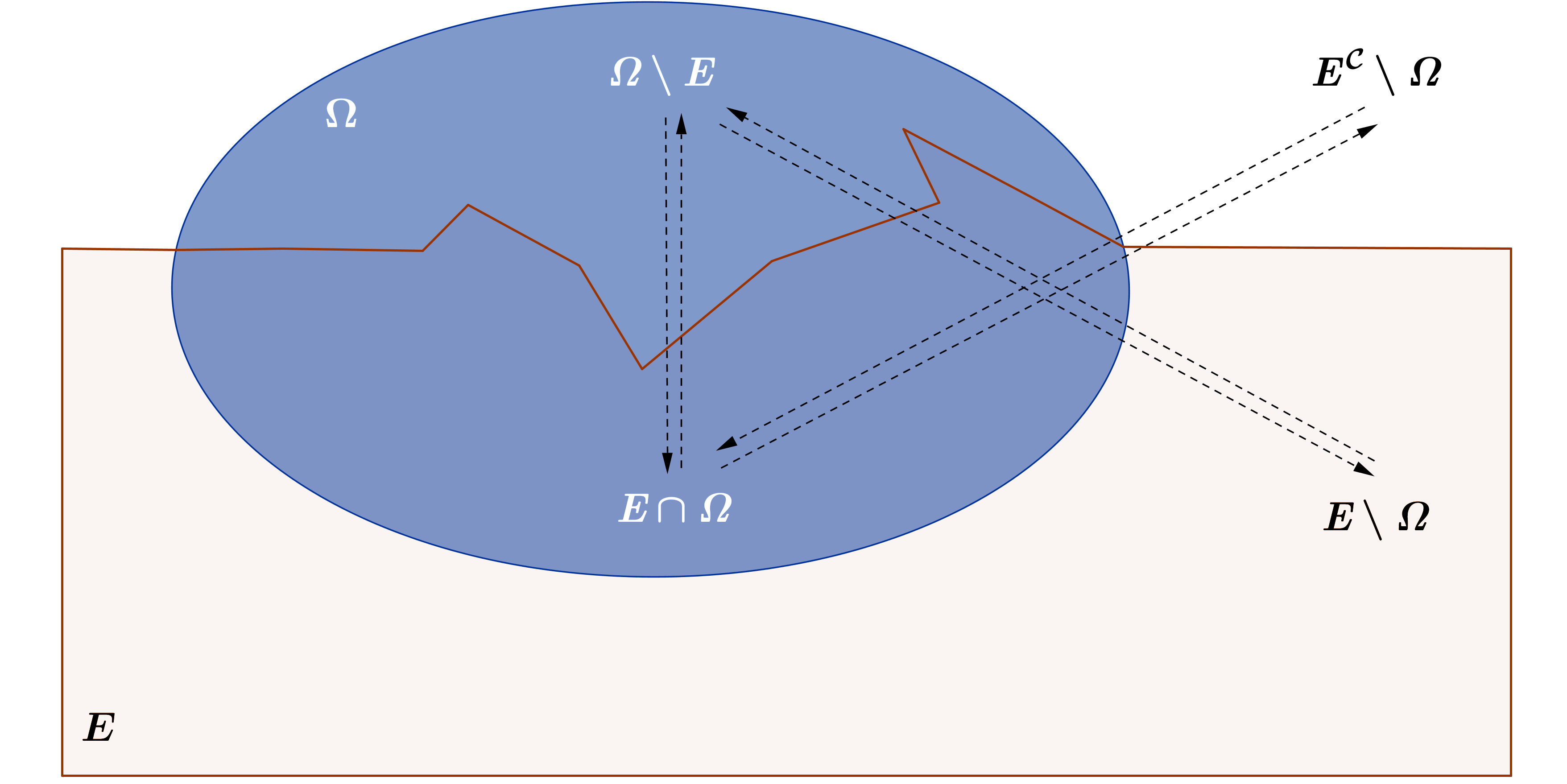}
	\caption{Fractional Perimeter}  
	\label{fign:FracPer}
	\end{minipage}
	\end{figure} 
	\end{center}
We define the interaction $I(A,B)$ of two disjoint subsets of $\Rn$ as
	\begin{equation}\label{nmsi1}
		\begin{split}
		I(A,B):=&\int_A \int_B \frac{dx\, dy}{|x-y|^{n+2s}}\\
			=& \int_{\Rn} \int_{\Rn} \frac{ \chi_A(x) \chi_B (x) }{|x-y|^{n+2s}}\, dx \, dy.
		\end{split}
	\end{equation}
Then (see~\cite{CRS10}),
one defines the nonlocal $s$-perimeter functional of $E$ in $\Omega$ as	
\begin{equation} \label{nmspf1} \text{Per}_s(E,\Omega) := I(E\cap \Omega,E^{\C})  + I(E\setminus \Omega,\Omega \setminus E). \end{equation}
Equivalently, one may write
	\[ \text{Per}_s(E,\Omega) = I( E\cap \Omega,\Omega \setminus E) + I(E\cap \Omega , E^{\C} \setminus \Omega)+ I(E\setminus \Omega,\Omega \setminus E). \]

\begin{defn}
Let~$\Omega$ be an open domain of~$\R^n$.
A measurable set $E\subset \Rn$ is $s$-minimal in~$\Omega$
if $\text{Per}_s(E,\Omega)$ is finite and if,
for any measurable set~$F$ such that $E\setminus\Omega =F\setminus \Omega$, we have that
	\[ \text{Per}_s(E,\Omega) \leq \text{Per}_s(F,\Omega).\]
A measurable set is $s$-minimal in $\Rn$ if it is $s$-minimal in any ball $B_r$, where $r > 0$.
\end{defn}

When~$ s\to \frac{1}{2}$,
the fractional perimeter~$\text{Per}_s$ 
approaches the classical perimeter, see \cite{BBM}. See also \cite{DAVILA} for the precise limit in the class
of functions with bounded variations, \cite{CV11,CV13} for a geometric
approach towards regularity and \cite{PONCE, ADM} for an approach based on $\Gamma$-convergence. See also \cite{vS-W-2004} for a different proof and Theorem 2.22 in \cite{LUCA} and the references therein for related discussions.
%as proved in \cite{CV11}.
A simple, formal statement, up to renormalizing constants, is the following:
\begin{thm}\label{TP12}
Let $R>0$ and $E$ be a set with finite perimeter in $B_R$. Then 
	\[\lim_{s \to \frac{1}{2} }\Big(\frac{1}{2} -s\Big) \text{Per}_s(E,B_r)=\text{Per }(E,B_r) \]
for almost any $r\in (0,R)$.
\end{thm}

The behavior of~$\text{Per}_s$ as~$s\to0$ is slightly more involved.
In principle, the limit as~$s\to0$ of~$\text{Per}_s$
is, at least locally, related to the Lebesgue measure
(see e.g.~\cite{VGMAZYA}). Nevertheless, the situation
is complicated by the terms coming from infinity,
which, as~$s\to0$, become of greater and greater importance.
More precisely, it is proved in~\cite{DFPV13-PER} that,
if~$\text{Per}_{s_o}(E,\Omega)$ is finite for some~$s_o\in(0,1/2)$, and the
limit
\begin{equation}\label{E:LIM:LE}
\beta_E := \lim_{s\to0} 2s
\,\int_{E\setminus B_1}\frac{dy}{|y|^{n+2s}}\end{equation}
exists, then
\begin{equation}\label{C:I:PA}
\lim_{s \to 0} 2s\, \text{Per}_{s}(E,\Omega)=
\left(|\partial B_1|-\beta_E\right)\,|E\cap\Omega| + \beta_E \,|\Omega\setminus E|.\end{equation}
We remark that, using polar coordinates,
$$ 0\le \beta_E\le 
\lim_{s\to0} 2s
\,\int_{\R^n\setminus B_1}\frac{dy}{|y|^{n+2s}}
= 
\lim_{s\to0} 2s\,|\partial B_1|
\int_{1}^{+\infty} \rho^{-1-2s}\,d\rho =|\partial B_1|,$$
therefore~$\beta_E\in [0,|\partial B_1|]$ plays the role
of a convex interpolation parameter in
the right hand-side of~\eqref{C:I:PA} (up to normalization constants).

In this sense, formula~\eqref{C:I:PA} may be interpreted by saying
that, as~$s\to0$, the $s$-perimeter concentrates itself
on two terms that are ``localized'' in the domain~$\Omega$,
namely~$|E\cap\Omega|$ and~$|\Omega\setminus E|$. Nevertheless,
the proportion in which these two terms count is given by
a ``strongly nonlocal'' interpolation parameter, namely
the quantity~$\beta_E$ in~\eqref{E:LIM:LE}
which ``keeps track''
of the behavior of~$E$ at infinity.

As a matter of fact, to see how~$\beta_E$ is influenced
by the behavior of~$E$ at infinity, one can compute~$\beta_E$
for the particular case of a cone.
For instance, if~$\Sigma\subseteq \partial B_1$,
with~$\frac{|\Sigma|}{|\partial B_1|}=: b\in[0,1]$,
and $E$ is the cone over~$\Sigma$ (that is~$E:= \{ tp,\;\, p\in\Sigma,\;\,
t\ge0\}$), we have that
$$ \beta_E = \lim_{s\to0} 2s\,|\Sigma|\,\int_1^{+\infty}\rho^{-1-2s}\,d\rho
=|\Sigma|=b\,|\partial B_1|,$$
that is~$\beta_E$ gives in this case exactly the opening of the cone.

We also remark that, in general, the limit
in~\eqref{E:LIM:LE} may not exist, even for smooth sets:
indeed, it is possible that the set~$E$ ``oscillates'' wildly at infinity,
say from one cone to another one, leading to the non-existence of the limit in~\eqref{E:LIM:LE}.

Moreover, we point out that
the existence of the limit
in~\eqref{E:LIM:LE}
is equivalent to the existence of the limit in~\eqref{C:I:PA},
except in the very special case~$
|E\cap\Omega| = |\Omega\setminus E|$, in which the limit in~\eqref{C:I:PA}
always exists. That is, the following alternative holds true:
\begin{itemize}
\item if $|E\cap\Omega| \ne |\Omega\setminus E|$, then the
limit in~\eqref{E:LIM:LE} exists if and only if
the limit in~\eqref{C:I:PA} exists,
\item if $|E\cap\Omega| = |\Omega\setminus E|$, then the
limit in~\eqref{C:I:PA} always exists (even when the one
in~\eqref{E:LIM:LE} does not exist), and
$$ \lim_{s\to0}2s \text{Per}_{s}(E,\Omega)=
|\partial B_1|\,|E\cap\Omega| =|\partial B_1|\,|\Omega\setminus E|.$$
\end{itemize}

The boundaries of $s$-minimal sets are referred to as \emph{nonlocal minimal surfaces}.
 
In \cite{CRS10} it is proved that $s$-minimizers satisfy a suitable integral equation (see in particular Theorem 5.1
in \cite{CRS10}), that is the Euler-Lagrange equation corresponding to the $s$-perimeter functional $\text{Per}_s$.
If~$E$ is $s$-minimal in $\Omega$ and~$\partial E$ is smooth enough,
this Euler-Lagrange equation can be written as
	\begin{equation} \label{ELsmin} 
\int_{\Rn} \frac{\chi_E(x_0+y) -\chi_{\Rn\setminus E} (x_0+y) }{|y|^{n+2s}}\, dy  =0,\end{equation}
for any $x_0\in \Omega \cap \partial E$.
 
Therefore, in analogy with the case of
the classical minimal surfaces, which have zero mean curvature,
one defines the \emph{nonlocal mean curvature} of~$E$ at~$x_0\in\partial E$ as
	\eqlab{ \label{nmc} H_E^s(x_0) := \int_{\Rn} \frac{
\chi_E(y) -\chi_{E^{\C}}(y) }{|y-x_0|^{n+2s}}\, dy.}
In this way, equation~\eqref{ELsmin} can be written as~$H_E^s=0$ along~$\partial E$.

It is also suggestive to think that the function~$\tilde\chi_E:=
\chi_{E}-\chi_{E^{\C}}$ averages out to zero at the points on~$\partial E$,
if~$\partial E$ is smooth enough, since at these points
the local contribution of~$E$ compensates the one of~$E^{\C}$.
Using this notation, one may take the liberty of writing
\begin{eqnarray*}
H_E^s(x_0)&=&\frac12
\int_{\Rn} \frac{\tilde\chi_{E}(x_0+y)+\tilde\chi_{E}(x_0-y)}{|y|^{n+2s}}\, dy\\
&=&
\frac12
\int_{\Rn} \frac{\tilde\chi_{E}(x_0+y)+\tilde\chi_{E}(x_0-y)-2\tilde\chi_E(x_0)}{
|y|^{n+2s}}\, dy
\\&=& \frac{-(-\Delta)^s \tilde\chi_{E}(x_0)}{C(n,s)},\end{eqnarray*}
using the notation of~\eqref{frlap2def}.
Using this suggestive representation, the
Euler-Lagrange equation in~\eqref{ELsmin} becomes
$$ {\mbox{$(-\Delta)^s \tilde\chi_{E}=0$
along~$\partial E$.}}$$
We refer to \cite{AV14} for further details on this argument.

It is also worth recalling that
the nonlocal perimeter functionals find applications
in motions of fronts by nonlocal mean curvature (see
e.g.~\cite{CAFFA-SOUG, IMBERT, CHAMBOLLE}),
problems in which aggregating and disaggregating terms
compete towards an equilibrium (see e.g.~\cite{I5}
and~\cite{I4}) and nonlocal free boundary problems
(see e.g.~\cite{CAFFA-SAVIN-VALDINOCI} and~\cite{DIPIERRO-SAVIN-VALDINOCI}).
See also~\cite{VGMAZYA}
and~\cite{VISENTIN}
for results related to this type of problems.
\bigskip

In the classical case of the local perimeter
functional, it is known that minimal surfaces are smooth in dimension $n\leq 7$. Moreover, if $n\geq 8$  minimal surfaces are smooth except on a small singular set of Hausdorff dimension $n-8$.
Furthermore, minimal surfaces that are graphs are called minimal graphs, and they reduce to hyperplanes
if~$n\le8$ (this is called the Bernstein property,
which was also discussed at the beginning of the Chapter \ref{S:NP}).
If $n\geq9$, there exist global minimal graphs that are not affine
(see e.g.~\cite{GIUSTI}).

Differently from the classical case,
the regularity theory for $s$-minimizers is still quite open.
We present here some of the partial results obtained in this direction:

\begin{thm}\label{THM 5.8} In the plane, $s$-minimal sets are smooth. More precisely:\\ 
a) If $E$ is an $s$-minimal set in $\Omega \subset \R^2$, then $\partial E \cap \Omega$ is a $C^{\infty}$-curve.\\
b) Let $E$ be $s$-minimal in $\Omega\subset \Rn$ and let $\Sigma_E \subset \partial E\cap\Omega$ denote its
singular set. Then $\mathcal{H}^d (\Sigma_E)=0$ for any $d>n-3$.
\end{thm}
See \cite{SV13} for the proof of this results
(as a matter of fact, in~\cite{SV13} only $C^{1,\alpha}$ regularity is proved,
but then \cite{BFV12} proved that $s$-minimal sets with~$C^{1,\alpha}$-boundary are automatically~$C^\infty$).
Further regularity results of the $s$-minimal surfaces can be found in 
\cite{CV13}. There, a regularity theory
when $s$ is near $\displaystyle \frac{1}{2}$ is stated, as we see in the following Theorem:

\begin{thm}\label{THM 5.9} There exists $\epsilon_0 \in \Big(0,\displaystyle \frac{1}{2}\Big)$ such that if $s \geq \displaystyle\frac{1}{2} -\epsilon_0$,  then\\
a) if $n\leq 7$, any $s$-minimal set is of class $C^{\infty}$, \\
b) if $n=8$ any $s$-minimal surface is of class $C^\infty$ except, at most, at countably many isolated points, \\
c) any $s$-minimal surface is of class $C^\infty$ outside a closed set $\Sigma$ of Hausdorff dimension $n-8$.
\end{thm}

\section{Graphs and $s$-minimal surfaces}

We will focus the upcoming material on two interesting results related to graphs: a Bernstein type result, namely the property that an $s$-minimal graph  in $\R^{n+1}$ is flat (if no singular cones exist in dimension $n$); we will then prove that an $s$-minimal surface whose prescribed data is a subgraph, is itself a subgraph.

The first result is the following theorem:
\begin{thm}\label{figv}
Let $E= \{ (x,t) \in \Rn \times \R \text{ s.t. } t<u(x)\}$ be an $s$-minimal graph, and assume there are no singular cones in dimension $n$ (that is, if $\mathcal{K} \subset \Rn$ is an $s$-minimal cone, then $\mathcal{K}$ is a half-space). Then $u$ is an affine function (thus $E$ is a half-space).
\end{thm}

To be able to prove Theorem \ref{figv}, we recall some useful auxiliary results. 
In the following lemma we state a dimensional reduction result (see Theorem 10.1 in \cite{CRS10}).
\begin{lemma} \label{dimrid}
Let $E=F\times \R$. Then if $E$ is $s$-minimal if and only if $F$ is $s$-minimal.
\end{lemma}
We define then the blow-up and blow-down of the set $E$ are, respectively
	\[E_0:=\lim_{r\to 0}E_r \quad \mbox{and} \quad E_\infty: =\lim_{r\to +\infty} E_r, \quad \mbox{where} \quad E_r=\frac{E}{r}.\]
A first property of the blow-up of $E$ is the following (see Lemma 3.1 in \cite{FV13}).
\begin{lemma}
If $E_\infty$ is affine, then so is $E$.
\label{halfspace}
\end{lemma}

We recall also a regularity result for the $s$-minimal surfaces (see \cite{FV13} and \cite{BFV12} for details and proof).
\begin{lemma}\label{reg}  Let~$E$ be $s$-minimal. Then:\\
a) If $E$ is Lipschitz, then $E$ is~$C^{1,\alpha}$.\\
b) If $E$ is~$C^{1,\alpha}$, then $E$ is~$C^{\infty}$.
\end{lemma}

We give here a sketch of the proof of Theorem \ref{figv} (see \cite{FV13} for all the details).
\begin{proof}[Sketch of the proof of Theorem \ref{figv}]
If $E\subset \R^{n+1}$ is an $s$-minimal graph, then the blow-down $E_\infty$ is an $s$-minimal cone (see Theorem 9.2 in \cite{CRS10} for the proof of this statement).  By applying the dimensional reduction argument in Lemma \ref{dimrid} we obtain an $s$-minimal cone in dimension $n$. According to the assumption  that no singular $s$-minimal cones exist in dimension $n$, it follows that necessarily $E_\infty$ can be singular only at the origin.

We consider a bump function $w_0 \in C^{\infty}( \mathbb{R}, [0,1])$ such that 
	\begin{equation*}
		\begin{split}
		&w_0(t)=0 \text{ in } \bigg(-\infty, \frac{1}{4}\bigg)\cup \bigg(\frac{3}{4}, +\infty\bigg) \\
		& w_0(t) =1 \text{ in } \bigg(\frac {2}{5},\frac{3}{5}\bigg)\\
		&w(t)=w_0(|t|).
		\end{split}
	\end{equation*}
The blow-down of $E$ is	
	\[E_\infty =\big\{(x',x_{n+1}) \text{ s.t. } x_{n+1}\leq u_\infty(x')\big\} .\]
For a fixed $\sigma \in \partial B_1$, let
	\[F_t:=\big\{(x',x_{n+1}) \text{ s.t. } x_{n+1}\leq u_{\infty} \big(x'+t\theta w(x')\sigma\big) -t\big\}\]
	be a family of sets, where $t\in (0,1)$ and $\theta >0$.
Then for $\theta$ small, we have that 
\begin{equation}\label{ABC567}
{\mbox{$F_1$ is below $E_\infty$.}}\end{equation} Indeed, suppose by contradiction that this is not true. Then, there exists $\theta_k \to 0$ such that
	\begin{equation}\label{UF78988799} u_\infty\big(x'_k+\theta_kw(x'_k)\sigma\big)-1 \geq u_\infty (x'_k).\end{equation}
But $x'_k \in \text{supp}w$, which is compact, therefore $\displaystyle x'_\infty:=\lim_{k\to +\infty} x'_k$ belongs to the support of $w$, and $w(x'_\infty)$ is defined. Then, by sending $k \to +\infty$
in~\eqref{UF78988799}
we have that
	\[u_\infty(x'_\infty) -1\geq u_\infty(x'_\infty),\]
which is a contradiction.
This establishes~\eqref{ABC567}.

Now consider the smallest $t_0\in (0,1)$ for which $F_t$ is below $E_\infty$. Since $E_\infty$ is a graph, then $F_{t_0}$ touches $E_\infty$ from below in one point $X_0=(x'_0,x_{n+1}^0)$, where $x'_0 \in \text{supp} w$. 
Now, since $E_\infty$ is $s$-minimal, we have that the nonlocal mean curvature (defined in \eqref{nmc}) of the boundary is null. 
Also, since $F_{t_0}$ is a $C^2$ diffeomorphism of $E_\infty$ we have that
	\begin{equation}\label{PLO90-1}
H_{F_{t_0}}^s(p) \simeq \theta t_0,\end{equation}
and there is a region where $E_\infty$ and $F_{t_0}$ are well separated by $t_0$, thus
	\[\big|\big(E_\infty \setminus F_{t_0}\big) \cap \big(B_3
\setminus B_2\big)\big| \geq ct_0,\]
for some~$c>0$.
Therefore, we see that
	\[H_{F_{t_0}}^s(p) =H_{F_{t_0}}^s(p) -H_E^s(p) \geq c t_0.\]
This and~\eqref{PLO90-1} give that~$\theta t_0 \ge c t_0$, for some~$c>0$
(up to renaming it). 
If~$\theta$ is small enough, this implies that $t_0=0$.

In particular,
we have proved that there exists $\theta >0$ small enough such that,
for any $t\in (0,1)$ and any~$\sigma\in\partial B_1$, we have that 
	\[  u_\infty\big(x'+t \theta w(x')\sigma\big)-t \leq u_\infty (x').\]
This implies that
	\[\frac{u_\infty\big(x'+t \theta w(x')\sigma\big)-u_\infty (x')}{t\theta}\leq \frac{1}{\theta},\]
hence, letting $t\to 0$, we have that 
	\[ \nabla u_\infty (x')w(x')\sigma \leq \frac{1}{\theta}, \text { for any } x\in \Rn\setminus \{0\}, \text{ and } \sigma \in B_1.\]
We recall now
that $w=1$ in $B_{3/5}\setminus B_{2/5}$ and $\sigma$ is arbitrary in $\partial B_1$.
Hence, it follows that
	\[|\nabla u_\infty(x)| \leq \frac{1}{\theta}, \text{ for any } x \in B_{3/5}\setminus B_{2/5}.\]
Therefore $u_\infty$ is globally Lipschitz. By the regularity statement in 
Lemma \ref{reg}, we have that~$u_\infty$ is $C^\infty$.
This says that~$u$ is
smooth also at the origin, hence (being a cone) it follows that~$E_\infty$ is necessarily
a half-space. Then by Lemma \ref{halfspace}, we conclude that E is a half-space as well.
\end{proof}

We introduce in the following theorem another interesting property related to $s$-minimal surfaces, in the case in which the fixed given data outside a domain is a subgraph. In that case, the $s$-minimal surface itself is a subgraph. Indeed:

\begin{thm}\label{thmgraph}
Let $\Omega_0$ be an open and bounded subset of $\R^{n-1}$ with boundary of class $C^{1,1}$ and let $\Omega:=\Omega_0 \times \R$. Let $E$ be an $s$-minimal set in $\Omega$. Assume that
		\eqlab{\label{thgr2}E\setminus \Omega =\{x_n <u(x'), \, x'\in \R^{n-1}\setminus \Omega_0\}}
		for some continuous function $u\colon \R^{n-1} \to \R$. Then
		\[ E\cap \Omega =\{x_n<v(x'), \, x' \in \Omega_0\}\]
		for some function $v\colon \R^{n-1}\to \R$.  
\end{thm}

The reader can see \cite{graph}, where this theorem and the related results are proved; here, we only state the preliminary results needed for our purposes
and focus on the proof of Theorem \ref{thmgraph}. The proof relies on a
sliding method, more precisely, we take a translation of $E$ in the $n^{\mbox{th}}$ direction, and move it until it touches $E$ from above. If the set $E\cap \Omega$ is a subgraph, then, up to a set of measure $0$, the contact between the translated $E$ and $E$, will be $E$ itself.

 However, since we have no information on the regularity of the minimal surface, we need at first to ``regularize'' the set by introducing the notions of supconvolution and subconvolution. With the aid of a useful result related to the sub/supconvolution of an $s$-minimal surface, we proceed then with the proof of the Theorem \ref{thmgraph}.
  
 The supconvolution of a set $E\subseteq \Rn$ is given by 
 	\[ E_{\delta}^{\sharp} := \bigcup_{x\in E} \overline {B_{\delta} (x)}.\]
 	\begin{center}
 \begin{figure}[htpb]
	\hspace{0.65cm}
	\begin{minipage}[b]{0.75\linewidth}
	\centering
	\includegraphics[width=0.75\textwidth]{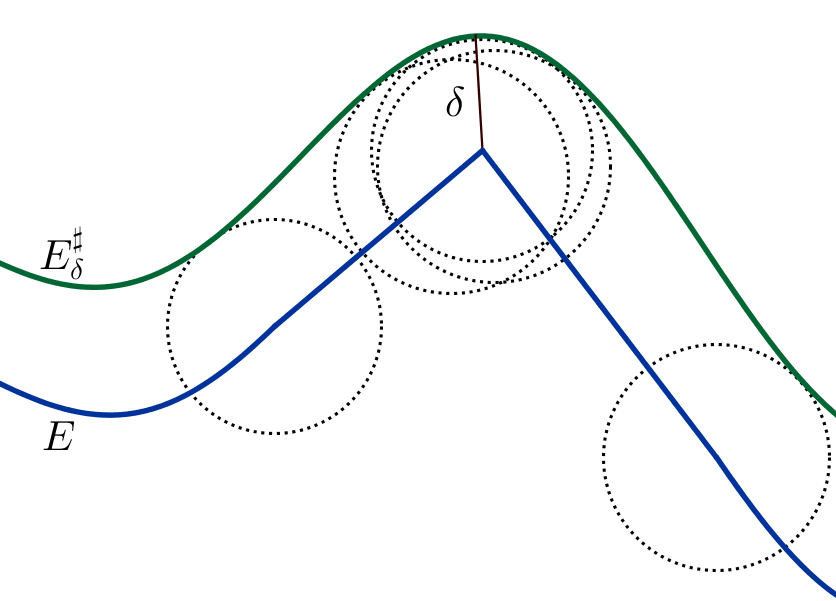}
	\caption{The supconvolution of a set}  
	\label{fign:sgrcor1}
	\end{minipage}
	\end{figure} 
	\end{center}
 	In an equivalent way, the supconvolution can be written as \[ E_{\delta}^{\sharp} =  \bigcup_{ { v\in \Rn}\atop {|v|\leq \delta}}   (E+v)  .\] Indeed, we consider $\delta>0$ and an arbitrary $x\in E$. Let $y\in \overline{ B_\delta (x)}$ and we define $v:=y-x$. Then
 		\[ |v|\leq|y-x|\leq \delta\quad \mbox{and} \quad y=x+v \in E+v. \]
 Therefore $\overline{B_\delta(x)} \subseteq E+v$ for $|v|\leq \delta$. In order to prove the inclusion in the opposite direction, one notices that taking $y\in E+v$ with $|v|\leq \delta$ and defining $x:= y-v$, it follows that
 \[ |x-y|=|v|\leq \delta.\] Moreover, $x\in (E+v)-v=E$ and the inclusion $E+v \in \overline{B_\delta(x)}$ is proved.
  	
 	      On the other hand, the subconvolution is defined as
 		\[ E_{\delta}^{\flat} := \Rn \setminus \left( (\Rn \setminus E)_{\delta}^{\sharp}\right).\] Now, the idea is that the supconvolution of~$E$
is a regularized version of~$E$ whose nonlocal minimal curvature
is smaller than the one of~$E$, i.e.:
\begin{equation}\label{oforUI89}
\int_{\R^n} \frac{ \chi_{\R^n\setminus E_{\delta}^{\sharp}}(y)
-\chi_{ E_{\delta}^{\sharp} }(y)
}{|x-y|^{n+2s}} \,dy \le
\int_{\R^n} \frac{ \chi_{\R^n\setminus E}(y)
-\chi_{ E }(y)
}{|\tilde x-y|^{n+2s}} \,dy
\le0,
\end{equation}
for any~$x\in\partial E_{\delta}^{\sharp}$, where~$\tilde x:= x-v\in
\partial E$ for some~$v\in\R^n$ with~$|v|=\delta$. Then, by construction, the set~$E+v$ lies in~$E_{\delta}^{\sharp}$,
and this implies~\eqref{oforUI89}.

Similarly, one has that the opposite inequality holds
for the subconvolution of~$E$, namely
\begin{equation}\label{oforUI89-bis}
\int_{\R^n} \frac{ \chi_{\R^n\setminus E_{\delta}^{\flat}}(y)
-\chi_{ E_{\delta}^{\flat} }(y)
}{|x-y|^{n+2s}} \,dy 
\ge0,
\end{equation}
By~\eqref{oforUI89} and~\eqref{oforUI89-bis}, we obtain:
 		 \begin{prop}\label{posubsup}
 		 Let $E$ be an $s$-minimal set in $\Omega$. Let $p\in \partial E_{\delta}^{\sharp}$ and assume that $\overline {B_\delta(p)}\subseteq\Omega$. Assume also that $E_{\delta}^{\sharp}$ is touched from above by a translation of $E_{\delta}^{\flat}$, namely there exists $\omega \in \Rn$ such that 
 		 \[ E_{\delta}^{\sharp} \subseteq E_{\delta}^{\flat} +\omega\] 
 		 and 
 		 \[ p\in (\partial E_{\delta}^{\sharp} )\cap (\partial E_{\delta}^{\flat}+\omega).\] Then \[E_{\delta}^{\sharp} = E_{\delta}^{\flat }+\omega. \]
 		 \end{prop}

\begin{proof}[Proof of Theorem \ref{thmgraph}]
One first remark is that
the $s$-minimal set does not have spikes which go to infinity: more precisely,
one shows that
	\eqlab{\label{thgr1} \Omega_0\times (-\infty, -M)\subseteq E\cap \Omega \subseteq \Omega_0 \times (-\infty,M) }
for some $M\geq 0$.
The proof of \eqref{thgr1} can be performed by sliding horizontally a large ball,
see~\cite{graph} for details.

After proving~\eqref{thgr1}, one can deal with the core of
the proof of Theorem \ref{thmgraph}.
The idea is to
slide $E$ from above until it touches itself and analyze what happens at the contact points.
For simplicity, we will assume here that the function~$u$
is uniformly continuous (if~$u$ is only continuous,
the proof needs to be slightly modified since the subconvolution
and supconvolution that we will perform may create new touching
points at infinity).

At this purpose, we consider $E_t=E+t e_n$ for $t\geq 0$. Notice that, by \eqref{thgr1}, if $t \geq 2M$, then $E\subseteq E_t$. Let then $ t$ be the smallest for which the inclusion $E\subseteq E_t$ holds.
We claim that  $t=0$.
		 If this happens, one may consider 
		 \[ v=\inf\{\tau \text{ s.t. }(x,\tau)\in E^{\C}\}\]
and, up to sets of measure $0$, $E\cap \Omega_0$ is the subgraph of $v$.
 
 The proof is by contradiction, so let us assume  that $ t>0$. According to \eqref{thgr2}, the set $E\setminus \Omega$ is a subgraph, hence the contact points between $\partial E$ and $\partial E_t$ must lie in $\overline \Omega_0 \times \R$. Namely, only two 
possibilities may occur: the contact point is
interior (it belongs to  $\Omega_0 \times\R)$, or it is at the boundary (on $\partial \Omega_0 \times\R$).
So, calling $p$ the contact point, one may have\footnote{As a matter
of fact, the number of contact
points may be higher than one, and even infinitely many
contact points may arise. So, to be rigorous,
one should distinguish the case in which all
the contact points are interior and the case in which
at least one contact point lies on the boundary.

Moreover, since the surface may have vertical portions
along the boundary of the domain, one needs to carefully define
the notion of contact points (roughly speaking, one needs to take
a definition for which the vertical portions
which do not prevent the sliding are not in the contact set).

Finally, in case the contact points are all interior, it is also
useful to perform the sliding method in a slighltly reduced domain,
in order to avoid that the supconvolution method produces
new contact points at the boundary (which may arise from
vertical portions of the surfaces).

Since we do not aim to give a complete proof of Theorem~\ref{thmgraph} here,
but just to give the main ideas and underline the additional
difficulty, we refer to~\cite{graph} for the full details of these arguments.}
that
	\eqlab{\label{caseone}  {\mbox{either }} p\in   \Omega_0 \times \R\quad \mbox{or}}
	\eqlab{\label{casetwo}  p\in \partial \Omega_0 \times \R. \quad \mbox{ }}

We deal with the first case in~\eqref{caseone} (an example of this behavior is depicted in Figure \ref{fign:thmgrph1}).
We consider $E_{\delta}^{\sharp} $ and $E_{\delta}^{\flat}$ to be the supconvolution, respectively the subconvolution of $E$.  We then slide the subconvolution until it touches the supconvolution. More precisely, let $\tau>0$ and we take a translation of the subconvolution, $E_{\delta}^{\flat}+ \tau e_n$. For $\tau$ large, we have that $E_{\delta}^{\sharp}\subseteq E_{\delta}^{\flat}+ \tau e_n$ and we consider $\tau_\delta$ to be the smallest for which such inclusion holds. We have (since $t$ is positive by assumption) that
	\[ \tau_\delta\geq \frac{t}{2}>0.\]
	Moreover, for $\delta$ small, the sets $\partial E_\delta^{\sharp} $ and $\partial (E_{\delta}^{\flat}+ \tau_\delta e_n)$ have a contact point which, according to \eqref{caseone}, lies in $\Omega_0\times \R$. Let $p_{\delta}$ be such a point, so we may write
 \[  p_\delta \in   (\partial E_\delta^{\sharp} ) \cap \partial (E_{\delta}^{\flat}+ \tau_\delta e_n) \quad \mbox{and } \quad   p_\delta \in \Omega_0\times \R.\]
 Then, for $\delta$ small (notice that $\overline{B_{\delta}(p)}\subseteq \Omega$), Proposition \ref{posubsup} yields that 
 \[E_\delta^\sharp =E_\delta^\flat +\tau_\delta e_n.\]
 Considering $\delta$ arbitrarily small, one obtains that
 \[E=E+\tau_0 e_n,\quad \mbox{with} \quad \tau_0>0.\] 
 But $E$ is a subgraph outside of $\Omega$, and this provides a contradiction. Hence, the claim that $t=0$ is proved.
\begin{center}
 \begin{figure}[htpb]
	\hspace{0.65cm}
	\begin{minipage}[b]{0.95\linewidth}
	\centering
	\includegraphics[width=0.95\textwidth]{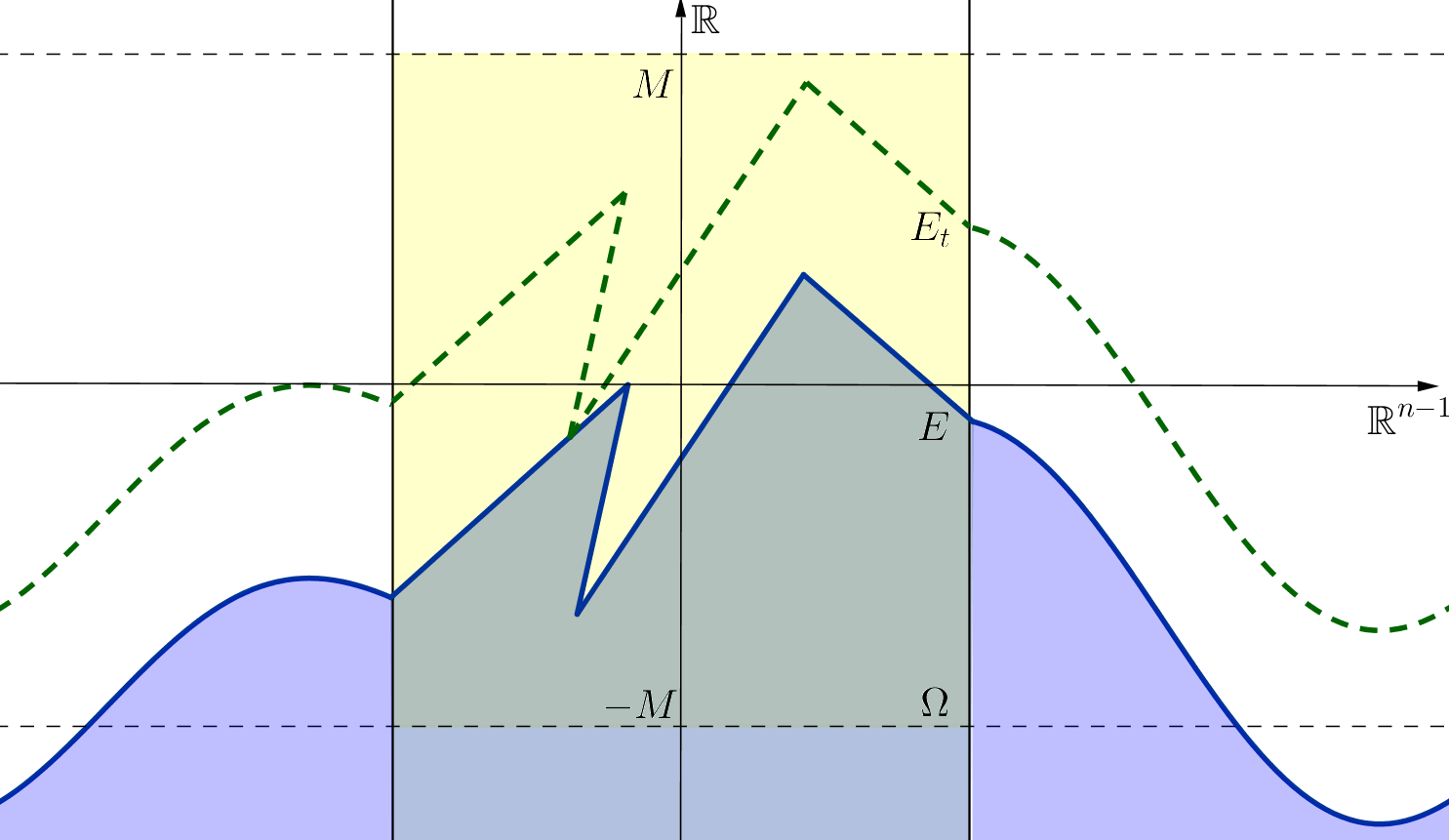}
	\caption{Sliding $E$ until it touches itself at an interior point}  
	\label{fign:thmgrph1}
	\end{minipage}
	\end{figure} 
	\end{center}
Let us see that we also obtain a contradiction when supposing that $t>0$ and that the second case \eqref{casetwo} holds. Let
 	\[ p=(p',p_n) \quad \mbox{and } \quad p \in  (\partial E)\cap (\partial E_t) . \]
 	Now, if one takes sequences $a_k\in \partial E$ and $b_k\in \partial E_t$, both that tend to $p$ as $k $ goes to infinity, since $E\setminus \Omega $ is a subgraph and $t>0$, necessarily $a_k, b_k$ belong to $\Omega$. Hence
 	 		\eqlab{ \label{pinclos} p\in \overline { (\partial E) \cap \Omega}  \cap \overline {(\partial E_t)\cap \Omega}.} 
 	Thanks to Definition 2.3 in \cite{CRS10}, one obtains that $E$ is a variational subsolution in a neighborhood of $p$. In other words, if $A\subseteq E\cap\Omega$ and $p\in \overline A$, then 
 			\[ 0\geq \mbox{Per}_s( E,\Omega)-\mbox{Per}_s(E\setminus A, \Omega) = I(A,E^{\C}) -I(A, E\setminus A)\] 
 			(we recall the definition of $I$ in \eqref{nmsi1} and of the fractional perimeter $Per_s$ in \eqref{nmspf1}).
  According to Theorem 5.1 in \cite{CRS10}, this implies in a viscosity sense (i.e. if $E$ is touched at $p$ from outside by a ball), that
 			\eqlab{\label{curvgz} \int_{\Rn}\frac{\chi_E(y)-\chi_{\Rn \setminus E}(y)}{|p-y|^{n+2s}}\, dy\geq 0.} 
In order to obtain an estimate on the fractional mean curvature in the strong sense, we consider the translation of the point $p$ as follows: 
\[p_t=p-t e_n=(p',p_n-t)=(p',p_{n,t}).\] Since $t>0$, one may have that either $p_n\neq u(p')$, or $p_{n,t}\neq u(p')$. 
 			
 			These two possibilities can be dealt with in a similar way, so we just continue with the proof in the case $p_n\neq u(p')$ (as is also exemplified in Figure \ref{fign:thmgrph2}).  Taking $r>0$ small, the set $B_r(p)\setminus \Omega$ is contained entirely in $E$ or in its complement. Moreover, one has from \cite{regint} that $\partial E\cap B_r(p) $ is a $C^{1,\frac{1}2+s}$-graph in the direction of the normal to $\Omega$ at $p$. That is: in Figure \ref{fign:thmgrph2}
the set~$E$ is $C^{1,\frac12+s}$, hence in the vicinity of~$p=(p',p_n)$, it appears to be sufficiently smooth.

So, let $\nu(p)=(\nu'(p),\nu_n(p))$ be the normal in the interior direction, then up to a rotation and since $\Omega$ is a cylinder (hence $\nu_n(p)=0$), we can write $\nu(p)=e_1$. Therefore, there exists a function $\Psi$ of class $ C^{1,\frac{1}2+s}$ such that $p_1=\Psi(p_2,\dots,p_n)$ and, in the vicinity of $p$, we can write $\partial E$ as the graph $G=\{x_1=\Psi(x_2,\dots,x_n)\}$. 
 			 
 			 Given \eqref{pinclos}, we deduce that there exists a sequence $p_k\in G$ such that $p_k\in \Omega$ and $p_k \to p$ as $k\to \infty$. From this it follows that there exists a sequence of points $p_k\to p$ such that 
 			 \eqlab {\label{pkE1} \partial E \mbox{ in the vicinity of }p_k \mbox{ is a graph of class }C^{1,\frac{1}2+s}} 
 			 and 
 			  	 \eqlab{\label{pkE2} \int_{\Rn}\frac{\chi_{E}(y)-\chi_{E^{\C}}(y)}{|p_k-y|^{n+2s}}\, dy =0.}
 			  	  \begin{center}
 \begin{figure}[htpb]
	\hspace{0.65cm}
	\begin{minipage}[b]{0.95\linewidth}
	\centering
	\includegraphics[width=0.95\textwidth]{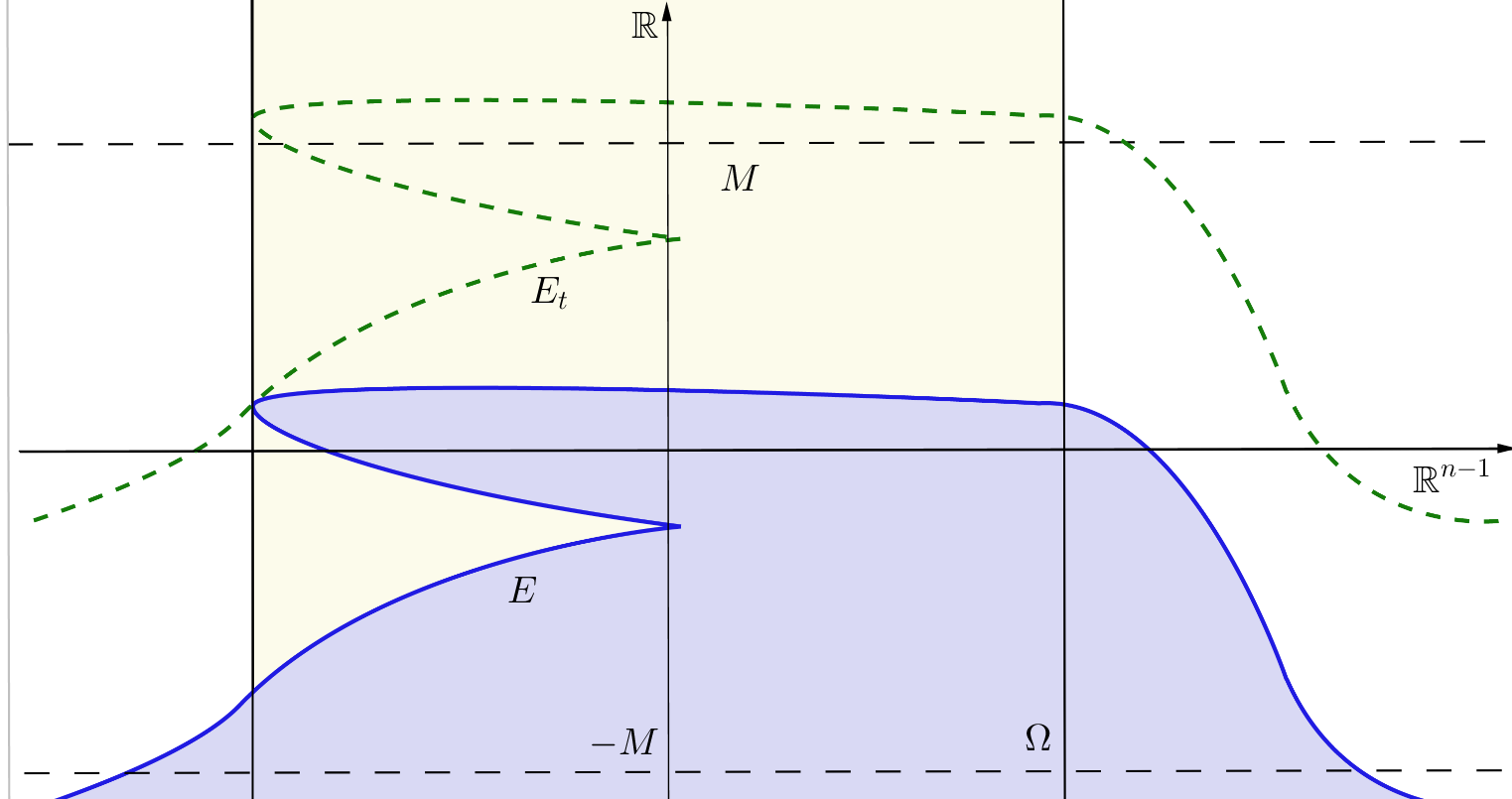}
	\caption{Sliding $E$ until it touches itself at a boundary point}  
	\label{fign:thmgrph2}
	\end{minipage}
	\end{figure} 
	\end{center}

From \eqref{pkE1} and~\eqref{pkE2}, and using 
a pointwise version of the Euler-Lagrange equation (see~\cite{graph}
for details), we have that
 			 \bgs{ \int_{\Rn}\frac{\chi_{E}(y)-\chi_{E^{\C}}(y)}{|p-y|^{n+2s}}\, dy =0.}	
 		Now, $E\subset E_t$ for $t$ strictly positive, hence
 			\eqlab{\label{curvgz11} \int_{\Rn}\frac{\chi_{E_t}(y)-\chi_{E_t^{\C}}(y)}{|p-y|^{n+2s}}\, dy >0.}	 	
 		Moreover, we have that the set $\partial E_t \cap B_{\frac{r}4}(p)$ must remain on one side of the graph $G$, namely one could have that
 			\bgs{ & E_t\cap	B_{\frac{r}4}(p) \subseteq \{x_1\leq \Psi(x_2,\dots,x_n)\} \mbox{ or }\\
 				& E_t\cap	B_{\frac{r}4}(p) \supseteq \{x_1\geq \Psi(x_2,\dots,x_n)\} .}
 	Given again \eqref{pinclos}, 	we deduce that there exists a sequence $\tilde p_k\in \partial E_t\cap \Omega$ such that $\tilde p_k \to p$ as $k\to \infty$ and $ \partial E_t\cap \Omega $ in the vicinity of $\tilde p_k$ is touched by a surface lying in $E_t$, of class $ C^{1,\frac{1}2+s} $. 
 	Then 	\bgs{ \int_{\Rn}\frac{\chi_{E_t}(y)-\chi_{E_t^{\C}}(y)}{|\tilde p_k-y|^{n+2s}}\, dy \leq 0.} 
Hence, making use of
a pointwise version of the Euler-Lagrange equation (see~\cite{graph}
for details), we obtain that
 	\bgs{ \int_{\Rn}\frac{\chi_{E_t}(y)-\chi_{E_t^{\C}}(y)}{| p-y|^{n+2s}}\, dy \leq 0.} But this is a contradiction with \eqref{curvgz11}, and this concludes the proof of Theorem~\ref{thmgraph}.
 	\end{proof}

On the one hand, one may think that Theorem \ref{thmgraph} has to be well-expected. On the other hand, it is far from being obvious, not only because the proof is not trivial, but also because the statement itself almost risks to be false, especially at the boundary.
Indeed we will see in Theorem \ref{STI-DSV-2}
that the graph property is close to fail at the boundary of the domain, where the $s$-minimal surfaces may present vertical tangencies and stickiness phenomena (see Figure \ref{eST2}).

\section{Non-existence of singular cones in dimension $2$}

We now prove the non-existence of singular
$s$-minimal cones in dimension $2$, as stated in the next result
(from this, the more general statement in Theorem~\ref{THM 5.8}
follows after a blow-up procedure):

\begin{thm}\label{THIS}
If $E$ is an $s$-minimal cone in $\mathbb{R}^2$, then $E$ is a half-plane.
\label{noconestwo}
\end{thm}

We remark that, as a combination of
Theorems \ref{figv} and~\ref{THIS},
we obtain the following result of Bernstein type:

\begin{cor}
Let $E= \{ (x,t) \in \Rn \times \R \text{ s.t. } t<u(x)\}$ be an $s$-minimal graph, and assume that $n \in \{1, 2\}$. Then $u$ is an affine function.
\end{cor}

\begin{center}
\begin{figure}[htb]
	\hspace{0.6cm}
	\begin{minipage}[b]{0.70\linewidth}
	\centering
	\includegraphics[width=0.70\textwidth]{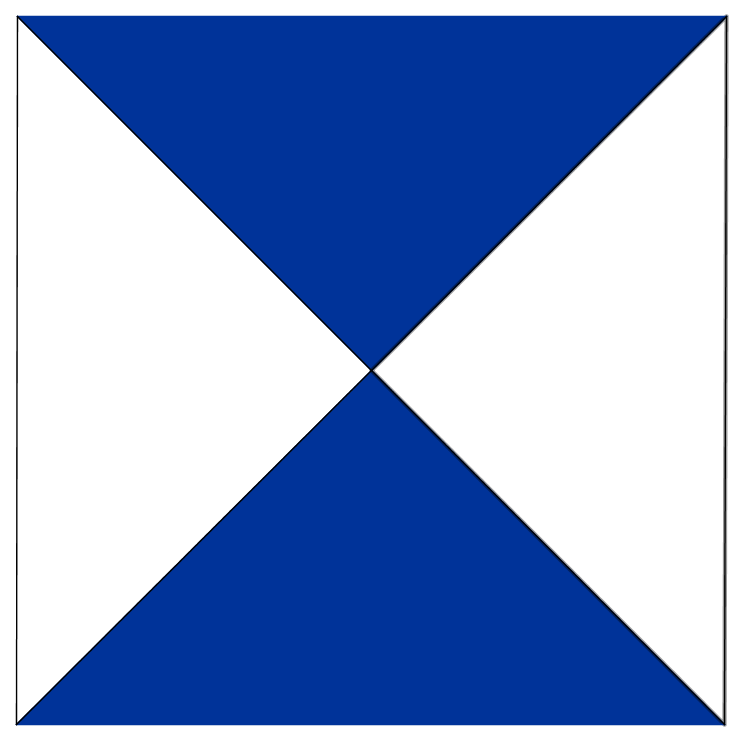}
	\caption{The cone $\mathcal{K}$}   
	\label{fign:Cone}
	\end{minipage}
	\end{figure}
	\end{center}
Let us first consider a simple example, given by the cone in the plane
	\[\mathcal{K} := \Big\{  (x,y)\in \R^2 \mbox{ s.t. } y^2>x^2\Big\},\]
see Figure \ref{fign:Cone}. 

\begin{prop}\label{THAT}
The cone $\mathcal{K}$ depicted in 
Figure \ref{fign:Cone} is not $s$-minimal in $\R^2$.
\end{prop} 

Notice that, by symmetry, one can prove that $\mathcal{K}$ satisfies \eqref{ELsmin}
(possibly in the viscosity sense). On the other hand, Proposition~\ref{THAT} gives that~${\mathcal{K}}$
is not $s$-minimal. This, in particular, provides an example of a set that
satisfies the Euler-Lagrange equation 
in~\eqref{ELsmin}, but is not~$s$-minimal (i.e.,
the Euler-Lagrange equation 
in~\eqref{ELsmin}
is implied by, but not necessarily equivalent to, the $s$-minimality property). 

\begin{proof}[Proof of Proposition~\ref{THAT}]
The proof of the non-minimality of $\mathcal{K}$ is due to an original idea by Luis Caffarelli.

Suppose by contradiction that the cone $\mathcal{K}$ is minimal in $\R^2$. 
We add to $\mathcal K$ a small square adjacent to the origin
(see Figure \ref{fign:Cone1}), and call $\mathcal{K}'$ the set obtained. Then $\mathcal{K}$ and $\mathcal{K}'$ have the same $s$-perimeter. This is due to the interactions considered in the $s$-perimeter functional and the unboundedness of the regions. We remark  that in Figure \ref{fign:Cone1} we represent bounded regions, of course, sets $A, B, C, D, A' ,B' ,C'$ and $D'$ are actually unbounded.
\begin{center}
\begin{figure}[htb]
	\hspace{0.5cm}
	\begin{minipage}[b]{0.95\linewidth}
	\centering
	\includegraphics[width=0.95\textwidth]{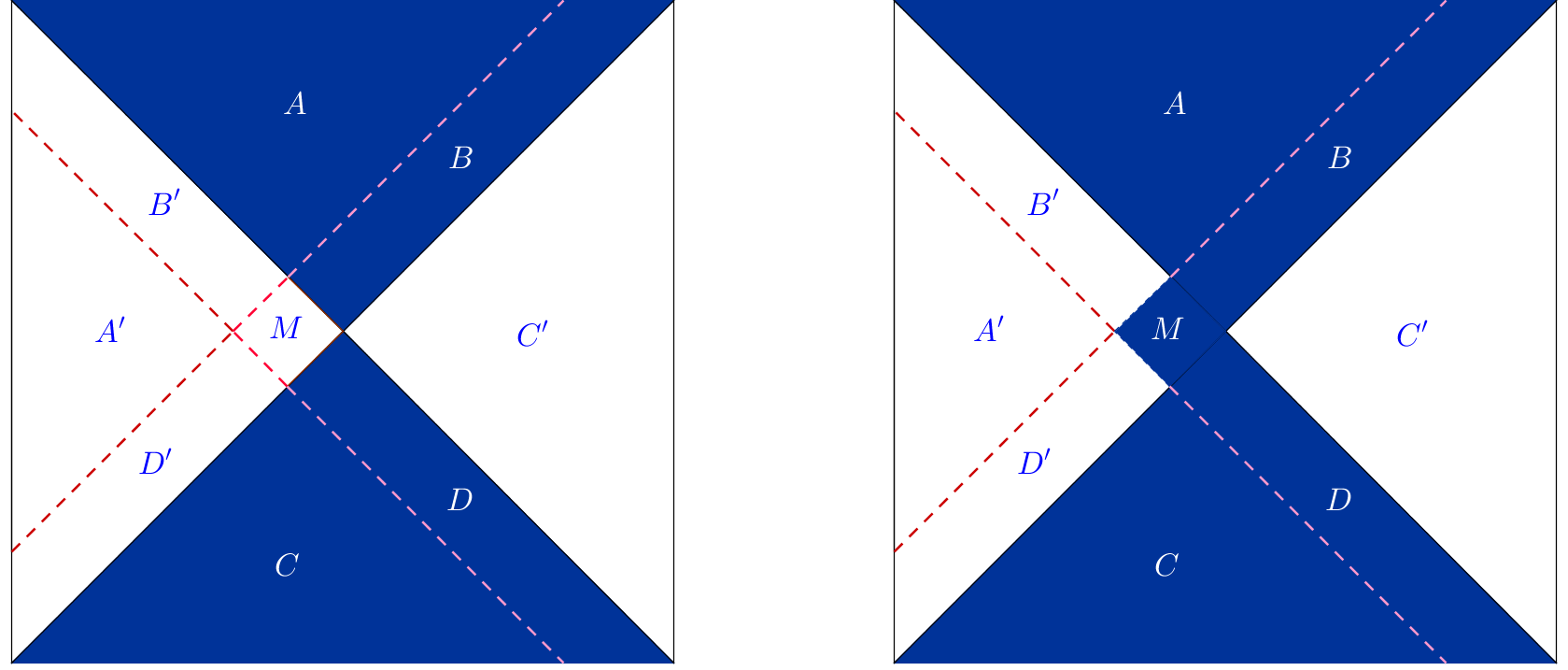}
	\caption{Interaction of $M$ with $A, B, C, D, A', B', C', D'$}   
	\label{fign:Cone1}
	\end{minipage}
\end{figure}
\end{center}

Indeed, we notice that in the first image, the white square $M$ interacts with the dark regions $A, B, C, D$, while in the second the now dark square $M$ interacts with the regions $A', B', C', D'$, and all the other interactions are unmodified. Therefore, the difference between the $s$-perimeter of $\mathcal{K}$ and that of $\mathcal{K'}$ consists only of the interactions $I(A,M)+I(B,M)+I(C,M)+I(D,M)-I(A',M)- I(B',M)-I(C',M)-I(D',M)$. But $A \cup B=A' \cup B'$ and $C\cup D=C'\cup D'$ (since these sets are all unbounded), therefore the difference is null, and the $s$-perimeter of $\mathcal{K}$ is equal to that of $\mathcal{K}'$. Consequently, $\mathcal{K}'$ is also $s$-minimal, and therefore it satisfies the Euler-Lagrange equation in \eqref{ELsmin} at the origin. But this leads to a contradiction, since the the dark region now contributes more than the white one, namely
	\[ \int_{\R^2} \frac{\chi_{\mathcal{K}'}(y) - \chi_{\R^2 \setminus \mathcal{K}'}(y)} {|y|^{2+s} } \, dy>0.\]
Thus $\mathcal{K}$ cannot be $s$-minimal, and this concludes our proof.
\end{proof}
\begin{center}
\begin{figure}[htb]
	\begin{minipage}[b]{0.75\linewidth}
	\centering
	\includegraphics[width=0.75\textwidth]{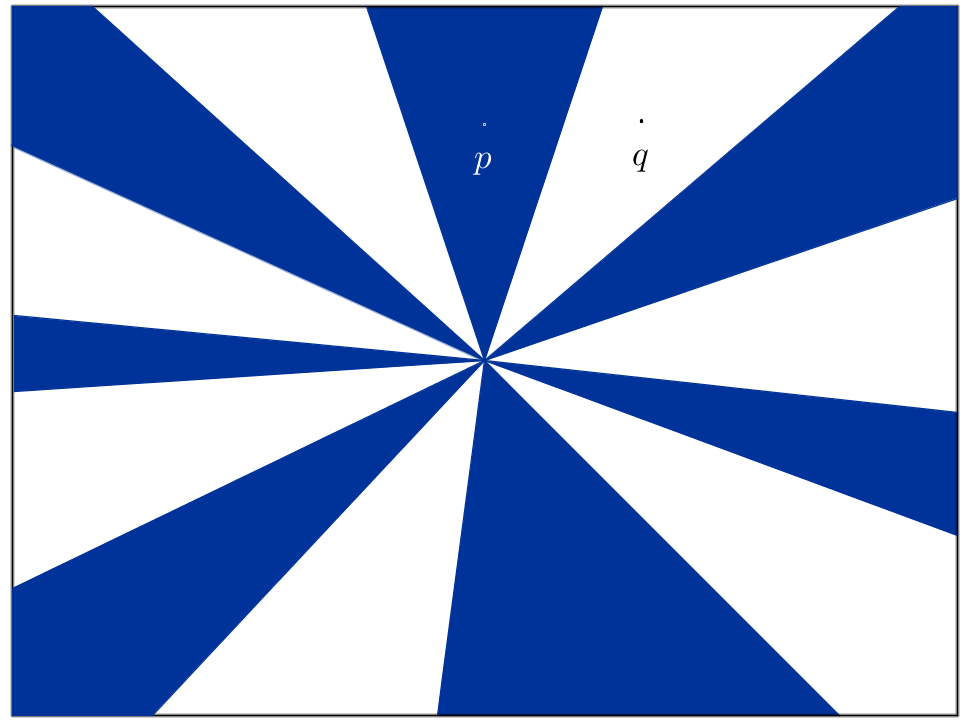}
	\caption{Cone in $\R^2$}   
	\label{fign:Cone2}
	\end{minipage}
\end{figure}
\end{center}

This geometric argument cannot be extended to a more general case (even, for instance, to a cone in $\R^2$ made of many sectors, see Figure \ref{fign:Cone2}).
As a matter of fact, the proof of Theorem~\ref{noconestwo}
will be completely different than the one of Proposition~\ref{THAT} and it will rely on
an appropriate domain perturbation argument.

The proof of Theorem~\ref{noconestwo} that we present here
is actually different than the original one in~\cite{SV13}.
Indeed,
in \cite{SV13}, the result was proved
by using the harmonic extension for the fractional Laplacian.
Here, the extension will not be used; furthermore, 
the proof follows the steps of Theorem \ref{dgdim2} and we will recall here just the main ingredients.

\begin{proof}[Proof of Theorem \ref{noconestwo}]

The idea of the proof is the following: if $E\subset \R^2$ is an $s$-minimal cone, then let $\tilde{E}$ be a perturbation of the set $E$ which coincides with a translation of $E$ in $B_{R/2}$ and with $E$ itself outside $B_R$. Then the difference between the energies of $\tilde{E}$ and $E$ tends to $0$ as $R\to +\infty$. This implies that also the energy of $E \cap \tilde{E}$ is arbitrarily close to the energy of $E$. On the other hand if $E$ is not a half-plane, the set $\tilde{E} \cap E$ can be modified locally to decrease its energy by a fixed small amount and we reach a contradiction.

The details of the proof go as follows.
Let \[ u:= \chi_E-\chi_{\R^2\setminus E}.\] {F}rom
definition \eqref{dguab} we have that
	\[ u(B_R,B_R)= 2 I(E\cap B_R,B_R\setminus E)\] and\[ u(B_R,B_R^{\C})=I(B_R\cap E,E^{\C}\setminus B_R) + I(B_R\setminus E,E\setminus B_R),\]
	thus 	
	\eqlab{ \label{peren1} \text{Per}_s (E,B_R) = \mathcal{K}_R (u),}
	where $\mathcal{K}_R (u) $ is given in \eqref{dgkenac} and $\text{Per}_s(E,B_R)$ is the $s$-perimeter functional defined in \eqref{nmspf1}. Then $E$ is $s$-minimal if $u$ is a minimizer of the energy $\K_R$ in any ball $B_R$, with $R>0$. 
	
	Now, we argue by contradiction, and suppose that $E$ is an $s$-minimal cone different from the half-space. Up to rotations, we may suppose that a sector of $E$ has an angle smaller than $\pi$ and is bisected by $e_2$. Thus there exists $M\geq 1$ and $p\in E\cap B_{M}$ on the $e_2$-axis such that $p\pm e_1 \in \R^2\setminus E$ (see
Figure \ref{fign:Cone2}). 

We take
$\varphi\in C^\infty_0(B_1)$, such that 
$\varphi (x)=1 $ in $B_{1/2}$.
For $R$ large (say $R > 8M$), we define 
\[ \Psi_{R,+}(y):=y+\varphi \Big(\frac{ y}{R}\Big)\,e_1 .\] 
We point out that,
for $R$ large, $\Psi_{R,+}$ is a diffeomorphism on~$\R^2$.

Furthermore, we define~$u_R^+(x):= u(\Psi_{R,+}^{-1}(x))$. Then
	\begin{equation*}
		\begin{aligned}
		&u_R^+(y)= u(y-e_1) &\text{ for } &p\in B_{2M} \\
	{\mbox{and }}\;\;
	& u_R^+(y)= u(y) &\text { for } &p \in \R^2 \setminus B_R.
		\end{aligned}
 \end{equation*}  
We recall the estimate obtained in \eqref{DG01}, that, combined with the minimality of $u$, gives  
 	\begin{equation*}	 {\mathcal{K}}_R( u_R^+)- {\mathcal{K}}_R( u)\leq \frac{C}{R^2} \mathcal{K}_R(u).\end{equation*}
But $u$ is a minimizer in any ball, and by the energy estimate in Theorem \ref{acenest1} we have that
   \[ \mathcal{K}_R(u_R^+)  -\mathcal{K}_R(u)  \leq CR^{-2s}.\]
This implies that
		\begin{equation} \label{urplus11} \lim_{R \to+ \infty} \mathcal{K}_R(u_R^+)  -\mathcal{K}_R(u) =0.
		\end{equation}
Let now
	\[ v_R(x):= \min \{ u(x), u_R^+(x) \} \quad \quad \text{ and } \quad \quad w_R(x):=\max \{ u(x), u_R^+(x) \}.\]
 We claim that $v_R$ is not identically $u$ nor $u_R^+$. Indeed
	\[\begin{split}  &u_R^+ (p) = u(p-e_1) = (\chi_{E} - \chi_{\R^2 \setminus E} )(p-e_1) =-1 \quad \mbox{and}  \\
				&u(p) = (\chi_{E} - \chi_{\R^2 \setminus E} )(p) = 1 .\end{split}\] 
On the other hand,
	\[\begin{split}  & u_R^+(p+e_1) = u(p)=1 \quad \mbox{and}  \\
& u(p+e_1) = (\chi_{E} - \chi_{\R^2\setminus E} )(p+e_1) = -1 .\end{split}\] 
By the continuity of $u$ and $u_R^+ $, we obtain that 
	\begin{equation}\label{neigP}
		v_R=u_R^+ < u \text{ in a neighborhood of } p \end{equation} and
		\begin{equation}\label{neigPe}  v_R =u <  u_R^+ \text{ in a neighborhood of } p+e_1.
		\end{equation}
Now, by the minimality property of~$u$,
	\[\mathcal K_R (u) \leq \mathcal K_R (v_R).\]  
Moreover (see e.g.
formula~(38) in~\cite{PSV13}), 
 	\[ \mathcal K_R (v_R)+\mathcal K_R (w_R)\le
\mathcal K_R (u) +\mathcal K_R (u_R^+).   \]
The latter two formulas give that 
	\begin{equation}
	\mathcal K_R (v_R) \leq \mathcal K_R (u_R^+).
	\label{wrur}
	\end{equation} 
We claim that 
\begin{equation}\label{I*8ui}
{\mbox{$v_R$ is not minimal for $\mathcal K_{2M}$}}\end{equation}
with respect to compact perturbations in $B_{2M}$.
Indeed, assume by contradiction that $v_R$ is minimal, then in $B_{2M}$ both $v_R$ and $u$ would satisfy the same equation. 
Recalling~\eqref{neigPe}
and applying the Strong Maximum Principle, 
it follows that $u=v_R$ in $B_{2M}$, which contradicts \eqref{neigP}.
This establishes~\eqref{I*8ui}.

Now, we consider a minimizer~$u_R^*$ of~$\mathcal K_{2M}$
among the competitors that agree with~$v_R$ outside~$B_{2M}$.
Therefore, we can define
	\[  \delta_R  : = \mathcal K_{2M} (v_R)- \K_{2M} (u_R^*). \]
In light of~\eqref{I*8ui}, we have that~$\delta_R>0$.

The reader can now compare
Step 3 in the proof of Theorem \ref{dgdim2}. There we proved that 
\begin{equation}\label{89*9*9}
{\mbox{$\delta_R$
remains bounded away from zero as~$R\to+\infty$.}}\end{equation}
Furthermore, since $u_R^*$ and $v_R$ agree outside $B_{2M}$ we obtain that 
	\[\mathcal K_R (u_R^*) +\delta_R = \mathcal K_R (v_R).\]
Using this, \eqref{wrur} and the minimality of $u$, we obtain that
	\[\delta_R  =\mathcal K_R (v_R) - \mathcal K_R (u_R^*) \leq \mathcal K_R (u_R^+) -\mathcal K_R (u).\]
Now we send $R$ to infinity, recall \eqref{urplus11} and~\eqref{89*9*9},
and
we reach a contradiction. Thus, $E$ is a half-space, and this concludes the proof of Theorem \ref{noconestwo}.
\end{proof}

As already mentioned,
the regularity theory for $s$-minimal sets is still widely open.
Little is known beyond Theorems~\ref{THM 5.8}
and~\ref{THM 5.9}, so it would be very
interesting to further investigate the regularity of~$s$-minimal surfaces
in higher dimension
and for small~$s$.\bigskip

It is also interesting to recall that if the $s$-minimal surface~$E$
is a subgraph of some function~$u:\R^{n-1}\to\R$
(at least in the vicinity of some point~$x_0=(x_0', u(x'_0))\in\partial E$)
then the Euler-Lagrange~\eqref{ELsmin} can be written
directly in terms of~$u$.
For instance (see formulas~(49) and~(50) in~\cite{BFV12}),
under appropriate smoothness assumptions on~$u$, formula~\eqref{ELsmin}
reduces to
\begin{eqnarray*}
0&=&
\int_{\Rn} \frac{\chi_{\Rn\setminus E}(x_0+y) -\chi_{E} (x_0+y) }{|y|^{n+2s}}\,dy
\\ &=&\int_{\R^{n-1}} F\left( \frac{u(x'_0+y')-u(x'_0)}{|y'|}\right)\,
\frac{\zeta(y')}{|y'|^{n-1+2s}}\,dy'+\Psi(x'_0),\end{eqnarray*}
for suitable~$F$ and~$\Psi$, and a cut-off function~$\zeta$
supported in a neighborhood of~$x_0'$.

\begin{center}
\begin{figure}[htb]
        \begin{minipage}[b]{0.95\linewidth}
        \centering
        \includegraphics[width=0.95\textwidth]{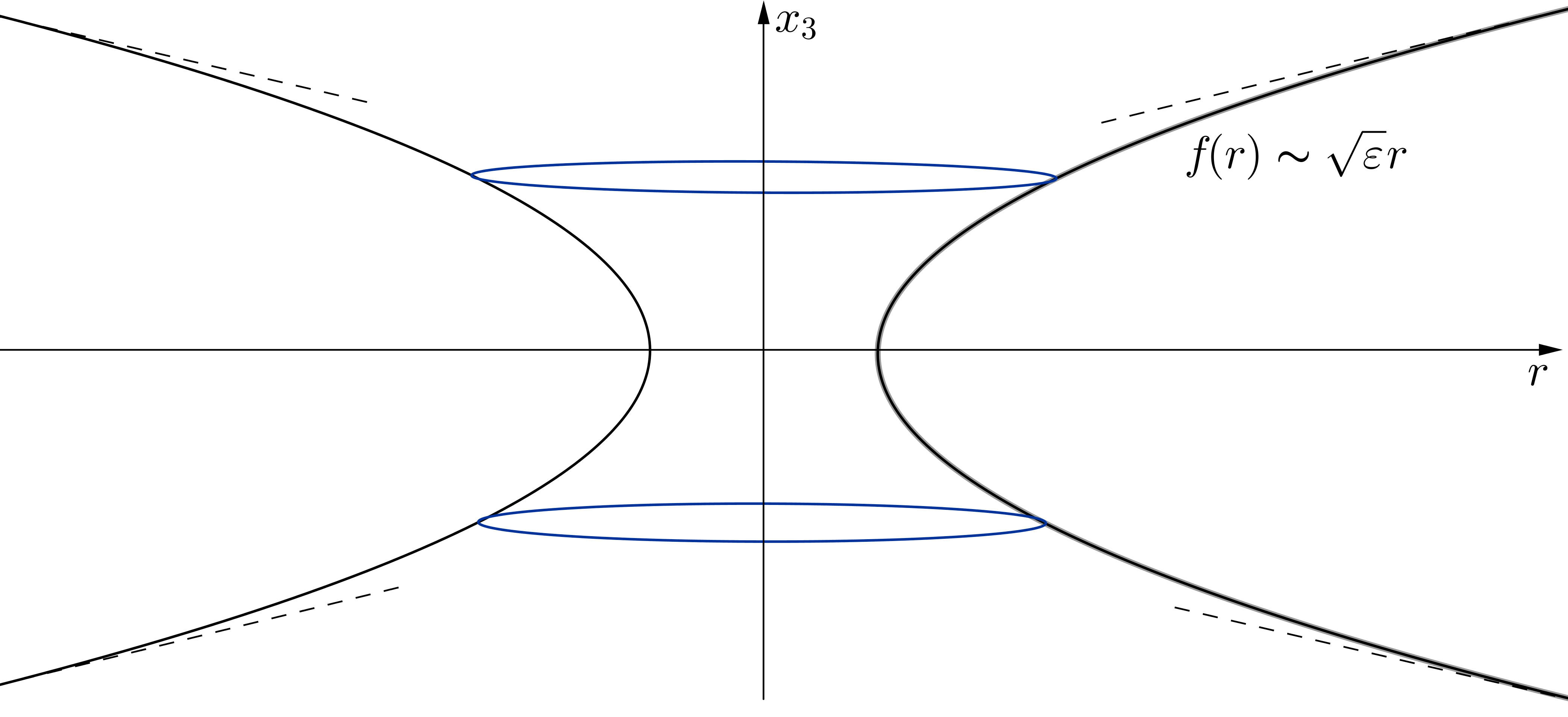}
        \caption{A nonlocal catenoid}
        \label{FRAC_CATENOID}
        \end{minipage}
\end{figure}
\end{center}

Regarding the regularity problems
of the $s$-minimal surfaces, let us mention the recent papers~\cite{Lawson}
and~\cite{Lawson2}. Among other very interesting results,
it is proved there that suitable singular cones
of symmetric type are unstable up to dimension~$6$
but become stable in dimension~$7$ for small~$s$
(these cones can be seen as the nonlocal analogue
of the Lawson cones in the classical minimal surface theory,
and the stability property is in principle weaker than minimality,
since it deals with the positivity of the second order derivative
of the functional).

This phenomenon may suggest the conjecture that the
$s$-minimal surfaces 
may develop singularities in dimension~$7$ and higher
when~$s$ is sufficiently small.
\bigskip

In~\cite{Lawson2}, interesting examples of
surfaces with vanishing nonlocal mean curvature
are provided for~$s$ sufficiently close to~$1/2$. Remarkably, the surfaces in~\cite{Lawson2}
are the nonlocal analogues of the catenoids, but, differently
from the classical case
(in which catenoids grow logarithmically),
they approach a singular cone at infinity, see Figure~\ref{FRAC_CATENOID}.

Also, these nonlocal catenoids are highly unstable from the
variational point of view, since they possess infinite Morse index
(differently from the standard catenoid, which has Morse index
equal to one, i.e. it is, roughly speaking,
a minimizer in any functional direction with the exception of one).
\bigskip

Moreover, in~\cite{Lawson2}, there are also examples
of surfaces with vanishing nonlocal mean curvature
that can be seen as the nonlocal analogues of two
parallel hyperplanes. Namely,
for~$s$ sufficiently close to~$1/2$, there exists
a surface of revolution made of two sheets
which are the graph of a radial function~$f=\pm f(r)$.
When~$r$ is small, $f$ is of the order of~$1+(\frac12-s) r^2$,
but for large~$r$ it becomes of the order of~$\sqrt{\frac12-s}\cdot r$.
That is, the two sheets ``repel each other'' and produce a linear
growth at infinity. When~$s$ approaches $1/2$ the two sheets are
locally closer and closer to two parallel hyperplanes,
see Figure~\ref{FRAC_2_SHEETS}.

The construction above may be extended to build families
of surfaces with vanishing nonlocal mean curvature
that can be seen as the nonlocal analogue of $k$
parallel hyperplanes, for any~$k\in\N$. These $k$-sheet surfaces
can be seen as the bifurcation, as~$s$ is close to~$1/2$,
of the parallel hyperplanes~$\{x_n = a_i\}$, for~$i\in\{1,\dots,k\}$,
where the parameters~$a_i$ satisfy the constraints
\begin{equation}\label{1.9}
a_1 >\dots> a_k,\qquad \sum_{i=1}^k a_i=0\end{equation}
and the balancing relation
\begin{equation}\label{1.10}
a_i =2\sum_{{1\le j\le n}\atop{j\ne i}}
\frac{(-1)^{i+j+1}}{a_i-a_j}.
\end{equation}
\begin{center}
\begin{figure}[htb]
        \begin{minipage}[b]{0.95\linewidth}
        \centering
        \includegraphics[width=0.95\textwidth]{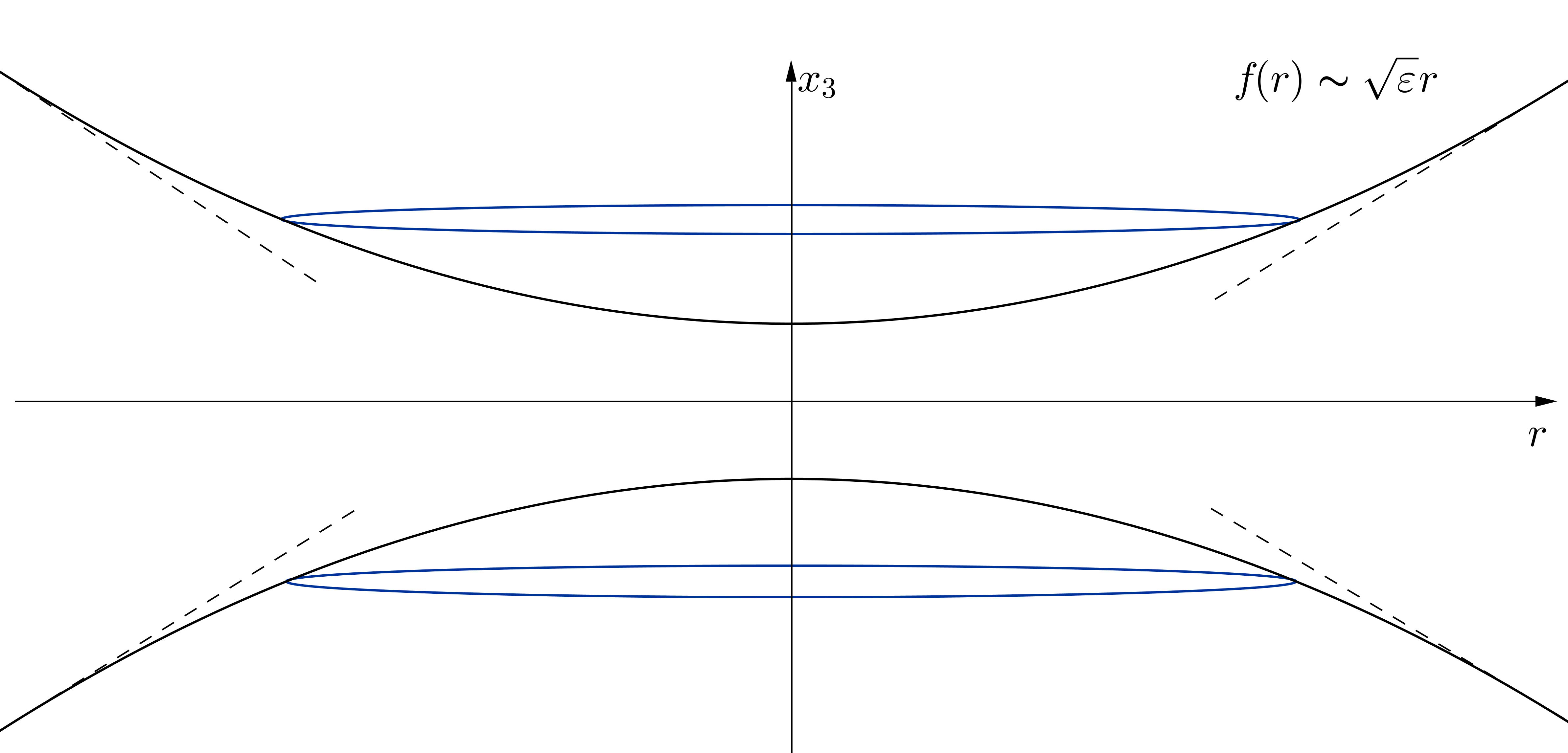}
        \caption{A two-sheet surface with vanishing fractional mean curvature}
        \label{FRAC_2_SHEETS}
        \end{minipage}
\end{figure}
\end{center}
It is actually quite interesting to observe that
solutions of~\eqref{1.10} correspond to (nondegenerate)
critical points of
the functional
$$ E(a_1,\dots,a_k):=\frac12\sum_{i=1}^k a_i^2+
\sum_{{1\le j\le n}\atop{j\ne i}}(-1)^{i+j}\log|a_i-a_j|$$
among all the $k$-ples~$(a_1,\dots,a_k)$ that satisfy~\eqref{1.9}.
\bigskip

These bifurcation techniques
rely on a careful expansion of the fractional perimeter functional
with respect to normal perturbations. That is, if~$E$
is a (smooth) set with vanishing fractional mean curvature,
and~$h$ is a smooth and compactly supported perturbation,
one can define, for any~$t\in\R$,
$$ E_h(t):= \{ x+t h(x)\nu(x),\; x\in\partial E\},$$
where~$\nu(x)$ is the exterior normal of~$E$ at~$x$.
Then, the second variation of the perimeter of~$E_h(t)$
at~$t=0$ is (up to normalization constants)
\begin{eqnarray*}
&& \int_{\partial E} \frac{h(y)-h(x)}{|x-y|^{n+2s}}
\,d{\mathcal{H}}^{n-1}(y) + h(x)\,
\int_{\partial E} \frac{\big(\nu(x)-\nu(y)\big)\cdot\nu(x)}{|x-y|^{n+2s}}
\,d{\mathcal{H}}^{n-1}(y) \\ &=&
\int_{\partial E} \frac{h(y)-h(x)}{|x-y|^{n+2s}}
\,d{\mathcal{H}}^{n-1}(y) + h(x)\,
\int_{\partial E} \frac{1-\nu(x)\cdot\nu(y)}{|x-y|^{n+2s}}
\,d{\mathcal{H}}^{n-1}(y)
.\end{eqnarray*}
Notice that the latter integral is non-negative, since
$\nu(x)\cdot\nu(y)\le1$.
The quantity above, in dependence of the perturbation~$h$,
is called, in jargon, ``Jacobi operator''.
It encodes an important geometric information,
and indeed, as~$s\to1/2$, it approaches the classical operator
$$ \Delta_{\partial E} h +|A_{\partial E}|^2\, h,$$
where~$\Delta_{\partial E}$ is the Laplace-Beltrami operator
along the hypersurface~$\partial E$ and~$|A_{\partial E}|^2$
is the sum of the squares of the
principal curvatures.
\bigskip
Other interesting sets that possess constant nonlocal mean curvature
with the structure of onduloids have been recently constructed 
in \cite{cabre-fall-weth} and \cite{davilaaa1}. This type of sets are periodic in a given direction 
and their construction has perturbative nature (indeed, the sets are 
close to a slab in the plane).

It is interesting to remark that the planar objects constructed
in~\cite{cabre-fall-weth} have no counterpart in the local framework, since
hypersurfaces of constant classical mean curvature with an onduloidal structure
only exist
in~$\R^n$ with~$n\ge3$: once again, this is a typical nonlocal effect,
in which the nonlocal mean curvature at a point is influenced by
the global shape of the set.\bigskip

While unbounded sets with constant nonlocal mean curvature
and interesting geometric features
have been constructed in~\cite{Lawson2, cabre-fall-weth},
the case of smooth and bounded sets is always geometrically trivial.
As a matter of fact, it has been recently proved independently in~\cite{cabre-fall-weth} and~\cite{figalli-maggi-ciraolo}
that bounded sets with smooth boundary
and constant mean curvature
are necessarily balls (this is the analogue
of a celebrated result by Alexandrov for surfaces
of constant classical mean curvature).
\bigskip 

\section{Boundary regularity}
The boundary regularity of the nonlocal minimal surfaces
is also a very interesting, and surprising, topic.
Indeed, differently from the classical case,
nonlocal minimal surfaces do not always attain boundary data
in a continuous way (not even in low dimension).
A possible boundary behavior is, on the contrary,
a combination of stickiness to the boundary
and smooth separation from
the adjacent portions. Namely, the nonlocal minimal surfaces
may have a portion that sticks at the boundary and that
separates from it in a~$C^{1,\frac12+s}$-way.
\begin{center}
\begin{figure}[htb]
        \begin{minipage}[b]{0.65\linewidth}
        \centering
        \includegraphics[width=0.65\textwidth]{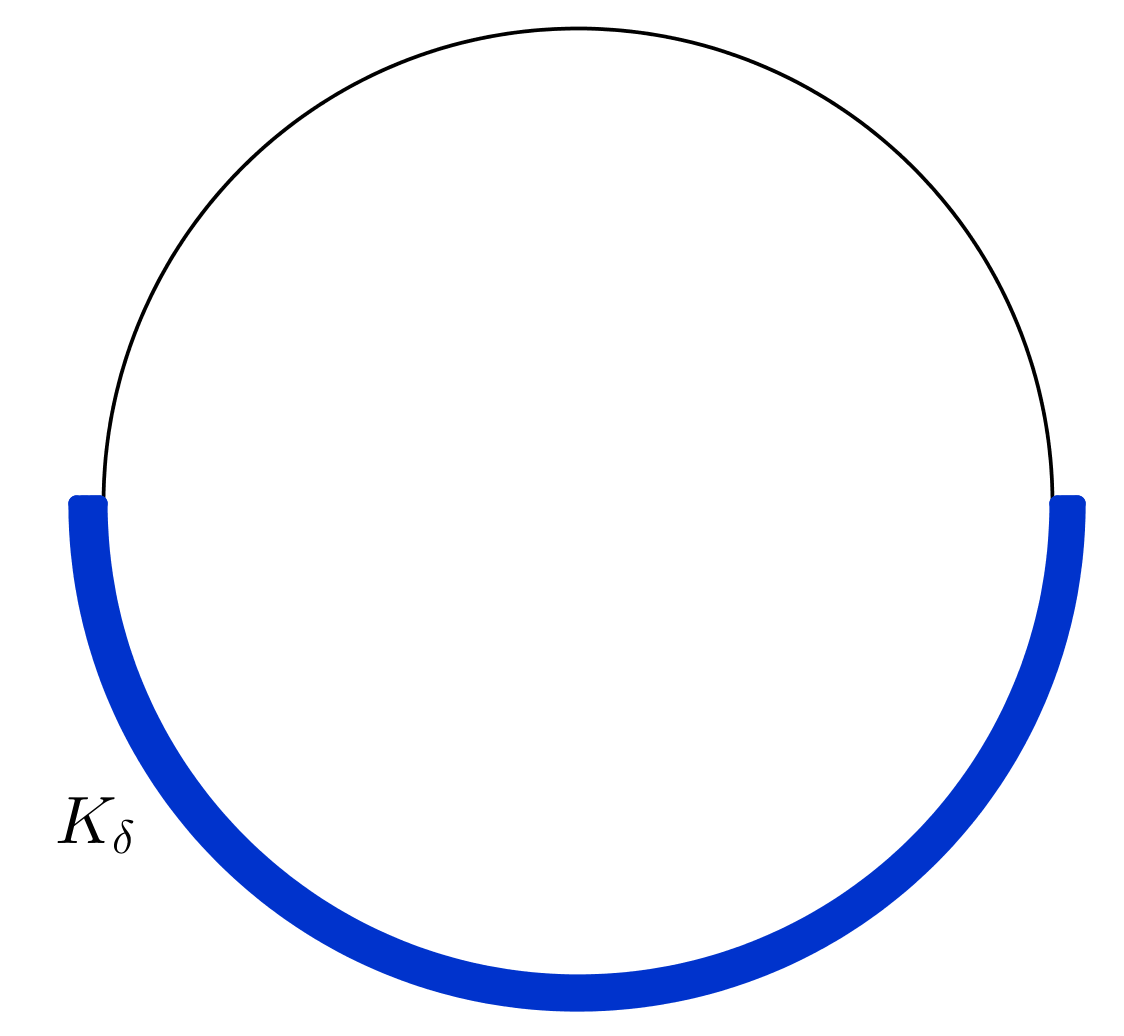}
        \caption{Stickiness properties of Theorem~\ref{STI-DSV-1}.}
        \label{eST1}
        \end{minipage}
\end{figure}
\end{center}
As an example, we can
consider, for any~$\delta>0$, the spherical cap
$$ K_\delta := \big( B_{1+\delta}\setminus B_1\big)\cap \{x_n<0\},$$
and obtain the following stickiness result:

\begin{thm}\label{STI-DSV-1}
There exists~$\delta_0>0$, depending on~$n$ and~$s$, such that
for any~$\delta\in(0,\delta_0]$, we have that
the $s$-minimal set in~$B_1$
that coincides with~$K_\delta$ outside~$B_1$
is~$K_\delta$ itself.

That is, the $s$-minimal set with datum~$K_\delta$ outside~$B_1$
is empty inside~$B_1$.
\end{thm}

The stickiness property of Theorem~\ref{STI-DSV-1}
is depicted in Figure~\ref{eST1}.

Other stickiness examples occur at the sides of slabs in the plane.
For instance,
given~$M>1$, one can 
consider the $s$-minimal set~$E_M$ in~$(-1,1)\times\R$
with datum outside~$(-1,1)\times\R$ given by
the ``jump'' set~$J_M:=J^-_M \cup J^+_M$,
where
\begin{eqnarray*}
&& J^-_M:= (-\infty,-1]\times (-\infty,-M)
\\{\mbox{and }}&&J^+_M:= [1,+\infty)\times (-\infty,M).\end{eqnarray*}
Then, if~$M$ is large enough,
the minimal set~$E_M$ sticks at the boundary of the slab:

\begin{thm}\label{STI-DSV-2}
There exist~$M_o>0$, $C_o>0$, depending on~$s$, such that
if~$M\ge M_o$ then
\begin{eqnarray}
&& [-1,1)\times [C_o M^{\frac{1+2s}{2+2s}},M]\subseteq E_M^c \label{SL-011}
\\{\mbox{and }}&& (-1,1]\times [-M,-C_o M^{\frac{1+2s}{2+2s}}]\subseteq E_M.
\label{SL-012}
\end{eqnarray}
\end{thm}

The situation of Theorem~\ref{STI-DSV-2}
is described in Figure~\ref{eST2}.
\begin{center}
\begin{figure}[htb]
        \begin{minipage}[b]{0.95\linewidth}
        \centering
        \includegraphics[width=0.95\textwidth]{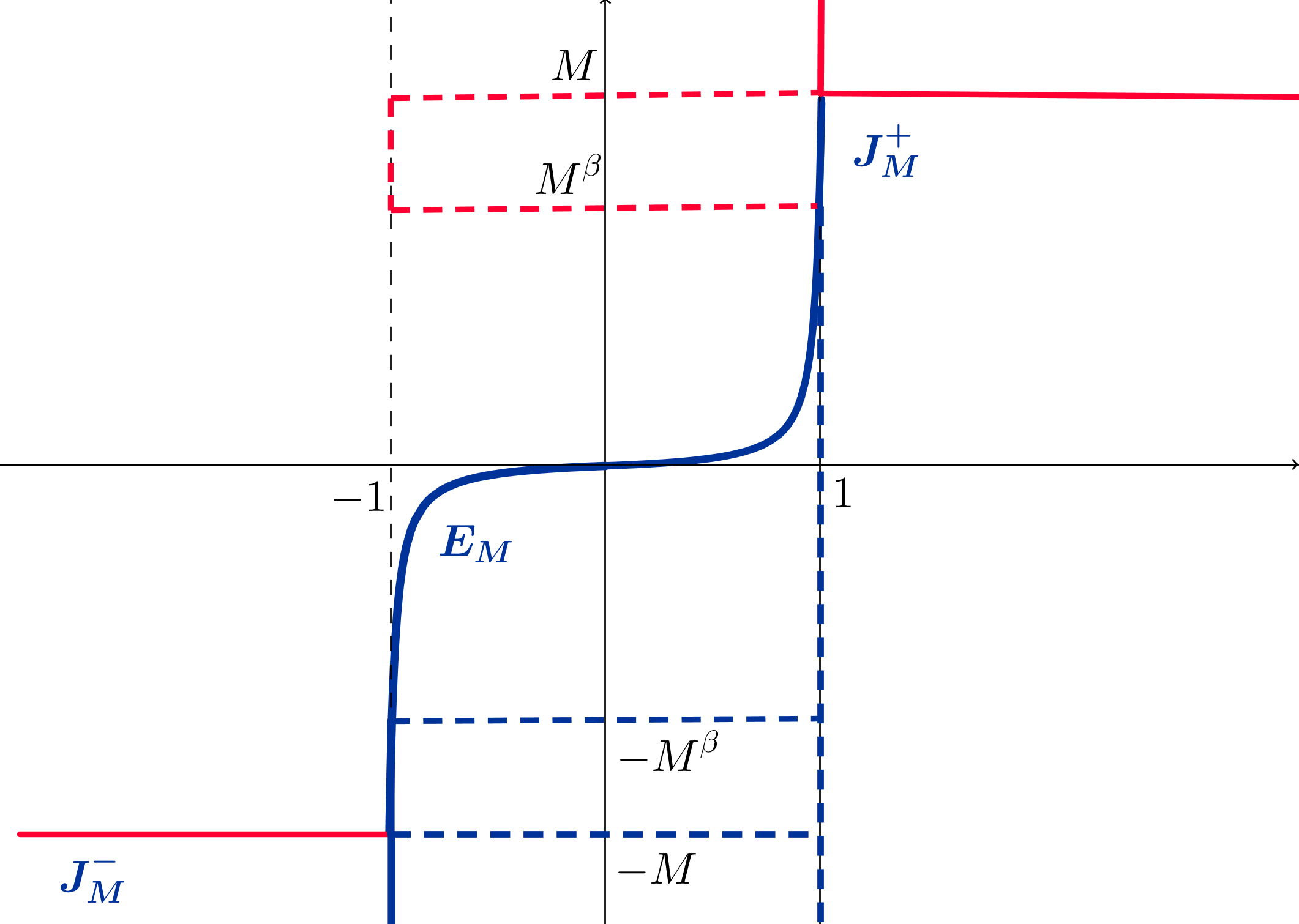}
        \caption{Stickiness properties of Theorem~\ref{STI-DSV-2}.}
        \label{eST2}
        \end{minipage}
\end{figure}
\end{center}
We mention that the ``strange''
exponent~$\frac{1+2s}{2+2s}$ in~\eqref{SL-011}
and~\eqref{SL-012} is optimal.\bigskip

For the detailed proof of
Theorems~\ref{STI-DSV-1} and~\ref{STI-DSV-2},
and other results on the 
boundary behavior
of nonlocal minimal surfaces, see~\cite{STICK}.
Here, we limit ourselves to give some 
heuristic motivation and a sketch of the proofs.
\bigskip

As a motivation for the (somehow unexpected) stickiness
property at the boundary,
one may look at Figure~\ref{eST1} and argue like this.
In the classical case, corresponding to~$s=1/2$,
independently on the width~$\delta$,
the set of minimal perimeter in~$B_1$ will always be
the half-ball~$B_1\cap \{x_n<0\}$. 

Now let us take~$s<1/2$.
Then, the half-ball~$B_1\cap \{x_n<0\}$ cannot be an $s$-minimal
set, since the nonlocal mean curvature, for instance, at the origin
cannot vanish. Indeed, the origin ``sees'' the complement 
of the set in a larger proportion than the set itself.
More precisely, in~$B_1$ (or even in~$B_{1+\delta}$)
the proportion of the set is
the same as the one of the complement, but outside~$B_{1+\delta}$
the complement of the set is dominant. Therefore, to ``compensate''
this lack of balance, the $s$-minimal set for~$s<1/2$ has to bend
a bit. Likely, the $s$-minimal set in this case will have the tendency
to become slightly convex at the origin, so that, at least nearby,
it sees a proportion of the set which is larger than the proportion of the complement
(we recall that, in any case, the proportion of the complement will
be larger at infinity, so the set needs to compensate at least near
the origin). But when~$\delta$ is very small, it turns
out that this compensation is not sufficient to obtain the desired
balance between the set and its complement: therefore,
the set has to ``stick'' to the half-sphere, in order to drop
its constrain to satisfy a vanishing nonlocal mean curvature equation.

Of course some quantitative estimates
are needed to make this argument work,
so we describe the sketch of the rigorous proof of
Theorem~\ref{STI-DSV-1} as follows.

\begin{proof}[Sketch of the proof of Theorem~\ref{STI-DSV-1}]
First of all, one checks that for any fixed~$\eta>0$,
if~$\delta>0$ is small enough, we have that
the interaction between~$B_1$ and~$B_{1+\delta}\setminus B_1$
is smaller than~$\eta$. In particular, by comparing with
a competitor that is empty in~$B_1$, by minimality we obtain that
\begin{equation}\label{ETA:DE} {\text{Per}}_s(E_\delta, B_1)\le \eta,\end{equation}
where we have denoted by~$E_\delta$ the
$s$-minimal set in~$B_1$
that coincides with~$K_\delta$ outside~$B_1$.

Then, one checks that
\begin{equation}\label{small:n}
{\mbox{the boundary of~$E_\delta$
can only lie in a small neighborhood of~$\partial B_1$}}\end{equation}
if~$\delta$ is sufficiently small.

Indeed, if, by contradiction,
there were points of~$\partial E_\delta$
at distance larger than~$\epsilon$
from~$\partial B_1$, then one could find
two balls of radius comparable to~$\epsilon$,
whose centers lie at distance larger than~$\epsilon/2$
from~$\partial B_1$ and at mutual distance smaller than~$\epsilon$,
and such that one ball is entirely contained in~$B_1\cap E_\delta$
and the other ball is entirely contained in~$B_1\setminus E_\delta$
(this is due to a 
Clean Ball Condition, see Corollary 4.3 in~\cite{CRS10}).
As a consequence,
${\text{Per}}_s(E_\delta, B_1)$ is bounded from below
by the interaction of these two balls, which is at least of
the order of~$\epsilon^{n-2s}$. Then, 
we obtain a contradiction with~\eqref{ETA:DE}
(by choosing~$\eta$ much smaller than~$\epsilon^{n-2s}$,
and taking~$\delta$ sufficiently small).

This proves~\eqref{small:n}. {F}rom this, it follows
that 
\begin{equation}\label{HJ:pr:ba}
\begin{split}
&{\mbox{the whole  set~$E_\delta$ must lie in a small neighborhood of~$\partial B_1$.}}\end{split}\end{equation}
Indeed, if this were not so, by~\eqref{small:n} the set~$E_\delta$
must contain a ball of radius, say~$1/2$. Hence,
${\text{Per}}_s(E_\delta, B_1)$ is bounded from below
by the interaction of this ball against~$\{x_n>0\}\setminus B_1$,
which would produce a contribution of order one,
which is in contradiction with~\eqref{ETA:DE}.

Having proved~\eqref{HJ:pr:ba}, one can use it to
complete the proof of
Theorem~\ref{STI-DSV-1} employing a geometric argument.
Namely, one considers the ball~$B_\rho$,
which is outside~$E_\delta$ for small~$\rho>0$, in virtue of~\eqref{HJ:pr:ba},
and then enlarges~$\rho$ untill it touches~$\partial E_\delta$.
If this contact occurs at some point~$p\in B_1$,
then the nonlocal mean curvature of~$E_\delta$ at~$p$ must be zero.
But this cannot occur (indeed, we know by~\eqref{HJ:pr:ba}
that the contribution of~$E_\delta$ to the nonlocal mean curvature
can only come from a small neighborhood of~$\partial B_1$,
and one can check, by estimating integrals, that this is not
sufficient to compensate the outer terms in which the complement of~$E_\delta$
is dominant).

As a consequence, no touching point between~$B_\rho$
and~$\partial E_\delta$ can occur in~$B_1$, which shows that~$E_\delta$
is void inside~$B_1$ and completes the proof of
Theorem~\ref{STI-DSV-1}.
\end{proof}

As for the proof of Theorem~\ref{STI-DSV-2}, 
the main arguments are based on sliding a ball of suitably large radius
till it touches the set, with careful quantitative estimates.
Some of the details are as follows (we refer to~\cite{STICK}
for the complete arguments).

\begin{proof}[Sketch of the proof of Theorem~\ref{STI-DSV-2}]
The first step is to prove
a weaker form of stickiness as the one claimed in Theorem~\ref{STI-DSV-2}.
Namely, one shows that
\begin{eqnarray}
&& [-1,1)\times [c_oM\,,M]\subseteq E_M^c \label{SL-011-PRE}
\\{\mbox{and }}&& (-1,1]\times [-M,\,-c_o M]\subseteq E_M,
\label{SL-012-PRE}
\end{eqnarray}
for some~$c_o\in(0,1)$. Of course, the statements in~\eqref{SL-011}
and~\eqref{SL-012} are stronger than
the ones in~\eqref{SL-011-PRE}
and~\eqref{SL-012-PRE} when~$M$ is large, since~${\frac{1+2s}{2+2s}}<1$,
but we will then obtain them later in a second step.

To prove~\eqref{SL-011-PRE},
one takes balls of radius~$c_oM$ and centered at~$\{x_2=t\}$,
for any~$t\in [c_o M,\,M]$. One slides these balls from left to right,
till one touches~$\partial E_M$. When~$M$ is large enough
(and~$c_o$ small enough) this contact point cannot lie
in~$\{|x_1|<1\}$. This is due to the fact that
at least the sliding ball lies outside~$E_M$,
and the whole~$\{x_2>M\}$ lies outside~$E_M$ as well.
As a consequence, these contact points see a proportion of~$E_M$
smaller than the proportion of the complement
(it is true that
the whole of~$\{x_2<-M\}$ lies inside~$E_M$,
but this contribution comes from further away than the ones just
mentioned, provided that~$c_o$ is small enough).
Therefore, contact points cannot satisfy a vanishing mean curvature
equation and so they need to lie on the boundary of the domain
(of course, careful quantitative estimates are
necessary here, see~\cite{STICK},
but we hope to have given an intuitive sketch of the computations
needed).

In this way, one sees that all the portion
$[-1,1)\times [c_oM\,,M]$
is clean from the set~$E_M$
and so~\eqref{SL-011-PRE} is established
(and~\eqref{SL-012-PRE} can be proved similarly).

Once~\eqref{SL-011-PRE} and~\eqref{SL-012-PRE} are established, one uses them
to obtain the strongest form expressed
in~\eqref{SL-011} and~\eqref{SL-012}. For this,
by~\eqref{SL-011-PRE} and~\eqref{SL-012-PRE},
one has only to take care of points
in~$\{ |x_2|\in [C_o M^{ \frac{1+2s}{2+2s} }, \,c_oM]\}$.
For these points, one can use again a sliding method,
but, instead of balls, 
one has to use suitable surfaces obtained by appropriate
portions of balls and adapt the calculations in order to evaluate
all the contributions arising in this way.

The computations are not completely obvious (and once again we refer
to~\cite{STICK} for full details), but the idea is, once again,
that contact points that are
in the set $\{ |x_2|\in [C_o M^{ \frac{1+2s}{2+2s} }, \,c_oM]\}$
cannot satisfy the balancing relation prescribed by
the vanishing nonlocal mean curvature equation.
\end{proof}

The stickiness property discussed above also has an interesting
consequence in terms of the ``geometric stability'' of
the flat $s$-minimal surfaces. 
For instance, rather surprisingly, 
the flat lines in the plane are ``geometrically
unstable'' nonlocal minimal surfaces,
in the sense that an arbitrarily small and compactly supported
perturbation can produce a
stickiness phenomenon at the boundary of the domain.
Of course, the smaller the perturbation, the smaller
the stickiness phenomenon, but it is quite relevant that
such a stickiness property
can occur for arbitrarily small (and ``nice'')
perturbations. This means that $s$-minimal flat objects,
in presence of a perturbation, may not only ``bend''
in the center of the domain, but rather ``jump''
at boundary points as well.
\bigskip

To state this phenomenon in a mathematical framework, one can consider,
for fixed~$\delta>0$ the planar sets
\begin{eqnarray*}
&& H:=\R\times(-\infty,0),\\
&& F_-:=(-3,-2)\times [0,\delta)\\
{\mbox{and }}&&
F_+:= (2,3)\times [0,\delta).\end{eqnarray*} 
One also fixes a set~$F$
which contains~$H\cup F_-\cup F_+$
and denotes by~$E$ be the
$s$-minimal set in~$(-1,1)\times\R$ among all the sets that coincide
with~$F$ outside~$(-1,1)\times\R$. 
\begin{center}
\begin{figure}[h]
        \hspace{1.0cm}
        \begin{minipage}[b]{0.95\linewidth}
        \centering
        \includegraphics[width=0.95\textwidth]{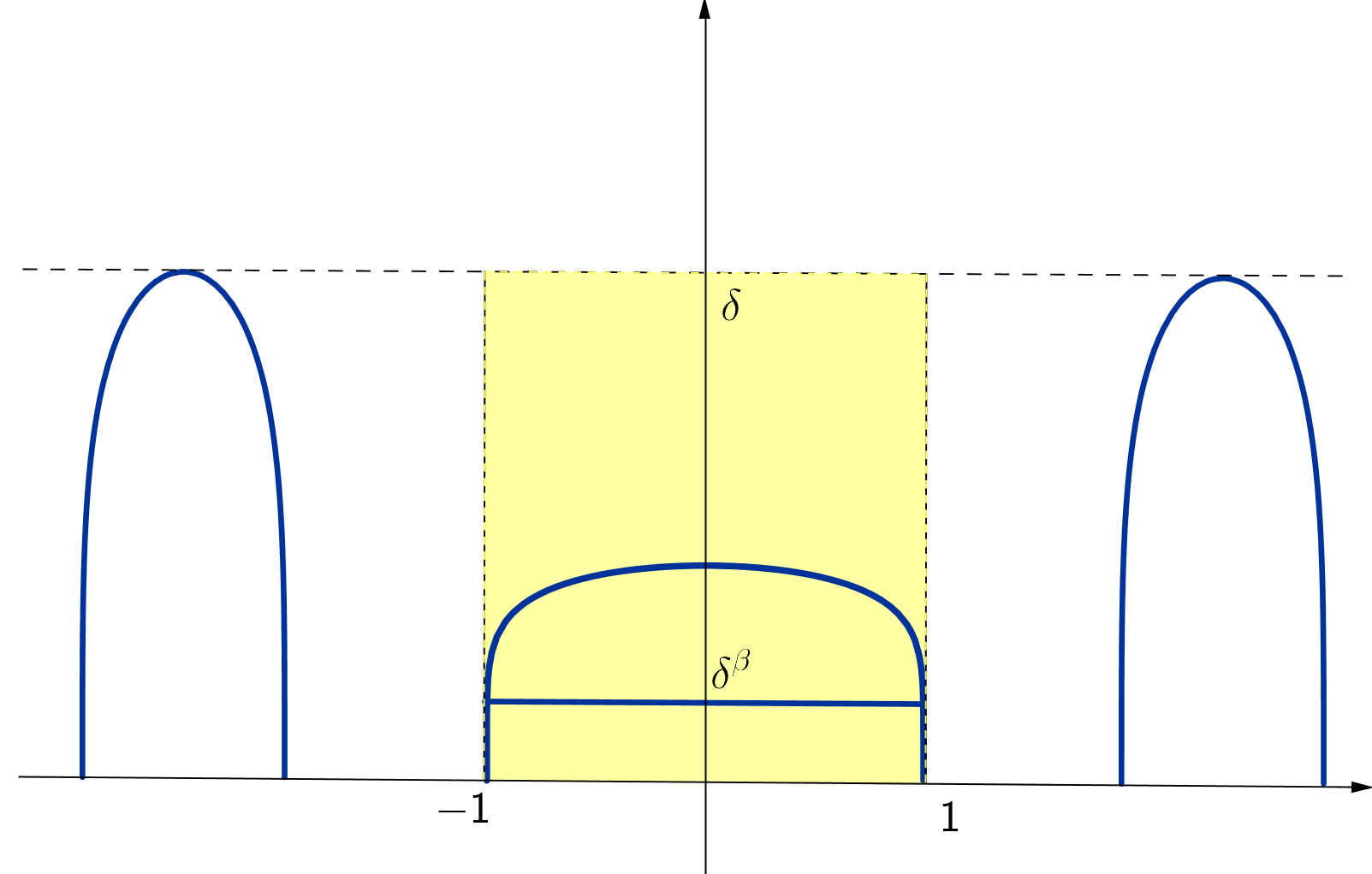}
        \caption{The stickiness/instability property in Theorem~\ref{UNS},
with~$\beta:=\frac{2+\epsilon_0}{1-2s}$}
        \label{PUNS}
        \end{minipage}
\end{figure}
\end{center}
Then, this set~$E$ sticks at the boundary of the domain, according
to the next result:

\begin{thm}\label{UNS}
Fix~$\epsilon_0>0$ arbitrarily small.
Then, there exists~$\delta_0>0$, possibly depending on~$\epsilon_0$,
such that, for any~$\delta\in(0,\delta_0]$,
$$ E\supseteq (-1,1)\times (-\infty, \delta^{\frac{2+\epsilon_0}{1-2s}} ].$$
\end{thm}

The stickiness/instability property in Theorem~\ref{UNS}
is depicted in Figure~\ref{PUNS}. We remark that
Theorem~\ref{UNS}
gives a rather precise quantification
of the size of the stickiness in terms of the size
of the perturbation: namely the size of the stickiness in Theorem~\ref{UNS}
is larger than the size of the perturbation
to the power~$\beta:=\frac{2+\epsilon_0}{1-2s}$,
for any~$\epsilon_0>0$ arbitrarily small.
Notice that~$\beta\to+\infty$
as~$s\to1/2$, consistently
with the fact that
classical minimal surfaces do not stick at the boundary.
\bigskip

The proof of Theorem~\ref{UNS}
is based on the construction of suitable auxiliary barriers.
These barriers are used to detach a portion of the set
in a neighborhood of the origin and their construction
relies on some compensations of nonlocal integral terms.
In a sense, the building blocks of these barriers are
``self-sustaining solutions'' that
can be seen as the geometric counterparts of
the $s$-harmonic function~$x_+^s$ discussed in Section~\ref{shf1}.

Indeed, roughly speaking, like the function~$x_+^s$,
these barriers ``see'' a proportion of the set
in~$\{x_1<0\}$ larger than what is produced by their tangent plane,
but a proportion smaller than that at infinity, due to their sublinear
behavior. Once again, the computations needed to
check such a balancing conditions are a bit involved,
and we refer to~\cite{STICK} for the complete details.
\begin{center}
\begin{figure}[h]
        \hspace{1.0cm}
        \begin{minipage}[b]{0.85\linewidth}
        \centering
        \includegraphics[width=0.85\textwidth]{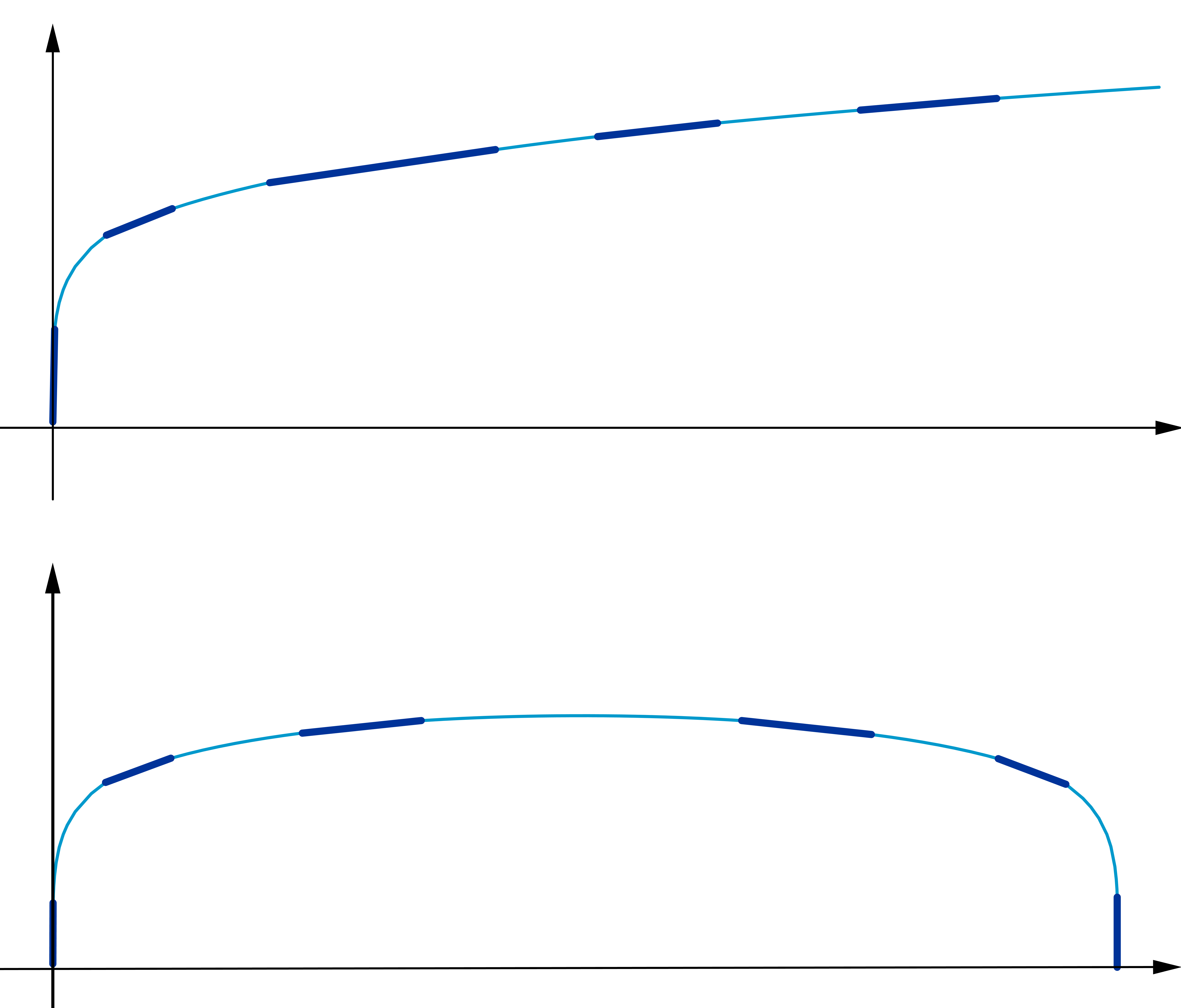}
        \caption{Auxiliary barrier for the proof of Theorem~\ref{UNS}}
        \label{barr}
        \end{minipage}
\end{figure}
\end{center}
\bigskip

To conclude this chapter, we make a remark on the connection between 
solutions of the fractional Allen-Cahn equation and $s$-minimal 
surfaces. Namely, a suitably scaled version of
the functional in \eqref{enfac}
$\Gamma$-converges to either the classical perimeter
or the nonlocal perimeter functional, depending
on the fractional parameter~$s$.
The $\Gamma$-convergence
is a type of convergence of functionals
that is compatible with the
minimization of the energy, and turns out to be very useful when dealing with variational problems indexed by a parameter. This notion was introduced by De Giorgi, see e.g.~\cite{degg}
for details. 

In the nonlocal case, some care is needed to introduce
the ``right'' scaling
of the functional, which
comes from the dilation invariance of the 
space coordinates and possesses a nontrivial energy
in the limit.
For this, one takes first the rescaled energy functional
 	\[ J_\eee (u, \Omega) := \eee^{2s} \K (u,\Omega)+ 
\int_{\Omega} W(u)\, dx,\]
where $\mathcal{K}$ is the kinetic energy defined in \eqref{kenac}.
Then, one considers the functional
\begin{equation*}\begin{split}&
F_\eee (u,\Omega) :=\left\{
\begin{matrix}
\eee^{-2s} J_\eee(u, \Omega) & {\mbox{ if $s\in(0,\,1/2)$,}} \\
\ \\
|\eee \log \eee|^{-1} J_\eee(u, \Omega) & {\mbox{ if $s=1/2$,}}\\
\ \\
\eee^{-1} J_\eee(u, \Omega) & {\mbox{ if $s\in(1/2,\,1)$.}}
\end{matrix}
\right.\end{split}\end{equation*}
The limit functional of~$F_\eee$ as~$\eee\to0$ depends on~$s$.
Namely, when~$s\in(0,1/2)$, the limit functional is
(up to dimensional constants that we neglect) the fractional
perimeter, i.e.
\begin{equation}\label{FF1}
F(u,\Omega):=\left\{\begin{matrix}
{\text{Per}}_s(E,\Omega)
& {\mbox{ if $u|_\Omega = \chi_E -\chi_{E^{\C}}$, for some set $E\subset
\Omega$}}
\\+\infty & {\mbox{otherwise.}}
\end{matrix}\right.
\end{equation}
On the other hand, when~$s\in[1/2,1)$, the limit functional of~$F_\eee$
is (again, up to normalizing constants) the classical
perimeter, namely
\begin{equation}\label{FF2}
F(u, \Omega):=\left\{\begin{matrix}
{\text{ Per}} (E,\Omega)
& {\mbox{ if $u|_\Omega = \chi_E -\chi_{ E^{\C}}$, for some set $E\subset
\Omega$}}
\\+\infty & {\mbox{otherwise,}}
\end{matrix}\right.
\end{equation}
That is, the following limit statement holds true:

\begin{thm}\label{TH1-SV-GAMMA}
Let~$s\in(0,\,1)$. Then,
$F_\eee$ $\Gamma$-converges to~$F$, as defined
in either~\eqref{FF1} or~\eqref{FF2},
depending on whether~$s\in(0,1/2)$ or~$s\in[1/2,1)$.
\end{thm}

For precise statements and further details, see~\cite{savingamma}.
Additionally, we remark that the level sets
of the minimizers of
the functional in \eqref{enfac}, 
after a homogeneous scaling in the space variables,
converge
locally uniformly to minimizers 
either of the fractional perimeter (if $s\in(0,1/2)$)
or of the classical perimeter (if~$s\in[1/2,1)$):
that is, the ``functional'' convergence stated in
Theorem~\ref{TH1-SV-GAMMA}
has also a ``geometric'' counterpart: for this, see
Corollary~1.7 in~\cite{SV14}.\bigskip

One can also interpret Theorem~\ref{TH1-SV-GAMMA} by saying that
a nonlocal phase transition possesses two parameters, $\eee$
and~$s$. When~$\eee\to0$, the limit interface
approaches a minimal surface either in the fractional
case (when~$s\in(0,1/2)$) or in the classical case (when~$s\in[1/2,1)$).
This bifurcation at~$s=1/2$ somehow states that
for lower values of~$s$ the nonlocal phase transition
possesses a nonlocal interface in the limit, but for larger values of~$s$
the limit interface is characterized only by local features
(in a sense, when~$s\in(0,1/2)$ the ``surface tension effect''
is nonlocal, but for~$s\in [1/2,1)$ this effect localizes).

It is also interesting to compare Theorems~\ref{TP12}
and~\ref{TH1-SV-GAMMA}, since
the bifurcation at~$s=1/2$ detected by Theorem~\ref{TH1-SV-GAMMA}
is perfectly compatible with the limit behavior of the
fractional perimeter, which reduces to the classical
perimeter exactly for this value of~$s$, as stated in Theorem~\ref{TP12}.
  
\chapter{A nonlocal nonlinear stationary Schr\"{o}dinger type equation}\label{S:NP:2}

The type of problems introduced in this chapter are connected to solitary solutions of nonlinear dispersive wave equations (such as the Benjamin-Ono equation, the Benjamin-Bona-Mahony equation and the fractional Schr\"{o}dinger equation). In this chapter, only stationary equations are studied and we redirect the reader to \cite{vazquez_nonlinear, vazquez_recent} for the study of evolutionary type equations.

Let $n\geq 2$ be the dimension of the reference space, $s \in (0,1)$ be the fractional parameter, and $\epsilon>0$ be a small parameter. We consider the so-called fractional Sobolev exponent 
		\begin{equation*}
			2_s^*:= 
				\begin{cases}
				\displaystyle \frac{2n}{n-2s} &\text { for } n\geq 3, \text{ or } n=2 \text{ and } s\in(0,1/2)\\
				+ \infty &\text{ for } n=1 \text{ and } s\in(0,1/2]	  
				\end{cases}
		\end{equation*}
and introduce the following nonlocal nonlinear Schr\"{o}dinger equation 
	\begin{equation}
		\begin{cases}
		\epsilon^{2s} \frlap u +u = u^p &\text{ in } \Omega \subset \Rn\\ 
		u=0 &\text{ in } \Rn \setminus \Omega, 
		\end{cases}
		\label{sch}
	\end{equation}
in the subcritical case $p\in (1, 2_s^*-1)$, namely when $p\in \displaystyle \bigg(1, \frac{n+2s}{n-2s}\bigg)$.

This equation arises in the study of the 
fractional Schr\"{o}dinger equation when looking for standing waves. 
Namely, the fractional Schr\"{o}dinger equation
considers solutions~$\Psi=\Psi(x,t):\R^n\times\R\to {\mathbb{C}}$ of
\begin{equation}\label{PLANCK}
i\hslash\partial_t \Psi = \big( \hslash^{2s} (-\Delta)^s +V\big)\Psi,\end{equation}
where~$s\in(0,1)$, $\hslash$ is the reduced
Planck constant and~$V=V(x,t,|\Psi|)$ is a potential.
This equation
is of interest in quantum mechanics
(see e.g.~\cite{L00} and the appendix in~\cite{DPDV14}
for details and physical motivations). 
Roughly speaking, the quantity~$|\Psi(x,t)|^2\,dx$
represents the probability density of finding a quantum particle
in the space region~$dx$ and at time~$t$.

The simplest solutions of~\eqref{PLANCK} are the ones
for which this probability density is independent of time,
i.e.~$|\Psi(x,t)|=u(x)$ for some~$u:\R^n\to[0,+\infty)$.
In this way, one can write $\Psi$ as~$u$ times
a phase that oscillates (very rapidly) in time: that is
one may look for solutions of~\eqref{PLANCK} of the form
$$ \Psi(x,t) := u(x)\, e^{i\omega t/\hslash},$$
for some frequency~$\omega\in\R$.
Choosing~$V=V(|\Psi|)=-|\Psi|^{p-1}=-u^{p-1}$,
a substitution into~\eqref{PLANCK} gives that
$$
\Big( \hslash^{2s} (-\Delta)^s u +\omega u - u^p \Big)\, e^{i\omega t/\hslash}
=\hslash^{2s} (-\Delta)^s \Psi -
i\hslash\partial_t \Psi +V\Psi=0,$$
which is~\eqref{sch}
(with the normalization convention~$\omega:=1$ and~$\epsilon:=\hslash$).

The goal of this chapter is to construct solutions of problem \eqref{sch} that concentrate at interior points of the domain $\Omega$ for sufficiently small values of $\epsilon$. We perform
a blow-up of the domain, defined as
	\[\Omega_\epsilon:=\displaystyle \frac{1}{\epsilon} \Omega=\bigg\{  \frac{x}{\epsilon}, x\in \Omega\bigg\}.\]
We can also rescale the solution of \eqref{sch} on $\Omega_\epsilon$, 
	\[u_\epsilon(x)=u(\epsilon x).\]
The problem \eqref{sch} for $u_\epsilon$ then reads
	\begin{equation}
		\begin{cases}
		 \frlap u +u = u^p &\text{ in } \Omega_\epsilon\\ 
		u=0 &\text{ in } \Rn \setminus \Omega_\epsilon.
		\end{cases}
		\label{dsch}
	\end{equation} 
When~$\epsilon\to0$, the domain~$\Omega_\epsilon$ invades the whole of the space.
Therefore, it is also natural to 
consider (as a first approximation)
the equation on the entire space  
	\begin{equation}
	 	\frlap u +u=u^p \text{ in } \Rn.
		\label{entsch}
	\end{equation}
The first result that
we need  is that there exists an entire positive radial least energy solution $w \in H^s(\Rn)$ of \eqref{entsch}, called the \emph{ground state solution}. 
Here follow some relevant results on this. The interested reader can find their proofs in \cite{FLS13}.
\begin{enumerate}
	\item The ground state solution $w\in H^s(\Rn)$ is unique up to translations. 
	\item The ground state solution $w\in H^s(\Rn)$ is nondegenerate, i.e., the derivatives $D_i w$ are solutions to the linearized equation
		\begin{equation}
		 	\frlap Z +Z= pZ^{p-1}.
			\label{linsch}
		\end{equation} 
	\item The ground state solution $w\in H^s(\Rn)$ decays polynomially at infinity, namely there exist two constants $\alpha, \beta >0$  such that 
		\[ \alpha |x|^{-(n+2s)} \leq u(x) \leq \beta |x|^{-(n+2s)}.\] 
\end{enumerate}
Unlike the fractional case, we remark that for the (classical) Laplacian, at infinity  the ground state solution decays exponentially fast. We also refer to~\cite{FL1D} for the one-dimensional case.

The main theorem of this chapter establishes the existence of a solution that concentrates at interior points of the domain for sufficiently small values of $\epsilon$. This concentration phenomena is written
in terms of the ground state solution~$w$.
Namely, the first approximation for the solution
is exactly the ground state~$w$,
scaled and concentrated at an appropriate point~$\xi$ of
the domain. More precisely, we have:

\begin{thm}\label{THSP}
If $\epsilon$ is sufficiently small, there exist a point $\xi \in \Omega$ and a solution $U_\epsilon$ of the problem \eqref{sch} such that
			\[ \bigg| U_\epsilon (x)- w\Big(\frac{x-\xi}{\epsilon}\Big)\bigg| \leq C \epsilon^{n+2s},\]
and $\text{dist}(\xi, \partial \Omega)\geq \delta>0$. Here, $C$ and $\delta$ are constants independent of $\epsilon$ or $\Omega$, and the
function $w$ is the ground state solution of problem \eqref{entsch}.
\label{nnscheq}
\end{thm}

The concentration point~$\xi$ in Theorem~\ref{THSP}
is influenced by the global geometry of the domain.
On the one hand, when~$s=1$, the point~$\xi$ is the one that maximizes
the distance from the boundary. On the other hand,
when~$s\in(0,1)$, such simple characterization of~$\xi$
does not hold anymore: in this case, $\xi$
turns out to be asymptotically the maximum of a (complicated, but rather explicit)
nonlocal functional: see \cite{DPDV14} for more details.

We state here the basic idea of the proof of Theorem~\ref{THSP}
(we refer again to \cite{DPDV14} for more details).
\begin{proof}[Sketch of the proof of Theorem \ref{nnscheq}] In this proof, we make use of the Lyapunov-Schmidt procedure. Namely, rather than looking for the
solution in an infinite-dimensional functional space,
one decomposes the problem into two orthogonal subproblems.
One of these problem is still infinite-dimensional,
but it has the advantage to bifurcate from a known object
(in this case, a translation of the ground state).
Solving this auxiliary subproblem does not provide a true
solution of the original problem, since a leftover
in the orthogonal direction may remain. To kill this
remainder term, one solves a second subproblem,
which turns out to be finite-dimensional
(in our case, this subproblem is set in~$\R^n$,
which corresponds to the action of the translations
on the ground state).

A structural advantage of the problem considered
lies in its variational structure.
Indeed, equation~\eqref{dsch} is 
the Euler-Lagrange equation of
the energy functional
	\eqlab{ \label{peren} 
I_\epsilon(u) =\frac{1}{2} \int_{\Omega_\epsilon} \Big( \frlap u(x)+u(x) \Big) u(x) \, dx - \frac{1}{p+1} \int_{\Omega_\epsilon} u^{p+1} (x) \, dx }
for any $u \in H^s_0(\Omega_\epsilon) := \{ u\in H^s(\R^n) \text { s.t. } u=0 \text{ a.e. in } \Rn \setminus \Omega_\epsilon\}$.
Therefore, the problem reduces to finding critical points of~$I_\epsilon$.

To this goal,
we consider the \emph{ground state} solution $w$ and  for any $\xi \in \Rn$ we let $w_\xi:=w(x-\xi)$. For a given $\xi \in \Omega_\epsilon$ a first approximation $\bar{u}_\xi$ for the solution of problem \eqref{dsch} can be taken as the solution of the linear problem 
	\begin{equation}
		\begin{cases}
		 \frlap \overline u_\xi +\overline u_\xi = w_\xi^p &\text{ in } \Omega_\epsilon,\\ 
		\overline u_\xi=0 &\text{ in } \Rn \setminus \Omega_\epsilon.
		\end{cases}
		\label{dlsch}
	\end{equation}
The actual solution will be obtained as a small perturbation of $\bar u_\xi$ for a suitable point $\xi = \xi (\epsilon)$.
We define the operator $\mathcal{L}:=\frlap +I$, where~$I$ is the identity and we notice that $\mathcal{L}$ has a unique fundamental solution that solves
	\[\mathcal{L}\Gamma = \delta_0 \quad \text{ in } \Rn.\]
The Green function $G_\epsilon$ of the operator $\mathcal{L}$ in $\Omega_\epsilon$ satisfies 	
	\begin{equation}
		\begin{cases}
		\mathcal{L} G_\epsilon(x,y) = \delta_y(x)  &\text{ if } x \in \Omega_\epsilon,\\ 
		 G_\epsilon (x,y)=0 &\text{ if } x \in \Rn \setminus \Omega_\epsilon. 
		\end{cases}
		\label{gsch}
	\end{equation}
It is convenient to introduce 
the regular part of $G_\epsilon$, which is often called
the Robin function. This function is defined by
	\begin{equation}\label{7.8bis}
H_\epsilon (x,y):= \Gamma(x-y) - G_\epsilon(x,y) \end{equation}
and it satisfies, for a fixed $y \in \Rn$,
	\begin{equation}
		\begin{cases}
		\mathcal{L} H_\epsilon(x,y) = 0 &\text{ if } x \in \Omega_\epsilon,\\ 
		H_\epsilon (x,y)=\Gamma(x-y) &\text{ if } x \in \Rn \setminus \Omega_\epsilon. 
		\label{rsch}
		\end{cases}
	\end{equation}
Then 
	\[\overline u_\xi (x)= \int_{\Omega_\epsilon} \overline u_\xi (y)\delta_0(x-y) \, dy,\]
and by \eqref{gsch}
	\[\overline u_\xi (x)=  \int_{\Omega_\epsilon} \overline u_\xi (y)\mathcal{L} G_\epsilon(x,y) \, dy.\]
The operator $\mathcal{L}$ is self-adjoint and thanks to the above identity and to equation \eqref{dlsch} it follows that 
	\[ \begin{split}
			 \overline u_\xi (x)= &\; \int_{\Omega_\epsilon} \mathcal{L} \overline u_\xi (y) G_\epsilon(x,y) \, dy\\
					=&\;  \int_{\Omega_\epsilon}
w_\xi^p(y) G_\epsilon(x,y)\, dy.
		\end{split}\]
So, we use~\eqref{7.8bis} and we obtain that
$$ \overline u_\xi (x)=
\int_{\Omega_\epsilon} w_\xi^p(y) \Gamma(x-y) \, dy-\int_{\Omega_\epsilon} w_\xi^p(y) H_\epsilon(x,y)\, dy.$$
Now we notice that, since $w_\xi$ is solution of \eqref{entsch} and $\Gamma$ is the fundamental solution of $\mathcal{L}$, we have that
		\[ \begin{split}
			\int_{\Rn} w_\xi^p(y)\Gamma(x-y) \, dy =&\;\int_{\Rn} \mathcal{L}w_\xi(y) \Gamma(x-y) \, dy\\
											=&\; \int_{\Rn} w_\xi (y)  \mathcal{L}\Gamma(x-y) \, dy\\
											=&\; w_\xi (x).
		\end{split}\]
Therefore we have obtained that 
 		\begin{equation}
	 \overline u_\xi (x) = w_\xi (x) -\int_{\Rn \setminus \Omega_\epsilon} w_\xi^p(y) \Gamma(x-y) \, dy- \int_{\Omega_\epsilon} w_\xi^p(y) H_\epsilon(x,y)\, dy.
		\label{uexp}
	\end{equation}
Now we can insert~\eqref{uexp} into the energy functional~\eqref{peren}
and expand the errors
in powers of~$\epsilon$.
For $\text{dist}(\xi,\partial \Omega_\epsilon)\geq \displaystyle \frac{\delta}{\epsilon}$ with $\delta$ fixed and small, the energy of $\overline u_\xi$ is a perturbation of the energy of the ground state~$w$
and one finds (see Theorem 4.1 in~\cite{DPDV14})
that
	\begin{equation}
		 I_\epsilon(\overline u_\xi) = I(w) + \frac{1}{2} \mathcal{H}_\epsilon(\xi) + \mathcal{O}(\epsilon^{n+4s}),
	\label{en}
	\end{equation}
where 
	\[\mathcal{H}_\epsilon (\xi):= \int_{\Omega_\epsilon}\int_{\Omega_\epsilon} H_\epsilon(x,y) w_\xi^p(x)w_\xi^p(y) \, dx \, dy\]
and $I$ is the energy computed on the whole space $\Rn$, namely
	\begin{equation}
	I(u) =\frac{1}{2} \int_{\Rn} \Big( \frlap u(x)+u(x) \Big) u(x) \, dx - \frac{1}{p+1} \int_{\Rn} u^{p+1} (x) \, dx.
	\label{enfuncrn}
	\end{equation}
In particular, $I_\epsilon(\overline u_\xi)$
agrees with a constant (the term~$I(w)$), plus a functional over a finite-dimensional space
(the term~$\mathcal{H}_\epsilon (\xi)$, which only depends on~$\xi\in\R^n$),
plus a small error.

We remark that the solution~$\overline u_\xi$
of equation~\eqref{dlsch} which can be obtained from~\eqref{uexp}
does not provide a solution for the original problem~\eqref{dsch}
(indeed, it only solves~\eqref{dlsch}):
for this,
we look for solutions $u_\xi$ of \eqref{dsch} as perturbations of~$\overline u_\xi$,
in the form
	\begin{equation}\label{POj}
u_\xi :=\overline u_\xi +\psi.\end{equation}
The perturbation functions $\psi$ are considered among those vanishing outside $\Omega_\epsilon$ and orthogonal to the space $\mathcal{Z}=\text{Span}(Z_1,\dots,Z_n)$, where $Z_i=\displaystyle \frac{\partial w_\xi}{\partial x_i}$ are solutions of the linearized equation \eqref{linsch}.
This procedure somehow ``removes the degeneracy'',
namely we look for the corrector~$\psi$ in
a set where the linearized operator is invertible.
This makes it
possible, fixed any~$\xi\in\R^n$, 
to find~$\psi=\psi_\xi$ such that the function~$u_\xi$,
as defined in~\eqref{POj}
solves the equation 
	\eqlab{ \label{qsci} \frlap u_\xi +u_\xi=u_\xi^p +\sum_{i=1}^n c_i Z_i \text{ in }\Omega_\epsilon.}
That is, $u_\xi$ is solution of
the original equation~\eqref{dsch}, up to an error
that lies in the tangent space of the translations
(this error is exactly the price that we pay
in order to solve the corrector equation for~$\psi$
on the orthogonal of the kernel, where the operator is nondegenerate).
As a matter of fact (see Theorem~7.6 in~\cite{DPDV14} for details)
one can see that the corrector~$\psi=\psi_\xi$ is of order~$\epsilon^{n+2s}$.
Therefore, one can compute~$I_\epsilon (u_\xi)
=I_\epsilon (\overline u_\xi+\psi_\xi)$ as a higher order
perturbation of~$I_\epsilon (\overline u_\xi)$.
{F}rom~\eqref{en}, one obtains that
\begin{equation}\label{9HU88I}
I_\epsilon (u_\xi)=
I(w) +\frac{1}{2}\mathcal{H}_\epsilon(\xi)+\mathcal{O}(\epsilon^{n+4s}),\end{equation}
see Theorem~7.17 in~\cite{DPDV14} for details.

Since this energy expansion now depends only on~$\xi\in\R^n$,
it is convenient
to define the operator $J_\epsilon \colon \Omega_\epsilon \to \R$ as
	\[ J_\epsilon(\xi):=I_\epsilon (u_\xi).\]
This functional is often
called the reduced energy functional.
{F}rom~\eqref{9HU88I}, we conclude that
	\begin{equation}\label{JK L}
J_\epsilon(\xi)=
I(w) +\frac{1}{2}\mathcal{H}_\epsilon(\xi)+\mathcal{O}(\epsilon^{n+4s}).\end{equation}
The reduced energy~$J$ plays an important role in this framework
since critical points of~$J$ correspond to true solutions
of the original equation~\eqref{dsch}. More precisely
(see Lemma 7.16 in~\cite{DPDV14}) one has that
$c_i=0$ for all $i=1, \dots, n$ 
in~\eqref{qsci}
if and only if 
	\begin{equation}\label{NEAT}
\frac{\partial J_\epsilon}{\partial \xi}(\xi)=0.\end{equation}
In other words, when $\epsilon$ approaches $0$, to find
concentration points, it is enough to find critical points
of~$J$,
which is a finite-dimensional problem.
Also, critical points for~$J$ come from critical points of~$\mathcal{H}_\epsilon$,
up to higher orders, thanks to~\eqref{JK L}.
The issue is thus to prove that~$\mathcal{H}_\epsilon$
does possess critical points
and that these critical points survive after
the small error of
size~$\epsilon^{n+4s}$: in fact, we show that~$\mathcal{H}_\epsilon$
possesses a minimum,
which is stable for perturbations.
For this, one needs
a bound for the Robin function $H_\epsilon$ from above and below.
To this goal,
one builds a barrier function $\beta_\xi$ defined for $\xi \in \Omega_\epsilon$  and $x \in \Rn$ as
	\[ \beta_\xi (x) := \int_{\Rn \setminus \Omega_\epsilon} \Gamma(z-\xi)\Gamma(x-z) \, dz.\]
Using this function in combination with suitable maximum principles,
one obtains the existence of a constant $c\in (0,1)$ such that
	\[ c H_\epsilon (x,\xi) \leq \beta_\xi(x) \leq c^{-1}H_\epsilon(x,\xi),\]
 for any $x\in \Rn$ and any $\xi \in \Omega_\epsilon$ with $\text{dist} (\xi, \partial \Omega_\epsilon)>1$, see Lemma~2.1
in~\cite{DPDV14}. {F}rom this it follows that  
	\begin{equation}\label{ATT}
\mathcal{H}_\epsilon (\xi) \simeq d^{-(n+4s)},\end{equation}
for all points $\xi \in \Omega_\epsilon$ such that $d \in [5, \delta / \epsilon]$.
So, one considers the domain~$\Omega_{\epsilon,\delta}$ of the points
of~$\Omega_\epsilon$ that lie at distance more than~$\delta/\epsilon$
from the boundary of~$\Omega_\epsilon$. By~\eqref{ATT},
we have that
  \begin{equation}\label{ATT-1}
\mathcal{H}_\epsilon (\xi) \simeq \frac{\epsilon^{n+4s}}{\delta^{n+4s}}
\;{\mbox{ for any }}\;\xi\in \partial \Omega_{\epsilon,\delta}.
\end{equation}
Also, up to a translation, we may suppose that~$0\in\Omega$.
Thus, $0\in\Omega_{\epsilon}$ and its distance from~$\partial\Omega_\epsilon$
is of order~$1/\epsilon$ (independently of~$\delta$).
In particular, if~$\delta$ is small enough, we have that~$0$
lies in the interior of~$ \Omega_{\epsilon,\delta}$,
and~\eqref{ATT} gives that
$$ \mathcal{H}_\epsilon (0) \simeq \epsilon^{n+4s}.$$
By comparing this with~\eqref{ATT-1}, we see that~$\mathcal{H}_\epsilon $
has an interior minimum in~$ \Omega_{\epsilon,\delta}$.
The value attained at this minimum
is of order~$\epsilon^{n+4s}$,
and the values attained at the boundary of~$\Omega_{\epsilon,\delta}$
are of order~$\delta^{-n-4s}\epsilon^{n+4s}$,
which is much larger than~$\epsilon^{n+4s}$, if~$\delta$ is small enough.
This says that the interior minimum of~$\mathcal{H}_\epsilon $
in~$ \Omega_{\epsilon,\delta}$ is nondegenerate and it survives to
any perturbation of order~$\epsilon^{n+4s}$, if~$\delta$ is small enough.

This and~\eqref{JK L} imply that~$J$ has also an interior 
minimum at some point~$\xi$ in~$ \Omega_{\epsilon,\delta}$. By 
construction, this point~$\xi$ satisfies~\eqref{NEAT},
and so this completes the proof of
Theorem~\ref{THSP}.
\end{proof}

The variational argument in the proof above (see in particular~\eqref{NEAT})
has a classical and neat geometric interpretation. Namely, the ``unperturbed'' functional (i.e. the one with~$\epsilon=0$) has a very degenerate geometry, since it has a whole manifold of minimizers with the same energy: this manifold corresponds to the translation of the ground state~$w$, namely it is of the form~$M_0:=\{ w_\xi,\;\xi\in\R^n\}$ and, therefore, it can be identified with~$\R^n$.
\begin{center}
\begin{figure}[htb]
\begin{minipage}[b]{0.65\linewidth}
	\centering
	\includegraphics[width=0.65\textwidth]{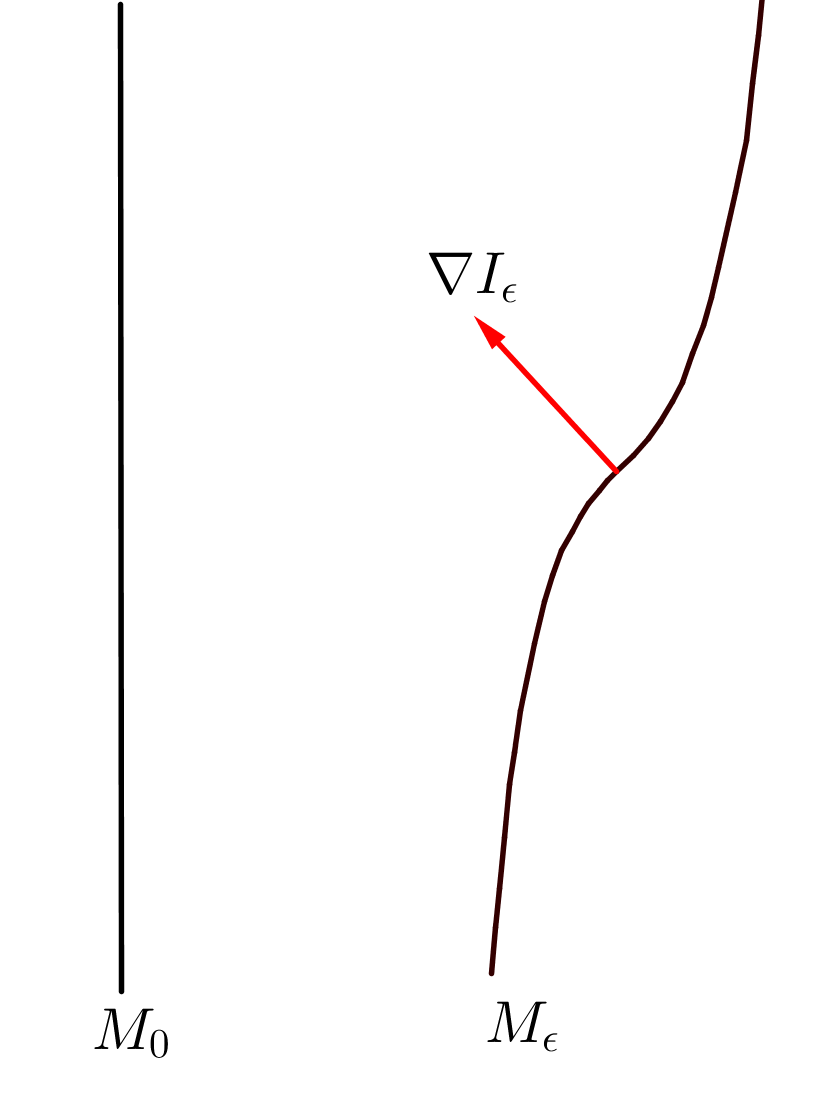}
	\caption{Geometric interpretation}   
	\label{fign:seq}
	\end{minipage}
\end{figure}
\end{center}

For topological arguments, this degenerate picture may constitute a serious obstacle to the existence of critical points for the ``perturbed'' functional (i.e. the one with~$\epsilon\ne0$). As an obvious example, the reader may think of the function of two variables~$f_\epsilon:\R^2\to\R$ given by~$f_\epsilon(x,y):=x^2+\epsilon y$.
When~$\epsilon=0$, this function attains its minimum along the manifold~$\{x=0\}$, but all the critical points on this manifold are ``destroyed'' by the
perturbation when~$\epsilon\ne0$ (indeed~$\nabla f_\epsilon(x,y) =(2x,\epsilon)$ never vanishes).

In the situation described in the proof of Theorem~\ref{nnscheq}, this pathology does not occur, thanks to the nondegeneracy provided in~\cite{FLS13}.
Indeed, by the nondegeneracy of the unperturbed critical manifold, when~$\epsilon\ne0$ one can construct a manifold, diffeomorphic to the original one 
(in our case of the form~$M_\epsilon:=\{ \overline u_\xi+\psi(\xi),\;\xi\in\R^n\}$), that enjoys the special feature of ``almost annihilating'' the gradient of the functional, up to vectors parallel to the original manifold~$M_0$ (this is the meaning of formula \eqref{qsci}).

Then, if one finds a minimum of the functional constrained to~$M_\epsilon$, the theory of Lagrange multipliers (at least in finite dimension) would suggest that the gradient is normal to~$M_\epsilon$. That is, the gradient of the functional is, simultaneously, parallel to~$M_0$ and orthogonal to~$M_\epsilon$. But since~$M_\epsilon$ is diffeomorphically
close to~$M_0$, the only vector with this property is the null vector, hence this argument provides the desired critical point.
\bigskip

We also recall that the fractional Schr\"odinger equation
is related to a nonlocal 
canonical quantization, which in turn produces
a nonlocal Uncertainty Principle.
In the classical setting, one considers the momentum/position
operators, which are defined in~$\R^n$, by
\begin{equation}\label{CQP}
P_k:=-i\hslash\partial_k \ {\mbox{
and }} \
Q_k:=x_k\end{equation}
for~$k\in\{1,\dots,n\}$. Then, the
Uncertainty Principle states that the operators~$P=(P_1,\dots,P_n)$
and~$Q=(Q_1,\dots,Q_n)$ do not commute (which makes
it practically impossible to measure simultaneously both momentum
and position).
Indeed, in this case a simple computation shows that
\begin{equation}\label{PLj89aaLL}
[Q,P]:=\sum_{k=1}^n [Q_k,P_k] = i\hslash n.\end{equation}
The nonlocal analogue of this quantization may be formulated
by introducing a nonlocal momentum, i.e. by replacing the operators
in~\eqref{CQP} by
\begin{equation}\label{CQ}
P_k^s:=-i\hslash^s\partial_k (-\Delta)^{\frac{s-1}2} \ {\mbox{
and }} \
Q_k:=x_k.\end{equation}
In this case
%, using that the Fourier transform 
%of the product is the 
%convolution of the
%Fourier transforms,
one has that
\begin{equation}\label{723}
\begin{split}
%(\widehat x_k * g)(\xi)\,&= 
 {\mathcal{F}}\big(x_k {\mathcal{F}}^{-1} g(x)\big) (\xi)
&= \int_{\R^n} \,dx \int_{\R^n} \,dy \; e^{2\pi ix\cdot (y-\xi)} x_k g(y)
\\
&= \frac1{2\pi i}
\,\int_{\R^n} \,dx \int_{\R^n} \,dy \;
\partial_{y_k} e^{2\pi ix\cdot (y-\xi)} g(y)
= \frac{i}{2\pi}\int_{\R^n} \,dx \int_{\R^n} \,dy \;
e^{2\pi ix\cdot (y-\xi)} \partial_k g(y)\\
&= \frac{i}{2\pi}
\int_{\R^n} \,dx\;
e^{-2\pi ix\cdot \xi} {\mathcal{F}}^{-1}(\partial_k g)(x)
= \frac{i}{2\pi} {\mathcal{F}}\big({\mathcal{F}}^{-1}(\partial_k g)\big)(\xi)
= \frac{i}{2\pi} \partial_k g(\xi),
\end{split}\end{equation}
for any test function~$g$.
In addition,
$${\mathcal{F}}(P_k^s f)= (2\pi)^{s}\hslash^s
\xi_k\,|\xi|^{s-1}\widehat f.$$
Therefore, given any test function~$\psi$, using this with $f:=\psi$
and~$f:=x_k\psi$,
and also~\eqref{723} with $g:={\mathcal{F}}(P_k^s \psi)$
and~$g:=\widehat\psi$,
we obtain that
\begin{equation*}\begin{split}
& {\mathcal{F}} \big(Q_k P_k^s \psi(x)-P_k^s Q_k\psi(x)\big)
= {\mathcal{F}} \big(x_k P_k^s \psi(x)\big) -{\mathcal{F}} \big (P_k^s (x_k\psi(x)) \big)\\
%=& \widehat x_k * {\mathcal{F}}(P_k^s \psi(x))
%-{\mathcal{F}} \big (P_k^s (x_k\psi(x)) \big)
=& \frac{i}{2\pi} \partial_k {\mathcal{F}}( P_k^s \psi)(\xi)
- (2\pi)^{s}\hslash^s
\xi_k\,|\xi|^{{s-1}} {\mathcal{F}}(x_k\psi(x))(\xi)
\\ 
=& (2\pi)^{s-1} i \hslash^s
\partial_k \big( \xi_k\,|\xi|^{{s-1}}\widehat \psi(\xi)\big)
- (2\pi)^{s}\hslash^s
\xi_k\,|\xi|^{{s-1}} \widehat x_k* \widehat\psi(\xi)\\
= &(2\pi)^{s-1} i \hslash^s
\partial_k \big( \xi_k\,|\xi|^{{s-1}}\widehat \psi(\xi)\big)
- (2\pi)^{s-1}i\hslash^s
\xi_k\,|\xi|^{{s-1}}\partial_k\widehat \psi(\xi)
\\ =& (2\pi)^{s-1} i \hslash^s 
\partial_k \big( \xi_k\,|\xi|^{{s-1}}\big)\,\widehat \psi(\xi)
=
(2\pi)^{s-1} i \hslash^s 
\left( |\xi|^{{s-1}} +(s-1)
\xi_k^2\,|\xi|^{s-3}\right)\,\widehat \psi(\xi)
.\end{split}\end{equation*}
Consequently, by summing up,
\[ \mathcal F \big([Q,P^s]\psi)=(2\pi )^{s-1} i \hslash^s\,|\xi|^{s-1}\,\left(n +s-1\right)\, \widehat \psi (\xi).\] 
So, by taking the anti-transform,
\eqlab{ \label{s-COMM-QP}
 [Q,P^s]\psi=\al 
i \hslash^s\,\left(
n +{s-1}
\right)\,
{\mathcal{F}}^{-1}
\big( (2\pi |\xi|)^{{s-1}} \widehat\psi\big)\\=\al 
i\hslash^s 
\,(n+s-1)\,
(-\Delta)^{\frac{s-1}2}\psi.}
Notice that, as $s\rightarrow1$, this formula reduces to the
the classical Heisenberg Uncertainty Principle in~\eqref{PLj89aaLL}.
\subsection{From the nonlocal Uncertainty Principle
to a fractional weighted inequality}

Now we point out a simple consequence of
the Uncertainty Principle
in formula~\eqref{s-COMM-QP}, which can be seen as
a fractional Sobolev inequality
in weighted spaces. The result
(which boils down to known formulas as~$s\to1$)
is the following:

\begin{prop}
For any~$u\in \mathcal{S}(\R^n)$, we have that
\begin{equation*}
\Big\| \, (-\Delta)^{\frac{s-1}{4}} u \, \Big\|_{L^2(\R^n)}^2
\le \frac{2}{n+s-1} \,
\Big\|\, |x|u \, \Big\|_{L^2(\R^n)}
\,\Big\| \,\nabla (-\Delta)^{\frac{s-1}{2}} u\,
\Big\|_{L^2(\R^n)}.\end{equation*}
\end{prop}

\begin{proof}
The proof is a general argument in operator theory.
Indeed, suppose that there are two operators~$S$ and~$A$,
acting on a space with a scalar
Hermitian product. Assume that~$S$ is self-adjoint
and~$A$ is anti-self-adjoint, i.e.
$$ \langle u,Su\rangle=\langle Su,u\rangle
\;{\mbox{ and }}\;
\langle u,Au\rangle=-\langle Au,u\rangle,$$
for any~$u$ in the space.
Then, for any~$\lambda\in\R$,
\begin{eqnarray*}
\| (A+\lambda S)u\|^2&=&
\|Au\|^2 + \lambda^2 \|Su\|^2 + \lambda\Big(
\langle Au,Su\rangle +\langle Su,Au\rangle\Big)\\
&=&
\|Au\|^2 + \lambda^2 \|Su\|^2 + \lambda
\langle (SA-AS)u,u\rangle .
\end{eqnarray*}
Now we apply this identity in the space~$C^\infty_0(\R^n)\subset L^2(\R^n)$,
taking~$S:=Q_k=x_k$ and~$A:=i P_k^s=\hslash^s\partial_k (-\Delta)^{\frac{s-1}2}$
(recall~\eqref{CQ}
and notice that~$iP_k^s$ is anti-self-adjoint,
thanks to the integration by parts formula).
In this way, and using~\eqref{s-COMM-QP}, we obtain that
\begin{eqnarray*}
0&\le&
\sum_{k=1}^n
\| (iP_k^s+\lambda Q_k)u\|^2_{L^2(\R^n)}\\
&=&
\sum_{k=1}^n \left[
\|P_k^s u\|^2_{L^2(\R^n)}
+ \lambda^2 \|Q_k u\|^2_{L^2(\R^n)} 
+ i\lambda \langle [Q_k ,P_k^s] u,u\rangle_{L^2(\R^n)}\right]
\\ &=&
\hslash^{2s} \Big\| \,\nabla (-\Delta)^{\frac{s-1}{2}} u\,\Big\|_{L^2(\R^n)}^2
+\lambda^2 \Big\|\, |x|u \, \Big\|_{L^2(\R^n)}^2 \\ 
&&+ i^2 \,\lambda\,(n + s - 1)\,\hslash^{s}
\langle (-\Delta)^{\frac{s-1}{2}} u,u\rangle_{L^2(\R^n)}
\\ &=&
\hslash^{2s} \Big\| \,\nabla (-\Delta)^{\frac{s-1}{2}} u\,\Big\|_{L^2(\R^n)}^2
+\lambda^2 \Big\|\, |x|u \, \Big\|_{L^2(\R^n)}^2\\
&&-\lambda\,(n + s -1)\,\hslash^{s}
\Big\| \, (-\Delta)^{\frac{s-1}{4}} u \, \Big\|_{L^2(\R^n)}^2.
\end{eqnarray*}
Now, if~$u\not\equiv0$,
we can optimize this identity by choosing
$$\lambda:=\frac{ 
(n + s -1)\,\hslash^{s}
\Big\| \, (-\Delta)^{\frac{s-1}{4}} u \, \Big\|_{L^2(\R^n)}^2}{2\,
\Big\|\, |x|u \, \Big\|_{L^2(\R^n)}^2 }$$
and we obtain that
$$ 0\le
\hslash^{2s} \Big\| \,\nabla (-\Delta)^{\frac{s-1}{2}} u\,\Big\|_{L^2(\R^n)}^2
- \frac{(n + s -1)^2\,\hslash^{2s}
\Big\| \, (-\Delta)^{\frac{s-1}{4}} u \, \Big\|_{L^2(\R^n)}^4}{4\,
\Big\|\, |x|u \, \Big\|_{L^2(\R^n)}^2},$$
which gives the desired result.
\end{proof}

\begin{appendix}
\chapter{Alternative proofs of some results} 
\section{Another proof of Theorem \ref{G1}}\label{1.APP-APP-TH}

Here we present a different proof of Theorem \ref{G1},
based on the Fourier transforms of homogeneous distributions.
This proof is the outcome of a pleasant discussion with
Alexander Nazarov.

\begin{proof}[Alternative proof of Theorem \ref{G1}]

We are going to use the Fourier transform of~$|x|^q$
in the sense of distribution, with~$q\in\Co\setminus\Z$.
Namely (see e.g. 
%%% formula~(2.40) 
Lemma~2.23 on page~38 of~\cite{KOLDO})
\begin{equation}\label{KSp}
{\mathcal{F}} (|x|^q) =
C_q\,|\xi|^{-1-q},
\end{equation}
with
\begin{equation}\label{J5K:1} C_q:=-2(2\pi)^{-q-1}\Gamma(1+q)\,\sin\frac{\pi q}{2}.
\end{equation}
We remark that the original value of the constant $C_q$ in~\cite{KOLDO} is here multiplied by a $(2\pi)^{-q-1}$ term, in order to be consistent with the Fourier normalization that we have introduced.
We observe that the function~$\R\ni x\mapsto |x|^q$  is locally integrable only when~$q>-1$, so it naturally induces a distribution only in this range of the parameter~$q$ (and, similarly, the function~$\R\ni\xi\mapsto |\xi|^{-1-q}$ is locally integrable only when~$q<0$): therefore, to make sense of the formulas above in a wider range of parameters~$q$
it is necessary to use analytic continuation and a special procedure that is called regularization:
see e.g. page~36 in~\cite{KOLDO}
(as a matter of fact, we will do a procedure of this type
explicitly in a particular case in~\eqref{pOL8901}).
Since~$\R\ni x\mapsto|x|^q$
is even, we can write~\eqref{KSp}
also as
\begin{equation}\label{KSp2}
{\mathcal{F}}^{-1} (|\xi|^q) =
C_q\,|x|^{-1-q}.
\end{equation}
We observe that, by elementary trigonometry,
$$ \sin\frac{\pi (s+1)}{2}=
-\sin\frac{\pi (s-1)}{2}
\;{\mbox{ and }}\;
\sin\frac{\pi (s-2)}{2}=\sin\frac{\pi s}{2}.$$
Moreover,
$$ \Gamma(2+s)=(1+s) \Gamma(1+s)
\;{\mbox{ and }}\;
\Gamma(s)=(s-1) \Gamma(s-1).$$
Hence, by~\eqref{J5K:1},
\begin{equation}\label{gJK:100}
\begin{split}
& \frac{1-s}{1+s} \cdot 
C_{s+1}\,C_{s-2}
= \frac{1-s}{1+s} \cdot 
4(2\pi)^{-2s-1} \Gamma(2+s)\,\Gamma(s-1)\,\sin\frac{\pi (s+1)}{2}
\,\sin\frac{\pi (s-2)}{2}
\\ &\qquad = 
4\Gamma(1+s)\,\Gamma(s)\,\sin\frac{\pi (s-1)}{2}
\,\sin\frac{\pi s}{2}
= 
C_s\,C_{s-1}.
\end{split}\end{equation}
Moreover,
$$ |x|^s +\frac{1}{s+1} \partial_x |x|^{s+1} = 2x_+^s.$$
So, taking the Fourier transform and using~\eqref{KSp} with~$q:=s$
and~$q:=s+1$, we obtain that
\begin{eqnarray*}
2{\mathcal{F}}(x_+^s) &=&
{\mathcal{F}}(|x|^s) +\frac{1}{s+1} {\mathcal{F}}\big(\partial_x |x|^{s+1}\big)\\
&=&
{\mathcal{F}}(|x|^s) +\frac{2\pi i\xi}{s+1} {\mathcal{F}}(|x|^{s+1})
\\ &=& C_s\,|\xi|^{-1-s}+\frac{2\pi i\xi}{s+1} C_{s+1}\,|\xi|^{-2-s}.
\end{eqnarray*}
As a consequence,
$$ 2 |\xi|^{2s} {\mathcal{F}}(x_+^s)
= C_s\,|\xi|^{-1+s}+\frac{2\pi i\xi}{s+1} C_{s+1}\,|\xi|^{-2+s}.$$
Hence, recalling~\eqref{723},
\begin{eqnarray*}
2 {\mathcal{F}}^{-1}\Big( |\xi|^{2s} {\mathcal{F}}(x_+^s)\Big)
&=&
C_s\,{\mathcal{F}}^{-1}(|\xi|^{-1+s})
+\frac{2\pi C_{s+1}\,i}{s+1} 
{\mathcal{F}}^{-1}(\xi) *
{\mathcal{F}}^{-1}(|\xi|^{-2+s})
\\ &=&C_s\,{\mathcal{F}}^{-1}(|\xi|^{-1+s})
-\frac{2\pi C_{s+1}\,i}{s+1}\cdot\frac{i}{2\pi} \partial_x
{\mathcal{F}}^{-1}(|\xi|^{-2+s})
\\ &=&C_s\,{\mathcal{F}}^{-1}(|\xi|^{-1+s})
+\frac{C_{s+1}}{s+1}
\partial_x
{\mathcal{F}}^{-1}(|\xi|^{-2+s}) 
.\end{eqnarray*}
Accordingly, exploiting~\eqref{KSp2} with~$q:=-1+s$
and~$q:=-2+s$,
\begin{eqnarray*}
2 {\mathcal{F}}^{-1}\Big( |\xi|^{2s} {\mathcal{F}}(x_+^s)\Big)
&=&
C_s\,C_{s-1}\,|x|^{-s}
+\frac{C_{s+1}\,C_{s-2}}{s+1}
\partial_x |x|^{1-s}\\
&=&C_s\,C_{s-1}\,|x|^{-s}
+\frac{1-s}{1+s}\cdot C_{s+1}\,C_{s-2} \,x\,|x|^{-1-s}
.\end{eqnarray*}
So, recalling~\eqref{gJK:100},
$$ 2 {\mathcal{F}}^{-1}\Big( |\xi|^{2s} {\mathcal{F}}(x_+^s)\Big)
= C_s\,C_{s-1}\,\Big( |x|^{-s} - x\,|x|^{-1-s}\Big).$$
This and \eqref{frlaphdef} give that
$$ (-\Delta)^s (x_+^s) = \tilde C_s 
\,\Big( |x|^{-s} - x\,|x|^{-1-s}\Big),$$
for some~$\tilde C_s$. In particular, the quantity above
vanishes when~$x>0$, thus providing an alternative proof
of Theorem~\ref{G1}.
\end{proof}

Yet another proof of Theorem \ref{G1}
can be obtained in a rather short, but technically quite advanced
way, using the
Paley-Wiener theory in the distributional setting.
The sketch of this proof goes as follows:

\begin{proof}[Alternative proof of Theorem \ref{G1}]
The function~$h:=x_+^s$ is homogeneous of degree~$s$.
Therefore its (distributional) Fourier
transform~${\mathcal{F}} h$ is homogeneous of
degree~$-1-s$ (see Lemma~2.21 in~\cite{KOLDO}).

Moreover, $h$ is supported in~$\{x\ge0\}$, therefore~${\mathcal{F}} h$
can be continued to an analytic function
in~$\Co_-:=\{ z\in \Co {\mbox{ s.t. }} \Im z<0\}$
(see Theorem~2 in~\cite{PALEY_DIS}).

Therefore, $g(\xi):=|\xi|^s {\mathcal{F}}h(\xi)$
is homogeneous of degree~$-1$ and is the trace of a function that is analytic
in~$\Co_-$. 

In particular, for any~$y<0$,
we have that
$$ g(-iy)= \frac{g(-i)}{y}=\frac{c}{-i y},$$ where~$c:=-i g(-i)$.
That is, $g(z)$ coincides with~$\frac{c}{z}$ on a half-line,
and then in the whole of~$\Co_-$,
by analytic continuation.

That is, in the sense of distributions,
$$ g(\xi)= \frac{c}{\xi-i0},$$
for any~$\xi\in\R$.

Now we recall the Sokhotski Formula (see e.g.~(3.10)
in~\cite{BLANC}), according to which
$$ \frac{1}{\xi\pm i0} =\mp i\pi\delta + \mbox{P.V.}\frac{1}{\xi},$$
where $\delta$ is the Dirac's Delta and the identity holds
in the sense of distributions. 
By considering this equation with the two sign choices and
then summing up, we obtain that
$$ \frac{1}{\xi+ i0} -\frac{1}{\xi- i0} =
- 2i\pi\delta.$$
As a consequence
$$ g(\xi)= \frac{c}{\xi-i0}=
\frac{c}{\xi+ i0} +2ic\pi\delta.$$
Therefore
$$ |\xi|^{2s} {\mathcal{F}}h(\xi)=
|\xi|^s g(\xi)=  \frac{c\,|\xi|^s}{\xi+ i0}+2ic\pi|\xi|^s \delta.$$
Of course, as a distribution, $|\xi|^s \delta=0$,
since the evaluation of~$|\xi|^s$ at~$\xi=0$ vanishes,
therefore we can write
$$ \ell(\xi):=|\xi|^{2s} {\mathcal{F}}h(\xi)=  \frac{c\,|\xi|^s}{\xi+ i0}.$$
Accordingly, $\ell$
is homogeneous of degree~$-1+s$ and it
is the trace of a function analytic in~$\Co_+
:=\{ z\in \Co {\mbox{ s.t. }} \Im z>0\}$.

Consequently, ${\mathcal{F}}^{-1}\ell$
is homogeneous of degree $-s$
(see again Lemma~2.21 in~\cite{KOLDO}),
and it is supported in~$\{x\le0\}$
(see again Theorem~2 in~\cite{PALEY_DIS}).
Since~$(-\Delta)^s x_+^s$ coincides (up to multiplicative constants)
with~${\mathcal{F}}^{-1}\ell$, we have just shown that
$(-\Delta)^s x_+^s =c_o x_-^{-s}$, for some~$c_o\in\R$,
and so in particular~$(-\Delta)^s x_+^s $ vanishes in~$\{x>0\}$.
\end{proof}

\section{Another proof of Lemma~\ref{LM-COST}}\label{APP-APP}

For completeness, we provide here an alternative proof of
Lemma~\ref{LM-COST} that does not use the theory of the fractional
Laplacian.

\begin{proof}[Alternative proof of Lemma~\ref{LM-COST}]
We first recall some basic properties of the modified Bessel functions
(see e.g.~\cite{ABRAMOWITZ}).
First of all (see formula~9.1.10 on page~360 of~\cite{ABRAMOWITZ}),
we have that
$$ J_s(z) := \frac{z^s}{2^s}
\sum_{k=0}^{+\infty} \frac{(-1)^k z^{2k} }{2^{2k} \,k!\,\Gamma(s+k+1)}
= \frac{z^{s} }{2^{s} \,\Gamma(1+s)} + \mathcal{O}(|z|^{2+s})$$
as~$|z|\to0$. Therefore
(see formula~9.6.3 on page~375 of~\cite{ABRAMOWITZ}),
\begin{eqnarray*}
&& I_s(z):= e^{-\frac{i\pi s}{2}} J_s (e^{\frac{i\pi}{2}} z)
\\ &&\qquad=e^{-\frac{i\pi s}{2}} \Big(
\frac{ e^{\frac{i\pi s}{2}} z^s }{2^{s} \,\Gamma(1+s)} + 
\mathcal{O}(|z|^{2+s}) \Big) \\
&&\qquad=
\frac{ z^s }{2^{s} \,\Gamma(1+s)} + 
\mathcal{O}(|z|^{2+s}) ,\end{eqnarray*}
as~$|z|\to0$.
Using this and formula~9.6.2 on page~375 of~\cite{ABRAMOWITZ},
\begin{eqnarray*} K_s(z)&:=&\frac{\pi}{2\sin(\pi s)}\Big(
I_{-s}(z)-I_s(z)\Big) \\
&=&
\frac{\pi}{2\sin(\pi s)}\left(
\frac{ z^{-s} }{2^{-s} \,\Gamma(1-s)}
-\frac{ z^s }{2^{s} \,\Gamma(1+s)}
+\mathcal{O}(|z|^{2-s})
\right).
\end{eqnarray*}
Thus, recalling Euler's reflection formula
$$ \Gamma(1-s) \,\Gamma(s) = \frac{\pi}{\sin(\pi s)},$$
and the relation~$\Gamma(1+s)=s\Gamma(s)$,
we obtain
\begin{eqnarray*}
K_s(z) &=&
\frac{\Gamma(1-s) \,\Gamma(s)}{2}\left(
\frac{ z^{-s} }{2^{-s} \,\Gamma(1-s)}
-\frac{ z^s }{2^{s} \,\Gamma(1+s)}
+\mathcal{O}(|z|^{2-s})
\right) \\ &=&
\frac{\Gamma(s)\, z^{-s}}{2^{1-s}}
-
\frac{\Gamma(1-s)\,z^s}{2^{1+s}s}+\mathcal{O}(|z|^{2-s}),\end{eqnarray*}
as~$|z|\to0$. We use this and formula~(3.100) in~\cite{MATHAR}
(or page~6 in~\cite{Oberhettinger}) and get that, for any small~$a>0$,
\begin{equation}\label{789OO-P}
\begin{split}
&\int_0^{+\infty} \frac{\cos(2\pi t)}{(t^2+a^2)^{s+\frac12}}\,dt=
\frac{\pi^{s+\frac12}}{ a^{s}\,\Gamma\left(s+\frac12\right) } K_s(2\pi a)
\\ &\qquad=
\frac{\pi^{s+\frac12}}{ a^{s}\,\Gamma\left(s+\frac12\right) }
\left[
\frac{\Gamma(s)}{2 \pi^s a^s}
-
\frac{\Gamma(1-s)\, \pi^s a^s}{2s}+\mathcal{O}(a^{2-s})
\right] \\
&\qquad=
\frac{\pi^{\frac12} \Gamma(s) }{ 2 a^{2s}\,\Gamma\left(s+\frac12\right) }
-
\frac{\pi^{2s+\frac12}\,\Gamma(1-s)}{ 2s\Gamma\left(s+\frac12\right) }
+\mathcal{O}(a^{2-2s}).
\end{split}
\end{equation}
Now we recall the generalized hypergeometric functions~${}_m F_n$
(see e.g. page~211 in~\cite{Oberhettinger}): as a matter of
fact, we just need that for any~$b$, $c$, $d>0$,
$$ {}_1 F_2 (b;c,d;0) = \frac{\Gamma(c)\Gamma(d)}{\Gamma(b)}\cdot
\frac{\Gamma(b)}{\Gamma(c)\Gamma(d)}=1.$$
We also recall the Beta function relation
\begin{equation} \label{betafunction} B(\alpha,\beta)=\frac{\Gamma(\alpha)\,\Gamma(\beta)}{
\Gamma(\alpha+\beta)},\end{equation}
see e.g. formula~6.2.2 in~\cite{ABRAMOWITZ}.
Therefore using
formula~(3.101) in~\cite{MATHAR}
(here with~$y:=0$, $\nu:=0$ and~$\mu:=s-\frac12$,
or see page~10 in~\cite{Oberhettinger}),
\begin{eqnarray*}
\int_0^{+\infty} \frac{dt}{(a^2+t^2)^{s+\frac12}}
&=& \frac{a^{-2s}}{2} B\left(\frac{1}{2},s\right)
{}_1 F_2\left( \frac{1}{2}; 1-s,\frac12;0\right)
\\ &=&
\frac{\Gamma\left(\frac{1}{2}\right)\,\Gamma(s)}{
2a^{2s}\Gamma\left(\frac{1}{2}+s\right)}.
\end{eqnarray*}
Then, we recall that~$\Gamma\left(\frac{1}{2}\right)=\pi^{\frac12}$,
so, making use of~\eqref{789OO-P}, for any small~$a>0$,
$$ \int_0^{+\infty} \frac{1-\cos(2\pi t)}{(t^2+a^2)^{s+\frac12}}\,dt
=
\frac{\pi^{2s+\frac12}\,\Gamma(1-s)}{ 2s\Gamma\left(s+\frac12\right) }
+\mathcal{O}(a^{2-2s}).$$
Therefore, sending~$a\to0$ by the Dominated Convergence Theorem
we obtain
\begin{equation}\label{DOMINATED}
\int_0^{+\infty} \frac{1-\cos(2\pi t)}{t^{1+2s}}\,dt
=\lim_{a\to0}
\int_0^{+\infty} \frac{1-\cos(2\pi t)}{(t^2+a^2)^{s+\frac12}}\,dt
=
\frac{\pi^{2s+\frac12}\,\Gamma(1-s)}{ 2s\Gamma\left(s+\frac12\right) }
.\end{equation}
Now we recall the integral representation of the
Beta function 
(see e.g. formulas~6.2.1 and 6.2.2 in~\cite{ABRAMOWITZ}), namely
\begin{equation}\label{89TTUHg}
\frac{\Gamma\left( \frac{n-1}{2}\right)\,\Gamma\left( \frac12+s\right)}{
\Gamma\left(\frac{n}{2}+s\right)}=
B\left( \frac{n-1}{2},\,\frac12+s\right)
=\int_0^{+\infty} \frac{\tau^{\frac{n-3}{2}}}{(1+\tau)^{ \frac{n}{2}+s } }
\,d\tau.\end{equation}
We also observe that in any dimension~$N$
the $(N-1)$-dimensional
measure of the unit sphere 
is~$\frac{N \pi^{\frac{N}{2}}}{\Gamma\left( \frac{N}{2}+1\right)}$,
(see e.g.~\cite{HUBER}).
Therefore
$$ \int_{\R^N} \frac{dY}{(1+|Y|^2)^{\frac{N+1+2s}{2}}}=
\frac{N \pi^{\frac{N}{2}}}{\Gamma\left( \frac{N}{2}+1\right)}
\int_0^{+\infty} \frac{\rho^{N-1}}{
(1+\rho^2)^{\frac{N+1+2s}{2}} }\,d\rho.$$
In particular, taking~$N:=n-1$ and using the change of
variable~$\rho^2=:\tau$,
\begin{eqnarray*}
&& \int_{\R^{n-1}} \frac{d\eta}{(1+|\eta|^2)^{\frac{n+2s}{2}}}
=
\frac{(n-1)\, \pi^{\frac{n-1}{2}}}{\Gamma\left( \frac{n-1}{2}+1\right)}
\int_0^{+\infty} \frac{\rho^{n-2}}{ 
(1+\rho^2)^{\frac{n+2s}{2}} }\,d\rho \\
&&\qquad=
\frac{(n-1)\, \pi^{\frac{n-1}{2}}}{2\,\Gamma\left( \frac{n-1}{2}+1\right)}
\int_0^{+\infty} \frac{\tau^{\frac{n-3}{2}} }{ 
(1+\tau)^{\frac{n+2s}{2}} }\,d\tau.
\end{eqnarray*}
By comparing this with~\eqref{89TTUHg}, we conclude that
\begin{eqnarray*}
\int_{\R^{n-1}} \frac{d\eta}{(1+|\eta|^2)^{\frac{n+2s}{2}}}&=&
\frac{(n-1)\, \pi^{\frac{n-1}{2}}}{2\,\Gamma\left( \frac{n-1}{2}+1\right)}
\cdot
\frac{\Gamma\left( \frac{n-1}{2}\right)\,\Gamma\left( \frac12+s\right)}{
\Gamma\left(\frac{n}{2}+s\right)}\\
&=& 
\frac{\pi^{\frac{n-1}{2}}
\,\Gamma\left( \frac12+s\right)
}{\Gamma\left(\frac{n}{2}+s\right)}.
\end{eqnarray*}
Accordingly, with the change of variable~$\eta:= |\omega_1|^{-1}
(\omega_2,\dots,\omega_n)$,
\begin{eqnarray*}
&& \int_{\R^n} \frac{1-\cos(2\pi\omega_1)}{|\omega|^{n+2s}}\,d\omega
\\&&\qquad= \int_\R \left( \int_{\R^{n-1}} 
\frac{1-\cos(2\pi\omega_1)}{ (\omega_1^2+\omega_2^2+\dots+
\omega_n^2)^{\frac{n+2s}{2}}
}\,d\omega_2\,\dots\,d\omega_n
\right)\,d\omega_1\\
&&\qquad =
\int_\R \left( \int_{\R^{n-1}}
\frac{1-\cos(2\pi\omega_1)}{ |\omega_1|^{1+2s}
(1+|\eta|^2)^{\frac{n+2s}{2}}
}\,d\eta
\right)\,d\omega_1
\\ &&\qquad =
\frac{\pi^{\frac{n-1}{2}}
\,\Gamma\left( \frac12+s\right)
}{\Gamma\left(\frac{n}{2}+s\right)}
\,
\int_\R 
\frac{1-\cos(2\pi\omega_1)}{ |\omega_1|^{1+2s} }
\,d\omega_1\\
&&\qquad =
\frac{2\pi^{\frac{n-1}{2}}
\,\Gamma\left( \frac12+s\right)
}{\Gamma\left(\frac{n}{2}+s\right)}
\,
\int_0^{+\infty}
\frac{1-\cos(2\pi t)}{ t^{1+2s} }
\,dt.
\end{eqnarray*}
Hence, recalling~\eqref{DOMINATED},
\begin{eqnarray*}
\int_{\R^n} \frac{1-\cos(2\pi\omega_1)}{|\omega|^{n+2s}}\,d\omega
&=&
\frac{2\pi^{\frac{n-1}{2}}
\,\Gamma\left( \frac12+s\right)
}{\Gamma\left(\frac{n}{2}+s\right)}
\cdot
\frac{\pi^{2s+\frac12}\,\Gamma(1-s)}{ 2s\Gamma\left(s+\frac12\right) }\\
&=&
\frac{ \pi^{\frac{n}{2}+2s} \,\Gamma(1-s)}{s\,\Gamma\left(\frac{n}{2}+s\right)}
,\end{eqnarray*}
as desired.
\end{proof}

To complete the picture, we also give a different
proof of~\eqref{DOMINATED} which does not use the theory
of special functions, but the Fourier transform in the sense of
distributions and the unique analytic continuation:

\begin{proof}[Alternative proof of \eqref{DOMINATED}]
We recall (see\footnote{To check~\eqref{TSUM01}
the reader should note that the normalization of the Fourier
transform on page~138 in~\cite{DONOGHUE} is different than
the one here in~\eqref{transF}.}
e.g. pages~156--157 in~\cite{DONOGHUE}) that
for any~$\lambda\in\Co$ with~$\Re\lambda\in(0,1/2)$ we have
$$ {\mathcal{F}} \big( |x|^{\lambda-1} \big)= \frac{\pi^{\frac12-\lambda}
\Gamma\left(\frac{\lambda}{2}\right)}{ \Gamma\left(\frac{1-\lambda}{2}\right)\,|x|^\lambda}$$
in the sense of distribution, that is
\begin{equation*}
\int_\R |x|^{\lambda-1}\widehat\phi(x)\,dx=
\frac{\pi^{\frac12-\lambda}
\Gamma\left(\frac{\lambda}{2}\right)}{ \Gamma\left(\frac{1-\lambda}{2}\right)}
\int_\R |x|^{-\lambda}\phi(x)\,dx\end{equation*}
for every~$\phi\in C^\infty_0(\R)$. The same result is obtained in \cite{B15}, in the proof of Proposition 2.4 b), pages 13--16.

As a consequence
\begin{equation}\label{TSUM01}
\int_\R \left(\int_\R |x|^{\lambda-1} \phi(y) \,e^{-2\pi i xy} \,dy
\right)\,dx=
\frac{\pi^{\frac12-\lambda}
\Gamma\left(\frac{\lambda}{2}\right)}{ \Gamma\left(\frac{1-\lambda}{2}\right)}
\int_\R |x|^{-\lambda}\phi(x)\,dx.\end{equation}
By changing variable~$x\mapsto-x$ in the first integral,
we also see that
\begin{equation}\label{TSUM02}
\int_\R \left(\int_\R |x|^{\lambda-1} \phi(y) \,e^{2\pi i xy} \,dy
\right)\,dx=
\frac{\pi^{\frac12-\lambda}
\Gamma\left(\frac{\lambda}{2}\right)}{ \Gamma\left(\frac{1-\lambda}{2}\right)}
\int_\R |x|^{-\lambda}\phi(x)\,dx.\end{equation}
By summing together~\eqref{TSUM01} and~\eqref{TSUM02},
we obtain
\begin{equation*}
\int_\R \left(\int_\R |x|^{\lambda-1} \phi(y) \,\cos(2\pi xy) \,dy
\right)\,dx=
\frac{\pi^{\frac12-\lambda}
\Gamma\left(\frac{\lambda}{2}\right)}{ \Gamma\left(\frac{1-\lambda}{2}\right)}
\int_\R |x|^{-\lambda}\phi(x)\,dx.\end{equation*}
It is convenient to exchange the names of
the integration variables in the first integral above: hence we write
\begin{equation*}
\int_\R \left(\int_\R |y|^{\lambda-1} \phi(x) \,\cos(2\pi xy) \,dx
\right)\,dy=
\frac{\pi^{\frac12-\lambda}
\Gamma\left(\frac{\lambda}{2}\right)}{ \Gamma\left(\frac{1-\lambda}{2}\right)}
\int_\R |x|^{-\lambda}\phi(x)\,dx.\end{equation*}
Now we fix~$R>0$, we point out that
$$ \int_\R |y|^{\lambda-1}\chi_{(-R,R)}(y)\,dy=
2\int_{0}^R y^{\lambda-1}\,dy = \frac{2R^\lambda}{\lambda},$$
and we obtain that
\begin{equation}\label{pOL8901}\begin{split}&
\int_\R \left(\int_\R |y|^{\lambda-1} \phi(x) \,
\big( \cos(2\pi xy)-\chi_{(-R,R)}(y)\big)\,dx
\right)\,dy\\ &\qquad=
\frac{\pi^{\frac12-\lambda}
\Gamma\left(\frac{\lambda}{2}\right)}{ \Gamma\left(\frac{1-\lambda}{2}\right)}
\int_\R |x|^{-\lambda}\phi(x)\,dx
- \frac{2R^\lambda}{\lambda}\int_\R \phi(x) \,dx.\end{split}\end{equation}
This formula holds true, in principle, for~$\lambda\in\Co$,
with~$\Re\lambda\in(0,1/2)$, but by the uniqueness
of the analytic continuation in the variable~$\lambda$
it holds also for~$\lambda\in(-2,0)$.

Now we observe that
the map
$$ (x,y)\mapsto |y|^{\lambda-1} \phi(x) \,
\big( \cos(2\pi xy)-\chi_{(-R,R)}(y)\big)$$
belongs to~$L^1(\R\times\R)$ if~$\lambda\in(-2,0)$, hence we can use Fubini's Theorem
and exchange the order of the repeated integrals
in~\eqref{pOL8901}: in this way,
we deduce that
\begin{equation*}\begin{split}&
\int_\R \left(\int_\R |y|^{\lambda-1} \phi(x) \,
\big( \cos(2\pi xy)-\chi_{(-R,R)}(y)\big)\,dy
\right)\,dx\\&\qquad=
\frac{\pi^{\frac12-\lambda}
\Gamma\left(\frac{\lambda}{2}\right)}{ \Gamma\left(\frac{1-\lambda}{2}\right)}
\int_\R |x|^{-\lambda}\phi(x)\,dx
- \frac{2R^\lambda}{\lambda}\int_\R \phi(x) \,dx.\end{split}\end{equation*}
Since this is valid for every~$\phi\in C^\infty_0(\R)$,
we conclude that
\begin{equation}\label{PO0u33}
\int_\R |y|^{\lambda-1}\,
\big( \cos(2\pi xy)-\chi_{(-R,R)}(y)\big)\,dy =
\frac{\pi^{\frac12-\lambda}
\Gamma(\frac{\lambda}{2})}{ \Gamma(\frac{1-\lambda}{2})}
|x|^{-\lambda}
- \frac{2R^\lambda}{\lambda},\end{equation}
for any~$\lambda\in(-2,0)$ and~$x\ne0$.
In this setting, we also have that
$$ \int_\R |y|^{\lambda-1}\,
\big( \chi_{(-R,R)}(y)-1\big)\,dy=
-2\int_R^{+\infty} |y|^{\lambda-1}\,dy=\frac{2R^\lambda}{\lambda}.$$
By summing this with~\eqref{PO0u33}, we obtain
\begin{equation*}
\int_\R |y|^{\lambda-1}\,
\big( \cos(2\pi xy)-1\big)\,dy =
\frac{\pi^{\frac12-\lambda}
\Gamma(\frac{\lambda}{2})}{ \Gamma(\frac{1-\lambda}{2})}
|x|^{-\lambda}
.\end{equation*}
Hence, we take~$\lambda:=-2s\in(-2,0)$, 
and we obtain that
\begin{equation*}
\int_\R \frac{\cos(2\pi xy)-1}{
|y|^{1+2s} } \,dy =
\frac{\pi^{2s+\frac12}
\Gamma(-s)}{ \Gamma\left(\frac{1}{2}+s\right)}
|x|^{2s}.\end{equation*}
By changing variable
of integration~$t:=xy$, we obtain~\eqref{DOMINATED}.
\end{proof}
\end{appendix}

%\addcontentsline{toc}{chapter}{Bibliography}
\bibliography{biblio}
\bibliographystyle{plain}

\end{document}